\documentclass[12pt,twoside]{article}
\usepackage[utf8]{inputenc}
\usepackage[T1]{fontenc}
 
\usepackage[xcdraw]{xcolor}

\usepackage{amsmath,amssymb,amsthm,mathabx} % For math equations, theorems, symbols, etc
\usepackage{xr-hyper}
\usepackage{xr}
\usepackage[final]{hyperref}
\usepackage{cleveref}
\usepackage{stackengine}

\usepackage[nottoc]{tocbibind}

\usepackage[a4paper,left=15mm,right=15mm,top=20mm,bottom=15mm]{geometry}

\usepackage[all]{xy}

\usepackage{fancyhdr}
\usepackage{refcheck}
\nocitenames 
\norefnames
\ignoreunlbld 

\makeatletter
\newcommand{\refcheckize}[1]{%
	\expandafter\let\csname @@\string#1\endcsname#1%
	\expandafter\DeclareRobustCommand\csname relax\string#1\endcsname[1]{%
		\csname @@\string#1\endcsname{##1}\wrtusdrf{##1}}%
	\expandafter\let\expandafter#1\csname relax\string#1\endcsname
}
\makeatother
\refcheckize{\cref}
\refcheckize{\Cref}

\makeatletter
\def\widebreve{\mathpalette\wide@breve}
\def\wide@breve#1#2{\sbox\z@{$#1#2$}%
	\mathop{\vbox{\m@th\ialign{##\crcr
				\kern0.08em\brevefill#1{0.8\wd\z@}\crcr\noalign{\nointerlineskip}%
				$\hss#1#2\hss$\crcr}}}\limits}
\def\brevefill#1#2{$\m@th\sbox\tw@{$#1($}%
	\hss\resizebox{#2}{\wd\tw@}{\rotatebox[origin=c]{90}{\upshape(}}\hss$}
\makeatletter

\usepackage{seqsplit}
\usepackage{xstring}

\definecolor{mycolor}{rgb}{0.122, 0.435, 0.698}

\hypersetup{
	colorlinks,
	citecolor=mblue,
	filecolor=mblue,
	linkcolor=mblue,
	urlcolor=mblue
}
\makeatletter
\DeclareOldFontCommand{\rm}{\normalfont\rmfamily}{\mathrm}
\DeclareOldFontCommand{\sf}{\normalfont\sffamily}{\mathsf}
\DeclareOldFontCommand{\tt}{\normalfont\ttfamily}{\mathtt}
\DeclareOldFontCommand{\bf}{\normalfont\bfseries}{\mathbf}
\DeclareOldFontCommand{\it}{\normalfont\itshape}{\mathit}
\DeclareOldFontCommand{\sl}{\normalfont\slshape}{\@nomath\sl}
\DeclareOldFontCommand{\sc}{\normalfont\scshape}{\@nomath\sc}
\makeatother

\usepackage[shortlabels]{enumitem}
\setlist{nolistsep} %

\newlist{asslist}{enumerate}{1} % also creates a counter called 'propenumi'
\setlist[asslist]{label=(\roman*), ref=\thethmT(\roman*)}
\crefalias{asslistenumi}{Assumption} 

\newlist{thmlist}{enumerate}{1} % also creates a counter called 'propenumi'
\setlist[thmlist]{label=(\alph*), ref=\thethmT(\alph*)}

\usepackage{tikz} % Required for drawing custom shapes

\usepackage[english]{babel} % English language/hyphenation
\usepackage{xcolor} % Required for specifying colors by name
\definecolor{ocre_old}{RGB}{243,102,25} % Define the orange color used for highlighting throughout the book
\definecolor{mblue}{rgb}{0.122, 0.435, 0.698}
\definecolor{ocre_n}{rgb}{0, 0.435, 0.698}
\definecolor{ocre}{rgb}{0.1, 0.14, 0.13}

\usepackage{imakeidx}

\makeatletter
\newtheoremstyle{ocrenumbox}% % Theorem style name
{0pt}% Space above
{0pt}% Space below
{\sl}% % Body font
{}% Indent amount
{\small\bf\sffamily\color{ocre}}% % Theorem head font
{\;}% Punctuation after theorem head
{0.25em}% Space after theorem head
{\small\sffamily\color{ocre}\thmname{#1}\nobreakspace\thmnumber{\@ifnotempty{#1}{}\@upn{#2}}% Theorem text (e.g. Theorem 2.1)
	\thmnote{\nobreakspace\the\thm@notefont\sffamily\bfseries\color{black}---\nobreakspace#3.}} % Optional theorem note

\newtheoremstyle{ocrenumhypbox}% % Theorem style name
{0pt}% Space above
{0pt}% Space below
{}% % Body font
{}% Indent amount
{\small\bf\sffamily\color{ocre}}% % Theorem head font
{\;}% Punctuation after theorem head
{0.25em}% Space after theorem head
{\small\sffamily\color{ocre}\thmname{#1}\nobreakspace\thmnumber{\@ifnotempty{#1}{}\@upn{#2}}% Theorem text (e.g. Theorem 2.1)
	\thmnote{\nobreakspace\the\thm@notefont\sffamily\bfseries\color{black}---\nobreakspace#3.}} % Optional theorem note

\newtheoremstyle{blacknumex}% Theorem style name
{5pt}% Space above
{5pt}% Space below
{\sl}% Body font
{} % Indent amount
{\small\bf\sffamily}% Theorem head font
{\;}% Punctuation after theorem head
{0.25em}% Space after theorem head
{\small\sffamily{\tiny\ensuremath{\blacksquare}}\nobreakspace\thmname{#1}\nobreakspace\thmnumber{\@ifnotempty{#1}{}\@upn{#2}}% Theorem text (e.g. Theorem 2.1)
	\thmnote{\nobreakspace\the\thm@notefont\sffamily\bfseries---\nobreakspace#3.}}% Optional theorem note

\newtheoremstyle{blacknumbox} % Theorem style name
{0pt}% Space above
{0pt}% Space below
{\normalfont}% Body font
{}% Indent amount
{\small\bf\sffamily}% Theorem head font
{\;}% Punctuation after theorem head
{0.25em}% Space after theorem head
{\small\sffamily\thmname{#1}\nobreakspace\thmnumber{\@ifnotempty{#1}{}\@upn{#2}}% Theorem text (e.g. Theorem 2.1)
	\thmnote{\nobreakspace\the\thm@notefont\sffamily\bfseries---\nobreakspace#3.}}% Optional theorem note

\newtheoremstyle{ocrenum}% % Theorem style name
{5pt}% Space above
{5pt}% Space below
{\sl}% % Body font
{}% Indent amount
{\small\bf\sffamily\color{ocre}}% % Theorem head font
{\;}% Punctuation after theorem head
{0.25em}% Space after theorem head
{\small\sffamily\color{ocre}\thmname{#1}\nobreakspace\thmnumber{\@ifnotempty{#1}{}\@upn{#2}}% Theorem text (e.g. Theorem 2.1)
	\thmnote{\nobreakspace\the\thm@notefont\sffamily\bfseries\color{bl!ack}---\nobreakspace#3.}} % Optional theorem note
\makeatother

\theoremstyle{ocrenumbox}
\newtheorem{thmT}{Theorem}[section]
\newtheorem{theoT}{Theorem}
\newtheorem{theoremeT}[thmT]{Theorem}
\newtheorem{lemT}[thmT]{Lemma}
\newtheorem{propT}[thmT]{Proposition}

\theoremstyle{ocrenumhypbox}
\newtheorem{hypT}[thmT]{Hypothesis}
\theoremstyle{blacknumex}

\theoremstyle{blacknumbox}
\newtheorem{definitionT}[thmT]{Definition}
\newtheorem{notationT}[thmT]{Notation}
\newtheorem{remarkT}[thmT]{Remark}

\newtheorem{paraT}[thmT]{}

\theoremstyle{ocrenum}

\newtheorem{corollaryT}[thmT]{Corollary}
\RequirePackage[framemethod=default]{mdframed} % Required for creating the theorem, definition, exercise and corollary boxes

\newmdenv[skipabove=7pt,
skipbelow=7pt,
backgroundcolor=black!5,
linecolor=ocre,
innerleftmargin=5pt,
innerrightmargin=5pt,
innertopmargin=5pt,
leftmargin=0cm,
rightmargin=0cm,
innerbottommargin=5pt]{tBox}

\newmdenv[skipabove=7pt,
skipbelow=7pt,
rightline=false,
leftline=true,
topline=false,
bottomline=false,
backgroundcolor=ocre!10,
linecolor=ocre,
innerleftmargin=5pt,
innerrightmargin=5pt,
innertopmargin=5pt,
innerbottommargin=5pt,
leftmargin=0cm,
rightmargin=0cm,
linewidth=4pt]{eBox}	

\newmdenv[skipabove=7pt,
skipbelow=7pt,
rightline=false,
leftline=true,
topline=false,
bottomline=false,
linecolor=ocre,
innerleftmargin=5pt,
innerrightmargin=5pt,
innertopmargin=0pt,
leftmargin=0cm,
rightmargin=0cm,
linewidth=4pt,
innerbottommargin=0pt]{dBox}	

\newmdenv[skipabove=7pt,
skipbelow=7pt,
rightline=false,
leftline=true,
topline=false,
bottomline=false,
linecolor=gray,
backgroundcolor=black!5,
innerleftmargin=5pt,
innerrightmargin=5pt,
innertopmargin=5pt,
leftmargin=0cm,
rightmargin=0cm,
linewidth=4pt,
innerbottommargin=5pt]{cBox}

\newenvironment{theorem}{\begin{tBox}\begin{theoremeT}}{\end{theoremeT}\end{tBox}}

\newenvironment{hyp}{\begin{tBox}\begin{hypT}}{\end{hypT}\end{tBox}}
\newenvironment{theo}{\begin{tBox}\begin{theoT}}{\end{theoT}\end{tBox}}

\newenvironment{defi}{\begin{dBox}\begin{definitionT}}{\end{definitionT}\end{dBox}}	
\newenvironment{notation}{\begin{dBox}\begin{notationT}}{\end{notationT}\end{dBox}}	
\newenvironment{para}{\begin{dBox}\begin{paraT}}{\end{paraT}\end{dBox}}	
	
\newenvironment{lem}{\begin{dBox}\begin{lemT}}{\end{lemT}\end{dBox}}	
\newenvironment{prop}{\begin{dBox}\begin{propT}}{\end{propT}\end{dBox}}

\newenvironment{cor}{\begin{dBox}\begin{corollaryT}}{\end{corollaryT}\end{dBox}}	

\makeatletter
\renewcommand{\@seccntformat}[1]{\llap{\textcolor{ocre}{\csname the#1\endcsname}\hspace{1em}}} 
\renewcommand{\section}{\@startsection{section}{1}{\z@}
{-4ex \@plus -1ex \@minus -.4ex}
{1ex \@plus.2ex }
{\normalfont\large \bf \color{ocre}}}
\renewcommand{\subsection}{\@startsection {subsection}{2}{\z@}
{-3ex \@plus -0.1ex \@minus -.4ex}
{0.5ex \@plus.2ex }
{\normalfont\large\bf\color{ocre} }}
\newcommand{\ul}{{\underline {l}}}
\newcommand{\uk}{{\underline {k}}}
\newcommand{\unl}{\underline }

\newcommand{\vFq}{{{\nu F_q}}}

\newcommand{\GvF}{{\bG^{\vFq}}}

\newcommand{\phieinszwei}{{\phi_1.\phi_2}}

\makeatletter
\def\author#1{\gdef\autrun{\def\and{\unskip, }#1}\gdef\@author{#1}}

\def\address#1{{\def\and{\\\hspace*{18pt}}\renewcommand{\thefootnote}{}%
\footnote {#1}}}%}

\makeatother
\def\email#1{e-mail: #1}

\newcommand{\tw}[1]{{}^{#1}\!}

\newcommand{\Ad}{{$\mathbf{A}(d)$}}
\newcommand{\Bd}{{$\mathbf{B}(d)$}}

\newcommand{\otw}{\text{otherwise}}
\newcommand{\inv}{^{-1}}

\theoremstyle{definition}

\newtheorem{rem}[thmT]{Remark}

{\color{ocre_n}}
\theoremstyle{plain}

\newtheorem{ass}[thmT]{Assumption}

\theoremstyle{definition}
\numberwithin{equation}{section}
\numberwithin{table}{section}

\def\norm#1#2{{\operatorname N}_{#1}(#2)}
\def\cent#1#2{{\operatorname C}_{#1}(#2)}

\newcommand{\Id}{\operatorname {Id}}

\newcommand{\id}{\operatorname {id}}

\newcommand{\wh}{\widehat}
\newcommand{\wt}{\widetilde}
\newcommand{\wT}{{\widetilde T}}
\newcommand{\wb}{\widebreve}
\newcommand{\wbrevT}{{\widebreve T}}

\newcommand{\AHs}{A_\bH(s)}

\newcommand{\bN}{{\mathbf N}}

\newcommand{\wbT}{{\widetilde {\mathbf T}}}

\newcommand{\wM}{{\wt M}}

\newcommand{\wG}{{\widetilde G}}
\newcommand{\wH}{{\widetilde H}}

\newcommand{\wN}{{\wt N}}

\newcommand{\wZ}{{\wt Z}}

\newcommand{\bC}{{\mathbf C}}

\newcommand{\La}{\ensuremath{\Lambda}}
\newcommand{\la}{\ensuremath{\lambda}}

\newcommand{\whphi}{\ensuremath{{\wh\phi}}}

\newcommand{\bT}{{\mathbf T}}

\newcommand{\tDtwepslsc}{\tD^\eps_{l,\mathrm{sc}}}

\newcommand{\tDksc}{\tD_{k,\mathrm{sc}}}
\newcommand{\tDlsc}{\tD_{l,\mathrm{sc}}}
\newcommand{\tBlsc}{\tB_{l,\mathrm{sc}}}
\newcommand{\sico}{{\mathrm{sc}}}

\newcommand{\twepsDlq}{\tDtwepslsc(q)}

\newcommand{\bB}{{\mathbf B}}

\newcommand{\bG}{{{\mathbf G}}}
\newcommand{\bW}{{{\mathbf W}}}

\newcommand{\bH}{{{\mathbf H}}}
\newcommand{\bL}{{{\mathbf L}}}

\newcommand{\bK}{{{\mathbf K}}}

\newcommand{\bM}{{\mathbf M}}
\newcommand{\bS}{{\mathbf S}}
\newcommand{\HF}{{{\bH}^F}}

\newcommand{\bX}{{\mathbf X}}
\newcommand{\wbG}{{\wt{\mathbf G}}}
\newcommand{\wbH}{{\wt{\bH}}}
\newcommand{\vtB}{{v_{\tB}}}
\newcommand{\Irr}{\mathrm{Irr}}
\newcommand{\Lin}{\mathrm{Lin}}
\newcommand{\Irrl}{\Irr_{\ell'}}

\newcommand{\bZ}{{\mathbf Z}}

\newcommand{\SL}{\operatorname{SL}}
\newcommand{\SU}{\operatorname{SU}}

\newcommand{\GL}{\operatorname{GL}}

\newcommand{\PGL}{\operatorname{PGL}}

\newcommand{\ZZ}{\ensuremath{\mathbb{Z}}}
\newcommand{\CC}{\ensuremath{\mathbb{C}}}

\newcommand{\VV}{{\mathbb{V}}}

\newcommand{\EE}{{\mathbb{E}}}

\newcommand{\TT}{\ensuremath{\mathbb{T}}}

\newcommand{\oTT}{\ensuremath{\overline {\mathbb T}}}
\newcommand{\RR}{\ensuremath{\mathbb{R}}}

\newcommand{\MM}{\ensuremath{\mathbb{M}}}

\newcommand{\ov}{\overline }

\newcommand{\R}{\operatorname{R}}
\newcommand{\obG}{{\overline {\mathbf G}}}

\newcommand{\xx}{\mathbf x }
\newcommand{\n}{\mathbf n }
\newcommand{\nn}{\mathbf n}
\newcommand{\maex}{maximal extendibility\ }
\newcommand{\Maex}{Maximal extendibility\ }

\newcommand{\h}{\mathbf h }\newcommand{\hh}{\mathbf h }

\renewcommand{\o}{\overline}

\newcommand{\Cent}{\ensuremath{{\rm{C}}}}
\newcommand{\NNN}{\ensuremath{{\mathrm{N}}}}

\newcommand{\Sym}{{\mathcal{S}}}

\def\restr#1|#2{\left.#1\right\rceil_{#2}}

\usepackage{imakeidx}
\makeindex[columns=2,intoc]

\def\III#1{\index{#1@$#1$}{#1}}
\def\II#1@#2{\index{#1@$#2$}{{#2}}}
\def\Ispezial#1@#2@#3{\index{#1@$#2$}{{#3}}}

\newcommand{\bII}{\mathbb I}

\newcommand{\tD}{\ensuremath{\mathrm{D}}}
\newcommand{\Cy}{\mathrm C}
\newcommand{\tC}{\mathrm C}
\newcommand{\tB}{\mathrm B}
\newcommand{\cE}{\mathcal E}

\newcommand{\cM}{\mathcal M}

\newcommand{\calO}{\mathcal O}

\newcommand{\calM}{\mathcal M}
\newcommand{\calL}{\mathcal L}

\newcommand{\calC}{\mathcal C}

\newcommand{\calG}{\ensuremath{\mathcal G}}
\newcommand{\calF}{\ensuremath{\mathcal F}}

\newcommand{\al}{{\alpha}}
\newcommand{\eps}{{\epsilon}}

\newcommand{\UCh}{\operatorname{Uch}}

\newcommand{\spannh}{\spann<h_0>}

\newcommand{\FF}{{\mathbb{F}}}
\newcommand{\FFtimes}{{\mathbb{F}^\times}}
\newcommand{\si}{\ensuremath{\sigma}}

\newcommand{\GF}{{{\bG^F}}}

\newcommand{\wGF}{{{{\wbG}^F}}}

\newcommand{\cO}{{\mathcal O}}
\newcommand{\SO}{{\operatorname{SO}}}

\makeatletter
\def\Set#1{\Set@h#1@}
\def\Lset#1{\Lset@h#1@}
\def\Set@h#1|#2@{\left\{\left.#1\vphantom{#2}\hskip.1em\,\right|\,\relax #2\right\}}
\def\Lset@h#1@{\left\{#1\right\}}
\def\CALC#1{\CALC@h#1@}
\def\CALC@h#1|#2@{\calC^{#1}(#2)}
\def\CALCrad#1{\CALCrad@h#1@}
\def\CALCrad@h#1|#2@{\calC_\radic^{#1}(#2)}

\def\CALCNC#1{\CALCNC@h#1@}
\def\CALCNC@h#1|#2@{\calC_{\radic,nc}^{#1}(#2)}

\def\restr#1|#2{\left.#1\right\rceil_{#2}}
\def\spann<#1>{\left\langle#1\right\rangle}
\def\spa#1{\left\langle#1\right\rangle}
\def\Sympm#1{\Sym_{\pm #1}}

\def\Spann<#1>{\Spann@h#1@}
\def\Spann@h#1|#2@{\left\langle\left.#1\vphantom{#2}\hskip.1em\right.\mid\relax #2 \right\rangle}
\def\Set#1{\Set@h#1@}
\def\Set@h#1|#2@{\left\{\left.#1\vphantom{#2}\hskip.1em\,\right.
\mid\relax #2\right\}}
\def\set#1{\set@h#1@}
\def\set@h#1@{\left\{#1\right\}}
\def\spann<#1>{\left\langle#1\right\rangle}
\makeatother

\newcommand{\UE}{{\underline{E}}}
\newcommand{\uE}{{\underline{E}}}

\newcommand{\Aut}{\mathrm{Aut}}

\newcommand{\Out}{\ensuremath{\mathrm{Out}}}

\newcommand{\Z}{\operatorname Z}

\newcommand{\calP}{\mathcal P}

\newcommand{\calX}{\mathcal X}

\newcommand{\DD}{\mathbb D}
\newcommand{\calN}{\mathcal N}

\newcommand{\forevery}{{\text{\quad\quad for every }}}
\newcommand{\und}{{\text{ and }}}

\newcommand{\lra}{\longrightarrow}
\newcommand{\tE}{\mathrm E}
\newcommand{\tA}{\mathrm A}
\newcommand{\tF}{\mathrm F}
\newcommand{\tG}{\mathrm G}

\newcommand{\cF}{\mathcal F}
\newcommand{\Norm}{\operatorname{N}}

\newcommand{\neins}{n_1^\circ}
\newcommand{\neinszwei}{n_1^\circ n_2^\circ}
\newcommand{\nii}{n_i^\circ}
\newcommand{\nzwei}{n_2^\circ}

\newcommand{\wrt}{{with respect to\ }}

\def\starStab#1|#2|#3{ (#1 #2)_{#3}= {#1}_{#3} {#2}_{#3}}
\def\starStabalign#1|#2|#3{ (#1 #2)_{#3}&= {#1}_{#3} {#2}_{#3}}
\def\starStabkla#1|#2|#3{ (#1 #2)_{#3}= {(#1)}_{#3} {#2}_{#3}}

\def\starStabklaneq#1|#2|#3{ (#1 #2)_{#3}\neq {(#1)}_{#3} {#2}_{#3}}
\def\starStabneq#1|#2|#3{ (#1 #2)_{#3}\neq {#1}_{#3} {#2}_{#3}}
\title{The McKay Conjecture on character degrees}
\author{Marc Cabanes 
 and 
Britta Sp\"ath }

\date{}

\begin{document}

\maketitle
\markboth{Marc Cabanes and Britta Sp\"ath}{The McKay Conjecture on character degrees}
\abstract{ We prove that for any prime $\ell$, any finite group has as many irreducible complex characters of degree prime to $\ell$ as the normalizers of its Sylow $\ell$-subgroups. This equality was conjectured by John McKay. 

The conjecture was reduced by Isaacs--Malle--Navarro (2007) to a conjecture on representations, linear and projective, of finite simple groups that we finish proving here using the classification of those groups.

We study mainly characters of normalizers $\NNN_\bG(\bS)^F$ of Sylow $d$-tori $\bS$ ($d\geq 3$) in a simply-connected algebraic group $\bG$ of type $\tD_l$ ($l\geq 4$) for which $F$ is a Frobenius endomorphism. We also introduce a certain class of $F$-stable reductive subgroups $\bM\leq \bG$ of maximal rank where $\bM^\circ$ is of type $\tD_{k}\times \tD_{l-k}$. The finite groups $\bM^F$ are an efficient substitute for $\NNN_\bG(\bS)^F$ or the $\ell$-local subgroups of $\bG^F$ relevant to McKay's abstract statement. For a general class of those subgroups $\bM^F$ we describe their characters and the action of $\Aut(\bG^F)_{\bM^F}$ on them, showing in particular that $\Irr(\bM^F)$ and $\Irr(\bG^F)$ share some key features in that regard.}
    
\address{M.C. : CNRS, Institut de Math\'ematiques Jussieu-Paris Rive Gauche, Place Aur\'elie Nemours, 75013 Paris, France, \email{cabanes@imj-prg.fr}}
\address{B.S. : School of Mathematics and Natural Sciences,
		University of Wuppertal, Gau\ss str. 20, 42119 Wuppertal, Germany, \email{bspaeth@uni-wuppertal.de}}\address{\textit{Mathematics Subject Classification (2020):} 20C15 (20C33 20G40)}
		
\tableofcontents
\section{Introduction}

The main theorem of this paper is as follows. 

\begin{theo}\label{theo_McK}Let $X$ be a finite group, $\ell$ a prime and $S$ a Sylow $\ell $-subgroup of $X$. Let $\Irr_{\ell ' }(X)$ denote the set of complex irreducible characters of $X$ whose degree is prime to $\ell $. Then $$ |\Irr_{\ell ' }(X)|=|\Irr_{\ell ' }({\rm N}_X(S))|.\eqno(\textrm{\bf MK})$$
\end{theo}

Every algebraist may have recognized that this is John McKay's conjecture on character degrees of finite groups. This paper provides the last of many steps in a proof of this conjecture using the classification of finite simple groups (CFSG). The other steps are mainly contained in the papers \cite{IMN}, \cite{MaH0}, \cite{ManonLie}, \cite{S12}, \cite{CS13}, \cite{KoSp}, \cite{MS16}, \cite{CS17A}, \cite{CS17C}, \cite{CS18B}, \cite{S21D1}, \cite{S21D2}, adding more than 400 pages to the CFSG and the background knowledge on representations of quasisimple groups -- thus fulfilling Jon Alperin's prediction that ``\textit{we have here a very easily stated conjecture about all finite groups which is not easily decided from a possible classification of all simple groups}'' \cite{A76}.

\subsection{The McKay conjecture}\label{subMK} McKay's equality ({\bf MK}) relates two numbers, one \textit{global} in the sense that it pertains to $X$, the other \textit{local} in that it is the same for the $\ell $-local subgroup N$_X(S)$.

It seems to originate in McKay's research on character tables of sporadic simple groups \cite{McK71}, an interest that would also lead him to the ``$\tE_8$-observation" \cite[p. 185]{McK80} and the so-called ``(monstrous) Moonshine'' on the character degrees of the Monster sporadic group (see \cite{CoNo79}). It could now be argued that the idea itself of the McKay Conjecture owes a lot to the CFSG as a project, and its proof now draws from the CFSG as a theorem.

While a proof of ({\bf MK}) as elementary as the statement itself seems unattainable, a legitimate wish is to find more structural statements implying character theoretic ones like ({\bf MK}), ({\bf AWC}) below or Dade's conjecture \cite[Conj. 6.3]{Da92}. This direction of research is exemplified by Brou\'e's conjecture \cite[Conj. 6.1]{Br90a} for blocks with abelian defect, see \cite{Rou23} for a recent survey of the issues raised. However, once those character theoretic equalities are checked (see \cite{S17} for a reduction of Dade's conjecture), they may well be helpful in establishing module theoretic statements, using arguments in the vein of \cite[Thm 1.1]{Br90b}.

\subsection{Early results and perspectives} \label{EarlyMK} Soon after the conjecture was sketched by McKay and made precise by Alperin, important verifications of ({\bf MK}) followed: For symmetric groups and $\ell =2$ already in \cite{Mac71} and \cite{McK}, for a large class of solvable groups by Isaacs in \cite{Isa73}, for general linear groups in characteristic $\ell $ in \cite{A76}, for many other finite groups of Lie type in characteristic $\ell $ by Lehrer in \cite{Lehrer}, for symmetric groups and general linear groups for arbitrary primes by Olsson in \cite{Olsson}. A strong form of ({\bf MK}) was proven for $\ell $-solvable groups by Wolf \cite{Wo78}.

The statement ({\bf MK}) itself has similarities with the so-called Harish-Chandra theory of cusp forms for finite groups of Lie type \cite{Sp70}. Fix $G=\GF$ a finite reductive group of Lie type where $F\colon \bG\to\bG$ is a Frobenius endomorphism defining the reductive group $\bG$ over a finite field. For an $F$-split Levi subgroup $\bL$ of $\bG$ and a so-called \textit{cuspidal} character $\lambda\in \Irr(\bL^F)$, parabolic induction allows us to define a subset $\Irr(\GF,(\bL,\lambda))\subseteq \Irr(\GF)$ which turns out to be parametrized via $$\Irr(\GF,(\bL,\lambda))\longleftrightarrow \Irr(\NNN_\GF(\bL,\lambda)/\bL^F)\eqno({\bf HC})$$ thanks to the Howlett--Lehrer--Lusztig theory of Hecke algebras (see for instance \cite[Thm 3.2.5]{GM}). This was generalized in the wake of the determination of $\ell $-blocks as partitions of $\Irr(G)$ for a classical group $G$ in characteristic $p\neq \ell$ by Fong--Srinivasan \cite{FS82}, \cite{FS86}, \cite{FS89}. In this generalization formalized by Brou\'e--Malle--Michel in all types, see \cite{BMM93}, one gets for any integer $d\geq 1$ similar subsets of unipotent characters $\cE(\GF,(\bL,\lambda))$ where $\bL$ (non $F$-split but still $F$-stable when $d\neq 1$) is the centralizer of a so-called \textit{$d$-torus} \cite{BM92} of $\bG$ and $\lambda$ is a so-called $d$-\textit{cuspidal} unipotent character of $\bL^F$. The above ({\bf HC}) is then true up to replacing parabolic induction by the Lusztig functor $R_\bL^\bG$ in the definition of $\Irr(\GF,(\bL,\lambda))$. 

In the meantime, a cluster of conjectures emerged around McKay's, starting with Alperin's weight conjecture \cite{A87}. Considering Brauer characters of a finite group $X$ with respect to the prime $\ell $, the conjecture posits that $$|\text{IBr}(X)|= |\text{Alp}_\ell(X)|\eqno({\bf AWC})$$ where Alp$_\ell(X)$ is the set of $X$-conjugacy classes of pairs $(Q,\pi)$ with $Q$ an $\ell $-subgroup of $X$ and $\pi$ an element of $\Irr(\NNN_X(Q)/Q)$ with codegree $|\NNN_X(Q)|/|Q|\pi(1)$ prime to $\ell $. Kn\"orr and Robinson reduced ({\bf AWC}) to a remarkable statement about ordinary characters, see \cite{KnRo89}, \cite[Thm 9.24]{Navarro_book}. This was in turn generalized by E.C. Dade into a broad conjecture \cite[Conj. 6.3]{Da92} implying both ({\bf MK}) and ({\bf AWC}), see \cite[Thms 9.26 and 9.27]{Navarro_book}. 

This and the many refinements brought to Dade's conjecture lead Brou\'e to introduce a strengthened version including all the extra refinements known in 2006 \cite{Br06}, which he referred to as MAKRODINU  (an acronym from the names of the authors of \cite{McK71}, \cite{A87}, \cite{KnRo89}, \cite{Da92}, \cite{IN_Annals}, \cite{Uno04}). See \cite{Tu08} and \cite{N_Annals} for other refinements. All suggest equivalences of algebras of a geometric nature over $\ell $-adic rings of coefficients.  
Indeed, Brou\'e's own conjecture on $\ell$-blocks with abelian defect \cite[Conj. 6.1]{Br90a} asserts an equivalence of derived module categories $$D(B)\cong D(b)\eqno({\bf ADC}) $$between an $\ell$-block $B$ of a finite group with abelian defect and its Brauer correspondent $b$, where both blocks are seen as algebras over $\overline{\ZZ}_\ell$. See \cite{Ok00} for the case of SL$_2(\ell^m)$, \cite[Thm 7.6]{ChR08} for the case of symmetric groups, \cite[Thm 4.33]{CR13} for a proof in some cases of principal blocks using CFSG. %Proving ({\bf iMK}) in full generality can also be seen as an incentive to find a generalization of Brou\'e's conjecture for arbitrary blocks.

\subsection{A conjecture made (quasi) simple} \label{red_iMK} In reducing ({\bf MK}) to a statement about simple groups, it seems impossible not to involve \textit{quasisimple} groups, i.e. perfect groups $G$ such that $G/\Z(G)$ is simple. To illustrate why, consider even the alternating groups: characters of $\mathfrak{A}_n$ ($n\geq 5$) are not of much help to find faithful characters of their double covers $2{\mathfrak A}_n$, let alone to prove counting statements about them. 

In a major breakthrough, the reduction theorem by Isaacs--Malle--Navarro for McKay's conjecture \cite{IMN} appeared in 2007. It introduces a so-called ``inductive McKay condition'', formally stronger than ({\bf MK}), which, once checked for a given $\ell $ and all quasisimple groups $G$, implies ({\bf MK}) for $\ell $ and any finite group $X$. See the surveys \cite{Ti14}, \cite{KM19}, \cite{Ma17b} and the book \cite{Navarro_book} for developments after \cite{IMN}. This condition was reformulated in terms of centrally isomorphic character triples, a notion devised by Navarro and the second author \cite{JEMS_NS}, see \cite[Def 10.14]{Navarro_book}. 

Let us recall that a \textit{character triple} is any $(A,X,\chi)$ with finite groups $X\unlhd A$ and $\chi\in \Irr(X)$ an irreducible character of $X$ invariant under the conjugation action of $A$ on $X$, a terminology due to Isaacs \cite[Ch. 11]{Isa}. The inductive condition ({\bf iMK}) for a finite group $X$ and a prime $\ell$ is as follows (see Definition~\ref{def:2_1} below for the relation $\geq_c$ between character triples).

First, we assume the group theoretic condition that for a Sylow $\ell$-subgroup $S$ of $X$ \textsl{there exists $\NNN_X(S)\leq N\leq X$ such that $N$ is stable under the stabilizer $\Gamma :=\Aut(X)_S$ of $S$ in $\Aut(X)$, with $N\neq X$ when $\NNN_X(S)\neq X$,} and second, we assume that

\noindent\textbf{({\bf iMK})} \textsl{there exists a $\Gamma$-equivariant bijection $$\Irr_{\ell ' }(X)\to \Irr_{\ell ' }(N)$$ such that $(X\rtimes \Gamma_\chi, X,\chi)\geq_c (N\rtimes \Gamma_\chi, N,\chi')$ whenever $\chi\mapsto \chi '$.}

\medskip

The group $N$ can be taken to be $\NNN_{X}(S)$ but it is important to keep the freedom in many quasisimple groups to choose a nicer overgroup $N$. 
Then the Reduction Theorem~\ref{RossiMK} below simply states that for a given prime $\ell$, once ({\bf iMK}) is checked for all quasisimple groups $X$, then it is true for any finite group, see \cite[Thm B]{Rossi_McKay}. 
It is clear by induction that ({\bf iMK}) implies McKay's equality (\textbf{MK}).

It is this ({\bf iMK}) that we indeed prove for all finite groups and primes.

\begin{theo}
	Let $X$ be a finite group and $\ell$ a prime. For any Sylow $\ell$-subgroup $S$ of $X$ and $\Gamma:=\Aut(X)_S$ there exists a $\Gamma$-equivariant bijection 
	\[\Omega: \Irrl(X)\to \Irrl(\NNN_X(S)),\] 
	such that every $\chi\in\Irrl(X)$ satisfies 
	\[ (X\rtimes \Gamma_\chi,X,\chi)\geq_c (\NNN_X(S)\rtimes \Gamma_{\chi}, \NNN_X(S),\Omega(\chi)) .\text{\ \ \ \ } \] 
	
\end{theo}

\subsection{Quasisimple groups of Lie type} \label{sub_iMKred} Whenever Out$(X)$ is cyclic then ({\bf iMK}) above boils down to the existence of an equivariant bijection $\chi\mapsto \chi '$ such that the restrictions of $\chi$ and $\chi '$ to Z$(X)$ have same irreducible constituent. This essentially reduces ({\bf iMK}) for this type of quasisimple groups to cases previously checked for ({\bf MK}), see \cite[\S 3, \S 5]{ManonLie}. 
It then remains to check ({\bf iMK}) for quasisimple groups of Lie type, thus making the proof of McKay's conjecture a Lie theoretic effort. 

The reformulation of the inductive condition of \cite{IMN} in terms of character triple equivalences originated in \cite{S12}, and the main application given there was to check ({\bf iMK}) for quasisimple groups of Lie type whose defining characteristic is the same prime $\ell$. 

After that, the main question is, of course, to prove ({\bf iMK}) in the cases where $X$ is a quasisimple group of Lie type in characteristic $p\neq\ell$. 

Apart from a few exceptions, the universal coverings of simple groups of Lie type are of the form $G=\bG^F$ where $\bG$ is a simple simply connected linear algebraic group endowed with a Frobenius endomorphism $F\colon \bG\to \bG$ defining $\bG$ over a finite field of order $q$, a power of $p$. To account for so-called diagonal automorphisms one defines the inclusion $G\unlhd \wG$ associated to a regular embedding $\bG\leq \wbG$ of the algebraic groups. One also defines $E$ acting on $\wG$ as a group of field and graph automorphisms of $G$, so that $G=[\wG ,\wG]$ and $\wG\rtimes E$ induces the whole of $\Aut(G)$ on $G$. %Note that the extendibility requirement of ({\bf iMK}glo) for the inclusion $G\unlhd\wG$ is then covered by a deep theorem of Lusztig, see [Lu88, Prop. 10].

\medskip

In such a situation, we define the integer $d_\ell(q)$ as the multiplicative order of $q$ mod $\ell$ when $\ell\geq 3$, mod 4 when $\ell =2$. To check ({\bf iMK}) an important idea from \cite{MaH0} is to take for $N$ the normalizer in $G=\GF$ of a Sylow $d_\ell(q)$-torus $\bS$ of $\bG$. The strong relation between $\NNN_G(\bS)$ and a Sylow $\ell$-subgroup of $G$ is ensured by Brou\'e--Malle's results \cite{BM92} and a remark of the first author \cite{Ca94}. Let us mention that the relevance of $d_\ell(q)$-tori to the $\ell$-local analysis of $G$ goes much deeper, see \cite[\S 4]{Br06}, \cite{Ro23b}.

\subsection{Proving ({\textbf{iMK}}) for the non-defining primes} \label{sub_iMKnondef} Malle and the second author checked the inductive condition ({\bf iMK}) for the prime 2 and all quasisimple groups in \cite{MS16}, thus proving McKay's equality (\textbf{MK}) for the prime $2$. Proving ({\bf iMK}) for the prime 2 is made simpler by ({\bf HC}) providing an equivariant bijection and by the fact that character triples $(A,X,\chi)$ are easier to describe whenever $X$ is perfect and $|A/X|$ is prime to $\chi(1)$, see \cite[6.25]{Isa}. 

Returning to the general case, the choice of a bijection $$\Irrl(G)\to \Irrl(N)\eqno{(\mathbf{\Omega})} $$ has been described in \cite{MaH0} drawing mainly on \cite{BMM93} augmented with a discussion of character degrees in (\textbf{HC}) to give a common indexing set to both sides of $(\mathbf{\Omega})$, see also \cite{S09}, \cite{S10a} and \cite{S10b} for the $N$-side. This choice of the map being relatively settled, the main effort to check ({\bf iMK}) for quasisimple groups then adresses the control of the character triples on either side and the action of $\Out(G)$ using a method introduced as \cite[Thm 2.12]{S12} and recalled here through the variant \Cref{prop_23}. 

Recall the choice of $N$ as $\NNN_{G}(\bS)$ for $\bS$ a Sylow $d_\ell(q)$-torus of $\bG$. Through an elementary application of Clifford theory, the representations of $N$ are strongly related with the ones of $\Cent_G(\bS)$ once certain extendibility questions are solved.

For any given $d\geq 1$ and Sylow $d$-torus $\bS$ of $(\bG,F)$ we single out in \cite{CS18B} the following conditions where 
\[\wt N=\NNN_{\wG}(\bS)\unrhd N=\NNN_G(\bS)\unlhd \wh N=\NNN_{GE}(\bS) .\]

\noindent	{\Ad.} \textit{There is an $ \wh N$-stable $\wt N$-transversal in $\Irr(N)$ where each element extends to its stabilizer in $\wh N$.}

\noindent	{\textbf{\Bd}.} \textit{Every $\theta\in \Irr(N)\cup \Irr(\Cent_{\wG}(\bS))$ extends to its stabilizer $\wt N_\theta$. In the case of $\theta\in\Irr(\Cent_{\wG}(\bS))$ this can be done in an $\Irr(\wt N/N)\rtimes \NNN_{\wG E}(\bS)$-equivariant way.}

\medskip

For large values of $d$ -- forcing $\bS=\{1\}$ -- the condition \Ad{} becomes a quite challenging condition entirely about $\Irr(G)$ that can also be written as follows (see \cite[Sect. 1.C]{S21D2}):

\medskip
\noindent $\textbf{A}(\infty)$.\ \textsl{Any element of $\Irr_{}(G)$ has a $\wG$-conjugate $\chi$ such that $ (\wG E)_\chi=\wG _\chi E_\chi $ and $\chi $ extends to $G E_\chi$.}

\medskip

For $\ell\geq 3$ and prime to $q$, what has been said above explains why ({\bf iMK}) for $G$ and $\ell$ is then implied by the conjunction of $\textbf{A}(\infty)$, $\textbf{A}(d_\ell(q))$ and $\textbf{B}(d_\ell(q))$, see \cite[2.4]{CS18B}.

Condition $\textbf{A}(\infty)$ was finally reached in all types as \cite[Thm A]{S21D2}. As noted in \cite[3.5]{CS18B} the stabilizer statement in $\textbf{A}(\infty)$ settles the question of determining the action of $\Out(G)$ on $\Irr(G)$ for all quasisimple groups $G$. This question has been a natural one since the completion of the CFSG. The answer, and in fact the stronger statement  \textbf{A}$(\infty)$, is expected to have applications to any counting conjecture, see already \cite{FS23} or \cite{Ru22b} through \cite{Ru22a}. The proof of \textbf{A}$(\infty)$ for all types is spread in \cite{CS17A}, \cite{CS17C}, \cite{CS18B}, \cite{S21D1}, {\cite{S21D2}} (see also \cite{Ta18}, \cite{Ma17a}) with the main part being devoted to the types $\tD$ and $\tw 2\tD$. We refer to the introductions of those papers for the main issues raised and the methods used.

\subsection{The present paper} \label{sub_paper} The above leaves us to prove the conditions $\textbf{A}(d)$ and $\textbf{B}(d)$ for every $d\geq 1$ and all groups $G$ of Lie type. In types \textit{not $\tD$ nor} $^2\tD$ this is done in \cite{CS17A}, \cite{CS17C}, \cite{CS18B}, accounting for the most technical part of those texts. The main goal of the present paper is to show $\textbf{A}(d)$ and $\textbf{B}(d)$ in types $\tD$ and $\tw 2\tD$.

In Chapter~\ref{sec_recall} we recall more precisely than above the setting, the background results, and some of the main notation. In Section 2.A the results concern general character theory and in particular the existence of extensions of characters. Then we state the condition (\textbf{iMK}) in terms of centrally isomorphic character triples. We also recall the main criterion used in types already solved, \Cref{prop_23} being a slightly generalized form that will be useful later. General finite groups of Lie type are introduced in \Cref{ssec2C} with their generators and automorphisms. We then define the conditions \Ad{} and \Bd{} matching the criterion for (\textbf{iMK}) just mentioned. We also comment on the case of groups of type $\tD$ and the important inclusion $\bG\leq\ov\bG$ of a group of type $\tD_l$ into a group of type $\tB_l$ with a common maximal torus.

In Chapter~\ref{sec_3} we extend some of the results of \cite{S21D2}. We start by studying the centralizers of semisimple elements in $\tD_{l,\text{sc}}(\ov\FF_p)$ for $l\geq 4$. This leads to further results on the characters of $^{}\tD_{l,\text{sc}}(q)$ and $^{2}\tD_{l,\text{sc}}(q)$. This includes groups of ranks $\leq 3$ that will appear in our study of local subgroups and will help us to give a uniform treatment in the part of Chapter \ref{sec_nondreg_groupM} dealing with characters. An important feature is the partition $$\Irr(G)=\oTT\sqcup \EE\sqcup \DD$$ deduced naturally from \textbf{A}$(\infty)$ in terms of stabilizers and extendibility of those characters \wrt $G\unlhd \wG E$. In Sections \ref{ssec_IrrG} and \ref{ssec:3D} we get precise results on stabilizers and kernels of characters specific to each of the three subsets above, see \Cref{thm_sumup_D}.

In the proofs of the local Conditions \Ad{} and \Bd{} for types other than $\tD$ and $^2\tD$, the discussion splits into two main cases according to whether or not the Levi subgroup $\bC:=\Cy_\bG(\bS)$ is a torus, the solution in the latter case -- \textit{non-regular $d$'s} -- usually using $\textbf{A}(\infty)$ in smaller ranks. Here we will have to use more than just $\textbf{A}(\infty)$, drawing on Chapter~\ref{sec_3} and \cite[Ch. 5]{S21D2},  in a slightly different dichotomy also introducing an overgroup of $\NNN_{ \bG}(\bS)$ in Chapter \ref{sec_nondreg_groupM}.

In Chapter 4, the Conditions \Ad{} and \Bd{} are proven first for integers $d\geq 3$ that are additionally \textit{doubly regular} for $(\bG,F)$.
The integer $d$ is called doubly regular if $\bC$ and $\ov\bC:=\Cy_{\o\bG}(\bS)$ are tori (hence equal) when $\bS$ is a Sylow $d$-torus of $(\bG,F)$. The proof in that case is simpler than the usual regular case in other types and takes advantage of the case of type $\tB_l$ being already known from \cite{CS18B}.
This finally ensures ({\bf iMK}) for quasisimple groups $\GF$ of type $\tD_l$ or $^2\tD_l$ and primes $\ell$ such that $d_\ell(q)$ is doubly regular for $(\bG,F)$, see \Cref{thm:dreg}. 

In the non-doubly regular case, where $\ov \bC$ is not a torus, we bring here a simplification that would also simplify the proofs given for other types and helps keeping the case of types $\tD$ and $^2\tD$ to a reasonable size. 
We introduce a finite group $M$, a subgroup of $\bG^{F'}$ where $F'$ is a version of $F$ slightly altered to fit the technicality of the non-doubly regular case (notation is slightly different in Chapters 5 and 6). This group $M$ is isomorphic to a subgroup $M'$ of $\GF$ containing $\NNN_\GF(\bS)$ and normalizing $\bK_2':=\bG\cap[\ov\bC ,\ov\bC]$ and $\bK_1':=[\Cent_\bG(\bK_2'), \Cent_\bG(\bK_2')]$. 
The groups $\bK'_i{}^F$ ($i=1,2$) are of types $\tD$ in possibly small ranks and their images $K_i$ in $\bG^{F'}$ define a central product 
%with $l=l_1+l_2$ and $\eps =\eps_1 \cdot \eps_2$. 
 $K_1.K_2$ which is normal of index $2$ or $4$ in $M$. We then use the knowledge of the character theory of $K_1$ and $K_2$ gathered from \cite{S21D2} and Chapter~\ref{sec_3} to derive crucial information about $\Irr(M)$. The groups $\wh M$ and $\wt M$ being defined from $M$ similarly to $\wh N$ and $\wt N$ of \ref{sub_iMKnondef} above from $N$, we establish a theorem that reads roughly as follows (the precise statement is \Cref{thm_sec_Ad}).

\begin{theo} There exists some $\wh M$-stable $\wt M$-transversal ${\TT}(M)$ in $\Irr(M)$ such that moreover any element of ${\TT}(M)$ extends to its stabilizer in $\wh M$.
\end{theo}
\medskip

In fact a larger class of groups $M$ with no reference to an integer $d$ is defined in Section \ref{ssec_M}. The relevance of this class of subgroups to the non-doubly regular case is given in the trichotomy of \Cref{tricho}, showing how Theorem C along with \Cref{thm:dreg} and the known case of cyclic defect \cite{KoSp} essentially exhaust all cases to consider.  The action of automorphisms on $M$ is given in Section \ref{ssecE(M)}. Afterwards, we study the stabilizers and extendibility of the characters of $M$ in Section \ref{ssec_5B}.
The proof of Theorem C as \Cref{thm_sec_Ad} is our Sections 5.D-E where we discuss several relevant subsets of $\Irr(M)$ defined according to the restriction of characters to $K_1'.K_2'$ and the various subsets ${\TT}_i$, ${\EE}_i$, and ${\DD}_i$ of $\Irr(K_i)$ selected from the description given in \cite{S21D2}.

 From there the end of the proof of Conditions \Ad{} and \Bd{} in Chapter~\ref{sec6} uses the fact that $d$ is doubly regular for $(\bK_1,F')$ where $\bK_1$ is a factor of a central product $\bM^\circ =\bK_1.\bK_2$ of $F'$-stable simple simply-connected groups of type $\tD$. Let $\bS'$ be a Sylow $d$-torus of $(\bK_1,F')$. Thanks to Chapter~4, ({\bf iMK}) holds in $\bK_1^{F'}$, thus providing a bijection $\Omega_1\colon \Irrl(K_1)\to\Irrl(\NNN_{K_1}(\bS'))$ with strong properties in terms of the $\geq_c$ relation, yielding more character correspondences and establishing a version of $\textbf{A}(\infty)$ for $\NNN_{\bM}(\bS')^{F'}$, see Section 6.A. This is completed in Section 6.C with results showing that indeed $\NNN_\bG(\bS')\leq \bM$ and translating the results obtained for $(\bG,F')$ into similar ones for $(\bG,F)$ with a special care for automorphisms. This essentially completes the proof of \Ad. Meanwhile, Section 6.B establishes \Bd{} by applying results on the Clifford theory for $(\bM^\circ)^{F'}$ gathered in 
Section \ref{ssec5E}.

\bigskip

\noindent{\bf Acknowledgement:} 
Some of this research was conducted in the framework of the research training group \textit{GRK 2240: Algebro-Geometric Methods in Algebra, Arithmetic and Topology}, funded by the DFG. The first author is grateful to the University of Wuppertal for its hospitality.

The authors would like to thank 
Gunter Malle, Ga\"etan Mancini, Gabriel Navarro, Mandi Schaeffer-Fry and the anonymous referee  for their careful reading and their suggestions on various versions of the manuscript. We also thank Lucas Ruhstorfer for his interest and discussions on several points.
%\newpage

%%%%%%%%%%%%%%%%%%%%%%%%%%%%%%%%%%%%%%%%%%
\section{Background results}\label{sec_recall}
%%%%%%%%%%%%%%%%%%%%%%%%%%%%%%%%%%%%%%%%%%
The aim of this chapter is twofold. On the one hand, we recall some notation used later, as well as the inductive McKay condition ({\bf iMK}), with a criterion for verifying it in \ref{ssec_2D}. 
We also state the relevant results for groups of Lie type that ensure part of the required assumptions. We conclude with some group-theoretic results for groups of type $\tD_l$. 
 
\subsection{Characters and extensions}
Our notation tries to be as classical as possible, some being recalled in \cite[1.A]{S21D2}. We are dealing a lot with situations where $X\unlhd A$ are finite groups and $\chi\in\Irr(X)$ is invariant under the conjugation action of $A$, i.e. $\chi\in\Irr(X)^A$. Extendibility is then a major issue.

\begin{notation} \label{defMaxExt} Let $X\unlhd A $ be finite groups. Let $\calX\subseteq \Irr(X)$ and $\psi \in\Irr(A)$. 
	
	We denote by $\II{IrrA}@ {\Irr(A\mid\calX)}$ the set of irreducible components of induced characters $\chi^A$ for $\chi\in\calX$, and by $\II{IrrX}@{\Irr(\restr\psi|{X})}$ the set of irreducible components of the restriction $\restr\psi|{X}$. 
	
We say that \index{maximal\ extendibility}{ \textit{maximal\ extendibility holds with respect to $X\unlhd A$ for $\calX$}}, if $\calX$ is $A$-stable and every $\chi\in \calX$ extends to its stabilizer $A_\chi$. When this occurs, \index{extension map} an \textit{extension map} is any $A$-equivariant $\Lambda\colon \calX\to\sqcup_{X\leq I\leq A }\Irr(I)$ such that for any $\chi\in\calX$, $\Lambda(\chi)\in \Irr(A_\chi)$ is an extension of $\chi$. Such a map (in particular satisfying $A$-equivariance) always exists as soon as maximal extendibility holds with respect to $X\unlhd A$ for $\calX$. When no set $\calX$ is specified, maximal extendibility \wrt $X\unlhd A$ means it holds for $\calX=\Irr(X)$.
	
	We denote by $\II{Lin X}@{\Lin(X)}$ the set of linear characters of $X$. 
\end{notation}{} 

Whenever maximal extendibility holds with respect to $X\unlhd A$ for a subset $\calX$ of $\Irr(X)$, let Cliff$(A\mid\calX)$ be the set of $A$-conjugacy classes of pairs $(\chi,\eta)$ with $\chi\in\calX$ and $\eta\in \Irr (A_\chi\mid 1_X)$. After an extension map $\Lambda$ has been chosen, Clifford theory (see for instance \cite[Sect. 1.8]{Navarro_book}) leads to the bijection \begin{align}\label{Cliff}\text{Cliff}(A\mid\calX)\ {\xrightarrow{\sim}} \ \Irr(A\mid \calX) \text{\ \ \   by \ \ \   } (\chi,\eta)\mapsto (\Lambda(\chi)\eta)^A. \end{align} 
This fact and its variants are key for exploring the characters of local subgroups such as the group $M$ defined later in the paper. 

Let us gather here some situations where extendibility is ensured.

\begin{prop} \label{ExtCrit} Assume $X\unlhd A$. 
	\begin{enumerate}\item If $A/X\Z(A)$ is cyclic, then maximal extendibility holds \wrt $X\unlhd A$.
		
		\item If $A/X$ is abelian and maximal extendibility holds  \wrt $X\unlhd A$, then for any subgroup $X\leq J\leq A$ we have maximal extendibility \wrt $X\unlhd J$ and to $J\unlhd A$.
		
		\item Let $n\geq 1$ and denote by $\II{Sn}@{\Sym_n}$ the corresponding symmetric group. If maximal extendibility holds \wrt $X\unlhd A$ for some set $\TT\subseteq \Irr(X)$ then it holds \wrt $X^n\unlhd A\wr \Sym_n$ for the subset $\TT^n:=\{ \chi_1\times\dots\times \chi_n\mid \chi_i\in\TT \}$ of $\Irr(X^n)$.	
		\item If $X$ is abelian and $E\leq A$ satisfies $A=XE$, then maximal extendibility holds \wrt $X\unlhd A$ for $\{\la\in\Lin(X)\mid \la([E_\la,E_\la]\cap X)=1 \}$. The latter is the whole $\Irr(X)$ when in addition \maex holds for $E\cap X\unlhd E$, e.g. when $A$ is a semidirect product $A=X\rtimes E$.
		\item Let $\chi\in \Irr(X)$ and assume $A=XV$ where $V\leq A$ is such that $\chi_0:=\restr\chi|{V\cap X}\in\Irr(V\cap X)$ and extends to some $\wt\chi_0\in\Irr(V_\chi)$. Then $\chi$ extends to $A_\chi =X V_\chi$ and there exists some $\wt\chi\in\Irr(X V_\chi)$ extending $\chi$ with $\restr{\wt\chi}|{V_\chi} = {\wt\chi_0} $. 
		
	\end{enumerate}
\end{prop}

\begin{proof}
	For (a), use \cite[Cor. 11.22]{Isa}. For (c), see \cite[Lem. 2.6]{S21D2}. For (e), see \cite[Lem.~4.1]{S10b}.
	
	 In (b) only maximal extendibility \wrt $J\unlhd A$ is nontrivial. But then we can use the fact that since $A/X$ is abelian then \maex is equivalent to elements of $\Irr(A)$ having multiplicity-free restrictions to $X$, see for example \cite[1.A]{S21D2}. Now restrictions to $J$ can't have any multiplicity $\geq 2$. 
	
For (d) note first that $[X,E_\la]$ is a subgroup of $\ker(\la)$. So if moreover $\la ([E_\la,E_\la]\cap X)=1 $ then indeed $[A_\la ,A_\la]\cap X\leq \ker(\la)$ since $A_\la =XE_\la$. We can then change $X\unlhd A_\la$ into $X/(X\cap [A_\la,A_\la])\unlhd A_\la/[A_\la,A_\la]$ thus reducing the problem to the case when $A_\la$ is abelian. The extension problem is then easy. When in addition \maex holds for $E\cap X\unlhd E$, then for any $\la\in \Irr(X)$, one has $\restr\la|{E\cap X}=\restr\wt\la|{E\cap X}$ for some (linear) character $\wt\la$ of $E_\la\leq E_{\restr\la|{E\cap X}}$ and therefore $\la ([E_\la,E_\la]\cap X)=1 $ since $\wt\la ([E_\la,E_\la] )=1 $. 
\end{proof}  

When a semidirect product $B\rtimes C$ acts on a group $X$ we are interested in characters of $X$ whose stabilizer in $B C$ decomposes as $B' C'$ with $B'\leq B $ and $C'\leq C$. This property can be given a different formulation in terms of $C$-stable $B$-transversals (\cite[Lem. 2.4]{S21D1}). See also the reformulation of the $\textbf{A}(\infty)$ condition in \Cref{Ainfty} or \Cref{prop_23}(iii).

The following statement allows one to deduce an extension map from a given extension map for a given transversal.

\begin{prop}\label{Ext_f} Assume $X\unlhd A\unlhd\wh A$ with $X\unlhd\wh A$. Let $\wh A_0\leq \wh A$ with $\wh A= A\wh A_0 $ and $A=A_0X$ for $A_0:= \wh A_0\cap A$. Let $X_0:=\wh A_0\cap X$ and assume $A/A_0\cong X/X_0$ abelian.
	
	Assume \maex holds \wrt $X_0\unlhd X$ and assume there is an $\wh A_0 $-equivariant extension map \wrt $X_0\unlhd A_0$ for $\TT$ an $\wh A_0$-stable $X$-transversal of $\Irr(X_0)$. 
	
	Then $X\unlhd A$ satisfies \maex with a $\Lin(A/A_0)\rtimes \wh A_0$-equivariant extension map $\wt \Lambda$. Namely when $\rho\in\Irr(X)$, $a\in \wh A_0$ and $\la\in \Lin(A/A_0)$ is seen as character of $A$ then $$\wt\La(\rho^a)=\wt\La(\rho)^a \text{\ \ and \ \ } \wt\La(\restr \la|{X}\rho)= \restr \la|{A_\chi}\wt\La(\rho).$$
\end{prop} 
 
\noindent\begin{minipage} {0.5\textwidth} 
		\setlength{\parindent}{1em}
		\noindent \textit{Proof.} We essentially recall the proof of \cite[Thm 4.2]{CS18B}. 
		Let $\Lambda$ be the extension map \wrt $X_0\unlhd A_0$ for $\TT$. Let $\rho\in \Irr(X)$. By the transversality property of $\TT$, $\rho$ lies over a unique character $\rho_0\in \TT\subseteq \Irr(X_0)$. Since maximal extendibility is assumed with respect to $X_0\unlhd X$, Clifford theory ensures that there exists an extension $\wt \rho_0\in\Irr(X_{\rho_0})$ of $\rho_0$ with $\wt \rho_0^{X}= \rho $, see (\ref{Cliff}). Note that according to \Cref{ExtCrit}(e) there exists a common extension $\rho_0'$ of $\wt \rho_0$ and $\restr \Lambda(\rho_0)|{A_{\wt \rho_0}}$ to $X_{\rho_0} A_{\wt \rho_0}$. Then $\restr(\rho_0')^{A_{\rho}}|{X}=(\restr\rho_0'|{X_{\rho_0}})^X=\rho$ since $A_\rho= X A_{\wt \rho_0} $ and $A_{\rho_0}= X_{\rho_0}(A_0)_{\rho_0}$ holds by the assumption on $\TT$ as recalled before the Proposition. So $\wt \Lambda(\rho):=(\rho_0')^{A_{\rho}}$ is an extension of $\rho$ to $A_\rho$. It is then easy to check from its construction and the $\wh A_0$-equivariance of $\La $ that $\wt \Lambda$ is $\Lin(A/A_0)\rtimes \wh A_0$-equivariant.\qed\end{minipage}
	\hspace{0.1\textwidth}
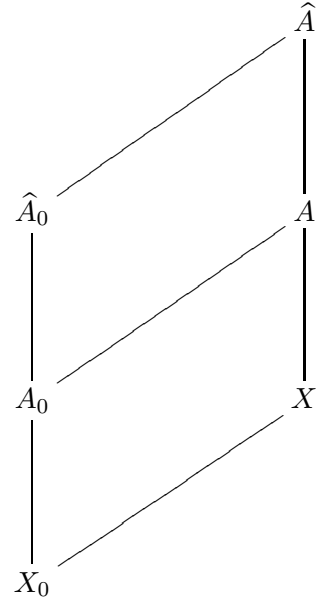
	\begin{minipage}{0.45\textwidth}
		$\xymatrix{&&&\wh A\ar@{-}[dd]\ar@{-}[llldd]\\ &&&\\ \wh A_0\ar@{-}[dd]&&&A\ar@{-}[dd]\ar@{-}[llldd]\\ &&&\\ A_0\ar@{-}[dd]&&&X\ar@{-}[llldd]\\ &&&\\ X_0&&&\\ &&&}$
	\end{minipage}%\renewcommand{\qedsymbol}{}

The following statement describes situations where the stabilizer of a character writes as a semidirect product and has extensions whose stabilizers have analogous properties.

\begin{lem} \label{lem_stab_extensions}
Let $A$ be a finite group, $  X\unlhd A$ and $\wt X\unlhd A$ such that $X \leq \wt X$ and $\wt X/X$ is abelian. Suppose that there exists a subgroup $Y\leq A$ such that $A=\wt X Y$, $X=Y\cap \wt X$ and $Y/X$ is abelian. Let $L\lhd A$ with $L=J(Y\cap L)$ for $J:=\wt X\cap L$ and abelian $L/X$. Let $\phi\in\Irr(X)$ with $A_\phi=\wt X_\phi Y_\phi$. Assume that $\phi$ extends to $\wt X_\phi$ and to $(Y_\phi)_{\wh\phi}$ for every $\wh\phi\in\Irr(J_\phi\mid \phi)$.
\begin{enumerate}
\item 
Then $\phi$ extends to $J_\phi$
and every extension $\wh \phi$ of $\phi$ to $J_\phi$ satisfies 
\[(A_\phi)_{\wh \phi}=\wt X_{\wh \phi} (Y_\phi)_{\wh \phi}.\]
\item Every $\kappa\in\Irr(J \mid \phi)$ satisfies $(\wt X Y)_\kappa=\wt X_\kappa Y_\kappa$ and extends to $J Y_\kappa$.
\item Every $\chi\in\Irr(L \mid \phi)$ satisfies $(\wt X Y)_\chi=\wt X_\chi Y_\chi$ and extends to $L Y_\chi$.
\item Let $z\in \Z(A)\cap Y$. Assume that $\phi$ has an extension $\phi'$ to $(Y_{ \phi})_{\wh \phi}$ with $z\in \ker(\phi')$. Then every $\chi\in \Irr( L \mid \phi)$ has an extension to $LY_\chi$ that contains $z$ in its kernel. 
\end{enumerate}
\end{lem}
Note that $Y_\phi$ in this situation stabilizes $J_\phi$ and hence acts on $\Irr(J_\phi\mid \phi)$. Now $(Y_\phi)_{\wh \phi}$ denotes the stabilizer of $\wh \phi$ in $Y_\phi$. By this construction  $(Y_\phi)_{\wh \phi}$ does not in general contain $J_\phi$, the group on which $\wh \phi$ is defined. 

\noindent
\begin{minipage} {0.5\textwidth} 
		\setlength{\parindent}{1em}
		\noindent \textit{Proof.} We first prove part (a):  As $J\leq \wt X$ and $\phi$ extends to $\wt X_\phi$ by assumption it is clear that $\phi$ extends to $J_\phi$.
  Let $\beta\in\Irr(\wt X_\phi)$ be an extension of $\phi$. Then $\wh \phi$ and $\restr\beta|{J_\phi}$ are extensions of $\phi$ so $\wh \phi =\lambda(\restr\beta|{J_\phi})$ for some $\la\in\Lin(J_\phi/X)$. Since $\wt X/X$ is abelian, $\la$ is $\wt X_\phi$-invariant, but then $\lambda(\restr\beta|{J_\phi})=\wh \phi$ is also $\wt X_\phi$-invariant, i.e., $\wt X_{\wh \phi}=\wt X_\phi$. As 
  \[\wt X_\phi Y_\phi=A_\phi\geq (A_\phi)_{\wh  \phi} \geq \wt X_{\wh \phi},\] 
  this implies $(A_\phi)_{\wh \phi}=\wt X_{\wh \phi} (Y_\phi)_{\wh \phi}$ as required. 		
	\end{minipage}
	\hspace{0.06\textwidth}
	\begin{minipage}{0.45\textwidth}
		$\xymatrix{&&A\ar@{-}[lldd]\ar@{-}[rd]\\ &&&Z \ar@{-}[rd] \ar@{-}[ld]\\ \wt X\ar@{-}[dr]&&L \ar@{-}[rd]\ar@{-}[ld]&&Y\ar@{-}[lldd]\\ &J\ar@{-}[rd]&&\\&&X}$\!\!\!\! \!\!\!\! \!\!\!\! \!\!\!\! \!\!\!\! \!\!\!\! \!\!\!\! \!\!\!\! \!\!\!\! \!\!\!\! \!\!\!\!  \!\!\!\! \!\!\!\! $\xymatrix{ \\ \\ \\ \\ Y\cap L  }$
	\end{minipage}

 Let now $\kappa=(\wh \phi)^J$. As $\phi$ extends to $Y_{\wh \phi}$ there exists some extension $\wt \phi$ of $\wh \phi$ to $Z_\whphi=J_\phi Y_\whphi $ for $Z=J Y$ according to \Cref{ExtCrit}(e). Because of $Z_{\wh \phi}= J_{\wh \phi} Y_{\wh \phi}$, we have $JY_{\wh \phi }= JY_\kappa =Z_{\kappa}$ and $(\wt\phi)^{JZ_{\wh\phi}}$ is an extension of $\kappa$ to $Z_{\kappa}=JY_{\kappa}$. Then every extension of $\whphi$ to $Z_\whphi=(JY)_{\wh \phi} $ defines by induction a character of $Z_{\kappa}=JY_{\kappa}$. This proves the statement in (b). 
 
	Since $Y/X$ is abelian, we see that every extension $\wt \kappa$ of $\kappa$ to $L_\kappa$ is $Y_\kappa$-invariant. The character $\chi=\wt \kappa^L$ is irreducible and $Y_\kappa$-invariant. The group $A_{\wt \kappa}$ contains $Y_{\wt \kappa}=Y_\kappa$. Recall $A_{\kappa}=\wt X_{\kappa} Y_\kappa$ and hence 
\[ Y_\kappa=Y_{\wt \kappa}\leq A_{\wt \kappa }\leq A_\kappa= \wt X_\kappa Y_\kappa.\]
This shows $A_{\wt \kappa}=\wt X_{\wt \kappa } Y_{\wt \kappa}$, implying $A_\chi= \wt X_{\chi} Y_\chi$. 
Let $\wh \kappa$ be an extension of $\wt \kappa$ to $Z_{\kappa}$, which exists as $\kappa$ extends to $Z_\kappa$ and $Y/X$ is abelian. Then $\wh \kappa^{Z_\kappa L}$ is irreducible and an extension of $\chi$. This proves part (c). 

For part (d), by assumption we have that $\phi$ extends to $Y_{\whphi}$ such that $z$ is contained in the kernel of some extension or equivalenty that $\phi$ extends to $Z_{\whphi}/\spa{z}$. By the above construction, $\wh \phi$ extends to $J_\phi Y_{\wh \phi}/\spa{z}$ and hence $\kappa=(\wh\kappa)^J$ extends to $Z_\kappa /\spa{z}$ and accordingly $\chi=\wt\kappa^L$ extends to $Z_\chi/\spa{z}$.
\qed

\subsection{Character triples and inductive McKay condition}\label{sec2_B}

Character triples {\index{character triples}} $(A,X,\chi)$ (where $X\unlhd A$ and $\chi\in\Irr(X)^A$) and certain relations on them help characterizing the situations where $\chi$ does not extend to $A$. See [Isa, 11.24] for some early occurrence of the notion and \Cref{def:2_1} below for the relation $\II{geqc}@{\geq_c}$. We recall more recent results and a reduction theorem for McKay's conjecture in that language (see \cite{Navarro_book}, \cite{S18}, \cite{Rossi_McKay}). 

By a classical application of Schur's lemma, each character triple $(A,X,\chi)$ defines a projective representation $\calP\colon A\to \GL_{\chi(1)}(\mathbb C)$ whose restriction to $X$ is a linear representation affording the character $\chi$, see \cite[Ch. 11]{Isa}, \cite[Def.~5.2]{Navarro_book}.

\begin{defi} \textbf{ --- Centrally isomorphic character triples \cite[Def. 10.14]{Navarro_book}.} \label{def:2_1}
	Let $(A,X,\chi)$ and $(H,M,\chi')$ be two character triples with $\Cent_A(X)\leq H\leq A$, $A=XH$, and $H\cap X=M$. We write 
	\[ (A, X,\chi)\geq _c (H,M,\chi'),\] 
	if two projective representations $\calP$ and $\calP'$ of $A$ and $H$ associated with $\chi$ and $\chi'$ exist such that the factor sets of $\calP$ and $\calP'$ coincide on $H\times H$ and such that for every $x\in \Cent_A(X)$, the matrices $\calP(x)$ and $\calP'(x)$ are scalar matrices to the same scalar.
\end{defi}

When dealing with $\geq_c$ the following lemma will be useful. 
\begin{lem}\label{rem_chartrip}
	Let $X\unlhd A$ and $H\leq A$ with $A=XH$ and $\Cent_A(X)\leq H$. We write $M:=X\cap H$. Let $\chi\in \Irr(X)$ and $\chi'\in\Irr(M)$.
	\begin{thmlist}
		\item Then 
		\[(A,X,\chi)\geq _c (H,M,\chi'),\]
		if there exist $\wt \chi\in \Irr(A)$ and $\wt \chi'\in\Irr(H)$
		with $\restr \wt \chi| X=\chi$, $\restr \wt \chi'|M=\chi'$ and $\Irr(\restr\wt\chi|{\Cent_A(X)}  )=\Irr(\restr\wt\chi'|{\Cent_A(X)}  )$.
		\item Assume $(A,X,\chi)\geq _c (H,M,\chi')$. Let $\kappa\in\Irr(\Cent_A(X))$. Then $\chi$ has an extension to $A$ that belongs to  $\Irr(A\mid \kappa)$ if and only if $\chi'$ has an extension to $H$ that belongs to  $\Irr(H\mid \kappa)$. 
		\item Assume $(A,X,\chi)\geq _c (H,M,\chi')$ and let $J$ be a group with $X\leq J \leq A$. Then $(J,X,\chi)\geq _c (J\cap H,M,\chi')$
	\end{thmlist}
\end{lem}
\begin{proof} Parts (a) and (b) follow from Lemma 2.15 of \cite{S18}. Part (c) is straightforward.
\end{proof}

Let us now recall the inductive McKay condition (\textbf{iMK}) from our introduction.

\begin{defi}\label{iMK} \textbf{(iMK)} Let $X$ be a finite group, $\ell$ a prime, $S$ a Sylow $\ell$-subgroup of $X$ and $\Gamma :=\Aut(X)_S$. We say that $X$ \textit{satisfies the inductive McKay condition} (\textbf{iMK}) for the prime $\ell$ whenever
	
	\begin{enumerate}
		\item there exists $\NNN_X(S)\leq N\leq X$ such that $N$ is stable under $\Gamma $, with $N\neq X$ when $\NNN_X(S)\neq X$, and 
		
		\item there exists a $\Gamma$-equivariant bijection $$\Omega_{X,\ell}\colon \Irr_{\ell ' }(X)\to \Irr_{\ell ' }(N)$$ such that $(X\rtimes \Gamma_\chi, X,\chi)\geq_c (N\rtimes \Gamma_\chi, N, \Omega_{X,\ell}(\chi))$ for any $\chi\in\Irrl(X)$.
	\end{enumerate}
\end{defi}

In our later considerations, when proving (\textbf{iMK}) for a given $X$ it will be crucial to use the fact that it is known already for certain subgroups of $X$.

An important tool to deal with character triple equivalences is the following Butterfly Theorem \cite[Thm 2.16]{S18}, \cite[Thm~10.18]{Navarro_book}. Among other things, it is crucial to navigate through various variants of (\textbf{iMK}).

\begin{theorem}\label{Butterfly} 	Let $(A,X,\chi)$ and $(H,M,\chi')$ be two character triples with \[(A, X,\chi)\geq _c (H,M,\chi').\] We assume that 
	$X\unlhd A'$ and define $\kappa\colon A\to \Aut(X)$, $\kappa ' \colon A'\to \Aut(X)$ the maps induced by conjugation. If $\kappa(A)=\kappa '(A')$ then $$(A', X,\chi)\geq _c (\kappa'{}\inv\kappa(H),M,\chi').$$ 
\end{theorem} 

\begin{rem}[\textbf{Relation with }\cite{IMN}]\label{IMNiMK}
Let $\ell$ be a prime and let $X$ be the universal covering of a finite simple group $\ov X:=X/\Z(X)$. As an application of the above one can see that \textit{$\ov X$ satisfies the inductive McKay condition of \cite[\S 10]{IMN} for $\ell$ if, and only if, $X$ satisfies the condition (\textbf{iMK}) of \Cref{iMK} for $\ell$}. See also a version of the reduction theorem of \cite{IMN} using the $\geq_c$ relation in \cite[Def. 10.23, Thm 10.26]{Navarro_book}.

Indeed, as explained in \cite[Prop. 2.3]{S12}, $\ov X$ satisfies the conditions (1-8) of \cite[\S 10]{IMN} if and only if the group theoretic conditions of \Cref{iMK}(a) are satisfied by $X$, $\ell$, a Sylow $\ell$-subgroup $S$ of $X$, along with a subgroup $N$, and there is an $\Aut(X)_S$-equivariant bijection $\Omega\colon \Irr_{\ell ' }(X)\to \Irr_{\ell ' }(N)$ such that each pair $(\chi, \Omega(\chi))$ satisfies the condition (\textbf{cohom}) introduced in Definition 2.4 of \cite{S12}. The latter means that there exists a group $\ov A$ with $\ov G:= X/Z \unlhd \ov A$ for $Z:=\ker( \chi)\cap {\Z(G)}=\ker( \Omega(\chi))\cap {\Z(G)}$ such that, by the construction in the proof of \cite[Prop. 2.8]{S12}, $\ov A$ induces the whole $\Aut(X)_{S,\chi}$ on $X/Z$, and the characters $\ov{\chi}$ and $\ov{\Omega ( \chi)}$ of $X/Z$ and $\ov N:=N/Z$ associated with $\chi$ and $\Omega ( \chi)$ have \textit{extensions} to $\ov A$, resp. $\NNN_{\ov A}( \ov N)$ lying above the same $\eps\in\Irr(\Cent_{\ov A}(\ov G))$. According to \Cref{rem_chartrip}(a) and thanks also to the group theoretic properties of $N$ and $X$, this condition (\textbf{cohom}) is equivalent to the existence of an overgroup $\ov A$ of $\ov G$ inducing $\Aut(X)_{S,\chi}$ on $X/Z$ and such that 
\[ ( \ov A, \ov G,\ov \chi) \geq_c (\NNN_{\ov A}( \ov N), \ov N ,\ov{\Omega(\chi)} ) .\] 

But then the Butterfly \Cref{Butterfly} implies that $$((X/Z)\rtimes \Aut(X)_{S,\chi}, X/Z,\ov\chi)\geq_c ((N/Z)\rtimes \Aut(X)_{S,\chi}, N/Z, \ov{\Omega{}(\chi)}).$$ It is then trivial to lift that into the relation $$(X\rtimes \Aut(X)_{S,\chi}, X,\chi)\geq_c (N\rtimes \Aut(X)_{S,\chi}, N, \Omega{}(\chi))$$ defining (\textbf{iMK}).
\end{rem}

The above is key to rephrase Isaacs--Malle--Navarro's reduction theorem on the McKay conjecture as the following criterion taken from \cite{Rossi_McKay}.

\begin{theorem}\label{RossiMK} Let $\ell$ be a prime. 
If any universal covering $X$ of a finite nonabelian simple group (see \cite[Def 5.1.1]{GLS3}) satisfies the above {(\textbf{iMK})} for $\ell$, then any finite group satisfies it.
\end{theorem} 

\begin{proof} 
	To see the equivalence between this and \cite[Thm B]{Rossi_McKay} we must, for $X$ the universal covering of a nonabelian simple group, replace the overgroup $X\rtimes \Gamma$ in our (\textbf{iMK}) by any overgroup $A$. As noted in the comment after Conjecture A in \cite{Rossi_McKay}, this is a straightforward application of the above \Cref{rem_chartrip}(c) and the Butterfly \Cref{Butterfly}.\end{proof}

\begin{theorem}\label{CS1319} Let $\ell$ be a prime and $X$ the universal covering of a finite simple group $\ov X$. If $\ell\geq 5$, assume $\ov X$ is not a group of Lie type $\tD$ or $^2\tD$ in characteristic $\neq\ell$. Then $X$ satisfies the above {(\textbf{iMK})} for $\ell$.
\end{theorem} 

\begin{proof} 
By the Classification of Finite Simple Groups, a finite nonabelian simple group is either an alternating group, a sporadic group or a simple group of Lie type, see \cite[Ch. 47]{Asch}. We have seen in \Cref{IMNiMK} that for the universal cover $X$ of a simple group and a prime $\ell$, the condition (\textbf{iMK}) is equivalent to the simple group $X/\Z(X)$ satisfying the inductive McKay condition of \cite{IMN}. This is the condition checked in \cite[Thm 3.1 and Thm 5.1]{ManonLie} for alternating groups and sporadic groups. For simple groups of Lie type in characteristic $\ell$, this is checked as \cite[Thm 1.1]{S12}. For $\ell =2$ and all types, this is the main result of \cite{MS16}. For $\ell =3$ and all types, see \cite[Thm C]{S21D2} and its proof. For other cases of groups of Lie type, \cite[Ch. 16 and 17]{IMN} covers the types $^2\tB_2$ and $^2\tG_2$, \cite[Thm A]{CS13} covers the types $^3\tD_4$, $\tE_8$, $\tF_4$, $^2\tF_4$, and $\tG_2$, \cite{CS17A} the types $\tA$ and $^2\tA$, \cite{CS17C} the type $\tC$, \cite[Thm A]{CS18B} the types $\tB$, $\tE_6$, $^2\tE_6$ and $\tE_7$. This clearly leaves out only types $\tD$ or $^2\tD$ for primes $\ell\geq 5$ different from the defining characteristic.
\end{proof}

In the cases reviewed in the above proof, (\textbf{iMK}) was mostly ensured via the criterion \cite[Thm 2.12]{S12} leading to the conditions \textbf{A}$(\infty)$, \textbf{A}$(d)$ and \textbf{B}$(d)$ described in the Introduction, see also \Cref{iMKbyAdBd}. At some point we will need the following slight reinterpretation of \cite[Thm 2.12]{S12}.
\begin{prop}\label{prop_23}
	Let $\wG\unrhd G\geq N$ and $E$ be finite groups, such that $E$ acts on $\wG$ normalizing $G$,
	$\NNN_G(N)=N$, $\Cent_{\wG \rtimes E}(G)=\Z(\wG)$, $G\rtimes E=G \wh N$ and $\wG=G \wt N$, 
	where $\wt N:=\NNN_\wG(N)$ and $\wh N:=\NNN_{GE}(N)$. 
	Let $\calG \subseteq \Irr(G)$ and $\calN \subseteq \Irr(N)$ be $\wt N \wh N$-stable. Assume the following: 
	\begin{asslist}
		\item The quotient $\wG/G$ is abelian. Maximal extendibility holds \wrt $G\unlhd \wG$ for $\calG $ and \wrt $N \unlhd \wt N $ for $\calN $.
		\item \label{crit_S12_bij}
		Denoting $\wt\calG:= \Irr(\wG \mid \calG )$ and $\wt\calN:= \Irr(\wt N\mid \calN )$, there exists some $\Lin(\wG/G)\rtimes \wh N $-equivariant bijection 
		$$\wt \Omega :\wt\calG \lra \wt\calN,$$
		such that $\wt \Omega\big(\wt\calG \cap \Irr(\wG \mid \xi)\big) =\wt\calN \cap \Irr(\wt N \mid \xi)$ for every $\xi \in \Irr(\Z(\wG))$. 
		\item There exists some $E$-stable $\wG$-transversal $\calG_0 $ in $\calG $, such that every $\chi\in\calG_0 $ extends to $GE_\chi$.
		\item There exists some $\wh N$-stable $\wt N$-transversal $\calN_0 $ in $\calN $, such that every $\psi\in\calN_0 $ extends to $\wh N_\psi$.
	\end{asslist}
	Then there exists some $\wt N \wh N$-equivariant bijection 
	\[\Omega:\calG \lra \calN ,\] 
	such that 
	$$( (\wG E)_\chi,G,\chi )\geq_c((\wt N \wh N)_\chi,N,\Omega(\chi) )	\forevery \chi\in\calG .$$
\end{prop} 
\begin{proof} 
	The proof is essentially the same as the one of Theorem~2.12 in \cite{S12}. Let us comment on the differences between the two statements. 
	In \cite{S12} the focus was on the characters in $\Irrl(G)$ for some prime $\ell$ but the degree of the characters did not play any role in the arguments used. The group theoretical assumptions in \cite[Thm 2.12]{S12} are also related to a prime $\ell$ and the corresponding Sylow subgroup but they are only used through the fact that $N$ is self-normalizing in $G$ along with $G E=G \wh N$ and $\wG=G \wt N$ in $\wG E$. The assumptions made in (iii) and (iv) are equivalent to the original assumptions of \cite[Thm 2.12]{S12} in terms of stabilizers thanks to \cite[Lemma~2.4]{S21D1}.
	 
	The proof of \cite[Thm 2.12]{S12} then yields a $\wt N \wh N$-equivariant bijection $\Omega:\calG \lra \calN $ such that for any $\chi\in\calG$ the pair $(\chi,\Omega(\chi))$ satisfies the condition (\textbf{cohom}) introduced in Definition 2.4 of \cite{S12}. Then our discussion is very similar to the one given in \Cref{IMNiMK} above. The condition (\textbf{cohom}) implies then that there exists a group $\ov A$ with $\ov G:= G/Z \unlhd \ov A$ for $Z:=\ker( \chi)\cap {\Z(G)}$ such that $\ov A$ has the same image as $G(\wt N \wh N)_\chi$ in $\Aut(G)$, and the characters $\ov{\chi}$ and $\ov{\Omega ( \chi)}$ of $\ov G$ and $\ov N:=N/Z$, deflations of $\chi$ and $\Omega ( \chi)$, satisfy
	\[ ( \ov A, \ov G,\ov \chi) \geq_c (\NNN_{\ov A}( \ov N), \ov N ,\ov{\Omega(\chi)} ) . \]
	Now taking into account the Butterfly Theorem~\ref{Butterfly} we get 
	\begin{align*} ( (\wt G E)_\chi/Z, G/Z, \ov\chi) &\geq_c ((\wt N \wh N)_\chi/Z, N/Z , \ov{\Omega(\chi)} ) . \end{align*}
	
	The sought for relation $( (\wG E)_\chi,G,\chi )\geq_c((\wt N \wh N)_\chi,N,\Omega(\chi) )	$ is then an easy consequence of the definition of $\geq_c$.
\end{proof}

\subsection{Groups of Lie type, generators and automorphisms}\label{ssec2C}
We now make more precise the notation for $G$, $\wG$, $E$ and the conditions \textbf{A}$(\infty)$, \textbf{A}$(d)$ and \textbf{B}$(d)$ from our Introduction. Additional notation will be given in \Cref{sub2E} for type $\tD$.

%\begin{notation}\label{FinRedGps}
Let $\III{p}$ be a prime, $\III{q}=p^f$ for some $\III{f}\geq 1$, $\II{F}@{\FF}$ be an algebraic closure of $\II Fq@{\FF_q}$, the field with $q$ elements. We consider $\II Gbold@{\bG}$ a simple simply connected linear algebraic group over $\FF$ with the choice of a maximal torus and a Borel subgroup $\II{T0}@{\bT}\leq \bB$, thus fixing a root system $\II{Phi}@{\Phi(\bG,\bT)}$ with a basis $\II{Del}@\Delta$. There is a presentation by the generators $\II{xal}@{\xx_\al(t_1)}$, $\II{nalphat}@{ \nn_\al(t_2)}:=\xx_\al(t_2)\xx_{-\al}(-t_2\inv)\xx_\al(t_2)\in\NNN_{ \bG}(\bT)$ and $\II{halphat}@{\hh_\al(t_2)}:=\nn_\al(t_2)\nn_\al(1)\inv\in\bT$ for $\al\in \Phi(\bG,\bT)$, $t_1\in \FF$, and $t_2\in \FF ^\times$, subject to the Chevalley relations, see \cite[Thm 1.12.1]{GLS3}. In particular the commutator formula \cite[Thm 1.12.1(b)]{GLS3} has consequences that we use repeatedly, see also \ref{ZG_order}(d) below.

\begin{para}\label{Comm}
 If $\al,\beta\in \Phi(\bG ,\bT)$ are such that $(\ZZ \al+\ZZ\beta )\cap\Phi(\bG ,\bT)=\{ \pm\al ,\pm \beta \}$, then \[[\xx_\al(t),\xx_\beta(t')]=1\] for any $t,t'\in\FF$. This is in particular the case if $\al\perp\beta$ in type $\tD$, or in type $\tB$ with $\al\perp\beta$ and $\beta$ a long root. 
\end{para}

We denote by $\II{Fp}@{F_p}$ the bijective endomorphism of $\bG$ sending any $\xx_\al(t)$ to $\xx_\al(t^p)$. We denote by $\II {EG}@{E(\bG)}$ the group of abstract group automorphisms of $\bG$ generated by $F_p$ and the graph automorphisms of type $\xx_{\eps\delta}(t)\mapsto \xx_{\eps\delta '}(t)$ for $t\in \FF$, $\delta\in \Delta$, $\eps\in\{1,-1\}$ and $\delta \mapsto \delta'$ is a symmetry of the Dynkin diagram of $\Delta$. %Allowing only symmetries of order 1 or 2, we get the subgroup $\II UUE@{\UE(\bG)}\leq E(\bG)$ (with clearly an equality except in type $\tD_4$).

We denote by $\II F@{F\colon \bG\to\bG}$ a Frobenius endomorphism defining some $\FF_q$-structure on $\bG$ preserving $\bT\leq \bB$, namely $F=F_p^f\circ\sigma$ with $\sigma$ some graph automorphism (possibly trivial) as above. We denote by $\II{Lang}@{\calL}$ the \index{Lang map} {Lang map} on $\bG$ defined by $\calL(g)=g\inv F(g)$ for $g\in \bG$. \index{diagonal outer automorphism}

Let $\II{EGF}@{E(\GF)}$ be the image of $\Cent_{E(\bG)}(F)$ in $ \Aut(G)$ by restrictions to $G=\GF$. The kernel of the latter is the subgroup generated by $F$ (see \cite[Lem. 2.5.7]{GLS3}), so for instance stabilizers $E(\GF)_{\bS}$ for $F$-stable subsets $\bS$ of $\bG$ make sense. 

We assume chosen a so-called regular embedding $\bG\leq \II{Gtilde}@{\wbG}$ \index{regular embedding} with $\wbG$ also defined over $\FF_q$ with Frobenius endomorphism $F$ extending the one of $\bG$, and such that $\Z(\wbG)$ is connected and $\wbG =\bG \Z(\wbG)$. We can assume that the action of $E(\bG)$ extends to $\wbG$, see \cite[Sect. 2]{MS16} or \cite[Prop. 1.7.5]{GM}. The action of $\wbG^F$ on $\GF$ by conjugation provides all diagonal automorphisms of $\GF$, while $\wbG^F\rtimes E(\GF)$ can be formed and induces the whole $\Aut(\GF)$ on $\GF$ in the case where $\GF$ is quasisimple not of type $\tB_2$, $\tG_2$ or $\tF_4$, see \cite[Thm 2.5.12]{GLS3}.\index{diagonal outer automorphism}

\begin{notation}[Diagonal automorphisms of $\GF$]\label{not_diag}
	Concerning $\Out(\GF)$, note that $\Cent_{\wbG^F}(\GF)=\Z(\wbG^F)=\Z(\wbG)^F$ (see \cite[Lem. 6.1]{Cedric}) with $\bG\cap\Z(\wbG)=\Z(\bG)$. Therefore $\wbG^F/\GF\Cent_{\wbG^F}(\GF)=(\bG\Z(\wbG))^F/\bG^F\Z(\wbG)^F$ is isomorphic to the group of cofixed points $\II{ZGF}@{\wZ(\GF)}:=\Z(\bG)/[\Z(\bG),F]$ by Lang's theorem, the map being explicitly $gz\mapsto g\inv F(g)[\Z(\bG),F]$ whenever $g\in\bG$, $z\in \Z(\wbG)$ with $gz\in \wbG^F$. Thus, the group $\wZ(\GF)\rtimes E(\GF)$ acts on $\Irr(\GF)$.
\end{notation}

Note that $|\Z(\bG)^F|=|\Z(\bG)/[\Z(\bG),F]|$ and indeed since $\Z(\bG)$ is either cyclic or of order 4, the two groups $\Z(\bG)^F$ and $\Z(\bG)/[\Z(\bG),F]$ are isomorphic even as $E(\GF)$-groups. 

We often abbreviate $\II{G}@{G}:=\GF$ and $\II{Gtilde}@{\wG}:=\wbG^F$.

\begin{notation}[Overgroups ${\wb G}$]\label{Gbreve} A slightly different way of dealing with diagonal automorphisms of $\GF$ is as follows, see \cite[Rem. 2.16(a)]{S21D1}. From the fact that $\Cent_\bG(\GF)=\Z(\bG)$ (and therefore $\Z(\bG^F)=\Z(\bG)^F$) recalled above, it is easy to see that $\NNN_\bG(\bG^F)=\calL\inv(\Z(\bG))\leq \bG$. We set \[ {\wb G}=\calL\inv(\Z(\bG)).\] Observe more generally that $\GF=\calL\inv(1)\unlhd \calL\inv(Z)\unlhd {\wb G}$ with $|\calL\inv(Z)/\GF|=|Z|$ for any subgroup $Z\leq\Z(\bG)$. Moreover 
$\widebreve G$ induces on $\GF$ the whole group of diagonal automorphisms of $\GF$ since ${\wb G}/\Z(\bG)\cong (\bG/\Z(\bG))^F$ by the natural map. We denote by $E({\wb G})\leq \Aut({\wb G})$ the restriction of $\Cent_{E(\bG)}(F)$ to ${\wb G}$. Then, in a way similar to $\wbG^F\rtimes E(\GF)$, the overgroup ${\wb G}\rtimes E({\wb G})$ induces the whole $\Aut(\GF)$ on $\GF$. Note however that $\Cent_{\wbG^F E(G)}(G)=\Z(\wbG^F)$ while $\Cent_{ \wb G E(\wb G)}(G)=\Z(\bG)\spa{\restr F|{\wb G}}$.
\end{notation}

\begin{rem}
Note that in the above construction, one has ${\wb G}\leq \bG^{F^e}$ for $e$ the exponent of the finite group $\Z(\bG)\rtimes E'$ where $E'$ is the subgroup of $\Aut(\Z(\bG))$ generated by $\restr F|{\Z(\bG)} $. Indeed, if $g\in \bG$ and $g\inv F(g)=z\in \Z(\bG)$, then $g\inv F^e(g)=zF(z)\dots F^{e-1}(z)$ can be written as $(zF)^eF^{-e}$ in $\Z(\bG)\rtimes E'$.
\end{rem}

\subsection{The conditions (\textbf{iMK}), \Ad {} and \Bd}\label{ssec_2D}

Following the road map suggested by \Cref{prop_23} and starting with $G=\bG^F$ as in the preceding section, we review the choices of groups $\wG$, $E$, $N$ to be made and how the assumptions translate for a quasisimple group of Lie type. 
The choice for the group $\wG$ is then obviously $\wbG^F$, while $E$ is $E(\GF)$.

Assumption \ref{prop_23}(iii) has the following generalization, introduced in \cite{CS17C}. 

\begin{para}[Condition] $\II Ainf @{\mathbf{A}(\infty)}$:\label{Ainfty}
	 \index{Condition $\mathbf{A}(\infty)$}
	There exists some $E(G)$-stable $\wG$-transversal $\TT$ in $\Irr(G)$, where every $\chi\in \TT$ extends to $G E(G)_\chi$.\end{para}
 For the equivalence between the above version of \textbf{A}$(\infty)$ and the version in terms of stabilizers used in the Introduction, see \cite[Lem. 2.4]{S21D1}.
\begin{theorem}[{\cite[Theorem~A]{S21D2}}] \label{Julia2}
	Condition $\mathbf{A}(\infty)$ holds for any simple simply connected $\bG$ and Frobenius endomorphism $F$.
	%There exists some $E(\GF)$-stable $\wbG^F$-transversal $\TT$ in $\Irr(\GF)$, where every $\chi\in \TT$ extends to $\GF E(\GF)_\chi$.
\end{theorem}

Let us fix $d\geq 1$. Recall the notion of ($F$-stable) $d$-tori in $(\bG,F)$, see \cite[Def. 25.6]{MT} and the corresponding Sylow theory. 

Let $\bS$ be a Sylow (i.e. maximal) $d$-torus in $\bG$.
Let ${N}:=\NNN_{\GF}(\bS)$, ${\wh N}:=\NNN_{\GF E(\GF)}(\bS)$, ${\wt N}:=\NNN_{\wGF }(\bS)$ and ${\wt C}:=\Cent_{\wGF }(\bS)$. We recall the following conditions already seen in our introduction:

\begin{para}[Condition $\II Ad@{\mathbf{A}(d)}$]\label{cond_Ad}
	\index{Condition \Ad} 
	There exists some $\wh N$-stable $\wt N$-transversal $\MM$ in $\Irr(N)$, where every $\chi\in \MM$ extends to $(\GF E(\GF))_{\bS,\chi}$.
\end{para}

\begin{para}[Condition $\II Bd@{\mathbf{B}(d)}$]\label{cond_Bd} \index{Condition \Bd}%\hfill\break
		\begin{thmlist}
		\item Maximal extendibility holds with respect to $N\unlhd \wt N$ and $\wt C\unlhd \wt N$. 
		
		\item There exists some $\Lin(\wG/G)\rtimes \wh N$-equivariant extension map $\wt \Lambda$ \wrt $\wt C\unlhd \wt N$. 
	\end{thmlist}
\end{para}

For $\ell$ an odd prime not dividing $q$, denote by $\II{dlq}@{d:=d_\ell(q)}$ the multiplicative order of $q$ in $(\ZZ/\ell\ZZ)^\times$. Then the strong relation between $d$-tori and $\ell$-subgroups allows us to choose $N$ as above to apply \Cref{prop_23} with conditions \Ad{} and \Bd{} essentially completing the assumptions \ref{prop_23}(i) and (iv). We get

\begin{theorem}[{\cite[Thm~2.4]{CS18B}}]\label{iMKbyAdBd}
	Assume that $\ell$ is a prime with $\ell\nmid 6q$, $\GF/\Z(\GF)$ is simple and $\GF$ is its universal covering group. Assume \textbf{B}$(d_\ell(q))$ and \textbf{A}$(d_\ell(q))$ are satisfied. Then the condition (\textbf{iMK}) of \Cref{iMK} holds for $\GF $ and $\ell$, taking for $N$ the subgroup $\NNN_{ \bG}(\bS)^F$ where $\bS$ is a Sylow $d_\ell(q)$-torus of $(\bG,F)$.
\end{theorem}
\begin{proof} Since \textbf{A}$(\infty)$ is always satisfied thanks to \Cref{Julia2}, \cite[Thm~2.4]{CS18B} tells us that the inductive McKay Condition of \cite{IMN} is satisfied for the simple group $\GF/\Z(\GF)$ and the prime $\ell$ with the above choice of $N$. But we have seen that it is equivalent to (\textbf{iMK}) in \Cref{IMNiMK}.
\end{proof}

\begin{rem}\label{rem_A1A2B1B2} If $d\in \{1,2\}$, $\bG=\tDlsc(\FF)$ and $F$ is a Frobenius endomorphism of $\bG$ defining it over $\FF_q$, then the Conditions \Ad{} and \Bd \ are known for $(\bG,F)$, see \cite[Thm 3.1]{MS16}. 
\end{rem}

\subsection{The curious case of types \texorpdfstring{$\tD_l$}{D} and \texorpdfstring{$\tB_l$}{B}}\label{sub2E}

In addition to what has been said about general groups $\GF$ with $\bG$ simple simply connected, we introduce here more notation, in particular related to the inclusion $\bG\leq \obG$ where $\bG$ has type $\tD_l$ and $\obG$ type $B_l$ for some $l\geq 4$, see also \cite[10.1]{S10a}, \cite[2.C]{MS16}, \cite[Sect. 2.1]{S21D1}. 

Let $e_1,\ldots ,e_l$ be the orthonormal basis of the $l$-dimensional Euclidean vector space $\oplus_{i=1}^l\RR e_i$ ($l\geq 4$). Set $\II {lu}@{\protect{\underline {l} }} =\{1, \dots ,l\}$ and 
$$\II Phi @{\Phi}:=\{ \pm e_i\pm e_j\mid i,j\in\ul , i\neq j \}\subseteq\II{Phi overline}@{\protect{\overline \Phi}}:=\{ \pm e_i, \pm e_i\pm e_j\mid i,j\in\ul , i\neq j \}.$$ 
These are root systems of type $\tD_l$ and $\tB_l$, respectively, with bases ${\Delta}:=\{\al_1, \al_2, \ldots , \al_l \}$ and $\II {DelOv}@{\overline \Delta}:=\{\ov \al_1, \al_2, \ldots , \al_l \}$, where $ \al_1:=e_1+e_2$, $\ov \al_1:=e_1$ and $\al_i:=e_i-e_{i-1}$ ($i\geq 2$), see \cite[Rem.~1.8.8]{GLS3}. Let $\II{G overline}@{\protect{\overline{ \mathbf G}}}:=\tB_{l,\mathrm{sc}}(\FF)$ be the simple simply connected linear algebraic group with root system $\ov \Phi =\Phi(\obG, \ov\bT)$ for some maximal torus $\ov\bT$. We recall ${ \xx_\al(t_1)}$, $\II{nalphat}@{ \nn_\al(t_2)}$ and $\II{halphat}@{\hh_\al(t_2)}$ its Chevalley generators, where $\al\in \ov \Phi$, $t_1\in \FF$ and $t_2\in \FF ^\times$. We denote by $ \bX_\al$ the group $ \xx_\al(\FF)$, $\II{N overline}@{\ov \bN}:=\spa{\nn_\al(t) \mid \al\in \ov \Phi, t \in \FF^\times} =\NNN_{ \obG}(\ov\bT)$ and 
$\II{Woverline}@{\protect{\overline W}}:=\ov\bN/\ov\bT$ the Weyl group of $\obG$ together with the canonical epimorphism $\II rhoov@{\protect{\overline \rho}: \ov \bN\lra \ov W}$. As explained in \cite[10.1]{S10a} and \cite[2.C]{MS16} the subgroup $\II{G}@{\bG}:=\Spann<\bX_\al| \al\in \Phi>$ is a simply connected simple group over $\FF$ with same maximal torus $\ov \bT=\bT$ and the root system $\Phi=\Phi(\bG,\bT)$ of type $\tD_l$. We set $\II{Nb}@{ \bN}:=\NNN_{ \bG}( \bT)$, 
$\II{W}@{\protect{ W}}:= \bN/ \bT$ and the surjection $\II{rho}@{\rho}:\bN\lra W$.

For $I\subseteq \ul$ and $\zeta\in\FF^\times$ we set $\II{hIzeta}@{ \h_I(\zeta):=\prod_{i\in I } \h_{e_i}(\zeta)}$. 

We assume chosen $\II{pivar}@{\varpi}\in\FFtimes$ with $\varpi^2=-1$ (note that $\varpi=1$ when $p=2$). We set $\II ncirc@{\neins:=\n_{e_1}(\varpi)}\in \NNN_{ \obG}(\bT)$. The Chevalley relations give easily the following statement.
\begin{lem}[{\cite[2.C]{MS16}}]\label{lem_gamma} 
	Let $\II{gamma}@{\gamma}: \bG \lra \bG$ be the graph automorphism defined by $\xx_{\epsilon \al_i}(t)\mapsto \xx_{\epsilon \al_i'}(t)$ for $t\in\FF$, $\eps =\pm 1$, $ i\in \ul$ and $(\al_1',\al_2',\al_3',\al_4',\dots ,\al_l'):=(\al_2,\al_1,\al_3,\al_4,\dots ,\al_l)$.
	Then $\neins$ normalizes $\bG$ and induces $\gamma$ on it, namely $\gamma(x)=x^{\neins}$ for any $x\in\bG$.
\end{lem}

Recall ${f}\geq 1$, ${q}=p^f$, and let $F\in \{ F_p^f, F_p^f\circ \gamma \}$ a Frobenius endomorphism of $\bG$. Then in the standard notation $\GF=\tD_{l,\mathrm{sc}}^\eps(p^f) $ where $\II eps@{\eps} =1$ or $-1$ according to $F= F_p^f$ or $F_p^f\circ \gamma $. We choose an extension of $F$ to $\obG$ as follows. If $F=F_p^f$ then $\o F$ is defined the same on the generators $\xx_\al(t) $ as in the preceding section. By contrast when $F=F_p^f\circ \gamma$ we define $\II{Fbar}@{\o F}:=F_p^f\circ \o\gamma$ where $\o\gamma$ is the inner automorphism of $\obG$ defined by $\o\gamma(x)=x^{\neins}$ for $x\in\obG$. In both cases $\obG^{\o F}=\tB_{l,\mathrm{sc}}(p^f) $.

 We denote by $\II Eun@{\underline{E}(\bG)}$ the subgroup of $E(\bG)$ generated by $F_p$ and $\gamma$. 
 We denote by $\II{Eunder}@{\UE(\GF)}$, 
 respectively $\Ispezial EunderbG@{\underline{E}(\protect{\breve{G}})}@{\underline{E}(\wb G)}$, the corresponding subgroup of $E(\GF)$, respectively $E(\wb G)$.

We describe below properties related to the centers of $\bG$ and $\GF$. These can easily be deduced from \cite[Thm 1.12.6]{GLS3}.

\begin{para}\label{ZG_order}
	\begin{thmlist}
		\item 	According to \cite[Table 2.2]{GLS3} the center of $\bG$ is the $2$-group generated by $\h_\ul(\varpi)$ and $\II h0@{h_0}:=\h_{\al_1}(-1) \h_{\al_2}(-1)=\h_{e_1}(-1)=(\neins)^2$. 
		
		\item 	We have $\spa{h_0}=[\Z(\bG),\gamma]= \Z(\bG\UE(\bG))=\Z(\obG)$. 
		
		\item $|\Z(\GF)|=4$ if and only if $p\neq 2$ and 
		\begin{asslist}
			\item $\eps=1$ and $4\mid (q-1)l$; or 
			\item $\eps=-1$ and $4\nmid (q-1)l$, in particular $2\nmid f$.
		\end{asslist}
		In all other cases: $|\Z(\GF)|=\gcd(2,q-1)$. 	
		\item If $\al,\beta\in \Phi(\obG,\bT)$ are both short roots and $\al\perp \beta$, then $[\nn_\al(t),\nn_\beta(t')]=h_0$ for any $t,t'\in\FFtimes$, see \cite[Bem. 2.1.7]{S07}.
	\end{thmlist}
\end{para}

%\newpage
%%%%%%%%%%%%%%%%%%%%%%%%%%%%%%%%%%%%%%%%%%%%%%%%%%%%%%%%%%%%%%%%%%%%%%%%%%%%%%%%%%%%%%%%%%%%%%%%%%%
%%%%%%%%
\section{Centralizers of semisimple elements and consequences for $\Irr(\GF)$}\label{sec_3}

From now on the groups $\bG$, $G=\bG^F$ are as defined in \Cref{sub2E} and we use the notation introduced there for groups of type $\tD$.

The aim of this chapter is to complement the results of \cite{S21D2} on elements of $\Irr(G)$ that don't extend to their stabilizer in $G\UE(G)$. We introduce the sets $\EE$ and $\DD$ in \Cref{def_TED} that will be of constant use in Chapter~\ref{sec_nondreg_groupM}. Through the equivariant Jordan decomposition of characters of \cite[Thm B]{S21D2}, analyzing elements of $\EE\cup\DD$ leads us to a study of centralizers of semisimple elements in the adjoint group of $\bG$. This extends the results of \cite{CS22}. It will be also of some use in our study of relative Weyl groups in \Cref{ssec4C}. 

Our main theorem is \Cref{thm_sumup_D} dealing with kernel and stabilizers of characters outside the transversal $\TT$ from \Cref{Julia2}. This includes a study of characters of groups of rank $1\leq k\leq 3$, called $G_{\underline{k}}$ in \ref{ssec:3D}, that occur naturally in centralizers of Sylow $d$-tori. 

\subsection{Centralizers of semisimple elements}
We describe here properties of the centralizers of semisimple elements $s_0\in \tDlsc(\FF)$. We keep $\bG=\tDlsc(\FF)$, $F$, $\bT$, and the notation for the associated roots as in \Cref{sub2E}.

\begin{notation}\label{not3.1} Let $\pi\colon\bG\to\bH$ be the adjoint quotient of $\bG$ with $F_p$, $\gamma$ and $F$ acting accordingly also on $\bH$. Let $\II piSO@{\pi_\SO}\colon \bG=\tDlsc(\FF)=\text{Spin}_{2l}(\FF)\to\SO_{2l}(\FF)$ be the the natural morphism with kernel $\spann<h_0>$, see \ref{ZG_order}. 
	
	 For every $I\subseteq \underline l $ we set  $ \II RoI@{\ov R_I}:=\ov \Phi\cap \spa{e_i\mid i \in I}_\ZZ ,\, 
	\II RI@{R_I}:=\Phi\cap \ov R_I =\Phi\cap \spa{e_i\mid i \in I}_\ZZ$ and $\II TI@{\bT_I} :=\bT\cap \spa{ \bX_\al \mid \al\in \ov R_{I} }$.
	
\end{notation}

\begin{rem}\label{rem:3:2}
	\begin{thmlist}
		\item Whenever $t_1,\dots ,t_l\in\FFtimes$, then $\pi_\SO(\prod_{i=1}^l \h_{e_i}(t_i))$ has the eigenvalues $\{ t_i^{\pm 2}\mid i \in \ul\}$ as an element of $\SO_{2l}(\FF) $, see the description of $\pi_\SO( \h_{e_i}(t_i))$ in \cite[2.7]{GLS3}.
		\item Assume $p$ is odd. Let $s_0\in \bG$ be semisimple. According to \cite[2A]{FS89}, $s_0$ and $s_0 h_0$ are $\bG$-conjugate if and only if $1$ and $-1$ are both eigenvalues of $\pi_\SO(s_0)$. Using then (a) above, an element $s_0\in \bT$ is $\bG$-conjugate to $s_0h_0$ if and only if $s_0$ can be written as $\prod_i \h_{e_i}(t_i)$ with $\{1,\varpi\}\subseteq \{ \pm t_i^{\pm 1}\mid i \in \ul\}$. 
	\end{thmlist}
\end{rem}

\begin{para} [Weyl groups and parabolic subgroups]\label{3:3'} We denote by $\Sym_{\pm l}$ the subgroup of the permutation group of $\ul\cup -\ul$ whose elements $\si$ satisfy $\si(-x)=-\si(x)$ for any $x\in \ul\cup -\ul$.
		The Weyl group ${\ov W} =\NNN_{\obG}(\bT)/\bT$ can be identified with $\II Spml@{\Sym_{\pm l}}$ as in \cite[Rem.~1.8.8]{GLS3} or \cite[Sect. 3.2]{S21D1}. 
		
		We write 
	$\II SDpml@{\Sym^\tD_{\pm l}}$ 
	for the normal subgroup of $\ov W$ consisting of permutations $\si$ with even $|-\ul\cap \si(\ul)|$. This coincides with $W=\NNN_{\bG}(\bT)/\bT$. 
	
	If $I\subseteq \ul$, then the ordinary symmetric group $\II SI@{\Sym_{I}}$ can be identified with the (parabolic) subgroup $\II SBI@{\Sym^\tB_I}$ of $\ov W=\Sym_{\pm l}$ fixing every element of $\ul \setminus I $ and stabilising $I$. For example $\Sym^\tB_I$ is the trivial group if $|I|= 1$. 
 %Old:   We define $\II SDpmI@{\Sym^\tD_{\pm I}}:=W\cap \Sym^\tB_I$.
 We define $\II SDpmI@{\Sym^\tD_{\pm I}}$ as the subgroup of elements in $W$ which fix any element of $\underline l\setminus I$ (hence also of $-\underline{l}\setminus -I$ and stabilize $I\cup -I$).

	Given a partition $\mathbb M$  of a set $M\subseteq \ul$  we set $\II SBM@{\Sym^\tB_{\mathbb M}}=\prod_{I\in\mathbb M}\Sym^\tB_I$ the direct product of subgroups $\Sym^\tB_I$ with $I\in \mathbb M$. 
\end{para}

\begin{para}[Elements of $\bT$ and the action of the Weyl group]\label{3:3}
	The elements $\spa{\h_{e_i}(t)\mid t\in \FF^\times , i \in \ul}$ generate $\bT$. Chevalley relations show 
	\[ \prod_{i=1}^l \h_{e_i}(t_i) =\prod_{i=1}^l \h_{e_i}(t'_i) \]
	if and only if $t'_i/t_i=\pm 1$ for every $i\in\ul$ and $\prod_{i=1}^l t'_i/t_i=1$. So the value of each $t_i$ is determined by $\prod_{i=1}^l \h_{e_i}(t_i)$ up to multiplication with $-1$.

 The Weyl group $\ov W=\Sym_{\pm l}$
 acts on those elements in the following way: 
	\[ \big(\prod_{i=1}^l \h_{e_i}(t_i)\big)^{(j,-j)}= \h_{e_j}(t^{-2}_j) \, \prod_{i=1}^l \h_{e_i}(t_i) 
	\und 
	\big(\prod_{i=1}^l \h_{e_i}(t_i)\big)^{(k,k')}= 	\h_{e_k}(t_{k'}) \h_{e_{k'}}(t_{k}) \big(\prod_{i\in \ul \setminus\{k,k'\}} \h_{e_i}(t_i)\big)
, \] 
where $j,k,k'\in \ul$ with $k\neq k'$.
\end{para}

From the action of $\ov W$ on $\bT$ described above we easily get the following. 
\begin{lem}\label{lem:3_3}
Let $\II Fcal@{\cF}\subseteq \FF^\times$ be a set of representatives in $\FF^\times$ under $x\mapsto -x$ and inversion $x\mapsto x\inv$, and such that $\{1,\varpi\}\subseteq\mathcal F$. Then every $s\in \bT$ has some $ \ov W$-conjugate $s' $ such that 
	\[ s'\in \prod_{i=1}^l \h_{e_i}(t'_i) \spannh, \]
	where $t'_i\in \mathcal F$ for every $i\in\ul$. 
\end{lem}

The corresponding centralizers are then described as follows. Note that whenever $X_1,X_2\leq X$ are finite groups with $[X_1,X_2]=1$ we write $X_1.X_2$ for the central product of the groups $X_1$ and $X_2$. 

\begin{lem}\label{3:5}
	Let $s'$ be as in \Cref{lem:3_3}. We write $I_\zeta(s'):=\{ i \mid t'_i\in \{\pm\zeta \} \}$ for $\zeta \in \mathcal F$. We abbreviate $I_1:=I_1(s')$, $I_\varpi:=I_\varpi(s')$ and set $
	R':=\bigsqcup_{\zeta\in\calF\setminus \{1,\varpi\}} \{ e_i-e_{i'}\mid i,i' \in I_\zeta(s') \text{ with } i \neq i' \}.$
Then $\Cent_\bG(s')$ is a central product of reductive groups normalized by $\bT$: 
	\[ \Cent_\bG(s')= \bC_1 . \bC_{4}. \bC_{R'}, \]
	where \begin{align*}
		 \bC_1&:=\spa{\bT_{I_1}, \bX_\al \mid \al\in R_{I_1} },\,\\ 
		\bC_4&:=\begin{cases}\spa{\bT_{I_\varpi}, \bX_\al \mid \al\in R_{I_\varpi} } \ \ \text{if $p$ is odd} \\  1  \ \  \ \ \text{if $p=2$, } \end{cases}\ \ \ \ \ \ \ \ \und \\
		\bC_{R'}&:=\spa{\bT_{\ul\setminus (I_\varpi \cup I_1)}, \bX_\al \mid \al\in R' }.
	\end{align*} 

Then, for $\bN:=\NNN_{\bG}(\bT)$ the group $ W_{s'}:=\Cent_{\bN}(s')/\bT$ satisfies
\[ W_{s'}=\begin{cases}\Sym^{\tD}_{\pm I_1 } \times \Sym^{\tD}_{\pm I_\varpi }\times \prod_{\zeta\in \mathcal F\setminus\{1,\varpi\}} \Sym^{\tB}_{I_\zeta(s')}\ \ \text{if $p$ is odd} \\  \Sym^{\tD}_{\pm I_1 } \times  \prod_{\zeta\in \mathcal F\setminus\{1\}} \Sym^{\tB}_{I_\zeta(s')}  \ \  \ \ \text{if $p=2$. } \end{cases}
\]
\end{lem}
\begin{proof} We can compute the centralizer of $s'$ from the root system as in \cite[3.5.3]{Ca85}, noting that the first statement reduces to the one about Weyl groups which can in turn be checked thanks to \ref{3:3} above. 
	
	Another way is to note that $\Cent_\bG(s')$ is connected and therefore $\Cent_\bG(s') =\pi_\SO\inv (\Cent_{\SO_{2l}(\FF)}^\circ(s'))$ so our claim reduces to a computation in $\SO_{2l}(\FF)$. The groups $\bC_1$, $\bC_4$ (the latter when $\varpi\neq 1$, i.e. $p\neq 2$) and $\bC_{R'}$ then correspond to the orthogonal groups on the eigenspaces associated with $1$ and $\varpi$, while $\bC_{R'}$ corresponds to the other eigenspaces of $\pi_\SO(s')$ in $\FF^{2l}$, see also \cite[1.13]{FS89} . 
\end{proof} 
%%%%%%%%%%%%%%%%%%%%%%%%%%%%%%%%%%%%%%%%%%%%%
The following statement is used to apply later \Cref{thm6_6} in the calculations of relative Weyl groups. We denote by $\UE(\HF)$ the group of automorphisms of $\HF$ generated by $F_p$ and $\gamma$.

We use the term $d$-\textit{regular} in the sense of \cite[Ch. 5]{BroueBook} and \cite{Sp74} for certain elements of $W$, see also
\Cref{def_varphi0} below.
\begin{cor} \label{cor3_8}

    	Let ${\varphi_0}$ be the automorphism of the character lattice $X(\bT)$, such that $F$ acts there as $q\varphi_0$. Let $\bT'$ be a maximal torus of $\bH$, let $s\in \bT'$ and set $\bC:=\Cent^\circ_\bH(s)\geq \bT'$, $P:=\NNN_{\bC}(\bT')/\bT'$. Let $d\geq 1$ be an integer and $\bS$ be a Sylow $d$-torus of $(\bH,F)$.
\begin{enumerate}[(a)]
	\item There is an inner automorphism $\iota$ of $\bH$ sending $\bT'$ to $\pi(\bT)$, and which induces an isomorphism
	\[ P\xrightarrow{\sim} \Sym^\tD_{\pm J}\times \Sym^\tD_{\pm J'}\times \Sym^{\tB}_{ \mathbb I}\leq W,\]
	for some $J, J'\subseteq \underline l $ with $J\cap J'=\emptyset$ and some partition $\mathbb I$ of $\underline l\setminus (J\cup J')$.
	\item Assume additionally $\Cent_\bH(\bS)$ is the torus $\bT '$ and $s\in \Cent_\bH(\bS)^F$. Then $\Cent_\bH(\bS)\leq \bC$ and there exists some $d$-regular element $w \varphi_0\in W\varphi_0$ normalising $\iota(P)$ such that $\Cent_{\iota(P)}(w\varphi_0)$ is isomorphic to $\NNN_{\bC}(\bS)^F/\Cent_\bH(\bS)^F$. 
	\item Keep the  assumptions about $s$ and $\bS$ from (b). Set $K:=\NNN_{\overline{ W}}(\iota(P))$ and $\widehat{ C}:=\Cent_{\bH^F\rtimes \underline{E}(\bH^F)}(s)$. There is a morphism 
	\[\NNN_{\widehat{ C}}(\bS)/\Cent_\bH(\bS)^F \longrightarrow \Cent_K(w\varphi_0), \] 
restricts into an isomorphism $\NNN_{\bC}(\bS)^F/\Cent_\bH(\bS)^F\cong\Cent_{\iota(P)}(w\varphi_0)$.
\end{enumerate}

\end{cor}
\begin{proof} 
	Up to inner automorphism, we can assume that $s=\pi(s')$ where $s'$ is as in \Cref{3:5} which then gives us the structure of the Weyl group of $\Cent_{\bG}(s')$, hence (a). 
	
	The maximal torus of $\bG$ corresponding to $\Cent_\bH(\bS)$ is the centralizer of a Sylow $d$-torus. This torus $\bT'$ is obtained from the maximally split torus $\pi(\bT)$ by twisting with $w\varphi_0$, where $\iota(F)\pi(\bT)=wF$ in the semidirect product $\bH E(\bH)$ and $w\varphi_0$ is some regular element of $W\varphi_0$ of order $d$, see \cite[\S 3.3]{Ca85} and \cite[3.5.7]{GM}. The quotient $\NNN_{\bC}(\bS)^F/\Cent_\HF(\bS)$ is therefore isomorphic to $\Cent_P(w\varphi_0)$ according to \cite[3.3.6]{Ca85}, which gives (b). The group $\NNN_{\bH \rtimes \spa{\gamma}}(\bS)/\Cent_\bH(\bS)$ is isomorphic to $\ov W$. With the standard discussion, see the proof of \cite[3.3.6]{Ca85}, we easily get the statement in (c). 
\end{proof}
%%%%%%%%%%%%%%%%%%%%%%%%%%%%%%%%%%%%%
\subsection{An automorphism of the centraliser of a semisimple element}
%%%%%%%%%%%%%%%%%%%%%%%%%%%%
The adjoint group $\bH$ being the dual of the simply-connected group $\bG$, Lusztig's Jordan decomposition of characters associates with each element of $\Irr(\twepsDlq)$ the $\bH^F$-orbit of a pair $(s,\phi)$, where $s\in\HF$ is semisimple and $\phi$ is a unipotent character of $\Cent_\HF(s)$. Under some assumptions on $s$ we find an automorphism of $\HF$ fixing $s$ and related to the graph automorphism $\gamma$ of $\twepsDlq$. This result complements \cite[Cor. 3.2]{CS22}.

When dealing with the semidirect product $\bG\rtimes \UE(\bG)$, recall that we consider any element of $\bG\UE(\bG)\setminus \bG$ (or $\bH\UE(\bH)\setminus \bH$) as an element of $\Aut(\bG)$ whence the notation $\sigma(g)$ for the product $\sigma g\sigma\inv$ whenever $\sigma\in \bG\UE(\bG)\setminus \bG$, $g\in\bG$. Recall the element $h_0\in\Z(\bG)$ from \ref{ZG_order}(a). We use the notation $\II Gss@{\bG_{\text{ss}}}$ and $\II Hss@{\bH_{\text{ss}}}$ to denote the sets of semisimple elements of $\bG$ and $\bH$.

\begin{prop} \label{wunderprop} 
	Let $s_0\in \bG_{\text{ss}}$ be such that $F(s_0)\in s_0 \spannh$ and $s_0$ is $\bG$-conjugate to $s_0 h_0$. 
	Set $s=\pi(s_0)\in \bH^F_{\text{ss}}$. 
	
	Then there exist some $a\in \bH^F$ with $s_0^a=s_0 h_0$, 
	some $\gamma'\in\HF\gamma$,
	and $F$-stable connected reductive subgroups $\bC_1$ and $\bC_{2}$ of $\bH$ such that 
	\begin{asslist}
		\item $\Cent^{\circ}_\bH(s)=\bC_1.\bC_{2}$ (central product);
		\item
		 $[\gamma',s_0]\in \spannh$ and therefore $\gamma'(s)=s$;
		\item $[a \gamma', \bC_1]=1$, and
		\item $[\gamma', \bC_{2}]=1$.
	\end{asslist}
\end{prop}

\begin{proof} Note that the statement is essentially about $\bH$ which is a quotient of $\bG/\spannh =\SO_{2l}(\FF)$. The automorphism $\gamma$ is induced by conjugation by the element $n_1^\circ\in\ov\bG$ thanks to \Cref{lem_gamma}, so the considerations below could be seen as happening in $\ov\bG/\spannh =\SO_{2l+1}(\FF)$ thus making more concrete the commutation arguments used.
	 
	Let $\bT_{0}$ be an $F$-stable maximal torus of $\bG$ with $s_0\in\bT_{0}$. Since all maximal tori of $\bG$ are $\bG$-conjugate, there exists some inner automorphism 
	$\iota_0:\bG\rtimes \UE(\bG)\lra \bG\rtimes \UE(\bG) $
	with $\iota_0(\bT_{0})=\bT$.
%	given by conjugation with some element of $\bG$. Then $\iota_0(F)$ stabilizes $\bT$ while being also in $\bG F$. Consequently, $\iota_0(F)=n F$ for some $n\in \NNN_{\bG}(\bT)$. 
	 
	Now \Cref{3:3} implies that $\iota_0(s_0)$ is $\NNN_{\bG\rtimes\spa{\gamma}}(\bT)$-conjugate to some element of the $\spannh$-coset\hfill \break
	$( \prod_{i=1}^l \h_{e_i}(t'_i)) \ \spannh $
	with all $t'_i$'s in $ \mathcal F$. So there is some 
	$\iota_1:\bG\rtimes \UE(\bG)\lra \bG\rtimes \UE(\bG)$, which is the conjugation by some element of $\bG\rtimes \spa{\gamma}$ such that $\iota_1(\bT_{0})=\bT$ and 
	\[ s':= \iota_1(s_0)\in \left(\prod_{i=1}^l\h_{e_i}(t'_i) \right) \quad\spannh\] 
 with all $t'_i$'s in $ \mathcal F$.
	Then $\Cent_{\bG}(s')$ can be written as 
	\[ \Cent_{\bG}(s')=\bC_1'.\bC'_{R'}.\bC'_4, \]
	where $\bC_1'$, $\bC_4'$ and $\bC'_{R'}$ are defined as in \Cref{3:5}.
	 
	Since $s_0$ and $s_0h_0$ are $\bG$-conjugate we can assume without loss of generality that $t'_1=\pm 1$ and $t'_l=\pm \varpi$ thanks to \Cref{rem:3:2} (b). 
	
	We have $\iota_1(F)\in \bG F$ in $\bG\UE(\bG)$ since $\iota_1$ is an inner automorphism there and $\UE(\bG)=\bG\UE(\bG)/\bG$ is abelian. Moreover $\iota_1(F)\in \NNN_{\bG}(\bT) F$ since $F(\bT_{0})=\bT_{0}$ implies $\iota_1(F)(\bT)=\bT$. Let $n\in \NNN_{\bG}(\bT) $ be such that $\iota_1(F)=n F$.
	
	The assumption $[s_0,F]\in \spannh$ implies $[ s',\iota_1(F)]=[ \iota_1(s_0),\iota_1(F)]\in \spannh$ by \ref{ZG_order}(b). Then $\pi_\SO(s')$ is fixed under the Frobenius endomorphism $nF=\iota_1(F)$ of $\SO_{2l}(\FF)$ and its eigenspaces are permuted by $nF$ as the corresponding eigenvalues. From the definition of $\bC'_1$, $\bC'_4$ and $\bC'_{R'}$, we get that their images under $\pi_\SO$ are $\iota_1(F)$-stable. 
	 
	Note that $h_0\in \bT_I$ for any nonempty $I\subseteq \ul$ and therefore $\bC'_1=\pi_\SO\inv( \pi_\SO(\bC'_1))$ and $\bC'_4 =\pi_\SO\inv(\pi_\SO(\bC'_4 ))$ while $\bC'_{R'}$ also contains $h_0$ unless it is trivial. So from what we have seen in $\SO_{2l}(\FF)$ we get that $\bC'_1$, $\bC'_4 $ and $\bC'_{R'}$ are all $\iota_1(F)$-stable.
	 
	What we said above about $\iota_1(F)=n F$ implies that $n\in \NNN_{\bG}(\bT) $ can be written in the form $n_1 (\neins)^{i_1} n_4  (\nzwei)^{i_2} n'$, where $n_1\in \bC_1'$, $n_4 \in \bC '_4 $, ${\neins:=\n_{e_1}(\varpi)}\in \ov\bG$, $\II ncircz@{\nzwei:=\n_{e_l}(\varpi)}\in \neins\bG$, $i_1,i_2\in \{0,1\}$ and $n'\in \Cent_{\obG}(\bC_1'.\bC '_4 )$. 
	 
	Since $\bC_1'$, $\bC_4  '$ are connected $\iota_1(F)$-stable, Lang's Theorem (\cite[21.7]{MT}) implies that there exists some $c\in \Cent_{\bG}(s')$ such that $(\iota_1(F))^c= vF$ where $v=(\neins)^{i_1} (\nzwei)^{i_2} n' $. Let $\iota\colon \bG\UE(\bG)\lra \bG\UE(\bG)$ be defined by $\iota(x)=\iota_1(x)^c$. Then $\iota(F)=vF$ and $\iota(\bG^F)=\iota(\bG^{vF})$. Note that $\iota(s_0)=(s')^c=s'$. 
	
	We have $ a':= \neinszwei\in \bG$ by checking its class mod $\bT$. Moreover using \Cref{3:3} we get $s'^{a'}=s' h_0$ while $[a', \bC'_{R'}]=1$ thanks to Chevalley's commutator formula. Note also that $[F,a']\in \spannh$. Let us now take 
    \[ a:=\pi(\iota^{-1}(a')) \ ,\ \   
 \bC_1:=\pi(\iota^{-1}(\bC'_1)),\ \und\ \ 
 \bC_2:=\pi(\iota^{-1}(\bC'_2))\ \  \text{for}\ \  
 \bC'_2:=\bC'_{R'}.\bC'_4 .\]
 
We clearly have $\Cent_{\bH}^\circ (s)=\pi (\Cent_\bG(s_0))=\pi\circ\iota\inv(\Cent_\bG(s'))=\pi\circ\iota\inv(\bC_1' .\bC'_{R'}.\bC'_4  )=\bC_1\bC_2$. 

Recall $F\in \{F_q, \gamma F_q\}$. For $F':=\iota(F)=v F$ we have $[F',a']= [F,a'] [v,a'] \in\spannh$ since 
$[n',a']\in\spannh$, see \Cref{ZG_order}(d). Therefore $a\in \bH^F$. 
 
Recall that $\gamma$ and $\neins$ induce the same automorphism on $\bG$. Hence we get $[\gamma ,\bC_2']=[\neins, \bC_2']=1$ thanks to the Chevalley commutator formula, and analogously $[a'\gamma,\bC'_1]=[\neins \nzwei\gamma,\bC'_1]=1$. Applying $\pi\circ\iota\inv$, this gives our claims (iii)--(iv) that $[\bC_1, a\gamma ']=[\bC_2, \gamma ']=1$ in $\bH\UE(\bH)$, where 
$\gamma ' \in \bH \UE(\bH)$ is the image of $\iota\inv(\gamma)\in \bG\UE(\bG)$ under $\pi$.
Note that $[\gamma ,\pi (s')]=1$ since $[\neins ,s']\in\spannh$ by \ref{3:3}, hence (ii). 

It remains to show that $\gamma'\in \HF \gamma$. By its definition, $\gamma'\in\bH \gamma$. Additionally $vF$ commutes with $\gamma$ in $\bH \UE(\bH)$, as \Cref{ZG_order} implies $[\gamma,\pi(v)]=\pi([\neins , n'])=1$. Now $\gamma\in \Cent_{\bH \gamma}(vF)=\Cent_{\bH \gamma}(\iota(F))$ implies $\gamma'\in \Cent_{\bH \gamma}(F)= \HF\gamma$. This completes our proof. \end{proof}

%\newpage%%%%%%%%%%%%%%%%%%%%%%%%%%%%%%%%%%%%%%%%%%%%%%%%%%%%%%%%%%%%%
\subsection{Consequences on characters of $\tDlsc(q)$}\label{ssec_IrrG}

Let $(\bG,F)$ and $\UE(\GF)$ be given as in \ref{sub2E}. We assume that a regular embedding $\bG\leq \wbG$ is chosen such that $\UE(\bG)$ acts on $\wbG$. This implies that $F$ and therefore $\UE(\GF)$ acts on $\wbG^F$.
The $\textbf{A}(\infty)$ property of
\Cref{Julia2} was originally introduced as saying that any element of $\Irr(\GF)$ has a $\wbG^F$-conjugate $\chi$ such that $(\wbG^F\UE(\GF))_\chi =\wbG^F_\chi\UE(\GF)_\chi$ and $\chi$ extends to $\GF \UE(\GF)_\chi$, see \cite[Thm 2.12(v)]{S12} and our Introduction. This condition on $\chi$ defines a subset $\o\TT\subseteq \Irr(\GF)$ that we complete below into a partition $$ \Irr(\GF)=\o\TT\sqcup \EE\sqcup \DD$$ that will be crucial in our description of $\Irr(\GF)$. This will be particularly useful when studying characters of certain local subgroups where $\bG$ is replaced by groups $\bG_\uk$ for $k\leq l$ (see \ref{ssec:3D}) occuring as $[\bL,\bL]$ for $\bL$ a minimal $d$-split Levi subgroup of our $\bG$. 

Most of the properties we single out revolve around the value of characters at $h_0$ (see \Cref{ZG_order}) and the action of the diagonal outer automorphism associated with $h_0$ through \Cref{not_diag}.

\begin{defi} \label{def_TED}
	Let $\ov \TT$, $\EE$ and $\DD$ be the following subsets of $\Irr(\GF)$:
	\begin{align*} 
		\II{Toverline}@{\protect{ \oTT}} &=&& \left\{ \chi\, \left |\, { \starStab \wGF | \UE(\GF)|\chi \text{ and }}{\chi \text{ extends to }\GF \UE(\GF)_\chi}\right.\right \},\\
		\II E@{\protect{\EE}} &=&&
		\left \{ \chi \, \left |\, {
			\starStab \wGF| \UE(\GF)|\chi \text{ and }}{\chi \text{ has no extension to }\GF \UE(\GF)_\chi}\right.\right \}, \und \\ 
		\II D@{\protect{\DD}}&=&&
		\left \{ \chi \, \left |\, 
		\starStabneq \wGF| \UE(\GF)|\chi \right.\right \}, 
	\end{align*}
so that $\Irr(\GF)= \ov \TT \sqcup \EE\sqcup \DD$.

We also define $ \II E'@{\protect{\EE'}}:= \ov\TT\cap \big( \bigcup_{x\in\wG}{}^x\EE \big)$ and $ \II D'@{\protect{\DD'}}:= \ov\TT \cap \big( \bigcup_{x\in\wG}{}^x\DD \big)$.
\end{defi}

\begin{notation}\label{that}
	Let $\II{that}@{\protect{\wh t}}\in \bT$ with $F(\wh t)=h_0\wh t$ and set $\wh G:=\GF\spa{\wh t}=\bG\cap\calL\inv(\spa{h_0})$, where $\calL:\wbG\longrightarrow\wbG$ is given by $x\mapsto x^{-1}F(x)$. 
\end{notation}

In the following, we prove that any element of $\EE\cup \EE'\cup \DD\cup \DD'$ is contained in a $\gamma$-stable $\wbG^F$-orbit, and is lying under a $\gamma$-invariant character of $\GF\spa{\wh t}$. This generalizes \cite[Prop.~5.3]{S21D1} where $\eps=1$, $\chi$ is cuspidal, $h_0\in\ker(\chi)$ and $\Norm _\bG(\GF)_\chi\leq \GF\spa{\wh t}$.

\begin{theorem} \label{thm_sumup_D_1}
	Let $l\geq 4$, $\eps\in\{\pm 1\}$, $\bG:=\tDlsc(\FF)$ with $\GF \cong\tDlsc^\eps(q) $. 
	\begin{thmlist}
		\item Let $\chi\in\EE\cup \EE '$. Then $h_0\in\ker(\chi)$ while $\chi$ is invariant under and extends to $\GF\spa{\wh t, \gamma}$. 
		\item Let $\chi\in\DD\cup \DD '$. Then $h_0\in\ker(\chi)$ and $\chi^\gamma =\chi^{\wh t}\neq\chi$ or $\chi^\gamma =\chi$ according to $\chi\in \DD$ or $\DD '$. 
	\end{thmlist}

\end{theorem}

Here are some first properties coming mostly from the structure of the outer automorphism group of $\GF$.
\begin{lem}\label{lem3:7} Recall $\wZ(G)=\Z(\bG)/[\Z(\bG),F]$ and its action on $\Irr(\GF)$ by diagonal outer automorphisms (see \Cref{not_diag}).
	\begin{thmlist} 
		\item If $\chi\in\DD$, then $|\wZ(G):\wZ(G)_\chi|=4$ and $h_0\in\ker(\chi)$.%, and $\chi^{e}\in \{\chi, \chi^{\wh t}\}$ for some $e\in \UE(G)\setminus \Cent_{\UE(G)}(\Z(G))$ and $\wh t:=h_0{[\Z(\bG), F]}$.
		\item If $\chi\in \EE$, then $h_0\in \ker(\chi)$, $\chi^\gamma=\chi$ and $\wZ(G)_\chi = \spa{\wh h_0}$ for $\wh h_0=h_0[\Z(\bG),F]\in \wZ(G)$.
	\end{thmlist}
\end{lem}
 
\begin{proof} First we prove part (a), so assume $\chi\in\DD$. We abbreviate $\wZ=\wZ(G)$ and $E=\UE(G)$. By \Cref{Julia2} there is $\chi'=\chi^z$ with $z\in\wZ$ such that $(\wZ E)_{ \chi '} = \wZ_{ \chi '} E_{ \chi '}$. Since $\chi\in\DD$, one has $(\wZ E)_\chi \neq \wZ_\chi E_\chi$ while \begin{align}
\label{eq31}	\big( (\wZ E)_\chi \big)^z=(\wZ E)_{ \chi '}.
	\end{align} This forces $\wZ E$ to be nonabelian and therefore $|\wZ|=4$. Then $F$ acts trivially on $\Z(\bG)$ and $\wZ=\Z(\bG)=\Z(\bG)^F$. 
	Also one can't have $\wZ_{ \chi '}=\spa{h_0}$ since the latter is central in $\wZ E$. But by (\ref{eq31}) above neither $z$ nor $zh_0$ belongs to $\wZ_{ \chi '}$, forcing $\wZ_{ \chi '}=1$, hence our statement that $|\wZ:\wZ_\chi|=4$.
	 Since $(\wZ E)_{ \chi '}=E_{ \chi '}$ can't centralize $z$, let us take $e'{}\in E_{ \chi '}$ with $ e'{}\neq{}^ze'{}\in (\wZ E)_{ \chi }$. 
 Concerning $\ker(\chi)$ note that $[\Z(G),e']=\spa{h_0}$ by \ref{ZG_order} and therefore $[\wZ, e'{}]=\spa{h_0}$ in $\wZ$ since $[\wZ, e'{}]\neq 1$. Then $^{e'{}}\chi '=\chi '$ implies $\spa{h_0}\leq \ker(\chi ')$ and also $\spa{h_0}\leq \Z(G)\cap\ker(\chi ')=\Z(G)\cap\ker(\chi^z )=\Z(G)\cap\ker(\chi )$ since diagonal automorphisms act trivially on $\Z(G)$.

	Next, we prove part (b). Let $\chi\in \EE$. Note that by \Cref{ExtCrit}(a) and the definition of $\EE$ this forces $\UE(\GF)_\chi$ to be noncyclic. Then $\chi^\gamma =\chi$, $F=F_p^f$ for some even $f$ and $\chi\in \Irr(\GF)^D$ for a subgroup $D\leq \UE(\GF)$ satisfying \cite[Hyp. 5.5]{S21D2}. Now our claims that $h_0\in\ker(\chi)$ and $\wZ_\chi=\spa{\wh h_0}$ are known from \cite[Lem. 5.11 and Prop. 5.21]{S21D2}.
\end{proof}
 
Let us recall briefly Lusztig's parametrization of $\Irr(\GF)$ (for $(\bG,F)$ any connected reductive group defined over $\FF_q$). Let $(\bG^*,F^*)$ be dual to $(\bG,F)$ in the sense of \cite[Def. 11.1.10]{DiMi2} or \cite[1.5.17]{GM}. Then the partition of $\Irr(\GF)$ into Lusztig's \textit{rational series} is $$\Irr(\GF)=\bigsqcup _{[s]}\cE(\GF,[s])$$ where $[s]$ ranges over the $\bG^*{}^{F^*}$-conjugacy classes of semisimple elements $s\in\bG^*{}^{F^*}$. We set $\UCh(\GF)=\cE(\GF,[1])$. Moreover for any $s\in\bG^*_{\text{ss}}{}^{F^*}$, one has a bijection $$\cE(\GF,[s])\xrightarrow{\sim}\II{Uch}@{\UCh}(\Cent_{\bG^*}(s)^{F^*}):=\Irr(\Cent_{\bG^*}(s)^{F^*}\mid \cE(\Cent_{\bG^*}^\circ(s)^{F^*},[1])) $$ thus allowing us to associate to each $\chi\in\Irr(\GF)$ a ${\bG^*}^{F^*}$-class of pairs $(s,\phi)$ with $s\in{\bG^*}_{\text{ss}}^{F^*}$ and $\phi\in\UCh{(\Cent_{\bG^*}(s)^{F^*}})$. This Jordan decomposition of characters can be chosen to be $\Out(\GF)$-equivariant when $\bG$ is simple simply connected, see \cite[Thm B]{S21D2}. 

All elements of $\cE(\GF,[s])$ have the same restriction to $\Z(\GF)$, see \cite[Prop. 9.11(b)]{Cedric}. We also recall in the next lemma some standard related facts. The proof is basic theory of dual groups (see more details in proofs of \cite[Lem. 4.4(ii)]{NavarroTiep_Annals_pprime} and \cite[Lem. 9.14]{Cedric}). We state everything in the case where $\bG$ is as defined in \ref{sub2E} and $(\bG^*,F^*)$ can then be taken to be $(\bH,F)$, thus resulting in more identifications. 
\begin{lem}\label{lem_dualZGF}
	The duality between $(\bH,F)$ and $(\bG,F)$ induces an $\UE(\GF)$-equivariant isomorphism $$\Lin(\Z(\GF))\cong\Z(\bG)/[\Z(\bG),F].$$ If $s\in \bH^F_{\text{ss}}$, $s_0\in\pi^{-1}(s)$ and $\chi\in \cE(\GF,[s])$, then the element of $\Irr(\restr \chi|{\Z(\GF)})$ corresponds to $[s_0,F][\Z(\bG),F]$. 
\end{lem}

We can now prove \Cref{thm_sumup_D_1}.
 
\begin{proof}[Proof of \Cref{thm_sumup_D_1}] We let $\chi\in\EE\cup \EE'\cup \DD\cup \DD'$. In view of the claim made and $[\wt Z(\bG),\gamma]\leq \spannh$ (see \Cref{ZG_order}) we can assume $\chi\in \EE\cup \DD\cup\DD '$. We have $h_0\in\ker(\chi)$ thanks to \Cref{lem3:7}.
	 
	Let $s\in \bH^F_{\text{ss}}$ and $s_0\in\pi^{-1}(s)$ with $\chi\in\cE(\GF,[s])$. Set $\Irr (\restr\chi|{\Z(\bG)^F})=:\{\xi\}$. Then $\xi(h_0)=1$, hence $\xi$ is $\gamma$-fixed by \Cref{ZG_order}. Now \Cref{lem_dualZGF} implies that $[s_0,F]$ is $\gamma$-fixed, hence an element of $\spannh$ by \Cref{ZG_order} again. 
	 
Let us show that $s_0^g=s_0h_0$ for some $g\in \bG$. Using the injective map $\II{omegas}@{\omega_s}\colon \II{AHs}@{A_\bH(s)}:=\Cent_\bH(s)/\Cent^\circ_\bH(s)\to \Z(\bG)$ from \cite[(1.5)]{S21D2}, this means by definition that we must prove that $h_0$ is in the image $\III{B(s)}$ of $\omega_s$. To see this note that by \Cref{lem3:7} again, $\wt Z(\GF)_\chi$ is trivial or generated by the class of $h_0$.
%, hence in both cases $\gamma$-stable in $\Out(G)$
 Taking $\wt\chi\in \Irr(\wbG^F\mid \chi)$ and seeing $B(s)^F$ as acting by linear characters on $\Irr(\wbG^F)$ (see \cite[Not. 2.14]{S21D2}), one has by Clifford theory that $B(s)_{\wt\chi}^F$ is the orthogonal of $\wt Z(\GF)_\chi$  (see \cite[Lem. 2.13(c)]{S21D2}). Therefore $h_0\in B(s)_{\wt\chi}^F$ which gives our claim. 

All assumptions of \Cref{wunderprop} are now satisfied. Let $a\in \HF$, $\gamma'\in \HF\gamma$, $\bC_1$ and $\bC_2$ be such that $\Cent_{ \bH}^\circ (s)=\bC_1\bC_2$ (central product) as there. 
 
 Let $(\wbG^*,F^*)$ be dual to $(\wbG,F)$ and let $\wt s\in \wt {\bG}^*{}^{F^*}$ such that $\wt\chi\in\cE(\wbG^F,[\wt s])$ with $\wt s\mapsto s$ by the map dual to the inclusion $\bG\to\wbG$. Let $\phi\in \UCh(\Cent_{\wt \bG^*}(\wt s)^{F^*})$, also be seen as an element of $\UCh( \Cent^\circ _\bH (s)^F)$ since $\bH$ and $\wbG^*$ are isogeneous, associated to $\wt \chi$ by the Jordan decomposition $\wt J$ of \cite[Prop. 2.15]{S21D2}. 
 From the definition of $\omega_s$ \cite[Equ. (1.5)]{S21D2}, $\phi $ is $c$-invariant under any $c\in \Cent_\HF(s)$ with $s_0^c=s_0 h_0$ since $A_\bH(s)^F_\phi =\omega_s\inv(B(s)^F_{\wt\chi})$ and we have seen that $h_0\in B(s)_{\wt\chi}^F$. Then $\phi^a=\phi$.
 
 Let us show that $\phi$ is also $\gamma '$-invariant. Restrictions induce a bijection $$\UCh(\Cent_{ \bH}^\circ (s)^F)\xrightarrow{\sim} \UCh(\bC_1^F)\times \UCh(\bC_2^F), \text{\ \ by\ \ }\phi\mapsto (\phi_1,\phi_2) $$ 
 see \cite[1.10]{S21D2}. Let $(\phi_1, \phi_2)$ be the character corresponding to $\phi$. Then $\phi^a=\phi$ implies $\phi_i^a=\phi_i$ for $i=1,2$ and therefore ${a\gamma '}$ fixes $\phi_1$ and $\phi_2$ thanks to \Cref{wunderprop}(iii) and (iv). Then it also fixes $\phi$ by the above bijection and we get our claim that 
 \[(s,\phi)^{\gamma '}=(s,\phi).\]
	
By the $E(\GF)$-equivariance of $\wt J$ and its compatibility with $\wt\bG^F$-orbits on $\Irr(\GF)$ (see \cite[Prop. 2.16]{S21D2}) we finally get that the $ \wbG^F$-orbit containing $\chi$ is $\gamma$-stable, as $\gamma'\in \HF\gamma$. By the definition of $\oTT$, this also implies $\chi=\chi^\gamma$ whenever $\chi\in \oTT$, which is the case when $\chi\in\DD '$. If on the contrary $\chi\in \DD$, then $\chi ':=\chi^g\in \oTT$ for some $g\in \wbG^F$ and we can't have $\chi =\chi ^\gamma$ since this would imply $g\gamma g\inv\gamma\inv\in \wbG^F_\chi =G\Z( \wbG^F)$ and therefore that $[g, \wbG^F\UE(\GF)]\leq \GF\Z( \wbG^F)$ thanks to the structure of $\wZ(\GF)\rtimes \UE(\GF)$. But then $\chi$ would have the same stabilizer as $\chi '\in \oTT$. So $\chi^\gamma\neq \chi$ and therefore $\chi^\gamma =\chi^{[g,\gamma]}=\chi^{\wh t}$ by \Cref{ZG_order}. This gives the two cases of our statement (b).
	
%	When $\chi\in \DD$ and therefore $\wG_\chi=G$ by \Cref{lem3:7}(a), this immediately implies the statement, as some $\wG$-conjugate of $\chi$ is hence $\gamma$-stable by \Cref{Julia2}. 
	
Assume now $\chi\in \EE$ and therefore $ \Norm_\bG(G,\chi)=\wh G$ by \Cref{lem3:7}(b). Recall that $[\gamma',s_0]\in \spannh$ according to \Cref{wunderprop}. Then we can apply \cite[Cor. 6.6]{S21D2} and obtain that $\chi$ has some extension $\wh \chi$ to $\wh G$ that is $\gamma$-invariant. This gives our claim (a).
\end{proof}

\begin{rem}\label{prop:DD_EE} The above proof gives the following slightly stronger statement for
$\chi\in\Irr(\GF\mid 1_{\spannh})$: \textit{If $ \wbG^F_\chi\neq \wbG^F$ and $ \wbG^F_\chi$ is $\gamma$-stable, then every $\wt \chi\in \Irr(\spa{\GF,\wh t}\mid \chi)$ is $\gamma$-invariant. 
}\end{rem}

%%%%%%%%%%%%%%%%%%%%%%%%%%%%%%%%%%%%%%%%%%%%%%%%%%%%%%%%%%%%
\subsection{Characters of $\mathrm{D}_{k,\mathrm{sc}}^\eps(q)$ with $1\leq k\leq l$}\label{ssec:3D}
In the following we establish a slight generalization of \Cref{thm_sumup_D_1} including groups of rank $\leq 3$ that appear naturally when studying $d$-split Levi subgroups of $\bG$, see \Cref{thm_sumup_D}. The groups added to the picture are essentially semisimple simply connected of type $\tA_1\times \tA_1$ or $\tA_3$ but we need some results  about them that are not formally contained in the \textbf{A}$(\infty)$ condition known already for groups of type $\tA$, see \cite[Thm 4.1]{CS17A}. For those groups the proofs are an adaptation of what has been done in the preceding section. 

For $I\subseteq \ul$ recall $\hh_I\colon \FFtimes\to \bT$, $h_0=\hh_{\{ 1\}}(-1)$ defined in \ref{sub2E}, and $\o\Phi\supseteq \o R_I\supseteq R_I=\Phi\cap\o R_I$ from \Cref{not3.1}.

\begin{notation} \label{not:3_16}
	For $k\in \ul$, let $\II{Tk}@{\bT_\uk}:=\bT\cap \spa{\bX_\al \mid \alpha \in \overline R_{\uk}} $ and \[\II{Gk}@{\bG_{\uk}}:=\bT_\uk\spa{\bX_\al \mid \alpha \in R_\uk }.\] 
	
	Let $F\in \{F_q,\gamma F_q\}$, $\III{G_\uk}:=\bG_\uk^F\ \unlhd\ \II{Gbr}@{{\breve G}_\uk}:=\calL^{-1}_F(\spa{h_0,\h_{\underline k}(\varpi)})\cap \bG_\uk$, and $\UE({\wb G}_\uk):=\spa{\restr\gamma|{{\wb G}_\uk}, \restr{F_p}|{{\wb G}_\uk}}\leq \Aut({\wb G}_\uk)$. 
	
		We define the sets of irreducible characters $\ov \TT$, $\EE$, $\EE'$, $\DD$, and $\DD'$ as in \Cref{def_TED}~: 
			\begin{align*}
		\II{Tkunderline}@{{ { \oTT_\uk}}} &= \left \{ \chi\in\Irr(G_\uk)\, \left |\, \text{{
				$\starStabkla {\wb G}_\uk| \UE({\wb G}_\uk)|\chi$ and }{$\chi$ extends to $G_\uk \UE({\wb G}_\uk)_\chi$}}\right.\right \},\\
		\II E@{ {\EE_\uk}} &= \left \{ \chi\in\Irr(G_\uk)\, \left |\, 
		\text{ 
			{$\starStabkla {\wb G} _\uk| \UE({\wb G}_\uk)|\chi$ \text{ and }}
			{ $\chi \not\in { \oTT_\uk}$ }}
		\right.\right \}, \ \ \II E'@{ {\EE'_\uk}} =\oTT_\uk\cap( \bigcup_{x \in{\wb G}_\uk}{}^x\EE_\uk) \\ 
		\II Dk@{ {\DD_\uk}}&=
		\left \{ \chi\in\Irr(G_\uk)\, \left |\, \starStabklaneq {{\wb G}_\uk}| {\UE({{\wb G}}_\uk)}|\chi \right.\right \}, \und \II D'@{ {\DD'_\uk}} =\oTT_\uk\cap(\bigcup_{x\in{\wb G}_\uk}{}^x\DD_\uk).
	\end{align*}
\end{notation}
 
The following statement now covers all ranks $\geq 1$. We keep $l\geq 4$ and $k\in\ul$.

\begin{theorem}\label{thm_sumup_D}
Let $\wh t_\uk\in {\bT}_\uk$ be such that $F(\wh t_\uk) =h_0\wh t_\uk$.
	\begin{thmlist} 
		\item In every ${\wb G}_\uk$-orbit in $\Irr(G_\uk)$ there exists some $\chi$ with $( {\wb G}_\uk \UE({\wb G}_\uk))_\chi =( {\wb G}_\uk)_\chi \UE({\wb G}_\uk)_\chi$ and $\chi$ extends to $G_\uk \UE({\wb G}_\uk)_\chi$.
			\item Let $\chi\in\EE_\uk\cup\EE'_\uk$. Then $h_0\in\ker(\chi)$ and $\chi$ is invariant under and extends to $G_\uk\spa{\wh t_\uk ,\gamma}$. 
		\item Let $\chi\in\DD_\uk\cup\DD'_\uk$. Then $h_0\in\ker(\chi)$ and $\chi^\gamma =\chi^{\wh t_\uk}\neq\chi$ or $\chi^\gamma =\chi$ according to $\chi\in\DD_\uk$ or $\chi\in\DD_\uk '$. 
	%	\item If $\chi\in\Irr(G)$ with $h_0\in\ker(\chi)$ and $(\wt T_\uk)_\chi\leq T_\uk\spa{\wh t_\uk}$, then every $\wt \chi\in\Irr(G_\uk\spa{\wh t_\uk}\mid \chi)$ is $\gamma$-invariant.
	\end{thmlist}
\end{theorem}

\begin{proof} (1) Note that for $k\geq 4$ the groups $G_\uk$ are just of the type studied before with $l=k$ since 
	$ \bT_{\uk}\leq\spa{\bX_\al\mid \al\in R_{\uk}}$ by the Steinberg relations as soon as $k\geq 2$ and therefore $\bG_\uk=\spa{\bX_\al\mid \al\in R_{\uk}}$ is the derived subgroup of the Levi subgroup $\bT\bG_\uk$ hence of type $\tDksc$ thanks to \cite[12.14]{MT}. The action of ${\wb G}_\uk$ on characters of $G_\uk$ induces all and only diagonal outer automorphisms, so the stabilizer statement of (a) is covered by \Cref{Julia2}. The extension statement is also a consequence of the extension statement of \Cref{Julia2} since the groups $\UE(G_\uk)$ and $\UE({\wb G}_\uk)$ are such that $\UE(G_\uk)\to\!\!\!\!\to\UE({\wb G}_\uk)$ with a kernel $K$ acting trivially on $G_\uk$ and therefore if $\chi\in \Irr(G_\uk)$ extends to a character $\chi '$ of $G_\uk\rtimes \UE(G_\uk)_\chi= (G_\uk\rtimes \UE({\wb G}_\uk)_\chi)/K$ then one can inflate $\chi '$ into a character of $G_\uk\rtimes \UE({\wb G}_\uk)_\chi$ also extending $\chi$. As for statements (b) and (c), they are covered by \Cref{thm_sumup_D_1}. 
	
	There remains to prove the theorem in the cases where $k\in\{1,2,3\}$. Set $\eps = 1$ or $-1$ according to $F=F_q$ or $\gamma F_q$.
	 
(2)	Let us assume $k=1$. The group $\bG_{\underline{1}}$ is a torus of dimension $1$. Then every $\chi\in\Irr(G_{\underline{1}})$ is linear, hence extends to its stabilizer in the semidirect product $G_{\underline{1}} \UE({\wb G}_{\underline{1}})$ by \Cref{ExtCrit}(d). This gives (a) since ${\wb G}_{\underline{1}}$ acts trivially on $G_{\underline{1}}$. Note also that $\EE_{\underline{1}}=\DD_{\underline{1}} =\emptyset$.
	 
(3)	Next, we study the case where $k\in\{2,3\}$. 
Formally one could argue that for $k=3$ the proof of \Cref{Julia2} could be transferred and still applies but we give here an independent proof that does not require to check that the relevant parts of \cite{S21D2} apply for $k=3$.
Recall that by (1) above 
	$ \bT_{\uk}\leq\spa{\bX_\al\mid \al\in R_{\uk}}$ and $\bG_\uk$ is then a simply connected group of type $\tA_1\times \tA_1$ or $\tA_3$.
	We get that $\bG_{\underline 2}=\SL_2(\FF)\times \SL_2(\FF)$ and $\bG_{\underline 3}=\SL_4(\FF)$ with $F_p$ acting as $x\mapsto x^p$ on matrix entries, while $\gamma$ acts by swapping the two components in the first case and by the graph automorphism of order 2 preserving the usual torus/Borel subgroups of 2.C above in the second case, see also the proof of \cite[Lem.~6.30]{S21D1}.
	It follows that 
\begin{align}
	 G_{\unl 2}\cong \begin{cases} \SL_2(q)\times \SL_2(q)& \text{if }\eps=1 ,\\
		\SL_2(q^2)& \text{if }\eps=-1 ,
	\end{cases}\ \ \und \ \
	G_{\unl 3}\cong \begin{cases} \SL_4(q)& \text{if }\eps=1 ,\\
		\SU_4(q) & \text{if }\eps=-1.
	\end{cases}  \label{eq32}
		\end{align}
	Note that ${\wb G}_\uk$ acts by diagonal automorphisms of $G_\uk$. Applying \cite[Thm~4.1]{CS17A} to types $\tA_1$ ($\SL_2$) and $\tA_3$ gives our claim (a) in the cases $k=3$ and $(k,\eps)=(2,-1)$, using what has been said in (1) about the $\UE(G_\uk)$ versus $\UE({\wb G}_\uk)$ question. In the case of $(k,\eps)=(2,1)$, \cite[Thm~4.1]{CS17A} indeed shows that 
	there exists an $\spa{F_p}$-stable $\GL_2(q)$-transversal in $\Irr(\SL_2(q))$. 
	This implies that $\ov \TT_{\underline 2 } $ contains some $\spa{F_p, \gamma}$-stable $\wt G_{\underline 2}$-transversal. Moreover the group $G \UE(G)$ being a subgroup of the wreath product $\SL_2(q)\spa{F_p}\wr \Cy_2$, the extendibility claim holds by \Cref{ExtCrit}(c). We then get part (a) in all cases. 
 
 (4) We now turn to statements (b), (c) in the case when $k=2,3$ and make some preliminary remarks. 
 
Notice first that whenever $\Z(G_\uk)=\Z(\bG_\uk)^F$ is of order $\leq 2$ then (a) and the commutativity of $({\wb G}_\uk/\Z(\bG_\uk).G_\uk)\rtimes \UE({\wb G}_\uk) $ imply that $\oTT =\Irr(G_\uk)$ and therefore $\EE_\uk=\DD_\uk =\emptyset$, so (b) and (c) are trivial. 
 
So we assume that $\Z(G_\uk)$ is of order 4, and in particular $q$ is odd. Then $\Z(G_\uk)=\Z(\bG_\uk)$ and therefore $h_0\in {[\gamma ,\Z(G_\uk)]}$ since $h_0=[\gamma , \h_\uk(\varpi)]$, see \Cref{ZG_order}. This implies that the second part of (b) or (c) implies also that $h_0\in\ker\chi$. This is because, if $\chi$ is invariant under the composition of a diagonal automorphism with $\gamma$, then its restriction to $\Z(G_\uk)=\Z(\bG_\uk)$ has the same property and indeed the restriction of $\chi$ to $\Z(G_\uk)$ is $\gamma$-invariant, hence is trivial on ${[\gamma ,\Z(G_\uk)]}$. 

Concerning (b), as in the proof of \Cref{thm_sumup_D_1}(b), notice that one can consider $\EE_\uk$ alone.

 From now on we omit the index $\uk$.
 
(5)	Assume $k=2$. 	If $\eps=-1$, we have $|\Z(G_{ })|=2$ so (b) and (c) are trivial as said in (4). So we assume $\eps=1$. Then the group $G_{ }$ is isomorphic to $\SL_2(q)\times \SL_2(q)$, see above. 
What has been said about the wreath product action of $\gamma$ implies $\EE_{ }=\emptyset$ hence (b).

To prove (c) let's take $\chi\in\DD {_{ }}$ and let $\chi_0$ be a $\wb G$-conjugate of $\chi$ such that $\chi_0\in\oTT{_{ }}$. That the stabilizers of $\chi $ and $\chi_0$ are distinct implies that $\UE(G)_{\chi_0}\nleq \Cent_{\UE(G)}(\Z(G))=\spa{F_p}$ and $\chi_0$ is not $\wb G$-invariant. 
	The character $\chi_0\in \Irr(\SL_2(q)\times \SL_2(q))$ being now written as $\chi_1.\chi_2$, we see that at least one of $\{\chi_1,\chi_2 \}$ is not $\GL_2(q)$-invariant. 
	But $\UE(G)_{\chi_0}\nleq \spa{F_p}$ implies that $\chi_0$ is invariant under some element of $\gamma\spa{F_p}$, leading to $\chi_1$ and $\chi_2$ being both not $\GL_2(q)$-invariant.
	
	Now every character of $\SL_2(q)$ that is not $\GL_2(q)$-invariant is $\spa{F_p}$-invariant, see the character table of $\SL_2(q)$ in \cite[Table 5.4]{CedricSL2}. So we get that $\chi_1$ and $\chi_2$ are both $\spa{F_p}$-invariant. We can conclude that $\chi_0$ is $\spa{F_p,\gamma}$-invariant, which implies the second case of (c) for $\chi_0$. But $\chi_0$ being the general element of $\DD '$, this settles the second case of (c) in general. Noting that $[F_p,\Z (G)]=1$ we also get that since $\chi_0=\chi^z$ for some $z\in \wb G$, $\chi$ is $F_p$-invariant. This implies that it can't be $\gamma$-invariant, by the definition of $\DD$. Then $\chi^\gamma =\chi^{\wh t}$ since $[z,\gamma]\in \spa{\wh t }G=\calL\inv(\spannh)$. This finishes the proof of (c) when $k=2$.

From now on, we assume $k=3$ and we consider the regular embedding $\bG=\SL_4(\FF)\leq \wbG=\GL_4(\FF)$ with extended Frobenius $F$ defined by the same formula on matrix entries. The action of $\wb G$ on $G$ is also induced by the one of $\wG=\wbG^F$ and the action of $\wh t$ is also the one of any $\wt t\in \wG$ corresponding with $h_0=-\Id_4$ in the isomorphism $\wG/G\Z(\wG)\cong \Z(\bG)/{[\Z(\bG),F]}=\Z(G)=\Z(\bG)$, the last equalities a consequence of $|\Z(G)|=4$.

(6)	We consider first the case where $k=3 \und \eps=-1.$ Note that $\UE(\wb G)$ is cyclic thus implying $\EE=\emptyset$ and (b) trivially.
Concerning (c) note that $|\Z(G)|=4$ implies $4\mid (q+1)$ and $q=p^f$ with an odd $f\geq 1$. Then $ \UE(G)_2=\spa{\gamma}$ and $ \UE(G)_{2'}=\spa{F_p\gamma}=\Cent_{\UE(G)}(\Z(G))$. Let $\chi\in\DD$ and $\chi_0\in\oTT$ a $\wG$-conjugate of $\chi$. Then $\UE(G)_{ \chi_0}\lneq \Cent_{\UE(G)}(\Z(G))$ since otherwise any $\wG$-conjugate would also be in $\oTT$. This forces $\spa{\gamma}=\UE(G)_2\leq \UE(G)_{\chi}$. We get $\chi^\gamma =\chi$ but also $\chi_0^\gamma \neq \chi_0$ since the stabilizers of $\chi$ and $\chi_0$ in $(\wG/G\Z(\wG))\rtimes \UE(G)\cong \Z(G)\UE(G)$ can't have the same Sylow 2-subgroup, otherwise they would be equal, the $2'$-part being central. This implies $\chi_0^\gamma = \chi_0^t$ for some $t\in {[\gamma ,\wG]}\setminus G\Z(\wG)$ hence our claim (c) since $[\Z(G), \gamma]=\spannh$. 

%Part (d) 
	
(7)	From now on we assume $k=3$ and $\eps =1$, so that $G= \SL_4( q)$.

For part (c) we consider a character $\chi\in\DD_{}$. The arguments used to prove \Cref{lem3:7}(a) apply also here and we get that the $\wG$-orbit of $\chi$ has 4 elements. Let us show that this $\wG$-orbit is $\gamma$-stable. Let $\wt\chi\in \Irr(\GL_4(q)\mid \chi)$. The fact that $\wG_\chi =G\Z(\wG)$ implies that $\wt\chi$ is in a Lusztig series $\cE (\GL_4(q),[\wt s])$ such that the centralizer of the image of $\wt s$
in $\PGL_4(q)$ has a component group of order 4 (see for instance \cite[Prop. 2.15]{S21D2}). This implies that the eigenvalues of $\wt s$ are an orbit under the multiplication by the 4th root of unity $\varpi$. But then the transpose-inverse automorphism ${\gamma_0}$ of $\GL_4(q)$ sends $\wt s$ to some conjugate times a scalar. This means that $\wt\chi^{\gamma_0}$ is in the same Lusztig series as some $\lambda\wt\chi$ with $\lambda\in\Lin(\GL_4(q))$ by the action of automorphisms on Lusztig series in the connected center case (\cite[Prop. 2.15]{S21D2} again). But those Lusztig series have a single element since the centraliser of $\wt s$ is a torus. So $\wt\chi$ and $\wt\chi^{\gamma}$ have the same restriction to $G=\SL_4(q)$ which gives our claim. We then proceed like in the end of the proof of \Cref{thm_sumup_D}(b). The above implies that if $\chi_0=\chi^z$ is an element of $\oTT$ for some $z\in \wt Z(G)=\wG/\Z(\wG)G$ then $\chi_0^\gamma =\chi^{z'}_0$ for some $z'\in \wZ(G)$. This implies $\chi_0^\gamma =\chi^{z'}_0=\chi_0$ by the definition of $\oTT$, hence the second case of (c). Concerning the first case, note that $\chi^\gamma =\chi$ would imply that $z\gamma z\inv\gamma\inv\in \wZ(G)_\chi =1$ and in turn that $z$ centralizes $\gamma$ hence the whole $\wZ(G).\UE(G)$. But then $^z\chi_0=\chi$ would have the same stabilizer as $\chi_0$. So $\chi^\gamma \neq\chi$ and this implies
 $\chi^\gamma =\chi^{\wh t}$ since $[\wZ(G),\gamma]$ is the group of order 2 generated by the image of $\wh t$. We have proven part (c) in this case. 

(8) We keep $k=3$, $\eps =1$ and address (b). Let $\chi\in \EE$. We then have $(\wG \UE(G))_\chi =\wG_\chi \UE(G)_\chi$ with $\chi$ not extending to $G\UE(G)_\chi$. This forces $\UE(G)$ to be non-cyclic, hence $q$ to be a square and $F$ to act trivially on $\Z(\bG)$. It also implies  that $\UE(G)_\chi$ is not cyclic and therefore $\gamma\in \UE(G)_\chi$. As pointed out in (4) above this implies $h_0\in\ker\chi$.

First we show that  $\wG_\chi\notin \{G\Z(\wt G),\wG \}$. 
Let $\chi'\in\oTT$ be a $\wG$-conjugate of $\chi$ (which exists by (a)).  The equality $\wG_\chi=\wG$ is not possible as then $\chi=\chi'\in \oTT$.

Assume next that $\wt G_\chi= G \Z(\wt G)$ or equivalently ${\wb G}_{\chi}=G$ and hence ${\wb G}_{\chi'}=G$. 
By definition of $\oTT$ we have $(\wb G \UE(\wb G))_{\chi'}= G \UE(\wb G)_{\chi'}$ and $\chi'$ has an extension to  $G \UE( G)_{\chi'}$. 
Via $\UE(\wb G)/\spa{\restr F|{\wb G}}=\UE(G)$ we see that $\chi'$ extends to some $\wt \chi'\in \Irr(G \UE(\wb G)_{\chi'})$ such that $\spa{h_0,\restr F|{\wb G}}\leq \ker(\wt \chi')$. 
Let us write $\chi=\chi'^{ t}$ for some $t\in \wb G$. We have 
\[(\wb G \UE(\wb G))_{\chi}=( G \UE(\wb G)_{\chi'})^t=  G (\UE(\wb G)_{\chi'})^t \] with also
$(\wb G \UE(\wb G))_\chi=\wb G_\chi\UE(\wb G)_\chi$. Then $[\gamma ,t]\in G$ since the above implies $G\gamma^t\subseteq G\UE(\wb G)\cap \bG\gamma =G\gamma$. Then $[[\gamma ,t],F]=[[F,\gamma],t]=1$ and by the Three-Subgroup lemma \cite[(8.7)]{Asch} $[F,t]\in \Cent_{ \Z(\bG)}(\gamma)=\spannh$.  Since $\chi\neq \chi'$ we actually have $[F,t]=h_0$ and we can assume $t=\wh t$.
If now $\chi =(\chi')^{\wh t}$, then $\wh \chi:=(\wt \chi')^{\wh t}$ is an extension of $\chi$ to $ G \UE(\wb G)_{\chi}$ with $\spa{h_0,F}
^{\wh t}=\spa{h_0,h_0F}\leq\ker( \wt \chi)$. In particular $F\in \ker(\wt \chi)$. Via  $\UE(\wb G)/\spa{F}=\UE(G)$ the character $\wh \chi$ defines also an extension of $\chi$ to $G \UE( G)_\chi$, contradicting $\chi\in\EE$. This contradiction establishes our claim that $\wG_\chi\notin \{G\Z(\wt G),\wG \}$, or equivalently that $\wt G_\chi$ has index $2$ in $\wt G$ while being also $\gamma$-stable. 

 Let $s\in \HF$ be such that $\chi\in\cE(G,[s])$. Denote $\wt \bH =\wbG^*=\GL_4(\FF)$. As in the above proof of (c), let $\wt s \in\wt \bH^F$ and $\wt \chi\in \Irr(\GL_4(q)\mid \chi)$, such that $\wt \chi\in \cE(\GL_4(q),[\wt s])$. 
	Then $|\wt G_\chi:G\Z(\wG)|=2$ implies 
 \begin{align}\label{eq33}
 	|A_\bH(s)^F_\phi|=2, 
 \end{align}
	where $\phi:=\Psi^{\wGF}_{\wt s}(\wt \chi)\in \UCh(\Cent_{\wt \bH}(\wt s)^F)$ is associated to $\wt\chi$ via Jordan decomposition of characters $\II{PsiGs}@{\Psi^{(\wGF)}_{\wt s}}: \cE(\wGF,[\wt s]) \lra \UCh(\Cent_{\wt \bH}(\wt s)^F)$, see \cite[Prop.~2.15]{S21D2}.
	
	Assume $|A_\bH(s)|=4$. Then as in the discussion made for (c) one gets that $\Cent_\wbH(\wt s)$ is a torus and $\phi$ is the trivial character. This leads to $A_\bH(s)^F_\phi=A_\bH(s)^F=A_\bH(s)$ since $F$ acts trivially on $\Z(\bG)$ hence on $\AHs$. This contradicts (\ref{eq33}) above.
	
	We therefore get that \begin{align}\label{eq34}
		A_\bH(s)=A_\bH(s)^F_\phi\text{ and has order 2}.
	\end{align} 
	
	Like for \Cref{thm_sumup_D_1}(a), our aim is to apply \cite[Cor. 6.6]{S21D2} whose original proof was for types $\tD_l$ ($l\geq 4$) but is easily seen to apply to our $G$ of type $\tA_3$. We have proven assumptions (i) and (ii) of \cite[Cor. 6.6]{S21D2}. Letting $\pi\colon \bH_0=[\wbH ,\wbH]\to\bH$ be the canonical surjection $\SL_4(\FF)\to\PGL_4(\FF)$, there remains to show that there exists $\gamma^*\in \HF\gamma$ such that $\gamma^*(s,\phi)=(s,\phi)$ and $[s_0,\gamma^*]\in \spannh$ for some (in fact any) $s_0\in\pi\inv(s)$. The first part is ensured by the equivariant Jordan decomposition for type $\tA$ (see \cite[Thm 8.2]{CS17A} or \cite[Thm B]{S21D2}), so we concentrate on the second point, namely having $[\pi\inv(s),\gamma^*]\subseteq \spannh$.
	
		Let $s_0\in\pi\inv(s)\subseteq \bH_0=\SL_4(\FF)$ lifting $s$. Then $|A_\bH(s)|=2$ from (\ref{eq34}) above implies that the set of eigenvalues of $s_0$ is closed under multiplication with $-1$. Let $\kappa$, $\kappa'$ be two of them, such that $\{\kappa, -\kappa, \kappa',-\kappa'\}$ is the set of all eigenvalues of $s_0$ with multiplicities. Because of $\det (s_0)=1$ we get $\kappa'=\pm \kappa\inv$ and hence the set of eigenvalues is $\{\kappa, -\kappa, \kappa\inv,-\kappa\inv\}$. Writing $\sim$ for the conjugation in $\bH_0=\SL_4(\FF)$, which for semisimple elements coincides with conjugation in $\GL_4(\FF)$, we see from the eigenvalues that $$s_0\sim s_0\inv\sim\gamma^*(s_0).$$ Since $\gamma^*(s)=s$ then $\gamma^*(s_0)=\omega s_0$ for some $\omega\in\FFtimes$, and $s_0\sim\omega s_0$ by the above. This forces in turn $\omega =\pm 1$ since $A_\bH(s)$ has order 2. So $\gamma^*(s_0)=\pm s_0$ and we get our claim that $[\pi\inv(s),\gamma^*]\subseteq \spannh =\{\pm\Id_4\}$. 
\end{proof}
%%%%%%%%%%%%%%%%%%%%%%%%%%%%%%%%%%%%%%%%%%%%%%%%%

The following definition makes precise what is mmeant by groups of type $\tD$ in ranks $\leq 3$, e.g. $\tD_{k,\sico}^\eps(q)$ like $G_{\underline{k}}$ above for $1\leq k\leq 3$. (Most textbooks consider types $\tD_l$ only for $l\geq 4$.) This will be used mainly in Chapters \ref{sec_nondreg_groupM} and \ref{sec6}.

\begin{defi} \label{Def_low_l} Assume $k\in \{1,2,3\}$ and $\eps =\pm 1$. Then we denote by $\tD_{k,\sico}(\FF)$ the subgroup $\bG_{\underline k}$ of $\bG=\tD_{4,\sico}(\FF)$ defined in \Cref{not:3_16}. We also write $\tD_{k,\sico}^\eps(q)$ for $\bG_{\underline{k}}^F$ where $F$ is such that $\bG^F=\tD_{4,\sico}^\eps(q)$. Then we get the isomorphisms with groups of type $\tA$ and a torus recalled in (\ref{eq32})
of the above proof.\end{defi}
 
We gather some extra statements that will be used alongside \Cref{thm_sumup_D}.

\begin{lem} \label{cor:3_24} Assume $k$ and $\wh t_\uk$ are like in \Cref{thm_sumup_D}. If $\chi\in\Irr(G_\uk)$ is such that $({\wb G}_\uk)_{\chi}\leq G_\uk\spa{\wh t_\uk}$, then $({\wb G}_\uk \UE({\wb G}_\uk))_{\chi}\leq G_\uk\spa{\wh t_\uk} \UE({\wb G}_\uk)$.
\end{lem}
\begin{proof} As in the proof of \Cref{lem3:7} we abbreviate $\wZ ={\wb G}_\uk/G_\uk$ , $E=\UE({\wb G}_\uk)$, $\wh h_0$ the class of $\wh t_\uk$. Note that $\wh h_0$ is central in the semidirect product $\wZ \rtimes E$ thanks to \ref{ZG_order} and that $E$ also acts trivially on $\wZ/\spa{\wh h_0}$ since this is of order $\leq 2$. By \Cref{thm_sumup_D}(a) there exists $z\in \wZ$ such that $(\wZ E)_{^z\chi}= Z_0E_0$ with $Z_0\leq \wZ$ and $E_0\leq E$. Since $\wZ$ is abelian we have $Z_0=\wZ\cap Z_0E_0= (\wZ\cap Z_0E_0)^z=\wZ\cap( Z_0E_0)^z=\wZ\cap (\wZ E)_\chi =\wZ_\chi\leq \spa{\wh h_0}$ by our hypothesis. Then $(\wZ E)_\chi =(Z_0E_0)^z\leq (\spa{\wh h_0}E)^z \leq \spa{\wh h_0}E$ since we have seen that $[\wZ ,E]\leq \spa{\wh h_0}$.
\end{proof}

\begin{lem} \label{max_ext_Gi_wGi}\label{rem_maxext_smallrank}
Assume $k\in \ul$, then maximal extendibility holds with respect to $G_\uk\unlhd{\wb G}_\uk $.
\end{lem}
\begin{proof} If $k\geq 3$, then the statement is \cite[Thm 2.17]{S21D1} thanks to our description of $\bG_\uk$ given in the proof of \Cref{thm_sumup_D}. When $k=1$ then ${\wb G}_\uk$ is abelian, while when $k=2$ and $F=F_q\gamma$ the quotient ${\wb G}_\uk/G_\uk$ is cyclic and \Cref{ExtCrit}(a) applies. When $k=2$ and $F=F_q$ then 
	$G_{\underline 2}= \SL_2(q)\times \SL_2(q) \unlhd \NNN_{\SL_2(\FF)}(\SL_2(q))\times \NNN_{\SL_2(\FF)}(\SL_2(q))$ and we have again the cyclic quotient situation on each factor.
\end{proof}

 %%%%%%%%%%%%%%%%%%%%%%%%%%%%%%%%%%%%%%%%%%%%%%%%%%%%%%%%%%%%%

\section{The doubly regular case}\label{sec:4}
In this chapter, we prove the inductive McKay condition (\textbf{iMK}) for the group $\GF=\tD_{l,\sico}^\eps(q)$ of \Cref{sub2E} and an odd prime $\ell$ under an arithmetic assumption on $d_\ell(q)$, the order of $q$ in $(\ZZ/\ell\ZZ)^\times$. This condition of being a\textit{ doubly regular} number for $(\bG,F)$, see \Cref{def_dreg} below, ensures that the Sylow $\ell$-subgroups of $\GF$ have abelian centralizers in $\bG^F=\tD_{l,\sico}^\eps(q)$ but also in the overgroup $\obG^F=\tB_{l,\sico}(q)$.
The criterion for (\textbf{iMK}) is the one of \Cref{iMKbyAdBd} through the conditions \Ad{} and \Bd{} of \ref{cond_Ad} and \ref{cond_Bd} defined for $(\bG ,F)$ and any integer $d\geq 1$, relating in particular to Sylow $d$-tori of $\bG$ as recalled in \Cref{ssec_2D}. 

The main result of the present chapter is as follows.

\begin{theorem}\label{thm:dreg} Let $\GF=\tD_{l,\sico}^\eps(q)$ for $l\geq 4$, $\eps=\pm 1$, $q$ a prime power.
	Let $d\geq 3$ and $\bS$ be a Sylow $d$-torus of $(\bG,F)$. 
\[ 	\text{Assume that $d\mid 2l$ and the ratio $\frac{2l}d $ is } \begin{cases} 	\text{ \text{even} \ when } & \eps =1 ,\textbf{} \\
		\text{ \text{odd} \ when } & \eps =-1,
	\end{cases} \] (in other words $d$ is a doubly regular integer for $(\bG,F)$, see \Cref{def_dreg}).
	Then:
	\begin{thmlist}
		\item Conditions \Ad{} and \Bd{} from \Cref{ssec_2D} hold for $(\bG,F)$.
		\item If $\ell$ is an odd prime with $\ell\nmid q$ and $d_\ell(q)=d$, then (\textbf{iMK}) from \Cref{iMK} holds for $\GF$ and $\ell$ with respect to $N=\NNN_\bG(\bS)^F$.
		\item (\textbf{iMK}) holds for $\tD_{4,\sico}^\eps(q)$ and any prime $\ell$.
	\end{thmlist}
\end{theorem} 
Note that for $d\in\{1,2\}$ and primes $\ell = 2$ or odd with $d_\ell(q)\in\{1,2\}$ the above points (a) and (b) are known from \cite[Thm~1]{MS16}. 

The proof will use a criterion for \Ad\ and \Bd\ from \cite{CS17C} devised for the more general case where $d$ is a regular number for $(\bG ,F)$, see \Cref{dreg_crit} below. This will split the proof of \Cref{thm:dreg}(a) into two parts. \Cref{ssec4B} deals with the Tits subgroup $\spa{\nn_\al(1)\mid \al\in \Phi(\bG,\bT)}\leq\NNN_{ \bG}(\bT)$. Then \Cref{ssec4C} shows properties of the relative Weyl groups $\NNN_\bG(\bS,\xi)^F/\Cent_\bG(\bS)^F$ for $\bS$ a Sylow $d$-torus of $\bG$ and $\xi\in\Irr(\Cent_\bG(\bS)^F)$. The proof of \Cref{thm:dreg}
concludes in \Cref{ssec_4D}. After that we derive some consequences of the constructions made, in particular correspondences of characters not necessarily of $\ell '$-degrees, that will be useful later in Section 6.A.

%%%%%%%%%%%%%%%%%%%%%%%%%%%%%%%%%%%%%%%
\subsection{The conditions \Ad\ and \Bd\ for regular numbers} \label{subsec_dreg_def}

Let $(\bG ,F)$ be a simple simply connected algebraic group defined over a finite field $\FF_q$ as in \Cref{ssec2C}. We recall the choice of a maximal torus and Borel subgroup $\bT\leq \bB$ both $F$-stable and the associated notation. Let \[\rho\colon \NNN_{ \bG}(\bT)\lra W(\bG,\bT)=\NNN_{ \bG}(\bT)/\bT\] be the canonical surjection onto the Weyl group. 
We see the latter as acting on the euclidean vector space ${\mathbb V} =\RR \otimes X(\bT)$, abbreviating the image of $W(\bG,\bT)$ as $W\leq \GL(\VV)$. We recall $\varphi_0\in \NNN_{ \GL(\VV)}(W)$, the element of $\GL(\VV)$ such that $F$ induces $q\varphi_0$ on $X(\bT)$, see \cite[Def. 22.10]{MT}. Writing $F=F_p^f\circ \si$ for $\si$ a graph automorphism as in \Cref{ssec2C}, $\varphi_0$ corresponds to the symmetry of the basis of the root system that defines $\si$.

For regular elements of Weyl groups and regular numbers we refer to \cite[Ch. 5]{BroueBook} and \cite{Sp74}. \index{regular element}\index{regular number}For the relation with $d$-tori, see \cite[Thm 25.10]{MT} and \cite[Sect. 3.5]{GM}.

\begin{defi}[Regular elements of $W\varphi_0$]\label{def_varphi0} An element of $W\varphi_0$ is called regular if it has an eigenvector in $\CC\otimes \VV$ which is not contained in the reflecting hyperplane to any $\alpha\in\Phi(\bG,\bT)$. It is called $\zeta$-regular if this eigenvector is for the eigenvalue $\zeta\in\CC^\times$.
	The integers $d\geq 1$ such that there is a $\zeta$-regular element in $W\varphi_0$ with $\zeta$ of order $d$ are called the \textit{regular numbers} for $(W,\varphi_0)$, or equivalently for $(\bG,F)$. An important property in relation with the polynomial order of $F$-stable subgroups of $\bG$ is that\textit{ $d$ is a regular number for $(\bG,F)$ if and only if a Sylow $d$-torus $\bS$ of $\bG$ is such that $\Cent_{ \bG}(\bS)$ is a torus} \cite[Example 3.5.7]{GM}. In particular we get that if $w\varphi_0\in W\varphi_0$ is $\zeta$-regular for some $\zeta$ of order $d$ and if $u\in\rho\inv(w) $ then the Sylow $d$-torus of $(\bT,uF)$ is a Sylow $d$-torus of $(\bG ,uF)$.
\end{defi}

See \cite[Table 1]{S10a} for a list of regular numbers for each type of $(W,\varphi_0)$.

 Keeping the pair $(\bG,F)$ of arbitrary type, we define the extended Weyl group (or Tits subgroup, see \cite{T}) already mentioned
$\III{V}:=\spa{\nn_\al(1)\mid \al\in \Phi(\bG,\bT)}\leq\NNN_{ \bG}(\bT)$ and the toral group $\III{H}=V\cap\bT$. They are both finite and we clearly have $\rho (V)=W(\bG,\bT)\cong V/H$.

Recall the choice of a regular embedding $\bG\leq \wbG$ and the group $E(\bG)$ acting on both $\bG$ and $\wbG$, see \Cref{ssec2C}. We set $\wt\bT=\Z(\wbG)\bT$. We form the semidirect product $\wbG\rtimes E(\bG)$ and see its elements outside of $\wbG$ as acting on $\bG$. 

Here is the slight variant of the criterion \cite[Thm~4.3]{CS18B} for \Ad\ and \Bd\ of \Cref{ssec_2D} that we will use.

\begin{prop}\label{dreg_crit}
	Let $d \geq 1$ be a regular number for $(\bG,F)$. Assume there exists an element $u \in V$ such that, denoting by $\bS$ the Sylow $d$-torus of $(\bT, uF )$, the following properties hold:
	\begin{asslist}
		\item $\rho (u)\varphi_0\in W\varphi_0$ is a $\zeta$-regular element of $W\varphi_0$ for some $\zeta\in \CC^\times$ of order $d$;
		
		\item $\rho(V_d) =W_d$ with $\II Vd@{V_d} := V ^{uF}$ and $\III{W_d}:= \Cent_W (\rho(u)\varphi_0)$.
		\item Set $\II{Vhatd}@{\protect{\wh V}_d}:=(V E(\bG) )^{uF}/\spa{uF} \geq V_d\geq \III{H_d}:= H^{uF}$. There exists an extension map $\Lambda_0$ with respect to $H_d \unlhd V_d$ such that 
		\begin{enumerate}
			\item[(iii.1)] $\Lambda_0$ is $\wh V_d$-equivariant; and
			\item[(iii.2)] if $E(\GF)$ is not cyclic, then for any $\la \in \Irr(H_d)$, $\Lambda_0(\la)$ extends to $(\wh V_d)_\la$.
		\end{enumerate}
		\item Set $ {T:=\bT^{uF}}=\Cent_{ \bG}(\bS)^{uF}$, $ N:=\NNN_{\bG }(\bS)^{uF}$, $\wh N:=\NNN_{\bG E(\bG)}(\bS)^{uF}/\spa{uF}$ and $\wh W_d :=\wh N/T$. 
		For $\wt \xi\in \Irr(\wt \bT^{uF} )$ set $W_{\wt \xi} := N_{\wt \xi}/T\unlhd \wh K(\wt \xi) := ({ \wh W_d})_{\restr\wt \xi|{T}}$. Then for every $\wt\xi \in \Irr(\wt \bT^{uF} )$, \maex holds for $W_{\wt \xi}\unlhd \wh K(\wt \xi)$.
		\item Maximal extendibility holds with respect to $N \unlhd \wt N :=\NNN_{\wbG}(\bS)^{uF}$.
	\end{asslist}
	Then conditions \Ad{} and \Bd{} from \ref{cond_Ad} and \ref{cond_Bd} are satisfied by $(\bG,F)$. 
\end{prop}

\begin{proof} We abbreviate $\II Ttilde@{\wt T}:=\wt \bT^{uF}$. The first point of the definition of \Bd{} in \Cref{cond_Bd} is given by (v). For the second point in \Bd, i.e., an equivariant form of \maex for $\wt T\unlhd \wt N$, it is a consequence of the assumptions (i), (ii) and (iii), as explained in the proof of {\cite[Thm~4.2]{CS18B}}, see also \Cref{Ext_f}.

	We now turn to \Ad. Note that the proposition reproduces {\cite[Thm~4.3]{CS18B}} verbatim except for (iv). We follow below the original proof of \Ad\ combining the ones of \cite[Prop. 5.13]{CS17A} and {\cite[Thm~4.3]{CS17C}}. We only stress the differences due to our simpler assumption (iv).

Assumption (iii) along with \Cref{Ext_f} yield an $\wh N$-equivariant extension map $\Lambda$ for $T\unlhd N$, see also \Cref{ExtMapHdVd}(e) below for an alternative proof. Let us recall also that if $\wt\xi\in\Irr(\wt T)=\Lin(\wt T)$ and $\xi:=\restr\wt\xi|{T}$, then $N_{\wt\xi}\unlhd N_\xi$ and both are normal in $\wh N_\xi$. 
For every $t\in\wt T$ the character $\nu_t \in\Lin(N_\xi)$ defined by $$\Lambda(\xi)^t=\Lambda(\xi)\nu_t $$ satisfies $\nu_t (N_{\wt\xi})=1$. For a fixed $\wt\xi\in\Irr(\wt T)$ the map $t\mapsto\nu_t$ is a surjection $\wt T_{}\to \Irr(N_\xi/N_{\wt\xi})=\Lin(N_\xi/N_{\wt\xi})$ with kernel $\wt T_{\La(\xi)}$.
Clifford theory (\ref{Cliff}) provides an $\wh N$-equivariant parametrization \begin{align*}
\Pi:\ &\calP \rightarrow \Irr(N)\\ \text{ } &(\xi,\eta) \longmapsto (\Lambda(\xi)\eta)^N
\end{align*}
of $\Irr(N)$ by the set $\calP$ of $N$-conjugacy classes of pairs $(\xi, \eta)$ with $\xi\in\Irr(T)$ and $\eta\in\Irr(W_\xi)$. By the above one clearly has $\Pi(\xi,\eta)^{t}=\Pi(\xi,\eta \nu_t )$ for every $t\in\wt T$. 

Letting now $\chi\in\Irr(N)$ we must check $(\wt N\wh N)_{\chi_0} =\wt N_{\chi_0}\wh N_{\chi_0} $ for some $\wt T$-conjugate $\chi_0$ of $\chi$. 

We choose $(\xi,\eta)$ as above such that $\chi =\Pi(\xi,\eta)$. Let $\wt\xi\in \Irr(\wt T)$ be an extension of $\xi$. 
	Let $\eta_0\in \Irr(\restr \eta|{W_{\wt\xi}} )$. By the assumption (iv),
	$\eta_0$ extends to $\wh K(\wt \xi)_{\eta_0}$ and hence 
	there exists some extension $\wt \eta_0$ of $\eta_0$ to $(W_\xi)_{\eta_0}$ that further extends to $\wh K(\wt \xi)_{\eta_0}$. The character $\eta':=\wt \eta_0^{W_{ \xi}}$ is $\wh K(\wt \xi)_{\eta_0}$-invariant and irreducible thanks to Clifford theory (\ref{Cliff}), so $(\xi ,\eta ')$ is another pair as above and we can define $\chi_0:=\Pi (\xi,\eta ')\in \Irr(N)$. Note that $\chi_0$ is in the $\wt T$-orbit of $\chi$ by the surjectivity of $t\mapsto \nu_t$ recalled above. % $\widetilde{\wt\eta}_0$ and by the construction extends also to $W_\xi \wh K(\wt \xi)_{\eta_0}$. 

	Let $x\in(\wt N \wh N)_{\chi_0}$. By the properties of $\Pi$ we can assume $x=\wt n \wh n$ with $\wt n\in \wt N_\xi$ and $\wh n\in \wh N_\xi$ after factoring out an element of $N$. 
	Then $\wt n=tn $ with $t\in \wt T$ and $n\in N_\xi$. 
	Arguing as in the proof of \cite[Prop. 5.13]{CS17A} we get 
	\begin{align*}
		\chi_0=\Pi(\xi,\eta')=\Pi(\xi,\eta')^{\wt n \wh n}=&\ \Pi (\xi , (\eta'\nu_t )^{n \wh n}) \\
	=&\ \Pi (\xi , (\eta'\nu_t )^{ \wh n}) \text{ using }(\eta')^n=\eta'\text{ since }n\in N_\xi.
		\end{align*} This implies $\eta'= (\eta'\nu_t)^{ \wh n}$ and therefore $\wh n$ stabilizes $\restr \eta|{W_{\wt \xi}}$. 
	Accordingly $\wh n\in \wh N_{ \xi}$ and $\wh n T\in W_\xi \wh K(\wt \xi)_{\eta_0}$. 
	By what was recalled above, $\eta'$ is $\wh n$-invariant. This yields $\Pi(\xi,\eta')^{\wh n}=\Pi(\xi,\eta')$ and therefore 
	$x\in \wt N_{\chi_0}\wh N_{\chi_0}$ as claimed.
	
	It remains to show that $\chi_0$ extends to $\wh N_{\chi_0}$. By assumption $\eta_0$ extends to $\wh K(\wt \xi)_{\eta_0}$ and hence there exists some extension $\wh \eta_0$ of $\wt \eta_0$ to $(\wh N_{\chi_{{}_0},\xi})_{\eta_0}/T$ since the latter is a subgroup of $\wh K(\wt \xi)_{\eta_0}$. The induced character $(\wh \eta_0)^{\wh K(\wt \xi)}$ is the extension of $\eta'$ required in the end of the proof of \cite[Thm 4.3]{CS17C} and denoted there as Res$^{{\wh N}_\xi}_{{\wh N}_{\xi ,\eta}}(\wh \eta)\in\Irr ({{\wh N}_{\xi ,\eta}})$.
	\end{proof}
	
\subsection{Doubly regular numbers and Sylow tori}\label{ssec_dregdef}
 
We now focus on the case when $\bG=\tDlsc(\FF)$ and $\ov \bG=\tBlsc(\FF)$ with $\bG\leq \ov \bG$, with Frobenius endomorphism $\ov F\colon \obG\to\obG$, $F=\restr \ov F|{\bG}$ and $\GF= \tDlsc^\eps(q)\leq \obG^{\ov F}=\tBlsc(q)$ as in \ref{sub2E}. As in \cite{S10a,S10b,MS16} we derive the results in type $\tD_l$ by some transfer of the analogous results in $\ov\bG$. We also assume that $d\geq 3$ satisfies the assumption of \Cref{thm:dreg}. In view of the list of regular numbers in types $\tD$ and $\tB$, see for instance \cite[Table 1]{S10a}, it corresponds to the following. 

\begin{defi}[Doubly regular integers for $\twepsDlq$]\label{def_dreg}
An integer $d\geq 3$ is called \textit{doubly regular} for $(\bG,F)$ if $d$ is regular for both $(\bG,F)$ and $(\obG, \ov F)$, i.e. $d\mid 2l$ with ratio satisfying $(-1)^{\frac{2l}d}=\eps$. \index{doubly regular}
\end{defi}

This is equivalent to the property that the centralizer in $\obG$ of a Sylow $d$-torus of $(\bG ,F)$ is a (maximal) torus.
Using \Cref{def_varphi0}, this can be checked by noting that the ratio of polynomial orders $P_{\obG,\ov F}(X)/P_{\bG,F}(X)$ is $X^l(X^l+\eps)$, not divisible by the $d$-th cyclotomic polynomial if $d$ is doubly regular.

Recall we identify the Weyl group $\ov W$ of type $\tB_l$ with $\Sym_{\pm l}$, where for a set $I\subseteq \ZZ_{>0}$, $\II Spm@{\Sym_{\pm I}}$ denotes the group of permutations $\pi$ of $I\sqcup -I$ with $\pi(-i)=-\pi(i)$ for every $i\in I$ and $\Sym_{\pm l}$ corresponds to $I=\ul$.
Recall also $W$ the normal subgroup of $\ov W$ corresponding to the Weyl group of type $\tD_l$. Then $\varphi_0$ associated with $F\colon \bG\to\bG$ is $1$ if $\eps =1$, while $\varphi_0=\ov\rho(\neins)=(1,-1)\in\ov W$ when $\eps =-1$.
We collect some statements on regular elements of $\ov W$ in the following.

\begin{lem}\label{lemreg4:8}
	Let $d\geq 1$ be an integer and $\zeta_d$ a primitive $d$-th root of unity. 
	\begin{thmlist}
		\item The set of $\zeta_d$-regular elements of $W\varphi_0$ forms a $W$-orbit. 
		\item $w_{\text{Cox}}:=(1,2,\dots ,l,-1,\dots ,-l+1,-l)$ is a Coxeter element of $\ov W$. Every regular element $w$ of $\ov W$ is conjugate to a power of $w_{Cox}$. If the order of $w$ is $d$, every $w$-orbit in $\ul\cup -\ul$ has length $d$.
		\item If $w$ is a regular element of $\ov W$ of order $d$ and $I\subseteq \ul$ is a union of $w$-orbits, then the projection $w_I$ of $w$ to $\Sym_{\pm I}$ is a regular element of order $d$ of $\Sym_{\pm I}$. All regular elements of $\ov W$ of a given order form a single conjugacy class of $\ov W$.
		\item Assume $d\geq 3$ is doubly regular for $(\bG,F)$. If $w$ is a regular element of $\ov W$ of order $d$, then it is a $\zeta_d$-regular element of $W\varphi_0$. 
	\end{thmlist}
\end{lem}
\begin{proof} Part (a) follows from the definition, see also \cite[5.14.1]{BroueBook}. Part (b) follows from Appendix 1 of \cite{BrMi} or from Remark 3.2 of \cite{S10b}. For Part (c), recall that $d$ is a regular number for $\ov W$ and therefore $d\mid 2l$ according to Table 1 of \cite{S10a}. Now the projection $w_I$ of $w$ is a regular element of $\Sym_{\pm I}$ since the orbits of $w$ and $w_I$ on $\pm I$ have the structure required in Remark 3.2 of \cite{S10b}. Part (d) is again derived from the characterization of $\zeta_d$-regular elements given in Table 1 of \cite{S10b}. \end{proof}

\subsection{The doubly regular case. Extended Weyl groups} \label{ssec4B}
%%%%%%%%%%%%%%%%%%%%%%%%%%%%%%%%%%%%%%%%%%%%%%%%%%%%%%%%%%

In the following, we assume $d$ to be doubly regular for $(\bG ,F)$ in the sense of \Cref{def_dreg} and show that \Cref{dreg_crit} applies in that situation to establish \Cref{thm:dreg}(a).
We check assumptions (i)-(iii) of \Cref{dreg_crit} in the present section and assumptions (iv) and (v) in the next. 

We define first an element $u\in V$. Recall that we have chosen $\varpi$ to be a primitive $\gcd(2,q-1)^2$-th root of unity, see \ref{sub2E}. 

\begin{notation} \label{notVdHd} Recall $V= \spa{ \n_\al(1) \mid \al \in \Phi(\bG,\bT)} $, $H=V\cap\bT$ and set $$V\leq \II{Voverline}@{\ov V}:= \spa{ \n_\al(1) \mid \al \in \Phi(\obG,\bT)} \leq \NNN_\obG(\bT)^{F_p}.$$
	Let $\II vtB@{\vtB}\in {\ov V}$ be chosen as in \cite[Sect. 5.A]{CS18B}, i.e.
	$$\vtB:=\big(\nn_{ e_1}(1)\nn_{ \al_2}(1)\cdots \nn_{\al_l}(1)\big)^{\frac {2l} d} \ \ \ \text{ and set }\ \ \ t=\hh_{\ul}(\omega)=\prod_{i=1} ^l \hh_{e_i}(\omega)\in \bT$$ 
 with 
	$\II omega@{\omega}\in\FFtimes$, a primitive $\gcd(2,q-1)^3$-th root of unity with $\omega^2=\varpi$. Let
	\begin{align}\label{def_u}
		\III{u}&:=\begin{cases} (\vtB)^t&\text{if } \eps =\,\,\,1,\\ (\vtB)^t\nn_{ e_1}(\varpi)\inv &\text{if } \eps=-1.\end{cases}
	\end{align}
	Let $ {V_d} = V\cap \bG^{u F}$ and $ {H_d}=V_d\cap \bT $.% as in \ref{dreg_crit}. 
%	
%	Recall $q=p^f$ for a prime $p$ and integer $f$. Let $\wh m=2|V|f$ and $\wh E=\Cent_{ \wh m}\times\spa{\gamma}$, the latter acting on $\bG^{F_p^{\wh m}}$ by $F_p$ for the generator of the first summand $\Cent_{ \wh m}$. We denote $\wh F\in\wh E$ such that it acts on $\bG^{F_p^{\wh m}}$ as $F\in \{F_p^m, F_p^m\gamma\}$ acts on $\bG$. 
\end{notation}

Basic properties of the Tits subgroup imply the following.

\begin{lem} \label{ovV_V}\begin{thmlist}
\item $\ov V\cap \bG =V$ and $\ov V\cap \bT=V\cap \bT=H$ is  elementary abelian of order $\gcd(2,q-1)^l$ with $V/H\cong W$, $\ov V/H\cong \ov W$.
\item $ \nn_{e_1}(1)^t= \nn_{ e_1}(1) \hh_{ e_1}(\omega^2)= \nn_{ e_1}(\varpi)=\neins.$
\item $u\in V=V^t$.
	\end{thmlist}
\end{lem}
\begin{proof} From the Chevalley relations or \cite{T}, we know for any type that the group $V=\spann<\nn_\al(1)\mid \al\in\Phi(\bG,\bT)>$ is generated by the $\nn_\delta(1)$ for $\delta\in \Delta$ in the notation of \Cref{ssec2C}. Those elements satisfy $\nn_\delta(1)^2=\hh_\delta(-1)$ and the braid relations. From the Coxeter presentation of $W(\bG,\bT)$ it is then easy to see that $H$ is the direct product $\Pi_{\delta\in \Delta}\spa{\hh_\delta(-1)} $,  with $V/H\cong W(\bG,\bT)$ by $\rho$. We get (a) by applying this simultaneously to $\bG$ and $\obG$, noting that $\ov V\cap \bT\geq V\cap \bT$ but have same cardinality. 
	
	(b) is clear from the definitions given in \Cref{ssec2C} and $e_i(t)=\omega^2=\varpi$ when $e_i$ ($i\in\ul$) is seen as an element of $ \Phi(\obG,\bT)\subseteq X(\bT)$. One shows similarly that $\nn_{\al}(1)^t=\nn_\al(\pm 1) $ when $\al =\pm e_i\pm e_j$ with $i\neq j$ in $\ul$. This implies $V^t=V$. 
	
	When $\eps =1$ the assumption that $d$ is doubly regular implies that $2l/d$ is even and therefore $v_\tB\in V$ since all squares of elements of $\ov V$ are in $V$. When $\eps =-1$ we have $v_\tB\in \ov V\setminus V = \nn_{ e_1}(1) V$ since $\nn_{ e_1}(1)\nn_{ \al_2}(1)\cdots \nn_{\al_l}(1)\in \ov V\setminus V$ and $2l/d$ is odd. But then (b) implies $v_\tB .{}^t\nn_{ e_1}(\varpi)\inv \in \nn_{ e_1}(1)V\nn_{ e_1}(1)\inv =V$ and we get our claim that $u\in V^t=V$.
This finishes the proof of (c). 
\end{proof}

We now prove the requirements (i)--(iii) of \Cref{dreg_crit} in a form slightly weakened for rank 4, replacing $E(\bG)$ with $\UE(\bG)$ as defined in \Cref{sub2E}.
% The following is a generalisation of Proposition 3.10 and Lemma 3.12 of \cite{MS16} where it was proven for the case where $d=1$. 

\begin{prop}\label{ExtMapHdVd} Let $u\in V$ be as in \eqref{def_u}.
	Let $\varphi_0$ be defined from $F$ as in \Cref{cor3_8}. 
	Then 
	\begin{thmlist}
		\item $\rho(u)\varphi_0$ is a $\zeta_d$-regular element of $W\varphi_0$ for some primitive (complex) $d$-th root $\zeta_d$ of $1$;
		\item $\rho(V_d)=\cent W{\rho(u)\varphi_0}$;
		\item there exists a $\Cent_{V \UE(\bG)}(uF)$-equivariant extension map $\Lambda_0$ with respect to $H_d\unlhd V_d$ and, 
		\item if $\eps =1$, then $\Lambda_0$ can be chosen such that for every $\la\in \Irr(H_d)$, $\Lambda_0(\la)$ extends to an irreducible character of $\Cent_{V \UE(\bG)}(u F)_\la/\spa{u F}$.
		\item Recall $\bN=\NNN_{ \bG}(\bT)$. There is a $\Cent_{\bN  \UE(\bG)}(uF)$-equivariant extension map \wrt $\bT^{uF}\unlhd\bN^{uF}$.
	\end{thmlist}
\end{prop}
\begin{proof} 
	% along with the divisibility conditions on $d$ and $l$ from \Cref{def_dreg}.
	Set $ {\ov V\! _d}:=\cent{\ov V}{{}^t{(uF)}}=\cent{\ov V}{v_\tB}$ as in \cite[Sect. 5.B]{CS18B}, $\ov H_d={\ov V\! _d}\cap \bT =\ov V\cap \cent{\bT}{{}^t{(uF)}} =\ov V\cap \bT^{uF}=H_d$ thanks to \Cref{ovV_V}(a). By \Cref{lem_gamma}, $W\varphi_0$ and $W\ov\rho(\neins)$ are equal when $\eps =-1$. Regarding $\rho(u)\varphi_0$ we have $\rho(u)\varphi_0=\ov\rho({}\vtB^t)=\rho({}\vtB)$ and this is a $\zeta_d$-regular element in $W\varphi_0$ since it is one in $\ov W$ as proven in \cite[Lem. 5.4]{CS18B}. Again, by \cite[Lem. 5.4]{CS18B} the groups $ \ov \bN =\NNN_\obG(\bT)$ and $ \bT$ satisfy $\ov \bN ^{\vtB F}=(\ov \bT )^{\vtB F} {\ov V\! _d}$ or equivalently $$\ov \rho({\ov V\! _d})=\cent{\ov W_{ }}{\ov \rho(\vtB)}.\eqno(1)$$ 
	
	On the other hand $V_d=\cent V {uF}= \cent {V^t} {(\vtB)^t }= (\cent {V} {\vtB })^t$ by \Cref{ovV_V}(c) and therefore $\rho(V_d)=\rho(\cent V {uF})=\cent W {\rho(\vtB) }$ as claimed. This ensures parts (a) and (b). 
	
	According to \cite[Thm 5.5]{CS18B} maximal extendibility holds with respect to $H_d=\ov H_d \unlhd {\ov V\! _d}$, hence also \wrt $H_d\unlhd 
	{\ov V\! _d}^t$. Let $\Lambda_\tB$ be an extension map \wrt $H_d\unlhd {\ov V\! _d}^t$, which can be assumed to be ${\ov V\! _d}^t$-equivariant. Since $V_d=V\cap \obG^{uF}=V^t\cap \obG^{uF}$ and $ {\ov V\! _d}^t=\ov V^t\cap \obG^{uF}$, we have $V_d\unlhd {\ov V\! _d}^t$. Then maximal extendibility holds \wrt $H_d\unlhd V_d$ with the extension map $\Lambda_0$ defined by $\Lambda_0(\la)=\restr\Lambda_{\tB}(\la)|{V_d}$ for $\la\in\Irr(H_d)$. Moreover, $\Lambda_0$ is ${\ov V\! _d}^t$-equivariant. So, to obtain our claim (c) about equivariance, it suffices to show that the automorphisms of $V_d$ induced by $\cent{V\UE(\bG) } {uF}$ are also induced by ${\ov V\! _d}^t=\Cent_{\ov V^t}{(uF)}$. Since $[F_p,\gamma]=[F_p,V]=1$ one has $\cent{V\UE(\bG) } {uF}= \cent{V\spa{\gamma}}{uF}\spa{F_p}$ where $\cent{V\spa{\gamma}}{uF} =\cent{V{}}{uF}=V_d$ or $\cent{V\spa{\gamma}}{uF}=V_d\spa{n\gamma}$ for some $n\in V$. It now suffices to show that $n\neins\in {\ov V\!_d} ^t$. Observe first that $\ov V^t=V\spa{\neins}$ by \Cref{ovV_V}(b) and therefore $V\spa{\gamma}$ acts on $\ov V^t$ by inner automorphisms. Then $n\gamma$ commutes with $uF$ and therefore with $(v_\tB)^t$ from the definition of $u$ in \Cref{notVdHd}. Then $n\neins$ commutes with $(v_\tB)^t$, therefore $n\neins\in \cent{\ov V^t}{(v_\tB)^t}={\ov V\!_d} ^t$. So the action of $n\gamma$ on $V_d$ is also induced by an element of ${\ov V\!_d} ^t$. This proves (c).

	For part (d), let $\pi_u\colon C:=\Cent_{V \UE(\bG)}(uF)\to \Cent_{V \UE(\bG)}(uF)/\spa{uF} $ be the quotient map. One has $\pi_u(C)_\la =\pi_u(C)_{\Lambda_0(\la)}$ by the equivariance of $\Lambda_0$ we have just proved. So to get our claim it suffices to show that $\Lambda_0(\la)$ extends to its stabilizer in $\pi_u(C)$. Now it extends first to $\pi_u(V_d\spa{F_p})_{\Lambda_0(\la)}=\pi_u((V_d)_{\Lambda_0(\la)}\spa{F_p})$ since $F_p$ acts trivially on $V_d$. But now $\pi_u((V_d)_{\Lambda_0(\la)}\spa{F_p})$ has index $\leq 2$ in $\pi_u(C)_{\Lambda_0(\la)}$ since $V\UE(\bG)\unrhd V\spa{F_p}$ has index 2 and therefore the intersections with $\bG^{uF}$ have index $\leq 2$ and the same holds in turn for the stabilizers of $\Lambda_0(\la)$ there. Then \Cref{ExtCrit}(a) gives our claim.

 We now turn to (e). Set $T:=\bT^{uF}$, $N:=\bN^{uF}$. Point (b) implies $T V_d=N$ while point (c) gives us the extension map $\Lambda_0$ \wrt $H_d\unlhd V_d$. According to \Cref{Ext_f}, we then get an extension map $\Lambda $ with respect to $T \unlhd N$, where for all $\la\in \Irr(T)$ and since $\restr\la|{H_d}\in\Irr(H_d)$, 
	\[ \restr \Lambda(\la)|{(V_d)_\la}= \Lambda_0(\restr\la|{H_d}) .\] 
	%Note that $N_\la \leq(V_d)_{\restr\la|{H_d}}T$ and therefore the above and $\restr \Lambda(\la)|{T}=\la$ define uniquely the linear character $\Lambda(\la)$. 
	Since $\La_0$ is $\Cent_{V \UE(\bG)}(uF)$-equivariant, $\Lambda$ is therefore also $\Cent_{V \UE(\bG)}(uF) $-equivariant. 
	
	So to get our claim about equivariance it is enough to show that $ \cent{\bN\UE(\bG)}{uF}$ acts on $T$ by elements of 
$\Cent_{V \UE(\bG)}(uF) $.	A sufficient condition is that $\bT^{uF}\cent{V\UE(\bG)}{uF}=\cent{\bN\UE(\bG)}{uF},$ in $\bG\rtimes \UE(G)$.
We clearly have $\bT^{uF}\leq \bT^{uF}\cent{V\UE(\bG)}{uF}\leq \cent{\bN\UE(\bG)}{uF}$ and using Lang's theorem the quotient mod $\bT^{uF}$ satisfies 
\[ \cent{V\UE(\bG)}{uF}/H_d\leq \cent{W\UE(\bG)}{uF}.\] 
We have to show that this is an equality. Since $\UE(\bG)=\spa{\gamma ,F_p}$ with $F_p$ acting trivially on both $V$ and $W$, it suffices to show $\cent{V\spa{\gamma}}{uF}/H_d = \cent{W\spa{\gamma}}{uF}$. As seen before, the term on the right is $\cent{\ov W{}}{\ov\rho(uF)}=\cent{\ov W{}}{\ov\rho(v_\tB)}$. On the other hand $\cent{V\spa{\gamma}}{uF}=\cent{V\spa{\gamma}}{v_\tB^t}\ni n\gamma$ implies $n\neins\in \cent{\ov V^t}{v_\tB^t}$ in the notation of the proof of (c) above. This gives our claim by equation (1).
 \end{proof}

\subsection{The doubly regular case. Relative Weyl groups }
\label{ssec4C}
Our next step in the proof of \Cref{thm:dreg} through the criterion in \Cref{dreg_crit} leads us to study certain subgroups of the so-called relative Weyl groups $W( \la):=\norm{\bG^{uF}}{\bT, \la}$ for $ \la\in\Irr( \bT^{uF})$ and $u$ as in \Cref{notVdHd}. In the following, we ensure assumption (iv) of \Cref{dreg_crit}. 

As in \cite[Sect. 6.1--2]{CS18B}, we determine $W( \la)$ using computations in the dual group, see also the proof of \cite[Thm 3.17]{MS16}. The character $ \la$ corresponds to some semisimple element $s\in \bH^F$ which centralizes a Sylow $d$-torus. As before, we assume that $d\geq 3$ and that $d$ is doubly regular in the sense of \Cref{def_dreg}. 

Let us recall from \ref{3:3'} the notation $\ov W=\Sym_{\pm l}\unrhd W={\Sym^\tD_{\pm l}}$ and for $I\subseteq \ul$, the (parabolic) subgroup $\Sym^\tB_I$ of $\ov W=\Sym_{\pm l}$ fixing every element of $\ul \setminus I $ and stabilizing $I$.  We also recall the notation $\Sym^\tB_{\mathbb M}=\prod_{I\in\mathbb M}\Sym^\tB_I\cong \prod_{I\in\mathbb M}\Sym_I$ for any partition $M=\sqcup_{I\in\mathbb M}I$ of a subset $M\subseteq \ul$. Recall ${\Sym^\tD_{\pm I}}$ the group of elements of $W$ that fix any element of $\underline{l}\setminus I$.

The following statement on permutation groups is key for ensuring assumption \ref{dreg_crit}(iv).
\begin{prop}\label{thm6_6}
	Let $l\geq 1$, $\III{J'},\III{J''}\subseteq \ul$ disjoint (possibly empty) subsets, $\bII$ a partition of $\ul\setminus(J'\cup J'')$, 
	\[\III{P}:=\Sym^\tD_{\pm J'}\times \Sym ^\tD_{\pm J''}\times\Sym^\tB_{\bII}\leq \ov W,\]
	$\III{K}:=\NNN_{\ov W}(P)$ and $w\in K$ a regular element of $\ov W$ seen as reflection group in $\GL_{{\mathbb R}}({\mathbb R}\ov\Phi)$. Then maximal extendibility holds \wrt $\Cent_P(w) \unlhd \Cent_K (w)$. 
\end{prop}
Note that in general $P$ is not a parabolic subgroup of $\ov W$.
\begin{proof}
		Let $d$ be the order of $w$ and set
		\[ P_0:=\Cent_P(w) \und K_0:=\Cent_K(w).\] 
Let us first determine the group $K$. Let $\tau\in \Sym^\tB_{J\cup J'}$ be such that \begin{itemize}
			\item[-] $\tau =1$ if $|J'|\neq |J''|$, or 
			\item[-] $\tau(J')=J''$, $\tau^2=1$ if $|J'|= |J''|$
		\end{itemize} 
		Let $S\leq \Sym^\tB_l$ be the subgroup of permutations $\pi$ with
		\begin{itemize}
			\item[-]$\pi (\bII)=\bII$,
			\item[-] $\pi (i)< \pi(i')$ for every $I\in \bII$ and $i,i'\in I$ with $i<i'$, and 
			\item[-] $\pi(i)=i$ for every $i \in J'\cup J''$.
		\end{itemize}
	
	Any element of $\ov W$ normalizing $P$ must have an image in the symmetric group on $\ul$ stabilizing $J'\cup J''$ and its complement, so it's easy to see that the determination of $K$ boils down to the cases where
	 $J'\cup J''=\ul$, or $J'= J''=\emptyset$. We get in the general case 
		\[K=\NNN_{\ov W}(P)= \left((\Sympm{J'}\times \Sympm{J''}) \rtimes \spann<\tau> \right ) \times \left(\prod_{I\in \bII} \spann<\Sym^\tB_{I},-\id_{I}>\right)\rtimes S\ \leq \ \Sym_{\pm (J'\cup J'')}\times \Sympm{\cup_{I\in\bII}I}.\] 
		
		It is clear that the question splits along the orbits of $K$ on $\ul$, the assumption of regularity remaining thanks to \Cref{lemreg4:8}(c). We now assume that $K$ is transitive on $\ul$ and treat separately the following cases \begin{itemize}
			\item[1a.] $J'=\ul$,
			\item[1b.] $J'\cup J''=\ul$ and $|J'|=|J''|$,
			\item[2.\ ] $J'= J''=\emptyset$ and all elements of $\bII$ have same cardinality. 
			 \end{itemize}

	(1)	 Assume now
		$J'\cup J''=\ul$, so that $$P=\Sym^\tD_{\pm J'}\times \Sym^\tD_{\pm J''}\ \unlhd \ K=(\Sympm{J'}\times \Sympm{J''}) \rtimes \spann<\tau> \ \ni \ w.$$
		
		Assume first that $ w( J'\cup -J')= J'\cup -J'$. Then we can write $w$ as $w'w''$ with $w'\in \Sym_{\pm J'}$ and $w''\in \Sym_{\pm J''}$. This leads to 
		$P_0=P_0'\times P_0'',$ 
		where $P_0':=\Cent_{\Sym^\tD_{\pm J'}}(w')$ and $P_0'':=\Cent_{\Sym^\tD_{\pm J''}}(w'')$. Set $\ov P_0':=\Cent_{\Sym_{\pm J'}}(w')$ and $\ov P_0'':=\Cent_{\Sym_{\pm J''}}(w'')$. 
		
In case (1a) above, i.e. $J'=\ul$, the sought \maex holds by \Cref{ExtCrit}(a) since $\ov P_0'/ P_0'$ is cyclic of order $\leq 2$. 

Assume now (1b), i.e. $|J'|= |J''|=l/2$. We first keep the assumption $ w( J'\cup -J')= J'\cup -J'$. According to \Cref{lemreg4:8}, $w'$ and $w''$ are regular elements of the same order of $\Sym_{\pm J'}$ and $\Sym_{\pm J''}$, respectively. As all regular elements of $\ov W$ of the same order are $W$-conjugate by \Cref{lemreg4:8}(a), there exists some involution $\tau_0\in \Sym_{\pm l}$ with $(w')^{\tau_0}=w''$. Then $K_0=(\ov P_0'\times \ov P_0'') \rtimes \spann<\tau_0>$ and \Cref{ExtCrit}(c) applies. 
		%%%%%%%%%%%%%%%%%%%%%%%%%%%%%%
		
    Assume now $ w(\pm J')=\pm J''.$ Note that $w^2$ is regular as well according to \Cref{lemreg4:8}(b). By the structure of regular elements recalled in \Cref{lemreg4:8}(b), if $P_0':=\Cent_{\Sym^\tD_{\pm J'}}(w^2)$, then $P_0'$ and $P_0$ are isomorphic via $x\mapsto x \cdot x^w$. 
		
		Set $\ov P_0':=\Cent_{\Sym_{\pm J'}}(w^2)$ and $\Delta \ov P_0':=\{ pp^w\mid p \in \ov P_0'\}$. We observe $K_0=\Cent_K(w)=\spann<\Delta \ov P_0',w>$ with $w\in \Z(K_0)$. We have maximal extendibility \wrt $P_0'\unlhd \ov P'_0$ by \Cref{ExtCrit}(a), and therefore also \wrt $P_0\unlhd \Delta \ov P_0'$ by the above description. But now, since $w\in \Z(K_0)$, any character of an intermediate group $A$ is stable under $w$ and extends to $A\spa{w}$ by \Cref{ExtCrit}(a) again. This shows \maex \wrt $P_0\unlhd K_0= \Delta \ov P_0'\spa{w}$. 
		%%%%%%%%%%%%%%%%%%%%%%%%%%
		
	(2)	 We now consider the case where $J'= J''=\emptyset$ and $K$ permutes transitively the elements of $\bII$.
	
		We have seen that $$P=\Sym_{\bII}\ \unlhd \ K= \left(\prod_{I\in \bII} \spann<\Sym^\tB_{I},-\id_{I}>\right)\rtimes S\ \ni \ w,$$ where $S\leq \Sym_{\pm \ul}$ is the subgroup preserving $\bII$ and the ordering on each subset $I\in \bII$. Note that $w$ acts on $\bII$ since $w$ stabilizes $P$. Recall that $d$ is the order of $w$. 
		
		Assume now that all $w$-orbits in $\bII$ have the same length $b$ (a divisor of $d$). Set $w':=w^{b} \in \Sym_{\pm \bII}$. Note that $w'$ is again a regular element according to \Cref{lemreg4:8}(b). For each $I\in \bII$ let $w'_I\in \Sym_{\pm I}$ be the permutation induced by $w'$. Note that $w'_I$ is a regular element of $\Sym_{\pm I}$ according to \Cref{lemreg4:8}(c), where we consider $\Sym_{\pm I}$ as a reflection group of type $\tB_{|I|}$. 
		
	The element $w^b$ can be written as the product of elements $w'_I\in \Sym_{\pm I}$ ($I\in\mathbb I$). Every $w'_I$ is a regular element of $\Sym_{\pm I}$ of order $\frac d b$. 
		%This implies that $\Cent_{\Sym_{\pm l}}(w')$ is isomorphic to the wreath product of $\Cent_{\Sym_{\pm I_0}}(w^b)\wr \Sym_{|\bII|}$ for some $I_0\in \mathbb I$. 
		We write $\cO$ for the set of $w$-orbits in $\bII$ and for every $O\in \cO$ let us fix some $I_O\in O$ . We define $\Delta_O: \Sym_{\pm I_O}\lra \Sym_{\pm \bigcup_{I\in O} I} $ by $x\longmapsto x \cdot x^w \cdot x^{w^2} \cdots x^{w^{b-1}}$.
		
		For $O\in \cO$ set $P_{I_O}:=\Cent_{\Sym_{I_O}}(w'_{I_O})$ and $P_O:=\Delta_O(P_{I_O})$. The group $P_0$ is the direct product of the groups $P_O$ ($O\in \cO$).
		
		Recall $K=\NNN_{\ov W}(P)=\left( \prod_{I\in \bII} \spa{\Sym_I, -\id_I} \right)\rtimes S$. Set $K_I:=\Cent_{\spa{\Sym_{I}, -\id_I}}(w'_I)$ and $K_O:=\spa{ \Delta_O(K_{I_O}), w_O}\leq K_0$ where $w_O$ is the projection of $w$ on $\Sym_{\pm \bigcup_{I\in O} I}$.

%Recall that we call an element of a given reflection group regular if it is $\zeta$-regular for a root of unity $\zeta$. The $\zeta$-regular elements form a conjugacy class in that reflection group. In the reflection group $\Sym_{\pm I}$ the regular elements of a given order form a conjugacy class, see the description of those elements for example in \cite[Sect. 5.3]{Sp74}.  
Since all regular elements of order $\frac db$ are conjugate in $\ov W$ thanks to \Cref{lemreg4:8}(c), we find a subgroup $S'\leq K$ such that 
	$$K_0= \spa{K_O\mid O\in \cO}\rtimes S',$$ 
where $S'\cong \Sym_{|\cO|}$.
		%Should be more elaborate here
		Note that $ w_O \in \Z(K_O)$ by the definition of $K_0$. 
		
		We can then conclude, as at the end of (1) above, that \maex holds \wrt $P_O\unlhd K_O$ since $\Delta_O(K_{I_O})/P_O \cong K_{I_O}/P_{I_O}$ is cyclic and $w_O$ is central in $K_O$. Then we get \maex \wrt $P_0\unlhd K_0$ by \Cref{ExtCrit}(c).
		
Assume not all $w$-orbits on $\bII$ have same cardinality. For $b\geq 1$ dividing $d$, let $\bII_{b}$ be the set of $I\in \bII$ such that $b$ is the smallest divisor of $d$ with $w^b(\pm I)= \pm I$.
	Then $K_0$ is the direct product of groups $K_{\bII_{b}}:=K_0\cap \Sym_{\pm \bII_{b}}$ ($b\geq 1$) while $P_0$ is the direct product of its intersections with those factors. We are then reduced to the case of a single $b$ treated above.
	\end{proof}

%\newpage
We are now back with $(\bG, F)$, $\bT$ and a doubly regular $d\geq 3$ as in \Cref{ssec4B}. Recall also the regular embedding $\bG\leq \wbG=\Z(\wbG)\bG$ with $F$ extended to $\wbG$ and $\wt\bT=\Z(\wbG)\bT$. The element $u\in \norm{\bG}{\bT}$ from \Cref{notVdHd} is chosen such that $\bT$ is the centralizer of a Sylow $d$-torus of $(\bG ,uF)$ by \Cref{def_varphi0}. 

We show below that in this case of a doubly regular $d$ the above proposition about centralizers in Weyl groups essentially implies the point (iv) in \Cref{dreg_crit}. The proof goes through a recasting of the question inside the dual group following ideas from \cite[\S 3.D]{MS16}, \cite[\S 6]{CS18B}.
\begin{cor} \label{thm41} 
	Let $\wt \la\in \Irr(\wt\bT^{uF})$. Set $\bW:=\NNN_{ \bG{}}(\bT)/\bT \unlhd \ov\bW:=\NNN_{ \bG\spa{\gamma}}(\bT)/\bT $ and $\la:=\restr\wt\la|{\bT^{uF}}$. 
	%Let $W(\wt\la):=\bW{}^{uF}_{\wt\la}=\NNN_{\wt\bG^{uF}}(\wt\bT,\wt\la)/{\wt\bT}^{uF}$.
	 Then \maex holds for $\bW{}^{uF}_{\wt\la}\unlhd {\ov \bW^{uF}_\la}$. 
\end{cor} 
\begin{proof} As explained in the proof of \cite[Prop. 6.2]{CS18B}, the duality between $\wbG\geq\bG$ and $\wbH\twoheadrightarrow\bG^*=\bH$ with Frobenius endomorphism $uF$ allows us to associate $(\bT,\la)$ to $ (\bT^*,s)$ and $(\wt\bT,\wh\la)$ to $ (\wt\bT^*,\wh s)$ with $\wt s\in (\wt\bT^*)^{uF}$ and $\wt s\mapsto s$. The stabilizers $\bW^{uF}_{\wt\la}\unlhd\bW^{uF}_{\la}$ become $\Cent_\bW(\wt s)^{uF}\unlhd \Cent_\bW( s)^{uF}$, see also \cite[Cor. 3.3]{CS13}. Note that everything is now written inside $\bH$ since $\Cent_{\wbH}(\wt s)/\Z(\wH)=\Cent_\bH^\circ(s)$. We denote by $\Cent_\bW^\circ( s):=\NNN_{\Cent_\bH^\circ(s)}(\bT^*)/\bT^*$ the Weyl group of the latter. 
The automorphism $\gamma$ acts on $\bH$ and $\bG$ and we can see also $\ov\bW$ as $\NNN_{\bH\rtimes \spa{\gamma}}(\ov\bT^*)/\ov\bT^*$ such that $\ov \bW_\la^{uF}$ corresponds via duality to $\Cent_{\ov\bW}( s)^{uF}$, hence we study $\bW_\la^{uF}\lhd \ov \bW_{\la}^{uF}$ via 
$$\Cent_\bW^\circ( s)^{uF}\unlhd \Cent_{\ov\bW}( s)^{uF}.$$
Note $\bW\spa{\gamma}=\ov\bW$ is a Weyl group of type $\tB_l$. In the latter this rewrites as $\Cent_{P}(u\varphi_0)\unlhd\Cent_K(u\varphi_0)$ where $P=\Cent_\bW^\circ(s)$, $ K=\Cent_{\ov\bW}(s)$ and $P\unlhd K$. By \Cref{cor3_8}(a), $P$ has the structure studied in \Cref{thm6_6} above, as $d$ was doubly regular for $(\bG,F)$ and hence $u\varphi_0$ is a regular element of $\ov\bW$. So \Cref{thm6_6} indeed gives our claim. 
\end{proof}

	%%%%%%%%%%%%%%%%%%%%%%%%%%%
	 
	\subsection{Proof of \Cref{thm:dreg}}\label{ssec_4D}
	In order to complete the proof of \Cref{thm:dreg} we now essentially have to check condition (v) of \Cref{dreg_crit} and that groups of type $\tD_4$ satisfy the condition (\textbf{iMK}) for all primes.
	
	For the first point we keep the notation of \Cref{notVdHd}.
	%%%%%%%%%%%%%%%%%%%
	\begin{lem}\label{NtildeN}
	 Maximal extendibility holds \wrt $\NNN_{\bG}( \bT)^{uF}\unlhd \NNN_{\wt\bG}(\wt\bT)^{uF}$.\end{lem}
	\begin{proof} Let us abbreviate $T=\bT^{uF}$, $N:=\NNN_{\bG}( \bT)^{uF}$, $\wt T=\wt\bT^{uF}$, $\wt N:= \NNN_{\wt\bG}(\wt\bT)^{uF}=N \wt T$. Note that since $\wt N/N$ is abelian, maximal extendibility can be proven by ensuring that the restriction of any irreducible character of $\wt N$ to $N$ is multiplicity-free (see for example \cite[1.A]{S21D2}).
		
		Let us first recall the parameterization of $\Irr (\wt N)$. We know from \Cref{ExtMapHdVd}(e) that there is an extension map $\Lambda$ \wrt $T\unlhd N$. By \Cref{Ext_f}, the map $\Lambda$ allows us to construct an extension map $\wt\Lambda$ \wrt $\wt T\unlhd\wt N$ with the property that $\restr\wt\Lambda (\wt\la)|{N_{\wt\la}} = \restr \Lambda (\la)|{N_{\wt\la}}$ for $\wt\la\in\Irr(\wt T)$, $\la:=\restr\wt\la|{T}$ and using $N_{\wt\la}\leq N_{\la}$, see \Cref{thm41}. 
By Clifford theory (\ref{Cliff}), every $\psi\in \Irr(\wt N)$ then writes as 
		\[\psi=(\wt\Lambda (\wt\la)\eta)^{\wt N}\] 
		for some $\wt\la\in\Irr(\wt T)$ and $\eta\in\Irr(\wt N_{\wt\la}/\wt T)$. Since $\wt N=\wt T N$, we observe $\wt N=N \, \wt N_{\wt \la}$ and hence the Mackey formula shows 
		\[ \restr\psi|{N}=(\restr(\wt \Lambda (\la)\eta)| {N_{\wt\la}} )^{ N}=(\Lambda ( \la)\eta^{W(\la)} )^{ N},
		\]
		where we set $W( \la):={N_{ \la}}/T$, and see $\eta$ as a character of $W(\wt\la):={N_{\wt\la}}/T$, ${\wt N_{\wt\la}}/\wt T$ and $\wt N_{\wt \la}$, respectively.
		(This also uses the above definition of $\wt\Lambda$ from $\Lambda$.)
		
		By \Cref{thm41}, maximal extendibility holds \wrt $W(\wt\la)\unlhd W(\la)$ since the group $\ov\bW_\la^{uF}$ considered there is an overgroup for $ W(\la)$. But $W(\la)/W(\wt \la)=N_\la/N_{\wt \la}$ is abelian (see the proof of \Cref{dreg_crit}), so \maex implies that $\eta^{W(\la)}$ is multiplicity-free. But then by Clifford correspondence, the character $\left (\Lambda ( \la)(\eta^{W(\la)\\})\right )^{ N}$ is also multiplicity-free, whence our claim.
	\end{proof}

We can now complete the proof of \Cref{thm:dreg}.
\begin{proof}[Proof of \Cref{thm:dreg}] 
		Set $\bG=\tDlsc(\FF)$ ($l\geq 4$), $F:\bG\lra \bG$, $d\geq 3$ as in the assumptions. Notice first that (a) implies (b) thanks to \Cref{iMKbyAdBd} since $\bG^F/\Z(\bG)^F$ is always a simple group and \cite{ManonLie} ensures that we can take $\GF$ to play the rôle of its universal covering group. So we just verify (a) and (c).
		
		We first check (a) by establishing
	 \Ad\ and \Bd\ through \Cref{dreg_crit} whose assumptions we now review. We let $V\leq \NNN_{ \bG}(\bT)$ and $u\in V$ as in \ref{notVdHd}, so that $\rho(u)\varphi_0$ is a $\zeta_d$-regular element of $W\varphi_0$ thanks to \Cref{ExtMapHdVd}(a). Then assumptions \ref{dreg_crit}(i) and (ii) are ensured by \Cref{ExtMapHdVd}(a) and (b). Assumption \ref{dreg_crit}(v) holds according to \Cref{NtildeN} since $\Cent_\bG(\bS)=\bT$ with $\bS$ the Sylow $d$-subtorus of $(\bT,uF)$ imply $\NNN_\bG(\bS)^{uF}=\NNN_\bG(\bT)^{uF}$ and $\NNN_\wbG(\bS)^{uF}=\NNN_\wbG(\wt\bT)^{uF}$.

	 Assume now $(l,\eps)\neq (4,1)$, so that $E(\bG^F)=\UE(\bG^F)$. Then assumption \ref{dreg_crit}(iii) is ensured by \Cref{ExtMapHdVd}(c) and (d). Assumption \ref{dreg_crit}(iv) amounts to a strengthening of the above \Cref{thm41} where $\ov\bW=\NNN_{\bG\spa{\gamma}}(\bT)/\bT$ is replaced by the overgroup $\NNN_{\bG\UE(\bG)}(\bT)/\bT=\ov\bW\times \spa{F_p}$. The extendibility property is preserved since $F_p$ is central.

We now concentrate on the case of $\GF=\tD_{4,\sico}(q)$ to finish checking (a) and (c). By \Cref{CS1319}, we can content ourselves with taking a prime $\ell\nmid 6q$ for (c). For (a), thanks to \Cref{iMKbyAdBd}, we have to check \Ad \ and \Bd \ for $d=4 $, the only integer doubly regular for $\GF$ and $\geq 3$.
	The polynomial order of $(\bG ,F)$ is 
	\[P_{(\bG ,F)}(X)= X^{12} \mathbf{\Phi}_1^4 \mathbf{\Phi}_2^4 \mathbf{\Phi}_3 \mathbf{\Phi}_4^2 \mathbf{\Phi}_6 ,\] where $\II Phii@{\mathbf{\Phi}_i}$ is the $i$-th cyclotomic polynomial,
	see \cite[p. 75]{Ca85}. For primes $\ell\nmid 2q$ such that $d_\ell(q)=d\in\{3,6\}$ we observe that $\ell>3$, $\mathbf{\Phi}_d$ occurs with exponent 1 and $\mathbf{\Phi}_{d\ell^a}$ is not present for $a>1$. Then $|\GF|_\ell =\mathbf{\Phi}_d(q)_\ell$, so a Sylow $\ell$-subgroup of $\GF$ is a subgroup of a Sylow $d$-torus of polynomial order $\mathbf{\Phi}_d$, hence cyclic (see for instance proof of \cite[Prop. 25.7]{MT}). For $\ell$-blocks of quasisimple groups with cyclic defect groups a so-called \textit{inductive Alperin--McKay condition} holds according to \cite[Thm~1.1]{KoSp}. This provides us with a stronger version of the required (\textbf{iMK}), see also 2nd paragraph of \cite[Sect. 6.C]{CS18B}.

It remains to check \textbf{A}(4) and \textbf{B}(4). We resume reviewing the assumptions of \Cref{dreg_crit}, remembering that only \ref{dreg_crit}(iii) and \ref{dreg_crit}(iv) were left incomplete. Denote by $\II{gamma3}@{\gamma_3}\in E(\bG)$ a graph automorphism of order 3.
Let $u\in V$, $V_4=\Cent_V(u)\geq H_4$ from \ref{notVdHd}, and $\wh V_4:= \Cent_{V\rtimes {\spa{\gamma, \gamma_3,F_p}}}(u)/\spa{uF_q}\geq \widecheck{V}_4:= \Cent_{V\rtimes {\spa{\gamma,F_p}}}(u)/\spa{uF_q}\geq V_4$ in $\bG^F$. Note that the centralizer of $\ov\rho(u)$ in $\ov W$ is a 2-group (use \Cref{lemreg4:8}(b) to determine $\rho(u)$) while $H_4$ is also a 2-group, hence $V_4$ is also one.
		
What has been checked of assumption \ref{dreg_crit}(iii) by applying \Cref{ExtMapHdVd}(c-d) ensures that the inclusion $H_4\unlhd \widecheck{V}_4$ satisfies maximal extendibility. Now if $\la\in\Irr(H_4)$, then its extension $\widecheck{\la}$ to $(\widecheck{V}_4)_\la$ can be chosen to have the image of $\spa{F_p}$ (central) in its kernel, see the proof of \Cref{ExtMapHdVd}(d). The quotient of $\widecheck{V}_4$ by the image of $\spa{F_p}$ is a 2-group, while $\wh{V}_4/\widecheck{V}_4$ has order 1 or 3, so \cite[6.28]{Isa} implies that $\widecheck{\la}$ can be chosen to be $(\wh{V}_4 )_{\la}$-invariant. It then further extends to $(\wh{V}_4 )_{\la}$ by \Cref{ExtCrit}(a). So we get \maex for $H_4\unlhd \wh{V}_4 {}$ hence our claim since the associated extension map can always be chosen to be $\wh{V}_4$-equivariant as recalled in \Cref{defMaxExt}.

For assumption \ref{dreg_crit}(iv), by the argument used before when $E(\GF)=\UE(\GF)$ it suffices to check maximal extendibility for $\bW^{uF}_{\wt\la}\unlhd \wh\bW^{uF}_\la$ where $\wh\bW:=\NNN_{ \bG\spa{\gamma, \gamma_3}} (\bT)/\bT\unrhd \ov\bW=\NNN_{ \bG\spa{\gamma}} (\bT)/\bT $. We have \maex for $\bW^{uF}_{\wt\la}\unlhd \ov\bW^{uF}_\la$ thanks to \Cref{thm41}. But since $\ov\bW^{uF}$ is a $2$-group, we get that $\ov\bW^{uF}_\la/\bW^{uF}_{\wt\la}$ is a Sylow 2-subgroup of $\wh\bW^{uF}_\la/\bW^{uF}_{\wt\la}$ while every Sylow 3-subgroup of $\wh\bW^{uF}_\la/\bW^{uF}_{\wt\la}$ is cyclic of order 1 or 3. We then get the sought maximal extendibility for $\bW^{uF}_{\wt\la}\unlhd \wh\bW^{uF}_\la$ by applying \cite[Cor. 11.31]{Isa} and \Cref{ExtCrit}(a). 
		\end{proof}

	%%%%%%%%%%%%%%%%%%%%%%%%%%%%%%%%%%%%%%%%%%%%%%%%
\subsection{Extending Malle's bijection}\label{ssec_4E}
%%%%%%%%%%%%%%%%%%%%%%%%%%%%%%%%%%%%%%%%%%%%%%%%
	For later applications, we construct a character correspondence extending the one given by (\textbf{iMK}) only assuming Conditions \Ad{} and \Bd{}. 

Proving (\textbf{iMK}) for $\GF$ and a prime $\ell$ not dividing $2 q$ with the choice $N=\NNN_{ \bG}(\bS)^F$ for $\bS$ a Sylow $d_\ell(q)$-torus of $(\bG,F)$ gives us a $\Gamma:=\Aut(\GF)_\bS$-equivariant bijection 
\[ \Omega_{\GF,\ell}:\Irrl(\GF)\lra \Irrl(\NNN_\bG(\bS)^F),\] 
which satisfies for every $\chi\in\Irrl(\GF)$ the relation
\[ (\GF\rtimes \Gamma_\chi, \GF,\chi)\geq_c (N\rtimes \Gamma_{\Omega_{\GF,\ell}(\chi)}, N,\Omega_{\GF,\ell}(\chi))\]
or any variant obtained by applying the Butterfly \Cref{Butterfly}.

In the setting of \Cref{thm:dreg}(b) we get the following.
\begin{cor} \label{cor3_5}
		Assume that $\ell$ is a prime with $\ell\nmid 2 q$ such that $d:=d_\ell(q)$ is doubly regular for $(\bG,F)$ and let $\bS$ be a Sylow $d$-torus of $(\bG,F)$. Set $T:=\Cent_{ \bG}(\bS)^F\unlhd N:=\NNN_{ \bG}(\bS)^F$.
		\begin{thmlist}
			\item There exists some $(\GF E(\GF))_\bS$-equivariant extension map $\Lambda$ \wrt $T\unlhd N$. Furthermore, for every $\la\in\Irr(T)$, the character $\Lambda(\la)$ extends to $(\GF E(\GF))_{\bS,\la}$. 
			\item For $\la\in \Irr(T)$ with $\ell\nmid |N/N_\la|$, the character $\chi:=\Omega_{\GF,\ell}^{-1}(\La(\la)^{N})$ belongs to $\ov\TT$, i.e. $\starStab {\wt\bG^F} | E(\GF)|\chi$ and $\chi$ extends to $\GF E(\GF)_\chi$. 
		\end{thmlist}
	\end{cor}
	\begin{proof}
		Part (a) follows from \Cref{ExtMapHdVd}(e) in the cases where $E(\GF)=\UE(\GF)$. When $\GF =\tD_{4,\sico}(q)$, then $d=4$ and we have seen in the proof of \Cref{thm:dreg}(a) above that \maex holds for $H_4\unlhd \wh{V}_4$. As explained in the proof of \Cref{dreg_crit}, this implies our claim by %$(\GF E(\GF))_{\bS}=T\wh{V}_4$ and 
		\Cref{ExtCrit}(e).
		For (b), set $ \la \in \Irr(T )$, $\chi':= \Lambda(\la)^{N }$. We also abbreviate $G=\GF\unlhd \wG=\wbG^F$, $\wt N :=\NNN_{\wt G }(\bS)$ and $\wh N :=\NNN_{G E(G)}(\bS)$. In the notation of the proof of \Cref{dreg_crit}, we have $\chi '=\Pi( \xi, \eta)$ for $( \xi, \eta)=(\la ,1)$. The fact that $\eta =1$ implies that this proof can be followed with both $\eta_0$, $\wt\eta_0$, and $\eta'$ being trivial and therefore the choice $\chi_0=\Pi(\la, 1)$. This gives 
		\[\starStab {\wt N }{} | {\wh N } {}| {\chi'}\] with $\chi'$ extending to $ \wh N_{\chi '}$. 
		
		As $\Omega_{G,\ell}$ is $\wt N \wh N$-equivariant, $\chi:=\Omega_{G,\ell}^{-1}(\chi')$ satisfies 
		\begin{align*}
			(\wt G E(G))_\chi &= G (\wt N \wh N)_{\chi}= 
			G (\wt N \wh N)_{\chi'} \\
			&=G (\wt N_{\chi'} \wh N_{\chi'})= 
			G (\wt N_{\chi} \wh N_{\chi})=
			(G\wt N_{\chi})\, ( G \wh N_{\chi})=
			\wt G_{\chi} E(G)_{\chi}.
		\end{align*} 
		As recalled above the bijection $\Omega_{G,\ell}$ satisfies some $\geq_c$-relation that we can take to be
		\[ ((\wt G E(G))_\chi, G , \chi)\geq_c ((\wt N\wh N)_{\chi'}, N , \chi'). \]
		According to \Cref{rem_chartrip}(c) this implies 
		\[ ( G E(G)_\chi, G , \chi)\geq_c (\wh N_{\chi'}, N , \chi'). \]
		Now, \Cref{rem_chartrip}(b) and the fact that $\chi'$ extends to $\wh N_{\chi'}$ allow us to see that $\chi$ extends to $G E(G)_\chi$.
	\end{proof}
	%%%%%%%%%%%%%%%%%%%%%%%%%%%%%%%%%%%%%%%%%%%%%%%%%%%%%%%%%%%%%%%%%%%%%
	The groups that appear above depend on the integer $d$ but not on the prime $\ell$ leading to $d$. The condition (\textbf{iMK}) for two primes $\ell$ and $r$ with the same value $d$ give two bijections, all mapping to some characters of $N=\NNN_\GF(\bS)$. Going back to the construction of $\Omega_{\GF,\ell}$ through \cite{MaH0} and \cite[Sect. 6]{CS17A}, we associate to a fixed $d$ a character set $\calG_d\subseteq \Irr(\GF)$ and a bijection 
	\[ \Omega ': \calG_d\lra \Irr(N),\]
	such that 
	\[ (G\rtimes \Gamma_\chi, G,\chi)\geq_c (N\rtimes \Gamma_{\Omega '(\chi)}, N,\Omega'(\chi)) \text{ for every } \chi\in\calG_d \]
	and for every prime $\ell$ with $d_\ell(q)=d$, the map $\Omega'$ restricts to $\Omega_{\GF,\ell}$ on $\Irrl(N)$.

 For the construction of this map we assume $(\bG ,F)$ to be as in \Cref{ssec2C}, i.e., $\bG$ might be of type different from $\tD$. 
 \begin{defi}\label{calGd} Let $(\bG ,F)$ be as in \Cref{ssec2C} with a regular embedding $\bG\leq \wbG$. Assume we have dual groups $\wbG^*\twoheadrightarrow \bG^*$ with Frobenius endomorphisms denoted by the same letter $F$. Let $d\geq 1$ and let $\bS^*$ be a Sylow $d$-torus of $(\wbG^*,F)$.
		
		For any $\wt s\in \wbG^*_{\text{ss}}{}^F$, let $\cE(\wGF,[\wt s])$ be the associated rational series of characters of $\wbG^F$. Recall $ \UCh(\Cent_{\wbG^*}(\wt s)^F)\subseteq \Irr (\Cent_{\wbG^*}(\wt s)^F)$ the set of unipotent characters and the Jordan decomposition map \begin{align*}
			\UCh(\Cent_{\wbG^*}(\wt s)^F)&\xrightarrow{\ \ \sim\ \ } \cE(\wGF,[\wt s]),\\
			\la&\mapsto \chi_{\wt s,\la}^\wbG.
		\end{align*}

For $\bK^*$ an $F$-stable connected reductive subgroup of $\wbG^*$ containing $\bS^*$, let $\UCh_d(\bK^*{}^F)$ be the set of irreducible components of Lusztig's generalized characters $\R^{\bK^*{}}_{\Cent_{\bK^*}(\bS^*)}(\la)$ for $\la\in \UCh({\Cent_{\bK^*}(\bS^*)^F})$.

Set $$\II{Gdtilde}@{\protect{\wt{\calG}}_d} =\bigcup_{\wt s\in \Cent_{ \wbG^*}(\bS^*)^F_{\text{ss}}} \{ \chi_{\wt s,\la}^\wbG\mid\la\in \UCh_d(\Cent_{\wbG^*}(\wt s)^F)\}$$
and let 
 $$\II {Gd}@{ {\mathcal G}_d}:=\bigcup_{\wt\chi\in \wt{\calG}_d} \Irr(\restr\wt\chi|{\GF})\subseteq\Irr(\GF)$$ be the set of irreducible components of the restrictions $\restr\wt\chi|{\GF}$ for $\wt\chi\in\wt\calG_d$.
	\end{defi}

The following proposition is an adaptation of the construction recalled in \cite[Sect. 6]{CS17A}, assuming Conditions \Ad{} and \Bd{}. We give a proof in the case when $d$ is regular for $(\bG,F)$ which simplifies a bit the notation. We later apply it only for $d$ a doubly-regular number for $(\bG,F)$ and $\bG=\tDlsc(\FF)$, where the assumptions are ensured by \Cref{thm:dreg}(a).
	
	\begin{prop}\label{bij_Omega'}
		Let $(\bG,F) $ be as in \ref{ssec2C} and let $d\geq 1$. Let $\bS$ be a Sylow $d$-torus of $(\bG,F)$, $N:=\NNN_\GF(\bS)$, $\wh N:=\NNN_{\GF\rtimes E(\GF)}(\bS)$ and $\wN:=\NNN_\wGF(\bS)$. 
		\begin{thmlist}
			\item Assume Condition \Bd{}. Then there exists some $\Lin(\wGF/\GF)\rtimes \wh N$-equivariant bijective map 
			\[\wt \Omega':\wt \calG_d \lra \Irr(\wN),\] such that 
			\begin{enumerate}
				\item[(a.1)] $\Irr(\restr \chi|{\Z(\wGF)})=\Irr(\restr\wt \Omega'(\chi)|{\Z(\wGF)}) \forevery \chi\in\wt\calG_d$, and
				\item[(a.2)] $ \wt \Omega'(\Irr(\wGF\mid \Irrl(\GF)))=\Irr(\wN\mid \Irrl(N))$ for every prime $\ell$ with $d=d_\ell(q)$.
			\end{enumerate}
			\item Assume Conditions \Ad{} and \Bd{}. Set $\Gamma:=\Aut(\GF)_\bS$. Then there exists a $\Gamma$-equivariant bijective map 
			\[ \Omega':\calG_d\lra \Irr(N)\] 
			such that 
			\[ (G\rtimes \Gamma_\chi,G,\chi)\geq_c (N\rtimes \Gamma_{\Omega'(\chi)}, N,\Omega'(\chi)) \text{ for every $\chi\in \calG_d $} \]
			and $\Omega'(\Irrl(G))=\Irrl(N)$ for every odd prime $\ell$ with $d_\ell(q)=d$. 
		\end{thmlist}
	\end{prop} 
\begin{proof} The proof essentially follows \cite[Sect. 6]{CS17A}. 
We give a detailed proof for the case where $d$ is regular for $(\bG,F)$, allowing us to parameterise both characters via a set of pairs $\wt \cM$. This is the only case used in the paper, namely for $\bG$, $\wbG$, etc.  as defined in \Cref{sub2E} with $d$ doubly regular. %, ensuring both \Ad\ and \Bd\ thanks to \Cref{thm:dreg}(a)
A simplification comes from the fact that centralizers of Sylow $d$-tori are tori in both $\bG$, $\wbG$ and their dual groups. A more general case would require to deal with triples instead of pairs as parameters. 
		
Let $\wt \bT_d= \Cent_{\wbG}(\bS)$ and $\wbT_d^*=\Cent_{\wbG^*}(\bS^*)$ for $\bS^*$ a Sylow $d$-torus in $\wt\bG^*$. Both are $F$-stable tori with same type the regular element $\rho(u)\in W$ with regard to the maximally split tori in duality $\wbT$ and $\wbT^*$. So $\wbT_d$ and $\wbT^*_d$ can be taken as effecting the duality between $\wbG$ and $\wbG^*$. 

    Let $\Ispezial Mtil@\tilde {\mathcal M}@{\protect{\wt\cM}}$ be the set of pairs $(\wt s, \eta)$ where $\wt s\in \wbT_d^*{}^F$ and $\eta\in\Irr(\NNN_{\Cent_{\wbG^*}(\wt s)}(\wbT_d^*)^F/\wbT_d^*{}^F)$.
		
The so-called generalized $d$-Harish-Chandra theory, see \cite[Sect. 4.6]{GM}, implies that there is a bijection 
\begin{align*}
\UCh\colon	\Irr(\NNN_{\Cent_{\wbG^*}(\wt s)}(\wbT_d^*)^F/\wbT_d^*{}^F)&\to \UCh_d(\Cent_{\wbG^*}(\wt s)^F) \\
\eta&\mapsto \UCh(\eta)\end{align*} 
defined as $\eta\mapsto \pm I_{\wbT_d^*,1}^{\Cent_{\wbG^*}(\wt s)}(\eta ) $ in the notation of \cite[Thm~4.6.21]{GM}.

One then sets 
$$\Psi^{(G)}\colon \wt\cM\to \Irr(\wbG^F)\ \text{ by } (\wt s,\eta)\mapsto \chi^{\wbG}_{\wt s,\UCh(\eta)}.$$
		
Note that $\wt\calG_d$ is defined to be its image, taking into account that $\Cent_{ \wbG^*}(\bS^*)=\wbT_d^*$ is a torus and therefore $\UCh(\wbT_d^*{}^F)=\{1_{\wbT_d^*{}^F}\}$.
		
On the other hand Condition \Bd{} implies there exists some $\Lin(\wG/G)\rtimes \wh N$-equivariant extension map $\wt \Lambda$ \wrt $\wt C\unlhd \wt N$. Let us recall the isomorphism 
  $$\wt N_{\chi_{\wt s,1}^{\wbT_d}}/\wbT_d^F\cong \NNN_{\Cent_{\wbG^*}(\wt s)}(\wbT_d^*)^F/\wbT_d^*{}^F,$$ 
given by a well-defined duality map $i_{\wt s,1} $, see the proof of \cite[Cor. 3.3]{CS13}. For $\eta\in\Irr(\NNN_{\Cent_{\wbG^*}(\wt s)}(\wbT_d^*)^F/\wbT_d^*{}^F)$ we denote by $\eta^*=\eta\circ i_{\wt s,1}$ the corresponding character of $\wt N_{\chi_{\wt s,1}^{\wbT_d}}$. We obtain a surjective map 
\[ \Psi^{ (N)}:\wt \calM\lra \Irr(\wt N)\text{ given by }(\wt s,\eta)\mapsto (\wt \Lambda(\chi_{\wt s,1}^{\wbT_d})\eta^*)^{\wt N},\]
which makes sense by Clifford theory (\ref{Cliff}). By the considerations made in the proof of \cite[Thm 6.1]{CS17A} which only use that the characters of $\wbG^F$ involved are in $\wt\calG_d$, the maps $\Psi^{(G)}$ and $\Psi^{(N)}$ are constant on $\wt N^*$-orbits, and they are bijections once seen as on the quotient sets. Then one gets a bijection 
$$\wt\Omega '\colon \wt\calG_d\to\Irr(\wt N) \text{ with } \Psi^{(G)}(\wt s,\eta) \mapsto \Psi^{(N)}(\wt s,\eta)$$ 
and the properties announced thanks to \cite[Prop. 6.3]{CS17A} and \cite[Thm 6.1]{CS17A} whose main arguments are independent of the prime $\ell$. 

The above follows the constructions of Section 6 of \cite{CS17A} and uses the fact that $\Cent_{\wbG^F}(\bS)$ is a torus to simplify the technical construction of $\wt \Omega'$ as otherwise the parameters used are triples. If $d$ is not regular for $(\bG,F)$, the construction of $\wt \Omega'$ can be deduced from Section 6 of \cite{CS17A} in a similar manner by omitting the assumptions involving $\ell$.  Note that in both cases, as soon as $\ell$ is an odd prime with $d=d_\ell(q)$, we also have $ \Irr(\wGF\mid \Irrl(\GF))\subseteq \wt\calG_d$   and   $ \wt \Omega'(\Irr(\wGF\mid \Irrl(\GF)))=\Irr(\wN\mid \Irrl(N))$ (see the discussion on degrees in  \cite[pp 180-181]{CS17A}).
%According to \cite[Sect. 4.6]{GM}, implies that for every unipotent character $\lambda\in \UCh(\Cent_{\bK^*}(\bS^*)^F)$ there is a bijection \begin{align*} \UCh\colon	\Irr(\NNN_{\Cent_{\wbG^*}(\wt s)}(\bS^*)^F/\Cent_{\Cent_{\wbG^*}(\wt s)}(\bS^*){}^F)&\to \UCh_d(\Cent_{\wbG^*}(\wt s)^F\mid \la) \\\eta&\mapsto \UCh_\lambda(\eta)\end{align*} 
This finishes the proof of (a).

		For the proof of (b), we can apply \Cref{prop_23} thanks to Conditions \Ad{} and \Bd{}. Hence, there exists a bijection 
	\[ \Omega ':\calG_{d}\lra \Irr(N),\]
	such that 
	\[ ( G\rtimes \Gamma_\chi, G,\chi)\geq_c (N\rtimes \Gamma_{\chi}, N, \Omega'(\chi)) \text{ for every } \chi\in\calG_d.\]
	Assuming $d=d_\ell(q)$ for some odd prime $\ell$, as $ \wt\Omega'(\Irr(\wG \mid \Irrl(G )))=\Irr(\wt N\mid \Irrl(N))$ the map $ \Omega '$ satisfies $ \Omega'(\Irrl(G))=\Irrl(N)$ by Clifford theory. \end{proof}

%%%%%%%%%%%%%%%%%%%%%%%%%%%%%%%%%%%%%%%%%%%%%%%%%%%%%%%%%%

\section{The group \texorpdfstring{$M$}{M}. Characters and Clifford theory} \label{sec_nondreg_groupM}

From now on we work with the group $\bG\cong\tDlsc(\FF)$ ($l\geq 4$), seen as a subgroup of $\ov \bG\cong \tBlsc(\FF)$ with common maximal torus $\bT$, associated root system $\ov\Phi =\Phi(\ov\bG,\bT)$ and root subgroups $\bX_\al$ ($\al\in \ov \Phi$), see \ref{sub2E}. Recall also the Frobenius endomorphism $F\colon\bG\to\bG$ such that $\GF=\tDlsc^\eps(q)$.

In this chapter we introduce first a finite subgroup $M\leq \bG$ depending on some integers $l_1$, $l_2=l-l_1$, $\eps_1=\pm 1$ and $\eps_2=\eps_1\eps$. The group $M$ has a normal subgroup $M_0$ of index $2\gcd(2,q-1)$, which is a central product $G_1.G_2$ with $G_i\cong \tD_{l_i, \sico}^{\eps_i}(q)$ in the sense of \Cref{Def_low_l}. In an analysis split into two main cases, we also introduce in \Cref{ssecE(M)} a group $\UE(M)$ acting on $M$ that will allow us to make statements similar to the condition \textbf{A}$(\infty)$ for this group $M$. This will involve changing $F$ into a slightly different $\vFq$ that proves more suitable when dealing with $d$-tori. 
Our main aim is to deduce from the knowledge of $\Irr(G_i)$, the properties of $\Irr(M)$, mainly the statement in \Cref{thm_sec_Ad}, an analogue of \Ad\ for $M$. We show later in \Cref{lem6_11} that the integers $l_1$, $\eps_1$ can be chosen so that the associated group $M $ contains $\NNN_\GvF(\bS)$ for a Sylow $d$-torus $\bS$ of $(\bG,\vFq)$. %In the next Chapter, we construct a character correspondence between $\Irr(\NNN_\GvF(\bS))$ and some characters of $M$, so that the results on $M$ can be used to prove Conditions \Ad{} and \Bd{} for all integers $d$ which are not doubly regular for $(\bG,F)$.

The present chapter is structured as follows. First, we define the group $M$ and investigate the structure of $M$ as a group, in particular, we introduce the normal subgroup $M_0\unlhd M$. Afterwards, we introduce a group $\UE(M)$ acting on $M$ in \Cref{ssecE(M)}. 
%, a subgroup of some group $\GvF$ that is isomorphic to a subgroup of $\GF$ containing $\NNN_\bG(\bS')^F$, where $\bS'$ is a Sylow $d$-torus of $(\bG,F)$. We first

 \Cref{ssec_5B} starts with collecting some basic observations on the characters of $M$ and an approach how to study them via $M_0$, thus establishing some preliminary simplifications for the proof of \Cref{thm_sec_Ad}. Since $M_0$ is the central product $G_1.G_2$ of two groups of type $\tD$, we transfer in Section \ref{ssec5D} the results on the character sets $\oTT$, $\EE$, $\DD$ from Chapter \ref{sec_3} to the characters of the direct factors of $M_0$. This leads to a partitioning of $\Irr(M)$ into subsets along the characters of $M_0$ and splitting the proof of \Cref{thm_sec_Ad} according to those sets in Sections \ref{ssec5E} and \ref{ssec5F}. 

\subsection{The group \texorpdfstring{$M$}{M} } \label{ssec_M}
In the following we introduce a subgroup $\bM$ of $\bG= \tDlsc(\FF)$ and then set $M:=\bM^{\vFq}$, defined in a general way from the integers $\eps_1$, $\eps_2$, $l_1$ and $l_2$ as well as the prime power $q$ determining some Frobenius endomorphism $\vFq$. The conventions followed to define $\bG_1$ and $\bG_2$ in \Cref{not_groupM} are meant to limit the number of cases to review in proofs. The rôles of $\bG_1$ and $\bG_2$ are symmetric in the sense that a Sylow $d$-torus of $\bG$ will eventually be assumed to be included in one of them but both cases, $\bG_1$ or $\bG_2$ are permitted, see \Cref{sec_6.B}. The definitions made below may look arbitrary, but a glance at \Cref{tricho} and its proof can already provide an explanation.

 For a given Frobenius endomorphism $F'$ of $\bG$ recall the Lang map 
\[ \II LF'@{\calL_{F'}}:\bG\lra \bG \text{ given by }x\mapsto x^{-1}F'(x).\] 

\begin{notation}\label{not_groupM}
    Let $ \eps_1,\eps_2\in \{\pm 1\}$ and $l_1,l_2\geq 1$ with $\eps=\eps_1\eps_2$ and $l_1+l_2=l\geq 5$. Assume that 
	\begin{asslist}
		\item $\eps_1=-1$ if $\eps=-1$; and 
		\item $2\mid l_1$ if $2\nmid l$ and $(\eps_1,\eps_2)=(-1,-1)$.
	\end{asslist}
	For $\II{J1}@{J_1:=\underline{l_1}}$, $\II{J2}@{J_2:=\underline l \setminus J_1}$ and $i\in \{1,2\}$, let $\II{Rioverline}@{\protect{\overline R}_i}:= \ov \Phi\cap \spa{e_j\mid j \in J_i}_\ZZ$, $\II{Ri}@{R_i}:=\ov R_i\cap \Phi$, $\II{Ti}@{\bT_i}:= \bT\cap \spa{\bX_\al\mid \al\in \ov R_i}$, and
	\[ \II{Gi}@{\bG_i}:=\begin{cases} \bT_i&\text{if }l_i= 1,\\
		\spa{\bX_\al\mid \al\in R_i}& \otw .\end{cases} \] 
	
\end{notation}

Following \Cref{not:3_16}, we clearly have $\bG_1 =\bG_{\underline{l_1}}$ and $\bG_2\cong \bG_{\underline{l_2}}$. Recall that for $ k\geq 2$, $\bG_{\uk}\cong \tD_{k,\sico}(\FF)$, using the notation of \Cref{Def_low_l} when $k\leq 3$.

 Recall $ {\neins=\n_{e_1}(\varpi)} $ and set $ {\nzwei:=\n_{e_l}(\varpi)}$. The above and the commutators given in \Cref{Comm} and \Cref{ZG_order}(d) yield at once the following.

\begin{lem} \label{rem51}
	\begin{thmlist}
		\item $[\bG_1,\bG_2]=1$, $\bG_1\cap\bG_2=\spannh$ with $\bT\leq \bG_1.\bG_2$.
		\item $[\neins, \nzwei]=h_0$, $n_i^\circ$ normalizes $\bG_i$ and centralizes $\bG_{3-i}$ for $i=1,2$. 
	\end{thmlist}
\end{lem}

	If $\II{gammai}@{\gamma_i}$ ($i=1,2$) denotes the automorphism of $\obG$ given by conjugation with $n^\circ_i$, then $\bG_i$ is $\gamma_i$-stable and $\gamma_i$ induces a graph automorphism on $\bG_i$. Moreover, $\gamma_i$ defines an automorphism of $\bG$. We also denote by $\gamma_1$ the graph automorphism of $\wbG$. Note that $\gamma_2$ is the concatenation of $\gamma$ and an inner automorphism of $\bG$. As such $\gamma_2$ acts also on $\wbG$ and by abuse of notation we denote this automorphism of $\wbG$ also by $\gamma_2$. 
	
Before defining the group $\bM\geq \bM^\circ :=\bG_1.\bG_2$ and a slight replacement for the Frobenius endomorphism $F$ we need to introduce the elements $\wh t_1, \wh t_2\in \bT$. 
\begin{lem}\label{rem:5:2}
	Set $\II veecir@{v^\circ}:= (\neins)^{\frac{1-\eps_1}{2}} (\nzwei)^{\frac{1-\eps_2}{2}}$, $\II Zi@{\bZ_i}:=\spa{\h_{J_i}(\varpi) , h_0}$ and $\III{Z_i}:=\bZ_i^{v^\circ F_q} $. For $i=1,2$ we choose some $\II thati@{\wh t_i}\in \calL_{v^\circ F_q}^{-1}(h_0)\cap \bT\cap \bG_i$ with $\wh t_i\in\Z(\bG_i)$, when possible. Then:
	\begin{thmlist}
		\item if $i\in \{1,2\}$ and $|Z_i|=2$, then $\wh t_i\in \bZ_i$;
		\item if $(\eps_1,\eps_2)=(-1,-1)$ and $\{|Z_1|,|Z_2|\}=\{2,4\}$, then $\wh t_1\in \Z(\bG_1)$ and ${F_p}(\wh t_1)=\wh t_1$;
		\item If $(\eps_1,\eps_2)=(-1,-1)$ and $\{|Z_1|,|Z_2|\}=\{2\}$, then one of the following holds
		\begin{itemize}
			\item $F_p(\wh t_1)=\wh t_1$ and $2\mid l_1$; or 
			\item $F_p(\wh t_1\wh t_2)=\wh t_1 \wh t_2$ and $2\nmid l_1 l_2$. 
		\end{itemize}
	\end{thmlist}
\end{lem}
We see that always $\bZ_i\leq\Z(\bG_i)$, and $\bZ_i=\Z(\bG_i)$ unless $l_i=1$ in which case $\bG_i$ is a torus.
\begin{proof}
	Observe that for every $I\subseteq \ul$, $\h_I(\varpi)$ is $F_p$-fixed if $4\mid (p-1)$ or $2\mid |I|$.
	
	In part (a), $\h_{J_i}(\varpi)\in \bZ_i\setminus Z_i$ and can therefore be chosen as $\wh t_i$.
	
	Then consider part (b). In this case $(\eps_1,\eps_2)=(-1,-1)$ and $\{ |Z_1|,|Z_2|\}=\{2,4\}$. According to \Cref{ZG_order}, $\{|Z_1|,|Z_2|\}=\{2,4\}$ implies that $l_1$ and $l_2$ have opposite parities. This leads to $2\nmid l_1+l_2= l$. The assumption in \ref{not_groupM}(ii) implies $2\mid l_1$. According to \Cref{ZG_order}, $|Z_1|=2$ and hence $\wh t_1=\h_{J_1}(\varpi)$ by (a). This proves part (b).
	
	In part (c), we assume that $(\eps_1,\eps_2)=(-1,-1)$ and $|Z_1|=|Z_2|=2$. By part (a), we can choose $\wh t_i=\h_{J_i}(\varpi)$ for $i=1,2$. Such an element is $F_p$-fixed if and only if $l_i$ is even, otherwise $F_p(\wh t_i)=h_0\wh t_i$, see \ref{3:3}. 
	If $2\nmid l$, then $2\mid l_1$ by \ref{not_groupM}(ii) and hence $F_p(\wh t_1)=\wh t_1$. Otherwise $2\mid l=l_1+l_2$ and $F_p(\wh t_1\wh t_2)=\wh t_1\wh t_2$ if $2\nmid l_1$.
\end{proof}
%%%%%%%%%%%%%%%%%%%%%%%%%%%%%%%%%%%%%

\begin{defi} \label{def_M} Depending on $|Z_1|$ and $|Z_2|$, we fix elements $\II vee@{v}\in \bG\sqcup\{\neins \}$, $n\in \bG$ as in the Table \ref{tab:my_label} below. We also define there $\II nu@{\nu}\in \text{Inn}(\bG)\sqcup\{\gamma\}$ as the automorphism of $\bG$ induced by conjugation by $v$, so that $$\nu(g)=vgv\inv$$ for any $g\in \bG$, and $\nu$ extends to $\wbG$ as explained in \Cref{ssec2C} for $\gamma$. We then form $\II nuF@{\vFq} $ the corresponding endomorphism of $\wbG$. From \Cref{lem_gamma} and the choice of $v$ in \Cref{tab:my_label} one has clearly $\bG^{\vFq }\cong	\bG^{F}\cong \twepsDlq$ and
	$$\II {Gi}@ {G_i:=\bG_i^{\vFq }}\cong\tD^{\eps_i}_{l_i,\sico}(q)\text{ for }i=1,2$$ in the notation of \Cref{Def_low_l}. 
	We also define 
	\[ \II{Mbold}@{\bM}:= (\bG_1.\bG_2 )\spa{\neinszwei}=(\bG_1.\bG_2 )\spa{n}\geq \II M @ {M:=\bM^{\vFq }} \geq \III{M^\circ:=(\bG_1.\bG_2 )^{\vFq }}\geq \III{M_0:=G_1.G_2}.\] 
	
\end{defi}

\begin{table}[h!]
	\centering
	\begin{tabular}{|c|c|c||c|c|c|c|c|c|c|}
		\hline 
		$(\epsilon_1,\epsilon_2)$ & $\{|Z_1|,|Z_2|\}$&Condition &$v$&$n$&Comment& $\nu$ \\\hline
		$(\,\,\,\,\,1,\,\,\,\,1)$& \text{any}&& $1$& $\neins \nzwei $&& $v$\\
		$(-1,-1)$& $\{1\}$ or $\{4\}$ && $\neins\nzwei$& $\neinszwei$& &$v$\\
		$(-1,-1)$& $\{2\}$ or $\{2,4\}$& $2\mid l_1$& $\wh t_1\neinszwei$ & $\wh t_1 \neins \nzwei $ &&$v$ \\
		$(-1,-1)$& $\{2\}$ &$2\nmid l_1$& $\wh t_1\wh t_2\neinszwei$ & $\wh t_1\wh t_2\neinszwei$& $2\nmid l_1$ implies $2\mid l$&$v$\\
		$(-1,\, 1)$& $\{1\}$, $\{2\}$ or $\{2,4\}$& &$\neins$& $\wh t_j \neins \nzwei $& $j$ with $|Z_j|\in \{1,2\}$&$\gamma$ \\
		$(-1, \,\,1)$& $\{4\}$ && $\neins$& $\wh t_1 \neins \nzwei $&&$\gamma${ }\\ \hline 
	\end{tabular}
	\caption{Choice of $n$, $v$ and $\nu$}
	\label{tab:my_label}
\end{table}
%%%%%%%%%%%%%%%%%%%%%%%%%%%%%%%%%
Using $v$ we recover the groups $Z_i$ and obtain finite subgroups of $\bG_i$ and $\bG$ as $\vFq $-fixed points. 
\begin{lem} \label{lem:5_3}
	Let $ {v}$ be defined as in \Cref{tab:my_label}. Then for $i=1,2$ we have
	\begin{thmlist}
		\item $\III{Z_i}=\bZ_i^{\vFq }=\spa{\h_{J_i}(\varpi) , h_0}^{\vFq } $;
		\item $ \calL_{\vFq }(\wh t_1)= \calL_{\vFq }(\wh t_2)=h_0$.
	\end{thmlist}
\end{lem}
\begin{proof} We observe that $v^{-1} v^\circ \in \bT$ according to \Cref{rem:5:2}, so the actions of $v$ and $v^\circ$ on $\bT$ coincide. As a consequence we see that that the groups $Z_i$ and the elements $\wh t_i$ defined using $v^\circ$ satisfy similar properties with respect to $\vFq$.
\end{proof}

\begin{lem}[Structure of $M$] \label{lem:5_5} $M^\circ= M_0\spa{\wh t_1 \wh t_2}$ and $M=M^\circ \spa{n}\lneq 	\bG^{\vFq }$. Additionally, $G_1\unlhd M$ and $G_2\unlhd M$.
\end{lem}
\begin{proof}
The equality $M^\circ= M_0\spa{\wh t_1 \wh t_2}$ comes from Lang's theorem and \Cref{rem51}(a). The inequality $\GvF\neq M$ holds since $\GvF$ is perfect. We also have $G_i\unlhd M^\circ\spa{n}$ by the central product structure of $\bM^\circ$ and \Cref{rem51}(b).
	
The equality $M=M^\circ \spa{n}$ is true as soon as $[v F_q ,n]=1$. 

Note $[\neins \nzwei,F_p]=1$ since $\n_{e_j}(\varpi)=\h_{e_j}(\varpi)\n_{e_j}(1)$ for any $j\in \ul$ and therefore
\begin{align}\label{eq51} [\nii,F_p]:= \begin{cases} 1&\text{if } 4\mid (p-1),\\
		h_0& \otw . \end{cases} \end{align} 
	Recall $[\neins,\nzwei]=h_0$, see \Cref{rem51}(b). 
	
	If $v\in \{1,n\}$, then 
	\[[v F_q ,n]=[v,n][F_q,n]= [F_q,n]=1.\] 
	From \Cref{tab:my_label} we see that $v (v^\circ)^{-1}\in \Z(\bG_1.\bG_2)^{F_p}$ whenever $(\eps_1,\eps_2)=(-1,-1)$ or equivalently $v=n$. This verifies $[v F_q ,n]=1$ whenever $\eps=1$. 
	
	According to \Cref{tab:my_label} it remains to consider the case where $\eps=-1$. Then $v=\neins$ and we observe 
	\begin{align*}[v F_q ,n]= [ \neins F_q, \wh t_j \neins \nzwei]&= [ \neins F_q, \wh t_j ][ \neins F_q, \neins][ \neins F_q, \nzwei]= h_0 [F_q,\neins] [\neins,\nzwei] [F_q,\nzwei]= \\ 
		&= h_0 [F_q,\neins] h_0 [F_q,\nzwei] =1\text{ by (\ref{eq51}) above.}\end{align*}
	This finishes our proof.
\end{proof}
\begin{lem}[Action of $\wt T$ on $M$]\label{lem:5_actions_wbT}
	Set $\Ispezial{Tbreve}@{\protect{\breve T}}@{\wb T} 
	:=\calL\inv_{\vFq }(\Z(\bG))\cap \bT$, 
	$\II{Ttilde}@{\wt T}:=(\bT \Z(\wbG))^{ \vFq}$ , 
	$\III{T_0}:= G_1.G_2\cap \bT$ and let $\II {ttildei}@{\wt t_i}\in \calL_{\vFq }^{-1}(\h_{J_i}(\varpi))\cap \bT_i$ for $i=1,2$. Then
	\begin{thmlist}
		\item $\wt T$ and $\wb T$ induce the same automorphisms on $\GvF$;
		\item \label{lem:wtT} $\wb T=\spa{\wt t_1\wt t_2,\wh t_1,\wh t_2} T_0$;
		\item $\wb T$ and $\wt T$ normalize $M$, $M^\circ$, $M_0$, $G_1$ and $G_2$.
	\end{thmlist}
\end{lem} 
\begin{proof} Part (a) is standard, see \ref{not_diag} and \ref{Gbreve}.
	
	For part (b), we have $\h_\ul(\varpi)=\h_{J_1}(\varpi)\h_{J_2}(\varpi)$ and therefore $\wt t_1\wt t_2$ satisfies $\calL_{\vFq }(\wt t_1\wt t_2)=\h_{\ul}(\varpi)$. By \Cref{lem:5_5}, we have $\bT^{\vFq }=T_0\spa{\wh t_1\wh t_2}$. On the other hand $\wb T/ \bT^{\vFq }$ is isomorphic to $\Z(\bG)$ by $\calL_{\vFq }$, so we indeed get $\wb T=\bT^{\vFq }\spa{\wt t_1\wt t_2, \wh t_1}$ since by \Cref{lem:5_3} we are adding elements whose images under $\calL_{\vFq }$ generate $\Z(\bG)$, see \Cref{ZG_order}(a).
	
    We consider part (c). Recall $G_1\unlhd M$ and $G_2\unlhd M$ from \Cref{lem:5_5}. 
	Now $\wh t_1$ acts on $G_1$ as a diagonal automorphism associated with $h_0[\Z(\bG_1), \vFq ]$ and $\wt t_1$ acts on $G_1$ as a diagonal automorphism associated with $\h_{J_1}(\varpi)[\Z(\bG_1), \vFq ]$ in the parametrization of \Cref{not_diag}. On the other hand, $[\wh t_1,\bG_2]=[\wt t_1,\bG_2]=1$. We can describe similarly the action of $\wt t_2$ and $\wh t_2$ on $G_2$. We observe $[n,\wt t_1\wt t_2 ]\in \spa{\wh t_1 \wh t_2} M_0=M^\circ$. Conjugation with $\wt T$ and $\wb T$ then stabilises $M$, $\bT$, $\bG_1$ and $\bG_2$ by \Cref{lem:5_5} and we get our claim. 
\end{proof}

Before going further into describing $M$ and some of its automorphisms we show below the relevance to our work around Sylow $d$-tori of $(\bG,F)$ for $d\geq 3$. In particular we show that for non-doubly regular $d$'s one can almost always build a group $M$ such that one of the two groups $\bG_1$ and $\bG_2$ contains a Sylow $d$-torus of $(\bG,\vFq)$ and $d$ is doubly regular for that $(\bG_j,\vFq)$, see case (iii) of \Cref{tricho}. 

Recall that $\eps\in\{\pm 1\}$ with $\GF=\twepsDlq$. For $\zeta\in\CC^\times$ a primitive $d$-th root of unity we denote by $\II aGFd @{a_{(\mathbf{G},F)}(d)}$ the multiplicity of $\zeta$ as a root of the polynomial order 
$$P_{(\bG,F)}(X)=X^{l^2-l}(X^2-1)(X^4-1)\cdots (X^{2l-2}-1)(X^l-\eps),$$ see \cite[Table 1.3]{GM}.

\begin{lem}\label{tricho}
	We keep $\GF=\tD_{l,\sico}^\eps(q)$ with $l\geq 5$ and take $ d\geq 3$. Then one of the following three possibilities occurs:
	\begin{asslist} 
		\item $d$ is doubly regular for $(\bG,F)$; 
		\item $a_{(\bG ,F)}(d)\leq 1$ and therefore $a_{(\bG ,F)}(dm)=0$ for any odd $m\geq 3$; or
		\item there exist $j\in\{1,2\}$, $l_1, l_2>0$ and $\eps_1, \eps_2\in\{\pm 1 \}$ with $l_1+l_2=l$, $\eps=\eps_1\eps_2$  determining $\bG_1$, $\bG_2$ and $\nu$ as in \Cref{not_groupM} and \Cref{def_M} such that $a_{(\bG ,F)}(d)=a_{(\bG_j,\vFq )}(d)$, $l_j\geq 4$ and $d$ is doubly regular for $(\bG_j,\vFq )$.
	\end{asslist}
\end{lem}
\begin{proof}
	In the following, whenever $k\geq 2$ and $\delta=\pm 1$ we write $a_{k,\delta}(d)$ for the multiplicity of a given $d$-th root of unity as a root of \[(X^2-1)(X^4-1)\cdots (X^{2k-2}-1)(X^k-\delta)= (X^2-1)\cdots (X^{2k}-1)/(X^k+\delta).\] Then $a_{(\bG ,F)}(d)=a_{l,\eps}(d)$.
	
	Let us check first that if $a_{l,\eps}(d)\leq 1$, then $a_{l,\eps}(dm)=0$ for any odd $m\geq 3$. Let $\zeta\in\CC^\times$ be a $dm$-th root of 1. If $a_{l,\eps}(dm)\geq 1$, then $\zeta$ is a root of $X^l-\eps$ or some $X^{2i}-1$ for $1\leq i\leq l-1$. In the first case $\zeta^m$ is then also a root of $X^l-\eps$ and since $dm\mid 2l$, $d$ divides the even integer $2i_0=2l/m$, so $\zeta^m$ is a root of $X^{2i_0}-1$ with $1\leq i_0\leq l-1$, leading to $a_{l,\eps}(d)\geq 2$. In the second case, we get $dm\mid 2i$, so $\zeta^m$ is a root of both $X^{2i}-1$ and $X^{2i'}-1$ for $i'$ the integer $i/m$, implying again $a_{l,\eps}(d)\geq 2$.

	We now assume that (i) and (ii) are not satisfied, that is $a_{l,\eps}(d)\geq 2$ and $d$ is not doubly regular for $(\bG,F)$, i.e., $d\nmid 2l$ or it does but $(-1)^{2l/d}\neq \eps$, see \Cref{def_dreg}. 
	
	Set $d_0=d$ if $d$ is odd, $d_0=d/2$ if $d$ is even. Set $k:= a_{l,\eps}(d) d_0$ and set $\delta:=(-1)^{a_{l,\eps}(d)}$ if $2\mid d$, and $\delta=1$ otherwise. It is easy to see that $d$ being not doubly regular implies 
	\begin{align}\label{eq:6.3}a_{l,\eps}(d)=a_{k,\delta}(d)=\Big\lfloor \frac {l-1} {d_0} \Big\rfloor \ \ \text{ and therefore } \ \ k=d_0\cdot\Big\lfloor \frac {l-1} {d_0} \Big\rfloor . \end{align} 
	This can be checked directly on the polynomials recalled above or by arguing on the form of regular elements of order $d$ in the Weyl groups of $\bG$ and $\bG_{k}$, see \Cref{lemreg4:8}. 
	
	Recall $l\geq 5 $, $d\geq 3$ and hence $d_0\geq 2$. So $4\leq d_0 \cdot a_{l,\eps}(d)=k=d_0\cdot \lfloor \frac {l-1} {d_0} \rfloor\leq l-1$. Let us show that we can choose $(\eps_1,l_1),(\eps_2,l_2)$ such that the assumptions from \Cref{not_groupM} are satisfied and $(\delta,k)\in\{(\eps_1,l_1),(\eps_2,l_2)\}$. Indeed if $\delta =1$ the choice $(\eps_1,l_1)=(\delta,k)$ clearly satisfies \ref{not_groupM}(i) and (ii). If $\delta =-1=\eps$ then one takes $(\eps_2,l_2)=(\delta,k)$ while \ref{not_groupM}(ii) is empty. Finally, when $\delta=-1=-\eps$ then \ref{not_groupM}(i) is empty and one can always take $(\eps_1,l_1)$ or $(\eps_2,l_2)=(\delta,k)$ to satisfy the parity condition of \ref{not_groupM}(ii).
	
	We then get (iii) by taking $j\in\{1,2\}$ such that $(\eps_j,l_j)=(\delta,k)$ since (\ref{eq:6.3}) above shows that the multiplicity of the $d$-th cyclotomic polynomial in the order of $(\bG_j,\vFq )$ is the same as in the one of $(\bG,F)$ while $d$ is clearly doubly regular for $(\bG_j,\vFq )$ by choosing $2k=2d_0a_{l,\eps}(d) $ and $\delta$.
\end{proof}

The above has given a hint on how $M$ will be used in the next chapter to verify Conditions \Ad{} and \Bd. We give below a few more group-theoretic properties of $M$ that won't be used but explain that this construction is more natural, and less new, than it seems. 

\begin{rem}
\begin{thmlist}
\item If $(l_1,\eps_1)\neq (l_2,\eps_2)$, then the group $M$ is a maximal subgroup of $\bG^\vFq$, a member of the family $\calC_1$ in Aschbacher's classification, see \cite[\S 18]{MT}.  

\item Recall $\pi_\SO\colon \bG\to \SO_{2l}(\FF)$ the reduction mod $\spannh$. We have $\spannh\leq M$ and $\pi_\SO (\bM)=\Cent_{\SO_{2l}(\FF)}(s)$ for $s:=\pi_{\SO}(\h_{J_2}(\varpi))$, a centraliser of an involution in $\SO_{2l}(\FF)$. 

\item Let $j\in\{1,2\}$ such that $l_j\geq 2$ and therefore there exists $\bL$ a $\vFq$-stable Levi subgroup of $\bM^\circ$ such that $[\bL,\bL]= \bG_{j}$. Then $\bM$ satisfies
\[\bM=\NNN_{\bG}([\bL,\bL]).\] 
Indeed, computations with the roots of $\bG$ allow us to see $\bM^\circ=[\bL,\bL]\,  \Cent^\circ_\bG([\bL,\bL])$. 
This group is normal in $\NNN_\bG([\bL,\bL])$ and we check easily $\NNN_\bG(\bM^\circ, \bG_{3-j})= \bM^\circ \spa{n}$. This leads to $\bM=\NNN_{\bG}([\bL,\bL]) \und$
\[ M= \NNN_\bG([\bL,\bL]) ^{\vFq}.\] 

If moreover for some odd prime $\ell$, $\bL$ is $d$-split and there exists a $d$-cuspidal character $\zeta\in\UCh(\bL^\vFq)$ for $d=d_\ell(q)$, then $(\bL,\zeta)$ forms a unipotent $d$-cuspidal pair of $(\bG,\vFq)$ defining an $\ell$-block $B$ of $\bG^\vFq$, see \cite[Thm 22.9]{CE04}. Certain subgroups of $M$ are studied in \cite{CE99} and are shown to control the $\ell$-fusion in $B$, see \cite[Prop.~5]{CE99}. When moreover $\bL$ is a minimal $d_\ell(q)$-split Levi subgroup as is the case in our applications of Ch. 6, then $M$ contains a Sylow $\ell$-subgroup of $\bG^\vFq$ and controls the fusion of $\ell$-subgroups of $\bG^\vFq$.
\end{thmlist}
\end{rem}

\subsection{Some groups of automorphisms of \texorpdfstring{$M$}{M}}\label{ssecE(M)}

We now define the automorphism group $\UE(\bM)$ as a slight variant of $\UE(\bG)$ already encountered in the preceding sections. We then restrict it in a manner similar to the construction of $\UE(G)$ to obtain the finite groups $\UE(M)$, $\UE(M^\circ)$, and other variants suitable for $\Ispezial{Mbreve}@{\protect{\breve M}}@ {\wb M}:=M\wb T$.

Our further considerations are divided into two cases setting apart the case corresponding to the last line of Table~\ref{tab:my_label}.
With the definition of $\UE(\bG)$ as bijective group homorphisms we can form the abstract group $\bG\rtimes \UE(\bG)$.

\begin{defi}\label{EMdef}
	Let $\UE(\bM)\leq \bG\rtimes \UE(\bG)$ be the subgroup given by
		\[ \II{EbM}@{\UE(\bM)}:=\begin{cases}
			\spa{F_p,\gamma}=\UE(\bG)&\text{if } n\in \Z(\bG_1.\bG_2)\neinszwei, \\ 
   \spa{F_p^2,\wh t_1\gamma}&\text{otherwise}.\end{cases} \] 
	
Let $\II{EM}@{\UE(M)}$, resp. $\Ispezial{EMT}@{\UE( \protect {\breve M})}@{\UE(\wb M )}$ be the subgroup of $\Aut(\wbG^{\vFq })$, respectively of $\Aut(\wbG^{\vFq }\wb T)$, obtained by restriction of $\UE(\bM)$. Set 
\[\II EMcirc@{ \UE(M^\circ):=\UE(M)\spa {n,h_0}}\leq \GvF \UE(M)\leq \GvF\rtimes \Aut(\GvF),\] 
so that $M\UE(M)=M^\circ\UE(M^\circ)$ (see \Cref{lem:5_5}). Set $\UE({\wb M}^\circ):=\UE(\wb M)\spa{n,h_0}\leq \wbG^\vFq\rtimes \Aut(\wbG^\vFq)$.
\end{defi}

 From \Cref{rem:5:2}(a--b) it is easy to see that the first case considered above corresponds to the first five lines of Table~\ref{tab:my_label}, while the second case corresponds to the last line where $\eps_1=-\eps_2=-1 \text{ and } |Z_1|=|Z_2|=4$.
 We now gather information on those automorphism groups, starting with the first case of the above Definition.
 
\begin{hyp}\label{case1}
	Assume $\eps=1$ or $\{|Z_1|,|Z_2|\} \neq \{ 4\}$. \end{hyp}

\begin{lem}\label{lem:5_6}
	Assume \Cref{case1}. Then
	\begin{thmlist}
		\item $n$ acts as $\gamma_1\gamma_2$ on $M_0=G_1.G_2$. 
		\item $G_1$, $G_2$ and $\spa{n,h_0}$ are $\UE(M)$-stable.
  		\item $\UE(\wb M)$ stabilizes $\wb T$, $M$ and $\wb M$. The group $\UE(\wb M ^\circ)=\UE(\wb M)\spa{n,h_0}$ 
	 also stabilizes ${\wb G}_i:=G_i\spa{\wh t_i,\wt t_i}$ for $i=1,2$.
  \item If $\eps_1=\eps_2=-1$, then 
  $vF_q \in \Cent_{\GvF \UE(M)}(\GvF)$.
\end{thmlist}\end{lem}

\begin{proof}
	Part (a) follows from $n\in \neinszwei \Z(\bG_1.\bG_2)$ while (b) is clear from the proof of \Cref{lem:5_actions_wbT}, see also the relations recalled in \Cref{tab:properties_c1} below.
	
	For part (c), we see that $F_p$, $\gamma$ and hence $\UE(M)$ stabilise $\Z(\bG)$ and $\bT$. This shows that $\UE(M)$ stabilises $\wb T$ and $\GvF \wb T$. For part (d) note that by the construction $vF_q$ acts trivially on $\GvF$.
\end{proof}

For the above proof and later we need the commutators for some elements of $ \UE(M^\circ)$. All are easily deduced from \ref{Comm} and \ref{ZG_order}(d), see also \Cref{rem:5:2}.
\begin{table}[h!]
	\centering
	\begin{tabular}{|c|c|c||c|c|c|c|c|c|}
		\hline 
		$(\epsilon_1,\epsilon_2)$ & $\{|Z_1|,|Z_2|\}$&& $v$&$n$& $[F_p,n]$&$[\gamma_1,n]$&$[F_p,v]$ & $[\gamma_1,v]$\\\hline
		$(1,1)$ & any && $1$ & $\neinszwei$ & 1&$h_0$&1&1\\
		$(-1,-1)$& $\{1\}$ or $\{4\}$ && $\neinszwei$ & $\neinszwei$ & 1&$h_0$&1&$h_0$\\
		$(-1,-1)$& $\{2\}$ or $\{2,4\}$& $2\mid l_1$& $\wh t_1\neinszwei $& $\wh t_1 \neinszwei $ &1&1&$1$&$1$ \\
		% 	$(-1,-1)$& $\{2\}$& $\wh t_1 \neins \nzwei $& $\wh t_1 \neins \nzwei $ &1&1&$1$&$1$ \\
		$(-1,-1)$& $\{2\}$ & $2\nmid l_1$ & $\wh t_1\wh t_2\neinszwei$ & $\wh t_1\wh t_2\neinszwei$&1&1&$1$&$1$ \\
		$(-1, 1)$& $\{1\}$, $\{2\}$ or $\{2,4\}$&& $\neins$& $\wh t_j \neins \nzwei $& $[F_p,\wh t_j]$ & $h_0[\gamma_1,\wh t_j]$ & $h_0$ & $1$\\
		%$(-1, 1)$ & $\{4\}$ && $\neins$& $\wh t_1\neinszwei$& $[F_p,\wh t_1]$ & $h_0[\gamma_1,\wh t_1]$ & $h_0$ & $1$ \\ 
		\hline 
	\end{tabular}
	\caption{commutators of some elements in $M\UE(M)=M^\circ\UE(M^\circ)$}
	\label{tab:properties_c1}
\end{table}
\begin{lem} \label{lem_centZ2} Assume \Cref{case1}. Then:
	\begin{thmlist}
		\item $\UE(M^\circ)\cap M^\circ = \spannh$ and $[\UE(M^\circ),\UE(M^\circ)]\leq \spannh$;
		\item For $\III{Z}:=(\bZ_1.\bZ_2)^{\vFq }$ the Sylow $2$-subgroup of $\Cent_{\UE(M^\circ)}(Z)$ is abelian.
	\end{thmlist}
\end{lem}
\begin{proof} The inclusion $[\UE(M^\circ),\UE(M^\circ)]\leq \spannh$ is clear from \Cref{tab:properties_c1}. We also have $(\neinszwei)^2=h_0$ by \ref{ZG_order}(d). Then $n^2\in\spannh$ by \ref{ZG_order}(b) and since $n\in \Z(\bG_1.\bG_2)\neinszwei$. The overgroup $M^\circ$ acts by 2 or 4 diagonal automorphisms on $G_1.G_2$ while 
 $\UE(M^\circ)/\spannh =\spa{\gamma, n\spannh, F_p}$ acts by the graph and field automorphisms $\gamma_1$, $\gamma_1\gamma_2$ and $F_p$ on $G_1.G_2$. The classification of automorphisms of quasisimple groups makes that $\UE(M^\circ)\cap M^\circ$ acts trivially on $G_1.G_2$ and equals $\spannh$ since both $n^2 $, $[F_p,n]$ and $ [\gamma_1,n]$ belong to $ \spannh$. This proves (a).
	
	For part (b) set $C:=\Cent_{\UE(M^\circ)}(Z)=\spa{h_0,\Cent_{\spa{\gamma,F_p,n}}(Z) }$. If $|Z_1|=|Z_2|=4$, then 
	\[C\in \{ \spa{F_p,h_0}, \spa{F_p\gamma_1,h_0}, \spa{F_pn,h_0}, \spa{F_p n \gamma_1,h_0}\}\] using again the description of the automorphisms induced on $G_1$ and $G_2$.
	In all cases, $C$ is abelian.
	
	Assume $\eps_1=1$ and hence $\eps_2=1$. By the above we can assume $2\in\{|Z_1|,|Z_2|\}$ and even $|Z_1|=2$. Then $2\nmid f$ according to \Cref{ZG_order}(b) and hence $\spa{\gamma,n,h_0}$ is the Sylow $2$-subgroup of $\UE(M^\circ)$. If additionally $|Z_2|=4$, then $\spa{\gamma,h_0}$ is the Sylow $2$-subgroup of $C$. If $|Z_1|=|Z_2|=2$ and hence $|Z|=4$, then $C\leq \spa{n,h_0}$ which is again abelian. 
	
	Next we consider the case where $\eps=1$ and $\eps_1=\eps_2=-1$. By assumption $2\in\{|Z_1|,|Z_2|\}$. According to \Cref{rem:5:2}, $|Z_1|=2$. The results in \Cref{tab:properties_c1} show that $\UE (M^\circ) =\spa{F_p,\gamma_1,n}$ is abelian, since the generators then commute with each other. 
	
Next, assume $\eps=-1$. As above, we can assume that $2\in \{|Z_1|,|Z_2|\}$. 
If $|Z_1|=2$, then $n=\wh t_1\neinszwei$. According to \Cref{tab:properties_c1}, $[\gamma_1,n]=1$ in this case and the group $\spa{\gamma_1,n,h_0}$ is abelian. If additionally $|Z_2|=2$, then $2\nmid f$ and $\spa{\gamma_1,n,h_0}$ is the Sylow $2$-subgroup of $\UE (M^\circ)$.
If $|Z_2|=4$, then $C=\spa{\gamma_1, nF_p}$ or $C=\spa{\gamma_1,F_p}$. Since in this case $[\gamma_1,n]=1$, the group $\spa{\gamma_1,n,h_0}$ is abelian. 
It remains to consider the case where $|Z_1|=4$ and $|Z_2|=2$. In this case, $2\nmid f$ and $n=\wh t_2 \neinszwei$. 
Then the group $\spa{\gamma_1, n,h_0}$ is a Sylow $2$-subgroup of $\UE(M^\circ)$ and the centralizer of $Z_1.Z_2$ in this group is $\spa{\gamma_1 n,h_0}$, abelian again. 
\end{proof}

\begin{lem}[Comparison of $\UE(M)$, $\UE(G_1)$ and $\UE(G_2)$] \label{lem:5_actions_C1} Assume \Cref{case1}. Let $\UE(M)$ be as in \Cref{EMdef}.
	For $i=1,2$ set $\UE(G_i):=\spa{\gamma_i,F_p}\leq \Aut(G_i)$, and let 
	$\II{EMGi}@{\UE_M(G_i)}\leq \Aut(G_i)$ be the subgroup of automorphisms of $G_i$ induced by $\UE(M^\circ)$. Then 
	\[ \UE(G_1)=\UE_M(G_1) \und \UE(G_2)=\UE_M(G_2).\]
\end{lem}
\begin{proof} Recall that $n$ acts as $\gamma_1\gamma_2$ on $G_1.G_2$, see \Cref{lem:5_6}(a). Hence $\UE_M(G_1)=\spa{\gamma_1,F_p}$ and $\UE_M(G_2)=\spa{\gamma_2,F_p}$. This leads to the statement.\end{proof}

In the following, we verify some adaptations of the above statement to the missing case corresponding to the last line of Table~\ref{tab:my_label}.
\begin{hyp}\label{case2}
	Assume $\eps=-1$ and $\{|Z_1|,|Z_2|\}=\{4\}$. \end{hyp}
Recall $f$ denotes the integer with $q=p^f$.
\begin{rem}\label{rem:5:12} 
	If \Cref{case2} holds, then $2\nmid f$ and $(p-1)_2=2$, see \Cref{ZG_order}(c).
\end{rem}

\begin{lem} \label{lem_centZ1_c2} Assume \Cref{case2}. 
	\begin{thmlist}
		\item Then $n=\wh t_1 \neinszwei$ acts on $G_2$ as $\gamma_2$ and $n$ acts on $G_1$ as a concatenation of $\gamma_1$ and a diagonal automorphism associated to $h_0$ in the parametrization of \Cref{not_diag}.
		 
		\item The element $\wh t_1$ from \Cref{rem:5:2} can be chosen to be $\h_{e_1}(\zeta)$ for $\zeta$ some $2(p+1)_2 $-th root of $1$ in $\FFtimes$ and hence $[F_{p^2},\wh t_1]=1$. 
		
		\item \label{def_EM_c2} $\UE(M)$
is cyclic of order $2f$, $[n, \UE(M)]=\spannh$, and the groups $G_1$, $G_2$ and $\spa{n,h_0}$ are $\UE(M)$-stable.
		\item Recall $\wb M:=M \wb T$ and $\Ispezial{EMT}@{\UE(M \protect {\breve T})}@{\UE(\wb M )} \leq \Aut(\GvF \wb T)$ the subgroup obtained by restricting $\UE (\bM)$ to $\GvF \wb T$. 
		Then $\UE(\wb M)$ stabilizes $\wb T$, $M$, $G_i\spa{\wh t_i,\wt t_i}$ and $\wb M$.
		Analogously, $\UE(\wb M)\spa{n,h_0}$ stabilizes $\wb T M_0$.
	\end{thmlist}
\end{lem}
\begin{proof} We have $v=\neins$ and $n=\wh t_1 \neinszwei$ as defined by the last row of \Cref{tab:my_label} and the actions are clear from \Cref{rem51}(b). This gives (a). 
	
	According to \Cref{rem:5:12}, $(q-1)_2=(p-1)_2=2$ and $(p+1)_2=(q+1)_2$. So for $\zeta$ a $2(p+1)$-th root of $1$ in $\FFtimes$, $\h_{e_1}(\zeta)$ satisfies $[\h_{e_1}(\zeta), vF]=[\h_{e_1}(\zeta), \gamma F]= \h_{e_1}(\zeta^{-q} \zeta^{-1})=h_0$. Hence we can choose $\wh t_1=\h_{e_1}(\zeta)$. The order of $\zeta$ divides $ p^2-1$ and hence $[\wh t_1,F_p^2]=1$. This shows part (b). 
	
	We see that $\gamma(\wh t_1)=\h_{e_1}( \zeta^{-1})=\wh t_1^{-1}$ hence $(\wh t_1 \gamma)^2=1$. Furthermore, $F_p^2$ has odd order $f$ as an automorphism of $\bG^\vFq$, and commutes with $\wh t_1 \gamma_1$. We get that $\UE(M)$ is abelian, even cyclic of order $2f$ since $\wh t_1 \gamma$ has order $2$ by the above. So we get the first claim of (c).
	
	We have $[n,F_p^2]=1$ since $[\wh t_1,F_p^2] =[\neinszwei, F_p^2]=1$ by \Cref{rem:5:2}, and we see
	\[ [n, \wh t_1\gamma_1]= [\wh t_1 \neinszwei, \wh t_1 \gamma_1]= [\wh t_1 \neins, \wh t_1\gamma_1 ] [ \nzwei, \wh t_1\gamma_1 ] = h_0.\] 
	This leads to $[n,\UE(M)]=\spannh$, as claimed in (c). We see that $\UE(M)$ stabilises $\spa{n}$. For every $i\in \{1,2\}$, the groups $\bG_i$ and $G_i$ are $\spa{F_p^2,\wh t_1 \gamma_1}$-stable. 
	This completes the proof of part (c). 
	
	By definition $\UE(M)$ acts on $\bT_0$ and stabilises $\Z(\bG)$. This implies that $\UE(M)$ stabilizes $\GvF \wb T$, $M$ and $M\wb T$. This shows part (d).
\end{proof}

The following gathers results easily obtained from Lemmas \ref{lem_centZ1_c2} and \ref{ZG_order}.
\begin{table}[h!]
	\centering
	\begin{tabular}{|c|c||c|c||c|c|c|c|c|}
		\hline 
		$(\epsilon_1,\epsilon_2)$ & $\{|Z_1|,|Z_2|\}$& $v$&$n$& $[F_p^2,n]$&$[\wh t_1\gamma_1,n]$&$[F_p^2,v]$ & $[\wh t_1\gamma_1,v]$\\ \hline\hline
		$(-1, 1)$ & $\{4\}$ & $\neins$& $\wh t_1\neinszwei$& $1$ & $h_0$ & $1$ & $1$ \\ \hline 
	\end{tabular}
	\caption{Commutators of some elements in $M\UE(M)$}
	\label{tab:properties_c2}
\end{table}

\begin{lem} \label{lem_centZ2_C2} Assume \Cref{case2}.
	\begin{thmlist}
		\item $[\UE(M^\circ),\UE(M^\circ)]=\spannh$ and $\UE(M^\circ)\cap M^\circ = \spannh $;
		\item $\UE(M^\circ)/\spannh$ is abelian and $\Cent_{\UE(M^\circ)}(Z_1.Z_2)$ is abelian.
	\end{thmlist}
\end{lem}
\begin{proof} Recall that by \Cref{case2},  $\eps=-1$ and $\{|Z_1|,|Z_2|\}=\{4\}$, hence $2\nmid f$. The group $\UE(M)=\spa{F_p^2, \wh t_1 \gamma}$ is abelian according to \Cref{lem_centZ1_c2}(c). \Cref{tab:properties_c2} implies $[\UE(M),n]=\spannh$. 
	
	Recall $\UE(M)\cap M=1$. Using $\gamma(\wh t_1)=\wh t_1\inv$ from the proof of Lemma~\ref{def_EM_c2} and then \Cref{ZG_order}(d), we get % \begin{align} xx=& xx \cr xx=& xx \end{align}
		$ n^2= (\wh t_1 \neinszwei)^2= (\neinszwei)^2= h_0$. This gives
	 $\spa{n}\cap M^\circ=\spa{n^2}=\spannh$, whence the second part of (a). Together with the equality $[F_p^2,\wh t_1 \gamma]=1$, the results of \Cref{tab:properties_c2} lead to $[\UE(M^\circ),\UE(M^\circ)]=\spannh$, completing the missing part of (a). 
	
The above implies that $\UE(M^\circ)/\spannh$ is abelian. We have $[Z_1.Z_2,n]=[Z_1.Z_2,\gamma_1]=[Z_1.Z_2,n\gamma_1]=\spannh $.
Considering the action of $\UE(M^\circ)$ on $Z_1$ and $Z_2$ we obtain $\Cent_{\UE(M^\circ)}(Z_1)= \spa{h_0,F_p^2,n\gamma_1, \gamma_1F_p}$ and $\Cent_{\UE(M^\circ)}(Z_2)= \spa{h_0,F_p,\gamma_1}$. This leads to $\Cent_{\UE(M^\circ)}(Z_1.Z_2)=\spa{h_0,F_p\gamma_1, F_p^2}$ and this group is abelian.
\end{proof}
%%%%%%%%%%%%%%%%%%%%%%%%%%%%

\begin{lem}[Comparison of $\UE(M)$, $\UE(G_1)$ and $\UE(G_2)$] \label{lem:5_actions_C2}
	Assume \Cref{case2}. For $i=1,2$ set $\UE(G_i):=\spa{\gamma_i,F_p}\leq \Aut(G_i)$, and let $\II EuM@{\UE_M(G_i)}\leq \Aut(G_i)$ be the subgroup of automorphisms of $G_i$ induced by $\UE(M^\circ)$. Then $ \UE(G_2)=\UE_M(G_2)$. Moreover, $ \UE(G_1)$ and $\UE_M(G_1)$ are $\wt t_1$-conjugate in $\Out(G_1)$.
\end{lem}
\begin{proof} Note that in our case $\vFq =\gamma F_q=F$ as endomorphism of $\bG$. Since $n=\wh t_1\neinszwei$ acts as $\gamma_2$ on $G_2$, we see that $\UE(M^\circ)$ induces $\spa{\gamma, F_p^2}$ on it. But the latter has same action on $\GF$ as $\spa{\gamma, F_p} $ since $\gamma F_q$ acts trivially and $2\nmid f$ by \Cref{rem:5:12}. Therefore $\UE(G_2)=\UE_M(G_2)$, as claimed.
	Via the isomorphism between $\Out(\GF)$ and $\Z(\bG)_F\rtimes \spa{\gamma , F_p}$ using $\calL_F$ and described in \Cref{not_diag}, the group $\UE_M(G_1)$ corresponds to $\spa{h_0\gamma,F_p^2}$. Note that 
	\[ \spa{\gamma,F_p^2}^{\h_{\underline {l_1}}(\varpi)}= 
	\spa{\gamma^{\h_{\underline {l_1}}(\varpi)},F_p^2}= \spa{h_0\gamma,F_p^2}.\] This shows that $ \UE(G_1)$ and $\UE_M(G_1)$ are $ \wt t_1$-conjugate in $\Out(G_1)$ since $\h_{\underline {l_1}}(\varpi) = \calL_F(\wt t_1)$. \end{proof}

\subsection{The characters of \texorpdfstring{$M$}{M} and their Clifford theory}\label{ssec_5B}
%%%%%%%%%%%%%%%%%%%%%%%%%%%%%%%%%%%%%%%%%%%%%%%%%%%%%%%%%%
The aim of the remainder of this chapter is the following statement, later used in the proof of Condition \Ad{} and corresponding roughly to Theorem C of our Introduction. Recall that the groups $\wT$ from \Cref{lem:5_actions_wbT} and $\UE(M)$ from \Cref{EMdef} act on $M$. As expected we state that the actions of $\wt T M$ and $\UE(M)$ on $M$ are of ``transversal nature'', and \maex holds \wrt $M\unlhd M\UE(M)$ for enough characters of $M$. In the case when $\Cent_{M\UE(M)}(M)\neq \Z(M)$ (i.e. $\eps =1$ and $\eps_1=-1$) we provide a necessary refinement of that \maex for certain characters having $h_0$ in their kernel.

\begin{theorem}\label{thm_sec_Ad}
	Set $\II Mtilde@{\wt M:=M \wt T}$. 
	\begin{thmlist}
		\item Every $\wt M$-orbit in $\Irr(M)$ contains a character $\chi$ such that 
		\begin{asslist}
			\item $\starStab \wt M| \UE(M)| \chi$; and 
			\item $\chi$ extends to $M \UE(M)_\chi$.
		\end{asslist}
		Equivalently (see \cite[Lem.~2.4]{S21D1}), there exists some $\UE(M)$-stable $\wt M$-transversal $\TT(M)$ in $\Irr(M)$, such that \maex holds \wrt $M\unlhd M \UE(M)$ for $\TT(M)$.
		\item If $2\mid f$, $\eps=1$, $\eps_1=-1$, $\chi\in\TT(M)$ and $h_0\in\ker(\chi)$, then the extension of $\chi$ can be chosen to have $v(\underline{F}_p)^f$ in its kernel, where $\underline{F}_p$ is the image of $F_p$ in $\UE(M)$.
	\end{thmlist}
\end{theorem}

Recall $$M_0=G_1.G_2\ \unlhd\  M^\circ =(\bG_1.\bG_2)^{\vFq }=M_0\spa{\wh t_1\wh t_2}\ \unlhd\  M=M^\circ\spa{n}, $$
 with $M/M_0$ a non-cyclic group of order 4, see \Cref{lem:5_5}.

 The characters of $M$ are studied through the characters of $M_0$, since $M_0=G_1.G_2$, with $G_i\cong\tD_{l_i,\sico}^{\eps_i}(q)$ by \Cref{def_M} and therefore much information about $\Irr(G_i)$ is known from \Cref{thm_sumup_D}. Characters of $M$ are obtained from those of $M_0$ in the following way. Recall $\wb T=\calL_{\vFq }\inv(\Z(\bG))\cap \bT$ and $\wt T=(\bT\Z(\bG))^\vFq$ from \Cref{lem:5_actions_wbT}. 
\begin{lem}\label{not:5A}
	\begin{thmlist}
		\item Let $ {\phi\in \Irr(M_0)}$, $ {\wh \phi\in\Irr(M^\circ_\phi \mid \phi)}$ and $ {\chi_0}\in \Irr( M^\circ_\phi \spa{n}_{\wh \phi}\mid \wh\phi)$. If $M_\phi\neq M_0\spa{\wh t_1\wh t_2 n}$,
		then $\chi:=\chi_0^M$ is irreducible.
		\item \Maex holds \wrt $M_0\unlhd M_0 \wb T $ and $M_0\unlhd M_0 \wt T $. 
	\end{thmlist}
\end{lem}
Note that every character $\phi\in \Irr(M_0)$ with $M_\phi= M_0\spa{\wh t_1\wh t_2 n}$ has the $\widecheck T$-conjugate character $\phi':=\phi^{\wt t_1\wt t_2}$. The character $\phi'$ then satisfies 
\[ M_{\phi'}= (M_\phi)^{\wt t_1\wt t_2}=M_0 \spa{\wh t_1\wh t_2 n^{\wt t_1 \wt t_2}}= M_0\spa{n},\] 
since $[{\wt t_1 \wt t_2}, n]\in \wh t_1 \wh t_2 M_0 $. Hence part (a) of the statement does not apply to $\phi$ but applies to $\phi'$. 
\begin{proof}
	For part (a) note that the quotients $M^\circ/M_0$ and $\spa{n}_{\wh \phi} M^\circ_\phi/M^\circ_\phi$ are cyclic, see \Cref{lem:5_5}. Therefore, the characters $\wh \phi$ and $\chi_0$ are extensions of $\phi$ according to \Cref{ExtCrit}(a). One has $M^\circ_\phi\spa{n}_{\wh \phi}\leq M_{\wh \phi}$. Now if $M^\circ_\phi\spa{n}_{\wh \phi}=M_{\wh \phi}$, then $\chi$ is irreducible by Clifford theory (\ref{Cliff}). Clearly $M_{\wh \phi}\leq M_\phi$. 
 Recall that $M/M_0$ is a Klein $4$-group and the subgroups of $M$ containing $M_0$ are $M_0$, $M_0\spa{n}$, $M_0\spa{\wh t_1\wh t_2n}$, $M^\circ$ and $M$. The group $M_\phi$ is one of those groups. Whenever $M_\phi\geq M^\circ$ or $M_\phi=M_0$ the above construction leads to $M_{\wh \phi}=M^\circ_{\phi }\spa{n}_{\wh \phi }$. Otherwise $M_\phi= M_0\spa{n \wh t_1\wh t_2}$, and hence $M_{\wh \phi}\neq M^\circ_\phi\spa{n}_{\wh \phi}$ in this case. We see that $ \chi$ is irreducible unless $M_\phi=M_0\spa{\wh t_1\wh t_2 n}$. %If $M_\phi= M_0\spa{\wh t_1\wh t_2 n}$, then $\chi$ is the sum of two different characters. 
	
	Next, we consider the claim in part (b). Recall the definitions of $\wh t_1,\wh t_2,\wt t_1$ and $\wt t_2$ from \Cref{rem:5:2} and \Cref{lem:5_actions_wbT}. Set $\Ispezial{Gibreve}@{\protect{\breve G}_i}@{{\wb G}_i}:=\spa{\wh t_i,\wt t_i,G_i}$. When $l_i\geq 2$ this is consistent with \Cref{Gbreve} since $\calL_{\vFq }(\spa{\wh t_i,\wt t_i})=\Z(\bG_i)$, while when $l_i=1$ then ${\wb G}_i$ is abelian. So \Cref{max_ext_Gi_wGi} implies that \maex holds \wrt $G_i\unlhd {\wb G}_i$. On the other hand, the group $\wb T$ of \ref{lem:5_actions_wbT} satisfies \[\wb T\leq {\wb G}_1.{\wb G}_2\] since $\wb T=\spa{\wt t_1\wt t_2,\wh t_1,\wh t_2} (T \cap M_0)$ by \ref{lem:wtT}. Hence, maximal extendibility holds \wrt $G_1.G_2\unlhd {\wb G}_1.{\wb G}_2$ and $M_0=G_1.G_2\unlhd \wb T (G_1.G_2)=M^\circ \wb T$. As $({\wb G}_1.{\wb G}_2)/(G_1.G_2)$ is abelian, this implies \maex \wrt $M^\circ \unlhd M^\circ \wb T$, see \Cref{ExtCrit}(b). 
	
	Since $\wt T=(\bT\Z(\wbG))^{\vFq } \leq \wb T \Z(\wbG)$, we also get \maex \wrt $M_0\unlhd M_0 \wt T$ and \wrt $M^\circ \unlhd M^\circ \wt T$.
\end{proof}

Recall $M^\circ= (\bG_1.\bG_2)^{\vFq }= M_0 T$ for $\III{T}:=\bT^{\vFq }$, and \[\UE(M^\circ):= \UE(M) \spa{n,h_0}.\]
\begin{prop}\label{def_whEM} Let $\phi\in \Irr(M_0)$ and $\chi\in \Irr(M\mid \phi)$. If ${\wb T}_\phi\leq T \spa{\wh t_1}$ and $\chi^{\wh t_1}=\chi$, then $ \starStab \wt M| \wh M |\chi$ where $\wh M=M\UE(M)$. 
\end{prop}
\begin{proof} Recall $\wb M:=\wb T M$ and that $\UE(\wb M)$ stabilizes $\wb M$ according to \Cref{lem:5_6}(c) and \Cref{lem_centZ1_c2}(d). 
	The automorphisms of $M$ induced by $\wb T$ and $\wt T$ coincide, see \Cref{lem:5_actions_wbT}.
	The automorphisms of $M$ induced by $\UE(\wb M)$ and $\UE(M)$ coincide. Hence our claim is equivalent to the equality 
	\[ \starStab {\wb M}| {\UE}(\wb M)|\chi. \] 
	%since the automorphisms of $M$ induced by $\wb M$ and $\wt M$, as well as the automorphisms given by $\UE(M)$ and $\dot{E}(M)$ coincide.
	We have to study the group $(\wb M {\UE}(\wb M))_\chi$. We know that $(\wb M {\UE}(\wb M))/M $ is isomorphic to $\Z(\bG)\rtimes {\UE}(\wb M)$. Let $U$ be the subgroup of $\Z(\bG)\rtimes {\UE}(\wb M)$ corresponding to $(\wb M \UE(\wb M))_\chi/M$. 
To get our claim it suffices to ensure $U= (U\cap \Z(\bG)) \rtimes (U\cap \UE(\wb M))$ by showing 
 \begin{align}\label{eq52}
		\spannh\leq U \leq \spannh \rtimes \UE(\wb M).
	\end{align} 
	The assumption $\chi^{\wh t_1}=\chi$ with $\calL_{\vFq }(\wh t_1)=h_0$ already implies the inclusion \[ \spannh \leq U .\] 
	
	Next, we check $U \leq \spannh \rtimes \UE(\wb M)$. If $\phi_1\in \Irr(G_1)$ and $\phi_2\in \Irr(G_2)$ with $\phi=\phieinszwei$, then $({\wb G}_i)_{\phi_i}\leq G_i\spa{\wh t_i}$ for some $i\in\{1,2\}$ because of ${\wb T}_\phi\leq T\spa{\wh t_1}$, where $T=\bT\cap M^\circ $. Without loss of generality, we can assume $ ({\wb G}_1)_{\phi}\leq G_1\spa{\wh t_1}$. According to \Cref{cor:3_24}, this implies 
	\[({\wb G}_1 \rtimes \UE({\wb G}_1))_\chi\leq (G_1\spa{\wh t_1}) \rtimes \UE({\wb G}_1).\] 
	
	Recall from \Cref{EMdef} that $\UE(\wb M)$ also defines a group of automorphisms of $\GvF \wb T$ and $ \UE({\wb M}^\circ):=\UE(\wb M)\spa{n,h_0}$ stabilizes ${\wb M}^\circ=\wb T M^\circ$. Then $\phi$ satisfies 
	\[ (M_0\wb T \UE({\wb M}^\circ) )_\phi \leq M_0 \spa{\wh t_1,\wh t_2} \UE({\wb M}^\circ).\]
	Since $\chi\in \Irr(M\mid \phi)$, this shows 
	\[ M(\wb T \UE(\wb M))_\chi\leq M(\wb T \UE({\wb M}^\circ))_\phi\leq (M \spa{\wh t_1})\, \UE(\wb M).\]
	The latter corresponds to $\spannh\rtimes \UE(\wb M)$ under the isomorphism $(\wb M {\UE}(\wb M))/M \cong \Z(\bG)\rtimes {\UE}(\wb M)$, so we get (\ref{eq52})
	and accordingly
	\begin{align}
		\starStabalign \wb M | \UE(\wb M)|\chi.\qedhere
	\end{align} 
\end{proof}
 For the proof of \Cref{thm_sec_Ad}, we finish describing the characters of $M$ via the ones of $M_0$. Notice that $\UE(M^\circ)$ stabilizes $G_1$ and $G_2$.
%\newpage
\begin{prop}\label{prop7_2}
Let $\phi\in \Irr(M_0)$, $\wh \phi\in\Irr(M^\circ_\phi\mid \phi)$ and $E\leq \UE(M^\circ)_{\wh \phi}$. 
Let $E'_i\leq \Aut(G_i)$ be the group of automorphisms induced by $E$. Assume that for each $i\in\{1,2\}$, $\phi_i\in \Irr(\restr \phi|{G_i})$ extends to $M_0\rtimes E'_i$. 
	\begin{thmlist}
		\item Then $\phi$ extends to $M_0E $. 
		\item If $v\in\bG\setminus\{1_\bG\}$ (or equivalently $\eps_1=\eps_2=-1$, see Table~\ref{tab:my_label}), $2\mid f$ and $h_0\in\ker(\phi)$, then $\phi$ has an extension $\wt \phi$ to $M_0E$ with $vF_q \in\ker(\wt \phi)$.
	\end{thmlist}
\end{prop}
\begin{proof}
Note first that $M_0 E'$ is a subgroup of $(G_1E'_1)\times (G_2 E'_2)$ where $E'\leq \Aut(M_0)$ is associated with $E$ and hence
$\phi$ extends to $M_0 E'$.

Set $Z_0:=\Z(M_0)$ and 
%hence $Z_0=Z_1.Z_2$ if $1\notin\{l_1,l_2\}$. L
let $\{\eta\}= \Irr(\restr \phi|{Z_0})= \Irr(\restr \wh\phi|{Z_0})$. As $\phi$ is $E$-invariant and hence $E'$-invariant, $\eta$ is $E'$-invariant as well.

Assume first that $\eta$ extends to $Z_0 E$. 

We deduce that $\eta$ extends to $Z_0 E'$, since $\eta$ has degree $1$ and $Z_0 E'=Z_0\rtimes E'$, see \Cref{ExtCrit}(d). 
	We also observe  
\[  \Cent_{M_0\UE(M^\circ)}(M_0)= \begin{cases}
    Z_0\spa{v F_q }& \text{if }(\eps_1,\eps_2)=(-1,-1),\\Z_0& \otw 
\end{cases}
\]
and hence $\Cent_{M_0  E'}(M_0)=Z_0$, since $v F_q \in E$ corresponds to the trivial element in $E'$ whenever $(\eps_1,\eps_2)=(-1,-1)$. This leads to 
	\[ (M_0 E',M_0,\phi)\geq_c(Z_0 E',Z_0,\eta),\]
	by \Cref{rem_chartrip}(a). \Cref{Butterfly} then implies 
	\[ \left( M_0 E ,M_0,\phi\right)\geq_c 	\left( Z_0 E,Z_0,\eta \right ).\]
Since $\eta$ extends to $Z_0E$, \Cref{rem_chartrip}(b) implies that $\phi$ extends to $M_0E$. So we get (a) whenever $\eta$ extends to $Z_0E$.
	
For part (b) we already note the following. If $\wt \eta$ is an extension of $\eta$ to $Z_0E$, then \Cref{rem_chartrip}(b) again implies that some extension $\wt \phi$ of $\phi$ to $M_0E$ satisfies
\begin{align}\label{eq5_13}
		\Irr(\restr \wt \phi|{\Cent_{Z_0E}(M_0)})&=\{\restr \wt \eta|{\Cent_{Z_0E}(M_0)}\}.
	\end{align} 
	
Let us now check that $\eta$ extends to $Z_0E$. 

Assume first that $h_0\in \ker(\phi)$. Then $\spannh\leq\ker(\eta)$. According to \Cref{lem_centZ2} and \Cref{lem_centZ1_c2}, $[E,E]\cap M^\circ \leq [\UE(M^\circ),\UE(M^\circ)]\cap M^\circ\leq \spannh\leq \ker(\phi)$. \Cref{ExtCrit}(d) then implies that $\eta$ extends to $Z_0 E$.
 
We are left to consider the case where $h_0\notin \ker(\phi)$. Set $Z:=\spa{h_0,\h_{J_1}(\varpi),\h_{J_2}(\varpi)}^{\vFq }$ and $\wt \eta\in \Irr(\restr \wh \phi|{\wt Z})$. We observe that $\eta$ and $\wt \eta$ are $E$-invariant. On the other hand, computations similar to \ref{ZG_order}(b) show $[\UE(M^\circ),Z]=\spannh$. As $\eta$ is $E$-invariant and $h_0\notin\ker(\eta)$ this implies $E\leq \UE(M^\circ)_{\wt \eta}\leq \Cent_{\UE(M^\circ)}(Z)$. According to \Cref{lem_centZ2} and \Cref{lem_centZ2_C2} the Sylow $2$-subgroup of $\Cent_{\UE(M^\circ)}(Z)$ and hence of $E$ is abelian. This shows that $\eta$ has an extension to the Sylow $2$-subgroup of $Z_0 E$. Additionally observe that  $E(M^\circ)/Z_0$ has a cyclic Hall $2'$-subgroup and therefore $\eta$ extends to $Z_0 E$, see \cite[Thm~11.32]{Isa}. Hence our claim in all cases and this finishes the proof of (a).
	
Next we ensure the statement in part (b). Here we assume $v\in \bG\setminus\{1_\bG\}$ and $2\mid f$, hence $(\eps_1,\eps_2)=(-1,-1)$. We observe that $v F_q \in \Cent_{\GvF \UE(M)}(M_0)$. By assumption $h_0\in\ker(\phi)$ and therefore $h_0\in\ker(\eta)$. Now $\eta$ has an extension $\wt \eta$ to $Z_0E$ that can be assumed to have $vF_q$ in its kernel since $\spa{v F_q }\cap Z_0\leq \spannh\leq \ker(\eta)$. According to \eqref{eq5_13}, there exists an extension $\wt \phi$ of $\phi$ to $M_0 E_\phi$ with $v F_q \in\ker(\wt \phi)$, as required in (b). 
\end{proof}

\subsection{Characters of \texorpdfstring{$G_1$}{G1} and \texorpdfstring{$G_2$}{G2}}\label{ssec5D}
In order to prove \Cref{thm_sec_Ad} we split $\Irr(M)$ according to its constituents after restriction to $M_0=G_1.G_2$. In particular, we use the sets defined in \Cref{not:3_16} for the subgroups $G_i$ of type D possibly of rank $\leq 3$ in the sense of \Cref{Def_low_l}. \Cref{prop_char_G1G2} enumerates in a simplified fashion the cases that will be considered in the following sections.

 Recall that the group $\UE(M)$ and the associated group $\UE(M^\circ)$ from \Cref{EMdef} stabilize $G_1 $ and $G_2$, see \Cref{lem:5_6}(b) and \Cref{lem_centZ1_c2}(d). 

Recall $\UE({\wb M}^\circ)=\UE(\wb M )\spa{n ,h_0}$ and $\Ispezial{Gi}@{\protect{\breve G} _i}@{{\wb G}_i} =\spa{G_i,\wh t_i,\wt t_i}$. 
The group $\UE({\wb M}^\circ)$ also stabilizes ${\wb G}_1$ and ${\wb G}_2$, see \Cref{lem:5_6}(c) and \Cref{lem_centZ1_c2}(d). 

\begin{defi}\label{def:TEDi}
Let $i\in \{1,2\}$ and write $\Ispezial{EMGbi}@{\UE_M( \protect{\breve G}_i}@ {\UE_M({\wb G}_i)}\leq \Aut({\wb G}_i)$ for the group of automorphisms of ${\wb G}_i$ induced by $\UE({\wb M}^\circ)$. Let $\ov \TT_i$, $\EE_i$ and $\DD_i$ be the following subsets of $\Irr(G_i)$:
\begin{align*}
\II{Ti}@{\protect{\oTT_i}}& :=&&\left \{\chi \, \left| \, 
\starStabkla {\wb G}_i | {\UE}_M({\wb G}_i) | {\chi} \und
\chi\text{ extends to }G_i \UE_M(G_i)_{\chi} \right. \right\} ,\\
\II{Ei}@{\EE_i} &:=&&\left \{\chi\, \left | \, \starStabkla {{\wb G}_i} |{\UE}_M({\wb G}_i) |{\chi} \und \chi \text { has no extension to }G_i \UE_M(G_i)_{\chi} \right. 
\right\} &\und\\
\II{Di}@{\DD_i} &:=&& \left \{\chi \, \left| \starStabklaneq{{\wb G}_i} |{\UE}_M({\wb G}_i) |{\chi} \right.\right\}, &
\end{align*}
so that 
\begin{align*}
	 \Irr(G_i)= \ov \TT_i \sqcup \EE_i\sqcup \DD_i.
\end{align*}

We also define $ \II EE'i@{\protect{\EE'_i}} =\ov\TT_i\cap \big( \bigcup_{x\in{\wb G}_i}{}^x\EE_i \big)$ and $ \II DD'i@{\protect{\DD}'_i}=\ov\TT_i\cap \big( \bigcup_{x\in{\wb G}_i}{}^x\DD_i \big)$.
\end{defi}
The set ${\EE'_i}$ is the subset of all characters in $\ov \TT_i$, which are ${\wb G}_i$-conjugate to an element of $\EE_i$. The set ${\DD'_i}$ is the subset of all characters in $\ov \TT_i$, which are ${\wb G}_i$-conjugate to an element of $\DD_i$. The definitions of the above character sets are analogous to those of \Cref{def_TED} (and \Cref{not:3_16}). Accordingly, the characters satisfy the following statement. 
\begin{prop}\label{cor:6_11}
Let $i\in\{1,2\}$. Then the character sets of \Cref{def:TEDi} satisfy: 
\begin{thmlist} 
%\item $\Irr(G_i)=\oTT_i\sqcup \EE_i\sqcup \DD_i$;
\item $\oTT_i$ contains some $\wt T$-transversal in $\Irr(G_i)$.
\item If $\chi\in\EE_i\cup \EE'_i$, then $h_0\in\ker(\chi)$, $({\wb G}_i)_\chi=G_i\spa{\wh t_i }$, $\chi^{n}=\chi$ and $\chi$ extends to $G_i\spa{\wh t_i,\gamma_i}$.
\item If $\chi\in\DD_i\cup \DD'_i$, then $h_0\in\ker(\chi)$, $({\wb G}_i)_\chi=G_i$ and every $\wh \chi\in \Irr(\spa{G_i,\wh t_i}\mid \chi)$ is $n$-invariant. More precisely $\chi^n=\chi$ if $\chi\in\DD_i'$ and $\chi^{n\wh t_i}=\chi$ if $\chi\in\DD_i$.

Moreover, $\UE_M(G_i)_{\chi}$ is cyclic whenever $\chi\in\DD_i$. If $\chi\in \DD'_i$, then $(M\UE(M))_\chi\leq M_0\UE(M^\circ)_\chi$.

\item If $|Z_i|=2$, then $\DD_i=\EE_i=\emptyset$ and $\oTT_i=\Irr(G_i)$.
\item If $\eps_i=-1$, then $\EE_i=\emptyset$. If $2\nmid f$, then $\EE_1=\EE_2=\emptyset$.
\item If $\chi\in \Irr(G_i)\setminus \EE_i$, then $\chi $ extends to $G_i \UE_M(G_i)_\chi$. If $\chi\in \EE_i$, then $\chi$ extends to $G_i E''$ for any $E''\leq \UE_M(G_i)_{\chi}$ with an even index $| \UE_M(G_i)_\chi/E''|$.
\end{thmlist}
\end{prop}
\begin{proof} The groups $\UE_M(G_i)$ and $\UE(G_i)$ induce the same subgroup of $\Aut(G_i)$ for every $i\in\{1,2\}$ unless $i=2$ and \Cref{case2} holds, see \Cref{lem:5_actions_C1} and \Cref{lem:5_actions_C2}. In those cases, the sets obtained in \Cref{def_TED} (and \Cref{not:3_16}) for $\Irr(G_i)$ coincide with the present $\oTT_i$, $\EE_i$, $\DD_i$ via the isomorphism between $G_i$ and $\tD^{\eps_i}_{l_i, \text{sc}}(q)$ from \Cref{def_M}. If \Cref{case2} holds, then $\UE(G_1)$ and $\UE_M(G_1)$ are $\wt t_1$-conjugate in $\Out(G_1)$, so the sets from \Cref{def_TED} for $\Irr(G_i)$ are $\wt t_1$-conjugate to the ones from \Cref{def:TEDi}. 
Recall that in the case of (b) we can assume that $\eps_1=\eps_2=1$ and hence $n$ acts as $\gamma_i$ on $G_i$. In case of part (c), $n$ acts either as $\gamma_i$ or $\gamma_i\wh t_i$, but it always corresponds to an automorphism of $\UE_M(G_i)$ by construction.  
This implies that parts (a), (b) and the two first sentences of (c) follow from \Cref{thm_sumup_D} and \Cref{lem3:7}(a). 

We now show that $\UE_M(G_i)_\chi$ is cyclic whenever $\chi\in\DD_i$.
Let $\chi'\in \DD_i'$ be in the ${\wb G}_i$-orbit of $\chi$. Recall that this ${\wb G}_i$-orbit is of length $4$. Via the identification  
\[{\wb G}_i \UE_M({\wb G}_i)/ (G_i \Cent_{{\wb G}_i\UE_M({\wb G}_i)}(G_i)) \cong \Z(\GF)\rtimes \UE_M(G_i),\] 
the group $({\wb G}_i \UE_M({\wb G}_i))_{\chi'}$ then corresponds to a subgroup of $\UE_M(G_i)$ and $({\wb G}_i \UE_M(G_i))_{\chi}$ corresponds to a subgroup of the form $\UE_M(G_i)^z$ for some $z\in \Z(\GF)\setminus \spannh =\Z(\GF)\setminus \Cent_{\Z(\GF)}(\UE_M(G_i))$. The group $\UE_M(G_i)_\chi$ corresponds then to a subgroup of $ \UE_M(G_i)\cap \UE_M(G_i)^{z}$ which is cyclic by a straight-forward discussion. 

If $\chi\in\DD'_i\subseteq \oTT_i$, then the inclusion $(M\UE(M))_\chi=(M^\circ \UE(M^\circ))_\chi \leq M_0\UE(M^\circ)_\chi$ follows from the fact that $M^\circ$ acts by diagonal automorphisms on $G_i$ and $\UE(M^\circ)$ acts by $\UE_M(G_i)$ on $G_i$. Recall that $\chi$ is in this situation not stable under any non-inner diagonal automorphism.

For the proof of (d), assume $|Z_i|=2$. Then $Z_i \UE_M(G_i)$ coincides with $Z_i\spa{\gamma_i,F_p}$ and is abelian. Then (a) implies $\ov \TT_i=\Irr(G_i)$ and $\DD_i=\EE_i=\emptyset$ as also explained in (4) of the proof of \Cref{thm_sumup_D}. 

For part (e), assume first \Cref{case1}. Then $\UE_M(G_i)$ acts by $\spa{\gamma_i,\restr F_p|{\GvF}}$, which is cyclic whenever $2\nmid f$ or $\eps_i=-1$. This then implies that every character of $G_i$ extends to its stabilizer in $G_i\UE_M({ G}_i)$ by \Cref{ExtCrit}(a), and we get $\EE_i=\emptyset$. This is (e) in that case. 

Assume now that \Cref{case2} holds and therefore $2\nmid f$. Then $\UE_M(G_1)=\spa{\wh t_1\gamma,F_p^2}$ and $\UE_M(G_2)=\spa{\gamma_2,F_p}$ both act cyclically on $\GvF$. Then $\EE_1=\EE_2=\emptyset$ again as above. 

Part (f) is clear for $\chi\in\oTT_i$. If $\chi\in\DD_i$ the group $\UE_M(G_i)_\chi$ is cyclic by part (c). Otherwise if $\chi\in\EE_i$, then the group $E''$ is either cyclic or contains $\gamma_i$ since $\UE_M(G_i)/\spa{\gamma}$ is cyclic. In the latter case, let $F'\in \UE_M(G_i)_\chi$ be such that $\UE_M(G_i)_\chi=\spa{F',\gamma_i}$ and $E''\leq \spa{(F')^2,\gamma_i}$. Every extension $\wt \chi$ of $\chi$ to $G_i\spa{\gamma_i}$ is not $F'$-invariant by the definition of $\EE_i$, but since $F'$ can only permute the two extensions of $\chi$ we see that $\wt \chi$ is $(F')^2$-invariant and hence $\chi$ extends to $G_i E''$. 
\end{proof}
Because of the central product structure $M_0=G_1.G_2$ with $G_1\cap G_2=\spannh$ (see \Cref{rem51}), any subsets $\mathbb G_1\subseteq \Irr(G_1)$ and $\mathbb G_2\subseteq \Irr(G_2)$ define 
\[ \II G1G2@{\mathbb G_1.\mathbb G_2}:=\{ \chi_1.\chi_2\mid \chi_i \in \mathbb G_i \text{ with } 
\Irr(\restr\chi_1|{\spannh})=\Irr(\restr\chi_2|{\spannh}) \}\subseteq \Irr(M_0).\]

Recall $\wbrevT:=\calL\inv_{\vFq }(\Z(\bG))\cap \bT$ and $\UE(M^\circ)=\UE(M)\spa{n,h_0}$ by the definition in \Cref{lem_centZ2} and \Cref{lem_centZ2_C2}. The following statement helps us split the proof of \Cref{thm_sec_Ad} into cases. 
\begin{prop}[Properties of $\wt T$-orbits in $\Irr(M_0)$]\label{prop_char_G1G2}
Let $\phi'\in\Irr(M_0)$. There exists some $t\in \wt T$ such that $\phi=(\phi')^t$ satisfies one of the following properties: 
\begin{asslist}
\item $\phi \in\ov \TT_1.\ov \TT_2$; or
\item $\phi\in {\EE'}_1.\EE_2\sqcup \DD_1'.\DD_2\sqcup \DD_1'.\EE_2\sqcup \EE_1.\DD_2'$ with $\{|Z_1|,|Z_2|\}=\{4\}$. 
\end{asslist}
\end{prop}
\begin{proof} 
Recall that the action of $\wb T$ on $\GvF$ coincides with the action of $\wt T$. Set $\phi'_i\in\Irr(G_i)$ such that $\phi'=\phi'_1.\phi'_2$. Recall $ \Irr(G_i)=\ov \TT_ i \sqcup \EE_i \sqcup \DD_i$ , see \Cref{def:TEDi}, with $\EE_i=\DD_i=\emptyset$ whenever $|Z_i|\neq 4$ as recalled in the preceding proof. %By the action of $\wt T$ on $G_1.G_2$ we observe that $(\phi'_1)^x\in \ov \TT_1$ for some $x\in \wt T$ according to \Cref{cor:6_11}(a). The character $(\phi')^x=(\phi'_1.\phi'_2)^x$ is of the form $\phieinszwei$ with $\phi_1\in\ov\TT_1$ and $\phi_2\in\Irr(G_2)$. Now $\phi_2\in\Irr(G_2)=\ov\TT_2\sqcup \EE_2\sqcup\DD_2$. If $\phi_2\in\ov\TT_2$, then $\phieinszwei\in\ov\TT_1.\ov\TT_2$, as in (i). 

Assume $\phi_1'$ and all its $\wt T$-conjugates belong to $\ov \TT_1$. In this case, \Cref{cor:6_11}(a) allows us to choose $x\in \wt T$ so that $(\phi'_2)^x\in \ov \TT_2$ additionally, and therefore $\phieinszwei:=(\phi'_1.\phi'_2)^x\in\ov\TT_1.\ov\TT_2$, as claimed in (i). 

If we are not in case (i), we see then that there is a $\wt T$-conjugate $(\phi_1,\phi_2)$ of $(\phi_1',\phi_2')$ in $(\EE'_1\sqcup \DD'_1)\times (\EE_2\sqcup \DD_2)$, also forcing $|Z_1|=|Z_2|=4$. This gives (ii) up to the case $(\phi_1,\phi_2) \in \EE_1'\times \DD_2$. 
 \Cref{cor:6_11} implies $\EE_1'=\EE_1^{\wt t_1}$ and $\DD_2^{\wt t_2}=\DD_2'$. Hence if $\phi_1.\phi_2 \in \EE_1'.\DD_2$, a $\wb T$-conjugate of $\phi_1.\phi_2$ is contained in  $\EE_1.\DD'_2$. But this is also a $\wt T$-conjugate, so we are in the case (ii).
\end{proof}

The above case (i) implies a familiar transversality of stabilizers.

\begin{lem}\label{eq_star_oT} 
	Every $\phi\in\oTT_1. \oTT_2 $ satisfies $\starStab 
	\wt T| \UE(M^\circ) |\phi $.
\end{lem}
\begin{proof} For $i\in \underline{ 2}$, let $\phi_i\in \Irr(G_i)$ with $\phi=\phi_1.\phi_2$. 
%	Let us denote by $F_p^{(i)}$ the automorphism of $M_0=G_1.G_2$, that acts trivially on $G_{3-i}$ and as $F_p$ on $G_i$. 
Recall that $\oTT_i$ was defined using $\UE_M({\wb G}_i)$ which acts as $\UE_M(G_i)$ on $G_i$, hence $\phi_i\in\Irr(\restr \phi|{G_i})$ satisfies $\starStabkla {\wb G}_i| \UE_M({\wb G}_i)|{\phi_i}$.% and additionally $\phi_i$ extends to $G_i\UE_M(G_i)_{\phi_i}$. 
	
	In order to verify the statement, we have to prove that for every $\wt g\in \wb T$ and $e\in \UE(M^\circ)$ with $\phi^{\wt g e}= \phi$, one has $\phi^{\wt g }= \phi^{ e}= \phi$ in the first place. 
	
	Since $\wb T\leq T\spa{\wt t_1,\wt t_2,\wh t_1,\wh t_2}$ and $\UE(M^\circ)$ acts as $\UE_M(G_i)$ on $G_i$, one may find $\wt g_i\in {\wb G}_i$, $e_i\in \UE_M(G_i)$ such that $\wt g=\wt g_1\wt g_2$ and $e$ acts as $e_1e_2$ on $M_0$. Then $(\phi_1\phi_2)^{\wt g e}= \phi_1\phi_2$ implies $\phi_i^{\wt g_ie_i}=\phi_i$ for $i=1,2$ by restriction.
	From the definition of the sets $\oTT_i$ we now get $\phi_i^{\wt g_i}=\phi_i^{e_i}=\phi_i$ and therefore our claim that $\phi^{\wt g }=\phi= \phi^{ e_1e_2}=\phi^{ e}$.
\end{proof}%%%%%%%%%%%%%%%%%%%%%%%%%%%%%

\Cref{prop_char_G1G2} above clearly leads to checking \Cref{thm_sec_Ad} by considering successively $\wt T$-orbits in $\Irr(M)$ containing an element in $\Irr(M\mid \oTT_1.\oTT_2)$, $\Irr(M\mid \EE'_1.\EE_2)$ or $\Irr(M\mid \DD'_1.\DD_2\sqcup\DD'_1.\EE_2\sqcup\EE_1.\DD_2')$, see the next two sections. 

Beforehand we have nevertheless to prove a statement about maximal extendibility. Recall ${\wt M:=M \wt T}$. 
\begin{prop} \label{prop_maxext_M_wtM}
Maximal extendibility holds \wrt $M\unlhd\wt M=\wt T M$.
\end{prop}
\begin{proof} Let $\chi\in\Irr(M)$. Assume first that ${\wb T}_\chi/T$ is cyclic. Then since $\wt T \leq \wb T \Z(\wbG)$ by the definition of $\wb T$ in \Cref{lem:5_actions_wbT} and therefore $\wt M\leq \wb M \Z(\wbG)$, we have that $\wt M_\chi /M(\wt M\cap \Z(\wbG))$ is cyclic. Then $\chi$ extends to $\wt M_\chi$ by \Cref{ExtCrit}(a).
	
	We now assume that ${\wb T}_\chi/T$ is not cyclic. Then ${\wb T}_\chi =\wb T$ and 
we can assume $\chi^{\wt t_1\wt t_2}=\chi^{\wh t_1}=\chi$. If $\phi\in\Irr(\restr \chi|{M_0} )$, then $\phi^{\wt t_1\wt t_2}$ is in the $M$-orbit of $\phi$ since $\phi^{\wt t_1\wt t_2}$ also belongs to $\Irr(\restr \chi|{M_0} )$. %Let $x\in M$ such that $\phi^x=\phi^{\wt t_1\wt t_2}$. 
For $\phi_i\in\Irr(G_i)$ with $\phi=\phi_1.\phi_2$, we see that the $M$-orbit of $\phi_i$ is $\wt t_i$-stable. 
Therefore, the $\spa{n, \wh t_i}$-orbit of $\phi_i$ is $\wt t_i$-stable. On the other hand, \Cref{cor:6_11}(b) and (c) implies that every $\spa{n,\wh t_i}$-orbit containing some $\psi' \in \DD_i\sqcup\DD'_i\sqcup \EE_i\sqcup \EE'_i$ is not $\wt t_i$-stable. Accordingly $\phi_i\in \oTT_i\setminus (\EE'_i\cup \DD'_i)$. %By the definition of $\oTT_i$, the character $\phi_i$ is either $\wh t_i\wt t_i$- or $\wt t_i$-invariant. 
Taken together, we see that $\phi\in\oTT_1.\oTT_2$ and therefore  $$\starStab
\wt T| \UE(M^\circ) |\phi $$ by \Cref{eq_star_oT}.   In particular $M_\phi\neq M_0 \spa{\wh t_1\wh t_2n}$ and \Cref{not:5A}(a) then ensures that $\chi=\wt \phi^M$, where $\wt \phi$ is an extension of $\phi$ to $M_{\wh \phi}$ for some extension $\wh \phi\in\Irr(M^\circ_\phi)$ of $\phi$. \Cref{not:5A}(b) tells us that maximal extendibility holds \wrt $M_0\lhd M_0 \wt T$, so the character $\wt \phi$ extends to $M_0 \wt T_\phi$. According to \Cref{ExtCrit}(e), there exists an extension $\psi$ of $\wt \phi$ to $\wt T_{\wt \phi} M_{\wh \phi}$. The character $\psi^{M \wt T_{\wh \phi}}$ is by this construction an extension of $\chi$ to $M \wt T_\chi$ as required. \end{proof} 

\subsection{Proof of \Cref{thm_sec_Ad}. Characters in \texorpdfstring{$\Irr(M\mid \oTT_1.\oTT_2)$}{Irr(M|T1.T2)}}\label{ssec5E}
In this section we begin the verification of \Cref{thm_sec_Ad} through studying $\Irr(M)$ as a union of sets $\Irr(M\mid \calX)$ for the various subsets $\calX\subseteq \Irr(M_0)$ singled out in \Cref{prop_char_G1G2}. Here we start with $\Irr(M\mid \oTT_1.\oTT_2)$. 

Recall $\wb T:=\calL_{\vFq }^{-1}(\Z(\bG))\cap\bT$, ${\wb M}^\circ:=\wb T M^\circ$,
$\II EM0@{\UE(M^\circ)}:=\UE(M)\spa{n, h_0}$, $\UE({\wb M}^\circ):=\UE(\wb M) \spa{n,h_0}$, $\Ispezial Mhat@{\hat M}@{\wh M:=M \UE(M)}$ and  $\II Mhat0@{\wh M_0}:=M_0 \UE(M^\circ)$.

In our situation, \Cref{eq_star_oT} leads to the following.
\begin{lem} \label{Cor5_7}
Let $\phi\in\Irr(M_0)$ and $\wh \phi\in\Irr(M^\circ_\phi\mid \phi)$. Assume:
\begin{asslist}
\item $\phi$ satisfies $\starStab {\wb T}| \UE({\wb M}^\circ) |{\phi}$ or equivalently $\starStabkla  \wh M_0 |{\wt T}| {\phi}$;
\item $\phi_i\in\Irr(\restr \phi|{G_i})$ extends to $M_0 \UE_M(G_i)_{\phi_i}$ for every $i\in\{1,2\}$.
\end{asslist}
Then: 
\begin{thmlist}
    \item Every $\chi\in\Irr(M\mid \phi)$ satisfies $ (\wt M \wh M)_\chi = \wt M_\chi \wh M_{\chi}$ and extends to $\wh M_\chi$. 
     \item Every $\kappa\in\Irr(M^\circ\mid \phi)$ satisfies $ (\wt M \wh M)_\kappa = \wt M_\kappa \wh M_{\kappa}$ and extends to $\wh M_\kappa$. 
    \item If in addition $\eps_1=\eps_2=-1$ (or equivalently $v\in \bG\setminus\{1_\bG\}$) and $2 \mid f$ (hence $\{|Z_1|,|Z_2|\}=\{2\}$ according to \Cref{ZG_order}(c)), then every $\chi\in \Irr(M\mid \phi)$ with $h_0\in \ker(\chi)$ has an extension $\wt \chi$ to $\wh M_\chi =M \UE(M)_\chi$ with $vF_q \in \ker(\wt \chi)$.
\end{thmlist}
\end{lem}
\begin{proof} 
We apply \Cref{lem_stab_extensions} with $A:=\wt T \wh M$, $X=M_0$, $\wt X:=\wt T M_0$, $Y:=\wh M_0$ and $L:=M$ and $\phi$ as the character. We know that maximal extendibility holds \wrt $M_0 \unlhd M_0 \wt T$ according to \Cref{not:5A}(b) and therefore $\phi$ extends to $ M_0\wt T_\phi$. 

The quotient $\wt T M_0/M_0=\wt T/ (\wt T\cap M_0)$ is abelian. The assumption (ii) implies that $\phi$ extends to $(\wh M_0)_{\whphi}$ for $\wh\phi\in\Irr(M^\circ_\phi \mid \phi)$, see \Cref{prop7_2}. The group $\wh M_0/M_0=  (M_0 \UE(M^\circ))/M_0$ is isomorphic to a subgroup of $ \UE(M^\circ)/ (\UE(M^\circ)\cap M^\circ)$ and  $\UE(M^\circ)/(\UE(M^\circ)\cap M^\circ)$ is abelian as $[\UE(M^\circ), \UE(M^\circ)]\leq M^\circ$ according to \Cref{lem_centZ2} and \Cref{lem_centZ2_C2}, respectively.

In combination with the assumption (i), we see that all assumptions of  \Cref{lem_stab_extensions} are satisfied and the statements in (a) and (b) follow then from (b) and (c) of \Cref{lem_stab_extensions}.

It remains to consider the case where $2\mid f$, $\eps_1=\eps_2=-1$ and $h_0\in \ker(\chi)$. As seen before, \Cref{prop7_2}(b) tells us that $\phi$ has an extension to $(\wh M_0)_{\whphi}$ with $\spa{v F_q }$ in its kernel. Then the statement in (c) follows from \Cref{lem_stab_extensions}(d) with $z=v F_q $.
\end{proof}
In a first application of this statement we verify \Cref{thm_sec_Ad} for $\wt M$-orbits containing some character in $\Irr(M\mid \oTT_1.\oTT_2)$. Part (b) of the following statement is useful for the proof of Condition \Bd{}.
%, the existence of an extension map is required. In its construction, we use the following statement. 

\begin{prop}[Stabilizer and extensions of $\Irr(M\mid \ov\TT_1. \ov\TT_2 )$]\label{prop:TT}
Let $\phi \in \ov\TT_1.\ov\TT_2$.
\begin{thmlist}
\item Then every character $\chi\in\Irr(M\mid \phi)$ satisfies 
\begin{asslist}
	\item $\starStab \wt M |\wh M|\chi$; and 
	\item $\chi$ extends to some $\wt \chi\in\Irr(\wh M_\chi )$ such that $vF_q \in\ker(\wt \chi)$ whenever $v\in \bG\setminus\{1_\bG\}$ and $h_0\in\ker(\chi)$.
\end{asslist} 
\item \label{cor_maxext_eps_1}
Maximal extendibility holds \wrt $ M^\circ\unlhd \wh M=M \UE(M)$ for $\Irr(M^\circ\mid \oTT_1.\oTT_2)$. 
\item If $2\mid f$, $\eps_1=\eps_2=-1$ and $\chi\in\Irr(M\mid \oTT_1.\oTT_2)$ with $h_0\in\ker(\chi)$, then $\chi$ extends to $M \UE(M)_\chi/\spa{\vFq }$. 
\end{thmlist}
\end{prop}
\begin{proof}
Recall ${\wb G}_i:=\spa{G_i,\wh t_i,\wt t_i}$ and $\UE_M({\wb G}_i)$ is the subgroup of $\Aut({\wb G}_i)$ induced by $\UE({\wb M}^\circ)$, see \Cref{EMdef}.   
According to \Cref{eq_star_oT} ,  $\starStabkla \wh M_0| \wt T | {\phi}$. Let $\phi_i\in\ov \TT_i$ be such that $\phi=\phieinszwei$. By the definition of $\ov \TT_i$ the character $\phi_i$ satisfies $ \starStabkla {\wb G}_i | \UE_M({\wb G}_i) | {\phi_i} $ and extends to $G_i \UE_M(G_i)_{\phi_i}$. 
 
Hence the assumptions from \Cref{Cor5_7} are satisfied and its part (a) then implies that every $\chi\in\Irr(M\mid \phi)$ satisfies
\[ (\wt M \wh M)_\chi= \wM_\chi \wh M_\chi\]
and extends to $M \UE(M)_\chi$. This ensures part (a). Parts (b) and (c) are now clear from the rest of \Cref{Cor5_7}.
\end{proof}%%%%%%%%%%%%%%%%%%%%%%%%%%%%%%%%%%%%%%%%%%%%
\subsection{Proof of \Cref{thm_sec_Ad}. Above the other characters of \texorpdfstring{$M_0$}{M}}\label{ssec5F}
In this section we study $\wt M$-orbits $\cO$ in $\Irr(M)$ with $\cO\cap \Irr(M\mid \oTT_1.\oTT_2)=\emptyset$. According to \Cref{prop_char_G1G2}, $\cO$ then contains a character $\chi'\in\Irr(M\mid \phi)$ for some $\phi\in \EE'_1.\EE_2\sqcup \DD'_1.\DD_2\sqcup \DD'_1.\EE_2\sqcup \EE_1.\DD'_2$, while $|Z_1|=|Z_2|=4 $.
\begin{prop}\label{prop:DD}
Assume $ |Z_1|= |Z_2|=4 $. Let $\phi\in \EE'_1.\EE_2 \sqcup \DD'_1.\DD_2\sqcup \DD_1'.\EE_2\sqcup \EE_1.\DD_2'$. Then every $\chi\in\Irr(M\mid \phi)$ satisfies 
\begin{asslist}
	\item $\starStab \wt M |\wh M |\chi$ and 
	\item $\chi$ extends to $\wh M_\chi =M \UE(M)_\chi$.
\end{asslist} 
\end{prop}
\begin{proof} Set $\phi=\phi_1\phi_2$ with $\phi_i\in \Irr(G_i)$ for $i=1,2$. We are going to prove the following points for at least one $\chi\in \Irr(M\mid \phi)$.
	\begin{enumerate}[(1)] % (a), (b), (c), ...
		\item $\chi^{\wh t_1}=\chi$;
            \item $M_\phi\in\{M_0, M_0\spa{n}, M\}$ and $\phi$ extends to some $\chi_0\in\Irr(M_\phi\mid \phi)$, and 
		\item if $\phi_i\in \EE_i$ and we denote by $E_i'\leq \Aut(G_i)$ the subgroup induced by $\UE(M^\circ)_{\chi_0}$, then the index $|\UE_M(G_i)_{\phi_i}/E'_i|$ is even.       
         \item $(M \UE(M))_\phi\leq M_0 \UE(M^\circ)$ or $M_\phi=M$.	
	\end{enumerate}
	
	Let us see first how those points imply our Proposition.
Since $ \wt M \wh M $ acts trivially on $M/M_0$ it suffices to prove the points (i) and (ii) for some $\chi\in \Irr(M\mid \phi)$ to have them for all. Set $T_i:=T\cap G_i$. The assumption that $\phi=\phi_1\phi_2$ belongs to $  \EE'_1.\EE_2 \sqcup \DD'_1.\DD_2\sqcup \DD_1'.\EE_2\sqcup \EE_1.\DD_2'$ implies $\phi_i\in \DD_i\cup \DD'_i\cup \EE_i\cup\EE'_i$. According to \Cref{cor:6_11} (b) and (c) the inclusion $T_i\spa{\wh t_i,\wt t_i}_{\phi_i}\leq T_i\spa{\wh t_i}$ holds and $h_0\in\ker(\phi)$. This implies ${\wb T}_\phi\leq T \spa{\wh t_1}$ since $\wb T= (T_1.T_2)\spa{\wh t_1,\wh t_2,\wt t_1\wt t_2}$ according to \Cref{lem:5_actions_wbT}. 
Then (1) implies $\starStab \wt M| \wh M|\chi$ according to \Cref{def_whEM}. 

By (2), $M_\phi\in \{M_0\spa{n}_\phi, M\}$ and $\chi=\chi_0^M$ for some extension $\chi_0\in\Irr(M_\phi\mid \phi)$ of $\phi$, see also \Cref{not:5A}(a). Let $E_i'\leq \Aut(G_i)$ be induced by $\UE(M^\circ)_{\chi_0}$. Note that $E_i'\leq \UE_M(G_i)_{\phi_i}$, and $\phi_i\in\Irr(G_i)$ extends to $G_i \UE_M(G_i)_{\phi_i}$ according to \Cref{cor:6_11}(f), whenever $\phi_i \in \Irr(G_i) \setminus \EE_i$. If $i\in\{1,2\}$ with $\phi_i\in \EE_i$, then  by (3) the group $E_i'$ has even index in $\UE_M(G_i)_{\phi_i}$ and $\phi_i$ extends to $G_iE_i'$ according to \Cref{cor:6_11}(f). Hence $\phi$ extends to $M_0 E(M^\circ)_{\chi_0}$ according to \Cref{prop7_2}. 

Moreover $(M \UE(M))_{\chi_0}\leq M_{\chi_0} \UE(M^\circ)_{\chi_0}$ follows from (4), since $(M \UE(M))_{\chi_0}= M \UE(M)_{\chi_0}$ in case $\chi_0$ is an extension of $\phi$ to $M$, or $(M \UE(M))_{\chi_0}\leq  (M\UE(M))_\phi \leq M_0 \UE(M^\circ)_\phi$. Then this implies that $\chi_0$ extends to $(M \UE(M))_{\chi_0}=\wh M_{\chi_0}$, according to \Cref{ExtCrit}(e). Inducing this extension of $\chi_0$ from $\wh M_{\chi_0}$ to $\wh M_{\chi}$ gives an extension of $\chi$ as required in (ii). 

After recalling the consequences of \Cref{cor:6_11} we first check (1), (2) and (4) in each of the four cases for $\phi$. We finally show (3).

Indeed \Cref{cor:6_11} implies for the character $\phi_i$:
\begin{align*}
T_0 \spa{n,\wh t_i}_{\phi_i}=\begin{cases}
    T_0\spa{n}_{\phi_i}&\text{if } \phi_i\in \DD'_i,\\
    T_0\spa{n\wh t_i}_{\phi_i}&\text{if } \phi_i\in \DD_i,\\
    T_0\spa{n,\wh t_i}_{\phi_i}&\text{if } \phi_i\in \EE_i\sqcup\EE_i'.
\end{cases}
\end{align*}
This leads to 
\begin{align}\label{stab_phii}
M_{\phi}=\begin{cases}
  M&\text{if } \phi\in \EE_1'.\EE_2,\\
    M_0\spa{n}&\text{if } \phi\in \EE_1.\DD'_2\sqcup \DD'_1.\EE_2,\\
    M_0&\text{if } \phi_i\in \DD'_1.\DD_2.
\end{cases}
\end{align}
Hence in all cases $M_\phi=M^\circ_{\phi}\spa{n}_{\wh \phi}$ for $\wh \phi\in\Irr(M^\circ _\phi\mid \phi)$. According to \Cref{ExtCrit}(a), $\phi$ extends to $M_\phi$, whenever $M_\phi\neq M$.  

Let now first $(\phi_1,\phi_2)\in\EE'_1\times\EE_2$ and hence $M_\phi=M$. Note that $\EE'_1\neq \emptyset$ and $\EE_2\neq \emptyset$ implies $(\eps_1,\eps_2)=(1,1)$, $2\mid f$ and $4\mid(q-1)$, see \Cref{cor:6_11}(d) and (e). Consequently, $v=1$ and $n=\neins \nzwei$ by the definition of $v$ and $n$ in \Cref{tab:my_label}.

According to \Cref{cor:6_11}(b), every extension $\wh \phi_i$ of $\phi_i$ to $G_i\spa{\wh t_i}$ is $\gamma_i$-invariant. This implies $\wh \phi_i^{\nii}=\wh \phi_i$ according to \Cref{lem_gamma} and then we find a character $\psi_i\in \Irr( G_i\spa{\wh t_i, \nii}/\spannh)$ that is an extension of $\phi_i$. Since $G_1\spa{\wh t_1, \neins}/\spannh \times G_2\spa{\wh t_2, \nzwei}/\spannh$ is isomorphic to 
\[(G_1\spa{\wh t_1 ,\gamma_1}/\spannh  )\times (G_2\spa{\wh t_2, \gamma_2}/\spannh {} ),\] 
the character $\psi_1.\psi_2$ first defines a character of $M\spa{\wh t_1}$ and via restriction we obtain a $\spa{\wh t_1}$-invariant character $\chi\in\Irr(M\mid \phi_1.\phi_2)$. We see that $\chi$ is an extension of $\phi$ to $M$ and $\chi^{\wh t_1}=\chi$, ensuring (1), (2) and (4). 

If $(\phi_1,\phi_2)\in \DD'_1\times\DD_2$ or equivalently $\phi\in\DD'_1.\DD_2$, then this implies $M_\phi=M_0$. By Clifford theory (\ref{Cliff}), the character $\chi\in\Irr(M\mid \phi)$ satisfies $\chi=\phi^M$. As $\phi$ is $\wh t_2n$-invariant, the character $\chi$ seen as a character of $M$ is $\wh t_1$-invariant due to $M\wh t_1=M \wh t_2 n $. This ensures (1) and (2).
%Part (3) is clear, as neither $\phi_1$ nor $\phi_2$ is in $\EE_1\sqcup\EE_2$. 
Considering the action of $M \UE(M)$ on $G_1$ and using $\phi_1\in\DD'_1$ (see \Cref{cor:6_11}(c)), we see that $(M\UE(M))_\phi\leq M_0 \UE(M^\circ)$, ensuring (4). 

Assume next that $(\phi_1,\phi_2)\in  \DD_1'\times\EE_2$ and hence $M_\phi=M_0\spa{n}$, see \eqref{stab_phii}. This implies (2). Then $2\mid f$ according to \Cref{cor:6_11}(e) and hence \Cref{case2} does not hold. This leads to $n^2=h_0$ and provides the $\wb T$-equivariant isomorphism 
\[  M_0 \spa{n}/\spannh  \cong (G_1.G_2) \spa{\gamma_1\gamma_2}/\spannh .\]
 Note that  $\phi_2$ extends to
$G_2\spa{\wh t_2,\gamma_2}$ according to \Cref{cor:6_11}(b) and hence every extension of $\phi_2$ to $G_2\rtimes \spa{\gamma_2}$ is $\wh t_2$-invariant. An extension of $\phi_1$ to $G_1\spa{\gamma_1}$ and an extension of $\phi_2$ define an extension $\wt \phi$ of $\phi$ to $M_0\spa{n}$. Because of $[G_1\spa{\gamma_1},\wh t_2]=1$, $\wt \phi$ is $\wh t_2$-invariant. We observe that $\wt \phi$ can be taken as $\chi_0$ and $\chi=\wt \phi^M$. By construction, $\chi$ is $\wh t_2$-invariant. Since $\wh t_1\wh t_2\in M$ by \Cref{lem:5_5} this leads to $\chi^{\wh t_1}=\chi$, as required in (1).  As above, we derive from $\phi_1\in\DD_1'$ that (4) holds. 

If $\phi_1.\phi_2\in \EE_1.\DD'_2$, then the same argument applies and %Since ${\wb T}_\phi\leq T \spa{\wh t_1}$, \Cref{def_whEM} applies and proves $\chi^{\wh t_1}=\chi$, by 
ensures (1), (2) and (4).

Finally, we prove (3), i.e. $2\mid [\UE_M(G_j)_{\phi_j}:E'_j ]$ whenever $j\in \{1,2\}$ with $\phi_j\in\EE_j$. We can assume $\wh M/M$ to be non-cyclic and $\phi_1.\phi_2\in (\EE'_1.\EE_2)\sqcup(\EE_1.\DD'_2)\sqcup(\DD'_1.\EE_2)$. This implies $\eps_1=\eps_2=1$.
Fix $j\in\{1,2\}$ with $\phi_j\in \EE_j$ and let $j'\in\{1,2\} \setminus \{j\}$. In our situation we observe  $\phi_{j'}\in \oTT_{j'}$ and $\phi_{j'}^{n}=\phi_{j'}$ according to \Cref{cor:6_11}.
%%%%%%%%%%%%%% N E W

Writing $\underline{F}_p\in\UE(M) $ for the image of $F_p$, we let $a_i$ be a divisor of the order of $\underline{F}_p$ such that $\underline{F}_p^{a_i}$ generates $\spa{\underline{F}_p}_{\phi_i}$. 
By definition $\UE_M(G_j)_{\phi_j}= \spa{ ( \underline{F}_p^{a_j}), \gamma_j}$, 
and every extension of $\phi_{j}$ to $G_j\rtimes \spa{\gamma_j}$ is not $\underline{F}_p^{a_j}$-invariant because of $\phi_j\in\EE_j$. 
Analogously 
$\UE_M(G_{j'})_{\phi_{j'}}= \spa{ ( \underline{F}_p^{a_{j'}}), \gamma_{j'}}$ and the extension of $\phi_{j'}$ to $G_{j'}\rtimes \spa{\gamma_{j'}}$ is $\underline{F}_p^{a_{j'}}$-invariant as $\phi_{j'}\in\TT_{j'}$. 

Because of $\eps_1=\eps_2=1$ we have $n=\neinszwei$ and there exists an isomorphism between  $(G_1\spa{\gamma_1})/\spannh \times (G_2\spa{\gamma_2})/\spannh$ and $\spa{G_1,G_2,n,\gamma_1}/\spannh$. 
The extensions $\wt \phi_1$ and $\wt \phi_2$ define an extension $\wt \phi$ of $\phi$ to $M_0\spa{n}$ with
\[\UE(M^\circ)_{\wt \phi}=\spa{n,\gamma}\left( \spa{\underline{F}_p^{2a_j}}\cap \spa{\underline{F}_p^{a_{j'}}} \right).\] 
Recall that $E'_j$ denotes the subgroup of $\Aut(G_j)$ induced by $\UE(M^\circ)_{\chi_0}$, where $\chi_0$ is an extension of $\phi$ to $M_\phi$. Without loss of generality, we can assume that $\chi_0$ is an extension of $\wt \phi$ and hence $\UE(M^\circ)_{\chi_0}\leq \UE(M^\circ)_{\wt \phi}$. 
We then see that $E'_j$ has an even index in $\UE_M(G_j)_{\phi_j}=\spa{\gamma_j,\underline{F}_p^{a_j}}$ as required for (3). 
\end{proof}

The above pretty much 
finishes the verification of \Cref{thm_sec_Ad}: 
\begin{proof}[Proof of \Cref{thm_sec_Ad}]
Let $\calO$ be an $\wt M$-orbit in $\Irr(M)$, $\chi'\in\calO$ and $\phi'\in \Irr(\restr \chi'|{M_0})$. According to \Cref{prop_char_G1G2} there is some $t\in \wt T$ such that $\phi:=(\phi')^t$ is contained in 
\[ \oTT_1.\oTT_2 \sqcup \EE_1'.\EE_2 \sqcup\DD_1'.\DD_2\sqcup \DD_1'.\EE_2\sqcup \EE_1.\DD_2'.\]
The case $\phi \in \EE_1'.\EE_2 \sqcup\DD_1'.\DD_2 \sqcup \DD_1'.\EE_2\sqcup \EE_1.\DD_2'$ is only relevant if additionally $\{|Z_1|,|Z_2|\}=\{4\}$. Set $\chi:=(\chi')^t$ and hence $\chi\in \calO$. By the choice of $\phi$ we see that \[ \chi \in \Irr(M\mid \oTT_1.\oTT_2)\, \cup \, \Irr(M\mid \EE_1'.\EE_2 \sqcup \DD_1'.\DD_2 \sqcup \DD_1'.\EE_2\sqcup \EE_1.\DD_2').\]
Characters of these sets have been studied in \Cref{prop:TT} and \Cref{prop:DD}, respectively. Consequently
$\starStab \wt M|\wh M| \chi$, and $\chi$ extends to $M \UE(M)_\chi$. This gives part (a) of \Cref{thm_sec_Ad}.

Assume next $\eps_1=\eps_2=-1$ and $2\mid f$. According to \ref{ZG_order} we see that $\{|Z_1|,|Z_2|\}=\{2\}$ and therefore $\Irr(M_0)= \oTT_1.\oTT_2$. Then $\Irr(M)=\Irr(M\mid \oTT_1.\oTT_2)$ and we get part (b)  of \Cref{thm_sec_Ad} from \Cref{prop:TT}.
\end{proof}

%%%%%%%%%%%%%%%%%%%%%%%%%%%%%%%%%%%%%%%%%%%%%%%%%%%%%%%%%%
\section{The Conditions \texorpdfstring{\Ad{}}{A(d)} and \texorpdfstring{\Bd{}}{B(d)}}\label{sec6}
Now we use the results on the character theory of the group $M$ from the previous section to establish Conditions \Ad{} and \Bd{} from \ref{cond_Ad} and \ref{cond_Bd} for an integer $d\geq 3$ fixed throughout the whole chapter. Condition \Ad{} is about the character theory of $\NNN_\bG(\bS')^F$, where $\bS'$ is a Sylow $d$-torus of $(\bG,F)$. \Cref{tricho} shows that outside of the ``doubly 
regular'' case treated in Chapter~\ref{sec:4}, we can essentially replace $\NNN_\bG(\bS')^F$ by a group $M$ as in Chapter \ref{sec_nondreg_groupM}: In the setting of \Cref{ssec_M} 
a Sylow $d$-torus of $(\bG, \vFq )$ can be taken as a subgroup of either $\bG_1$ or $\bG_2$, and $d$ is doubly regular there.  Character correspondences between a subset of $\Irr(M) $ and $\Irr(\NNN_\bG(\bS')^F)$ allow us to transfer statements on $\Irr(M)$ to $\Irr(\NNN_\bG(\bS')^F)$ via centrally isomorphic character triples. These are constructed using results from the doubly regular case, where (\textbf{iMK}) was already shown, see \Cref{thm:dreg}. 
As a second step in \Cref{ssec6_A}, we use character correspondences to define an extension map, later verifying Condition \Bd. 

 Proving \Cref{thm:ndreg} finishes the proof of \Ad{} and \Bd{}. From this we derive (\textbf{iMK}) for quasisimple groups of type $\tD_l$, the final step of the proof of Theorem B and indeed McKay's equality.

\subsection{Character correspondences}
In the following, we continue using the notation introduced in Section \ref{sec_nondreg_groupM} around the group $M$.
We establish a character correspondence between some characters of $M$ and one of its subgroups under Assumption \ref{ass6_2} that essentially sums up the case (iii) of \Cref{tricho}.
This leads us to deduce some results about $\Irr(\NNN_\bG(\bS')^F)$ for some Sylow $d$-torus $\bS'$ of $(\bG,F)$. In order to apply the results of \Cref{sec:4} we assume the following. 

\begin{ass}\label{ass6_2}
	Let $d$ be an integer with $d\geq 3$. Assume that $d$ is doubly regular (see \Cref{def_dreg}) for $(\bG_{j},\vFq )$ for some $j\in\{1,2\}$ with $l_j\geq 4$.
\end{ass}

\begin{notation}\label{not:6:3}
	Let $\III{j}\in \{1,2\}$ be such that $d$ is doubly regular for $(\bG_{j},\vFq )$. Set $\II{K1}@{\bK_1:=\bG_{j}}$, $\II{K2}@{\bK_2:=\bG_{3-j}}$, $\eps'_1:=\eps_{j}=\eps\eps '_2$, $l'_1:=l_{j}=l-l'_2$, $K_i=\bK_i^\vFq$, $\II EMK1@{\UE_M(K_1)}:=\UE_M(G_{j})$, $\II EMK2@{\UE_M(K_2)}:=\UE_M(G_{3-j})$ (see \Cref{lem:5_actions_C1} and \Cref{lem:5_actions_C2}) and let $\II S@{\bS}$ be a Sylow $d$-torus of $(\bK_1,\vFq )$. We associate with $\bS$ the groups \[ \III{N_1}:=\NNN_{K_1}(\bS),\, \III{C:=\Cent_{M^\circ}(\bS)}\unlhd \III{N:=\NNN_M(\bS)},\, \II{Nhat}@{\wh N:=\NNN_{M \UE(M)}(\bS)} \und \II{Ntilde}@{\wt N:=\NNN_{M\wt T}(\bS)},\] where $\wt T:=(\bT\Z(\wbG))^{\vFq}$.
\end{notation}
\Cref{ass6_2} implies that (\textbf{iMK}) holds for $K_1$ and $\ell$ \wrt $N_1$, whenever $\ell$ is an odd prime with $d=d_\ell(q)$, see \Cref{thm:dreg}. Recall that $\Aut(K_1)$ is induced by bijective endomorphisms of $\bG$ commuting with $F$, see \Cref{ssec2C}, and hence $\Aut(K_1)$ acts on the set of $F$-stable subgroups of $\bK_1$. The group $\II{Gamma1}@{\Gamma_1}:=\Aut(K_1)_\bS$ is therefore well-defined and according to \Cref{thm:dreg}(b) there exists a $\Gamma_1$-equivariant bijection
\[ \Omega_1^\circ: \Irrl(K_1)\lra \Irrl(N_1), \]
such that 
\[ (K_1\rtimes \Gamma_{1,\psi}, K_1, \psi)\geq_c (N_1\rtimes \Gamma_{1,\Omega_1^\circ(\psi)} , N_1,\Omega_1^\circ(\psi)) \forevery \psi\in\Irrl(K_1).\]
Since $K_1=\bK_1^{\vFq }\cong \tD_{l'_1,\sico}^{\eps'_1}(q)$ according to \Cref{def_M}, and the integer $d=d_\ell(q)$ is doubly regular for $(\bK_1,\vFq )$, \Cref{bij_Omega'} defines a character set $\calG_d(K_1)\subseteq \Irr(K_1)$ and provides us with a character correspondence extending $\Omega_1^\circ$. 

Recall $K_1\unlhd M$ by \Cref{lem:5_5}(d). In the following, we continue to use the group $\UE(M)$ from \Cref{lem:5_5}(d), which acts by definition also on $\wt M:=\wt T M$. The following gives a character correspondence between some characters of $K_1$ and characters of ${N_1=\NNN_{K_1}(\bS)}$, which is additionally $\Gamma_1$-equivariant. Recall that here we use the order relation on character triples from \Cref{sec2_B}.

\begin{prop}\label{cons_iMcK_H1} We keep \Cref{ass6_2} and follow \Cref{not:6:3}.
	\begin{thmlist}
		\item Let $\calG_d(K_1)\subseteq \Irr(K_1)$ be the set from \Cref{calGd}. There exists some $\Gamma_1$-equivariant bijection
		\[\Omega_1: \calG_d(K_1) \xrightarrow{\sim} \Irr{(N_1)},\]
		such that 
		\[ ( K_1 \rtimes \Gamma_{1,\chi_1},K_1,\chi_1)\geq_c ( N_1\rtimes \Gamma_{1,\Omega_1(\chi_1)},N_1,\Omega_1(\chi_1)) \forevery \chi_1\in\calG_d(K_1).\]
		\item Set $\III{A:=\wt M \UE(M)}$. Then every $\chi_1\in\calG_d(K_1)$ satisfies
		\[ (A_{\chi_1},K_1,\chi_1) \geq_c (\NNN_A(\bS)_{\chi_1},N_1,\Omega_1(\chi_1)) \forevery \chi_1\in\calG_d(K_1).\]
		\item \label{cor6_4} Let $J$ be a group with $K_1\leq J\unlhd A=\wt M \UE(M)$. Then there exists an $\NNN_A(\bS)$-equivariant bijection 
		\[\Pi: \Irr(J\mid \calG_d(K_1)) \lra \Irr(\NNN_J(\bS)) \]
		such that every $\psi \in \Irr(J\mid \calG_d(K_1))$ satisfies
		\[(A_\psi, J, \psi)\geq_c (\NNN_A(\bS)_{\Pi(\psi)}, \NNN_J(\bS),\Pi(\psi)). \]
		In particular 
		\begin{asslist}
			\item If $\psi \in \Irr(J\mid \calG_d(K_1))$ and $\psi':=\Pi(\psi)$, then $ A_\psi=J \NNN_A(\bS)_{\psi'}$. For any $J\leq U \leq A$, the character $\psi $ extends to $U $ if and only if $\psi'$ extends to $\NNN_U(\bS)$. 
			\item Let $\phi_1 \in \calG_d(K_1)$ and $\phi_2\in\Irr(\Cent_J(K_1))$ such that $\phi_1.\phi_2$ is a well-defined irreducible character of $K_1.\Cent_J(K_1)$. Then 
			\begin{align}\label{eq_cor6:4}
				\Pi(\Irr(J\mid \phi_1.\phi_2))= \Irr(\NNN_J(\bS)\mid \Omega_1(\phi_1).\phi_2).
			\end{align}
		\end{asslist}
	\end{thmlist}
\end{prop} 
\begin{proof}
	According to \Cref{bij_Omega'}, the bijection $\Omega_1$ exists as required in (a). Recall $K_1\unlhd \wt M\UE(M)$ by \Cref{lem:5_actions_wbT}(c). According to the Butterfly Theorem \ref{Butterfly}, part (a) implies (b), see also \Cref{IMNiMK}. In Part (c), the existence of a bijection $\Pi$ satisfying the $\geq_c$ relation and (ii) is a consequence of (a) and (b) thanks to \cite[Prop. 2.4]{Rossi_McKay} and the construction described there. Property (i) follows from \Cref{rem_chartrip}(b). 
\end{proof}

We continue using \Cref{ass6_2} and the notation above. We deduce from $\Omega_1$ and its properties the following bijection. The next statement ensures Condition \Ad{} later, via an isomorphism between $\GF$ and $\bG^{\vFq }$. Recall $N:=\NNN_M(\bS)$, $\wt N:=\NNN_{\wt M}(\bS)$ and $\wh N:=\NNN_{M \UE(M)}(\bS)$.

\begin{theorem} \label{thm6_5} We keep \Cref{ass6_2}.
	 There exists some $\wh N$-stable $\wt N$-transversal $\TT(N)$ in $\Irr(N)$ such that maximal extendibility holds \wrt $N\unlhd \wh N$ for $\TT(N)$. Additionally if $2\mid f$, $\eps=1$ and $\eps_1=-1$, then every $\chi\in\TT(N)$ with $h_0\in\ker(\chi)$ has an extension $\wt \chi$ to $\wh N_\chi$ with $vF_q \in\ker(\wt \chi)$.
\end{theorem}
\begin{proof} Recall $\wh M:=M\UE(M)$ and $\wt M:=M \wt T$. According to \Cref{thm_sec_Ad}(a) there exists some $\wh M$-stable $\wt M$-transversal $\TT(M)$ in $\Irr(M)$, such that every $\chi\in\TT(M)$ extends to $\wh M_\chi$. On the other hand $\calG_d(K_1)$ is by definition $\Aut(K_1)$-stable. Set $\TT'(M):=\TT(M)\cap \Irr(M\mid \calG_d(K_1))$. 
	Since $\UE(M)$ and $\wt T$ permute the Sylow $d$-tori of $(\bK_1,\vFq )$ and all Sylow $d$-tori of $(\bK_1,\vFq)$ are $K_1$-conjugate (see \cite[Thm~25.11]{MT}) we see $M \UE(M)= M \wh N$ and $\wt M=M \wt N$. Therefore, $\TT'(M)$ is also a $\wh N$-stable $\wt N$-transversal in $\Irr(M\mid \calG_d(K_1))$.
	%According to \Cref{thm_sec_Ad} this exists. 
	
	For $J:=M$ and $A:=\wh M \wt M$, we apply Proposition~\ref{cor6_4} and obtain the bijection 
	\[\Pi:\Irr(M\mid \calG_d(K_1))\lra \Irr(\NNN_M(\bS)) \] 
	with the properties stated there. Then $\Pi$ is $\wh N \wt N$-equivariant as $\NNN_{A}(\bS)=\wh N \wt N$. Note that the set $\TT(N):=\Pi(\TT'(M))$ is a $\wh N$-stable $\wt N$-transversal in $\Irr(N)$ by the equivariance of $\Pi$. Let $\psi'\in \TT(N)$ and set $\psi:=\Pi^{-1}(\psi')\in \TT'(M)$. By the properties of $\TT'(M)$, we see that $\psi$ extends to its stabiliser in $\wh M$. Furthermore by the properties of $\Pi$ we have
	\[ (M \UE(M)_\psi, M,\psi)\geq_c (\wh N_{\psi '},N , \psi'). \]
	According to \Cref{rem_chartrip}(b), $\psi'$ extends to $\wh N_{\psi'}$, since $\psi$ extends to $\wh M_\psi$. 
	
	Note that in the case of $2\mid f$, $\eps=1$ and $\eps_1=-1$, every $\psi'\in\TT(N)$ with $h_0\in\ker(\psi')$ has an extension $\wt \psi'$ to $\wh N_{\psi'}$ with $vF_q \in\ker(\wt \psi')$, as $\psi$ has the analogous property by \Cref{thm_sec_Ad}(b).
\end{proof}
The following helps to ensure the first part of Condition \Bd{} from \Cref{cond_Bd}. %If $\GF=\GvF$ the statement is exactly what is required for Condition \Bd{}, otherwise the statements are related via an isomorphism. 
\begin{cor}\label{cor6_6}
	Maximal extendibility holds with respect to $N\unlhd \wt N$.
\end{cor}
\begin{proof} As in the above proof of \Cref{thm6_5} we apply Proposition~\ref{cor6_4}  with $J:=M$ and use the bijection $\Pi:\Irr(M\mid \calG_d(K_1)) \lra \Irr(N)$ with 	
\[ ((\wt M \UE(M))_\psi, M , \psi)\geq_c ((\wh N\wt N)_{\psi'}, N , \psi') \] 
for any $\psi'\in \Irr(N)$ and $\psi:=\Pi^{-1}(\psi')$.
	
By \Cref{rem_chartrip}(c) this restricts to 
\[ (\wt M_\psi, M , \psi)\geq_c (\wt N_{\psi'}, N , \psi'). \]
Recall that according to \Cref{prop_maxext_M_wtM}, $\psi$ extends to $\wt M_\psi$, since maximal extendibility holds with respect to $M\unlhd \wt M$. According to \Cref{rem_chartrip}(b) this implies that $\psi'$ extends to $\wt N_{\psi'}$. 
\end{proof}
%%%%%%%%%%%%%%%%%%%%%%%%%%%%%%%%%%%%%%%%%%%%%%%%%%
\subsection{Another extension map}\label{ssec6_A}
The aim of this section is to build a version of the extension map required in the second part of Condition \Bd{}, see \Cref{cond_Bd}. As before, we only verify an analogue for subgroups of $\wt\bG^{\vFq }$. Recall $M^\circ:=(\bG_1.\bG_2)^{\vFq }$ and $\UE(M^\circ):=\UE(M)\spa{n,h_0}$. We write $\II ToKj@{\protect{\oTT}(K_i)}$ for the set from \Cref{def:TEDi} with our choice of $K_1=G_j$ and $K_2=G_{3-j}$. We choose $\II TTK2@{\TT(K_2)}$ to be an $\UE_M(K_2)$-stable subset of $\oTT(K_2)$, which is at the same time a $\wt T$-transversal in $\Irr(K_2)$. (This is possible according to \Cref{cor:6_11}(a).) In the diagrams below double bars stand for normal inclusions of subgroups.

\noindent\begin{minipage}{0.4\textwidth}
 \setlength{\parindent}{1em}
$
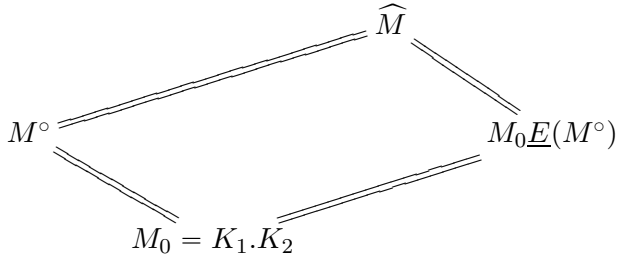
\xymatrix{
 &  & \widehat M &  \\
M^\circ \ar@{=}[rru] &  &  & M_0\underline E(M^\circ) \ar@{=}[lu] \\
 & M_0=K_1.K_2 \ar@{=}[lu] \ar@{=}[rru] &  & 
}$
\end{minipage}
\hspace{0.07\textwidth}
\noindent \begin{minipage}{0.4 \textwidth}\noindent
In the groups $\wh M$, $M_0\UE(M^\circ)$, $M_0$ and $M^\circ$, introduced before we consider their normalizers of $\bS$. Note that there is a well-defined action of $\UE(M)$ and $\UE(M^\circ)$ on the $\vFq$-stable tori of $\bG$ and hence the stabilizer of $\bS$ in those groups is well-defined. 

Set $\II{No}@{N^\circ:=\NNN_{M^\circ}(\bS)}$ and $\II{No}@{\wh N^\circ}:=\NNN_{M_0\UE(M^\circ)}(\bS)$. Note that $N^\circ\cap \wh N^\circ=\NNN_{M_0}(\bS)$ since $M^\circ\cap M_0 \UE(M^\circ)= M_0$ by \Cref{lem_centZ2}(a) and \Cref{lem_centZ2_C2}(a). 
\end{minipage}

\begin{minipage}{0.45\textwidth}
By its definition $\TT(K_2)$ is also $\wh N^\circ$-stable and is a ${\wt T}$-transversal in $\Irr(K_2)$. 
Recall the description of the action of $M_0\UE(M^\circ)$ and, therefore, of $\wh N^\circ$ on $K_1$ in \Cref{lem:5_actions_C1} and \Cref{lem:5_actions_C2}, respectively.
This allows us to transfer the main result from \Cref{ssec_4E} and we obtain an $\wh N^\circ$-equivariant extension map $\Lambda_1$ for $\Cent_{K_1}(\bS)\unlhd \NNN_{K_1}(\bS)$.
\end{minipage}
\hspace{0.07\textwidth}
\begin{minipage}{0.5\textwidth}
\setlength{\parindent}{1em}
\xymatrix{
 &  & \widehat N &  \\
N^\circ \ar@{=}[rru] &  &  & \widehat N^\circ \ar@{=}[lu] \\
 & N_1.K_2 \ar@{=}[lu] \ar@{=}[rru] &  & 
}\end{minipage}

\begin{lem}\label{prop:7_7} We follow \Cref{not:6:3}, keeping \Cref{ass6_2}. Recall ${ N_1=\NNN_{K_1}(\bS)}$ and $C=\Cent_{M^\circ}(\bS)\unlhd N^\circ$. 

Let
	 $\III{C_1:=\Cent_{K_1}(\bS)}$ and let $\II{Lambda1}@{\La_1}$ be the $\wh N^\circ$-equivariant extension map \wrt $C_1\unlhd N_1$ from \Cref{cor3_5}(a). 
	 	Let $\la\in\Irr(C_1)$, $\phi_2\in\TT(K_2)$ with $\la(h_0)\phi_2(1)=\la(1)\phi_2(h_0)$, so that $\la.\phi_2\in\Irr(C_1.K_2)$ and $\La_1(\la)^{N_1}.\phi_2\in\Irr(N_1.K_2)$ are well defined. 
	
	Then any $\psi\in\Irr(N^\circ\mid \La_1(\la)^{N_1}.\phi_2)$ extends to its stabilizer in $ \wh N$. Moreover, $\wh N_\psi=N^\circ  \wh N_\tau$ for $\{\tau \} = \Irr(C\mid \la.\phi_2)\cap \Irr(\restr \psi|C)$. 
\end{lem}

\noindent
\begin{minipage}{0.37\textwidth}
\noindent{\it Proof.}
 Let 	\[\Omega_1: \calG_d(K_1) \xrightarrow{\sim} \Irr{(N_1)},\] and
	\[\Pi: \Irr(M^\circ\mid \calG_d(K_1)) \xrightarrow{\sim} \Irr(N^\circ)\]
	be the bijections from Proposition~\ref{cor6_4} obtained by choosing $M^\circ$ as $J$ and $\wt M \UE(M)$ as $A$. 
	
	Let $\psi \in \Irr(N^\circ\mid \La_1(\la)^{N_1}.\phi_2)$ and set $\chi:=\Pi^{-1}(\psi)$. 
	By definition $\chi\in \Irr(M^\circ\mid \Omega_1^{-1}(\La_1(\la)^{N_1}).\phi_2)$ thanks to \eqref{eq_cor6:4} in \Cref{cons_iMcK_H1}. 
	By \Cref{cor3_5}(b), the character $\Omega_1^{-1}(\La_1(\la)^{N_1})$ of $K_1$ satisfies
	\[ \Omega_1^{-1}(\La_1(\la)^{N_1})\in \oTT(K_1). \] 	
 \end{minipage}
\hspace{0.03\textwidth} 
\noindent\begin{minipage}{0.5\textwidth}
\xymatrix{
 &  &  &  & \widehat M \\
 &  & \chi \ar@{.}[r] & M^\circ \ar@{=}[ru] & \widehat N \ar@{-}[u] \\
 &  & M_0 \ar@{=}[ru] & N^\circ \ar@{-}[u] \ar@{=}[ru]  \ar@{.}[r] & \psi \\
 & K_1 \ar@{=}[ru] & N_1.K_2 \ar@{=}[ru] \ar@{-}[u] & C \ar@{=}[u] \ar@{.}[r] & \tau \\
\Lambda_1(\lambda)^{N_1} \ar@{.}[r] & N_1 \ar@{=}[ru] \ar@{-}[u] & C_1.K_2 \ar@{=}[ru] \ar@{=}[u] \ar@{.}[r] & \lambda.\phi_2&  \\
 & C_1 \ar@{=}[ru] \ar@{=}[u] \ar@{.}[r] & \lambda  &  & 
}
\end{minipage}

\medskip
\noindent
 We now get $\chi\in \Irr(M^\circ\mid \oTT_1.\oTT_2)$ in the notation of \Cref{def:TEDi} and for this character, Proposition~\ref{cor_maxext_eps_1} tells us that $\chi$ extends to $\wh M_\chi$. By \Cref{cons_iMcK_H1}(c.i) we see that then also $\psi$ extends to $\wh N_{\psi}$. 
	
	We now have to compute $\wh N_\tau$. Clifford theory implies that $\wh N_\psi\leq \wh N_{\restr \psi|C}=N \wh N_{\tau}$ and hence it remains to ensure $\wh N_\psi \geq \wh N_\tau$. 
	Note first that $C_1.K_2$ has index $2$ or $1$ in $C$. By Clifford theory, then $\tau$ is either an extension of $\la.\phi_2$ or equals $(\la.\phi_2)^C$. 
	
	Assume that $\tau$ is an extension of $\la.\phi_2$ to $C$. Then there exists a common extension $\wt \tau$ of $\tau$ and $\restr\Lambda_1(\la)|{(N_1)_{\tau}}.\phi_2$ to ${(N_1)_{\tau}} C$ according to \Cref{ExtCrit}(e). 
	As $\La_1$ is $\wh N^\circ$-equivariant, the character $\restr\Lambda_1(\la)|{(N_1)_{\tau}}.\phi_2$ is $(\wh N^\circ)_{\la.\phi_2}$-invariant. Then the character $\wt \tau$ is $C\wh N^\circ_{\tau} =\wh N_\tau$-invariant, since $\wh N^\circ_{\la.\phi_2} C= \wh N_{\restr \tau|{C_1.K_2}} \geq \wh N_\tau $. Note that $\psi=\wt \tau^{N^\circ}$ and hence $\psi$ is $\wh N_\tau$-invariant.
	
	Now assume that $\tau=(\la.\phi_2)^C$. Then $\wt \tau:=(\La_1(\la).\phi_2)^{C {(N_1)_{\tau}}}$ is irreducible and an extension of $\tau$. We observe again that $\wt \tau$ by the construction using $\La_1$ is $C \wh N^\circ_{\la.\phi_2} $-invariant and $\wt \tau ^N=\psi$ by Clifford theory. It remains to see that $C\wh N^\circ_{\la.\phi_2} =\wh N_\tau$. 
	
	In order to see this equality, we compare the actions of $C$ and $\wh N^\circ$ on $C_1.K_2$. While $C$ induces a diagonal automorphism on $K_2$ which can also be induced by conjugation with an element of $\wt T$, $\wh N^\circ$ acts on $K_2$ as an element of $\UE_M(K_2)$. Recall $\phi_2\in \TT(K_2)$ and $\TT(K_2)$ is a $\wh N^\circ$-stable  $\wt T$-transversal in $\Irr(K_2)$. According to \cite[Lem.~2.4]{S21D1}, the properties of $\TT(K_2)$ lead to $(\wh N^\circ \wt T)_{\phi_2}= \wh N^\circ_{\phi_2} \wt T_{\phi_2}$, and hence  
	\[(C \wh N^\circ)_{\phi_2}= C_{\phi_2} \wh N^\circ _{\phi_2}.\]
	Since $C_1\leq \Z(C)$ we obtain that 	 
	\[\wh N_{\la.\phi_2}=(C \wh N^\circ)_{\la.\phi_2}= C_{\phi_2} \wh N^\circ _{\la.\phi_2}.\]
	As $\tau=(\la.\phi_2)^C$ and $\wh N= \wh N^\circ C$, this leads to 
	\[\wh N_\tau= C \wh N_{\la.\phi_2}= C (C_{\phi_2} \wh N^\circ _{\la.\phi_2})= C \wh N^\circ _{\la.\phi_2}. \] 
	Hence in all cases $\wt \tau$ is $\wh N_\tau$-invariant, implying that $\psi$ is $\wh N_\tau$-invariant, as well.  
	
	General Clifford theory shows that $\wh N_\psi= N^\circ \wh N_{\tau, \wt \tau}$. Taking into account that $\wt \tau$ is $\wh N_\tau$-invariant we obtain the equality $\wh N_\psi=N^\circ  \wh N_\tau$. \qed{}
\smallskip

Above, before \Cref{prop:7_7}, we have defined ${\TT(K_2)}\subseteq \oTT(K_2)$ and the former is by definition an $M_0\UE(M^\circ) $-stable $\wt T$-transversal in $\Irr(K_2)$.
Recall ${\wh N:=\NNN_{M \UE(M)}(\bS)}$ and ${C:=\Cent_{M^\circ}(\bS)}$. The latter will be shown to be ${\Cent_{\bG}(\bS)^{\nu F_q}=\Cent_{M}(\bS)}$ in \Cref{lem6_11} below.

\begin{prop}[Extension map \wrt $\Cent_{M}(\bS)\unlhd \NNN_{M}(\bS)$]\label{prop_Lambda}
	Set $\II Ctilde@{\wt C:=\Cent_{\wt M}(\bS)}$.
	\begin{enumerate}
		\item The set $\Irr(C\mid \TT(K_2))$ forms an $\wh N$-stable $\wt C$-transversal in $\Irr(C)$, in particular $\starStab \wt C|\wh N|\rho$ for every $\rho\in\Irr(C\mid \TT(K_2))$.
		\item There exists some $\wh N$-equivariant extension map $\II Lambda @{\Lambda}$ \wrt $C\unlhd N$ for $\Irr(C\mid \TT(K_2))$.
	\end{enumerate}
\end{prop}
\begin{proof}  
	For part (a) observe that $\TT(K_2)$ is $\UE(M^\circ)$-stable. Furthermore, because of $[N_1,K_2]=1$, we get that the set $\Irr(C\mid \TT(K_2))$ is $C\wh N_1 $-stable. The group $\wh N$ induces the same outer automorphisms of $K_2$ as $ C \wh N_1$. Hence, $\Irr(C\mid \TT(K_2))$ is $\wh N$-stable. Since $\TT(K_2)$ is a $\wt T$-transversal, $\Irr(C\mid \TT(K_2))$ is a $\wt C$-transversal in $\Irr(C)$. This ensures part (a). 
	
	For the proof of part (b), we show that \maex holds \wrt $C\unlhd \wh N$ for $\Irr(C\mid \TT(K_2))$. This will imply the existence of an associated $\wh N$-equivariant extension map and by restriction we will actually get our claim \wrt $C\unlhd  N$.
	
	Let $\tau\in \Irr(C\mid \TT(K_2))$ and show that it extends to $\wh N_{\tau}$. Note that with $\la\in \Irr(C_1)$ and $\phi_2\in \Irr(\restr \tau|{K_2})\cap \TT(K_2)$ 
	 \Cref{prop:7_7} applies. 
	Let $\wt \tau$ be the $\wh N_\tau$-invariant extension of $\tau$ to $N_\tau$ constructed in the proof of \Cref{prop:7_7}. Then $\psi=\wt \tau^{N^\circ}$ and $\psi$ extends to some $\wt \psi\in\Irr(\wh N_\psi)$. %with $\wh N_\psi= N \wh N_\tau$. 

By Clifford correspondence (\ref{Cliff}), some character $\tau'\in \Irr((\wh N_\psi)_\tau\mid\tau)$ satisfies $\tau'^{\wh N_{\psi}}=\wt \psi$. Recall that by \Cref{prop:7_7} we have $\wh N_\psi= N^\circ  \wh N_\tau$ and hence $(\wh N_\psi)_\tau=\wh N_\tau$. Taking into account the degrees, we see that $\tau'$ is an extension of $ \tau$ and, by construction, even of $\wt \tau$ to $\wh N_{\tau}$.  This finishes our proof.
\end{proof}
Next we establish a statement implying later Condition \Bd{} via the isomorphism between $\GF$ and $\bG^{\vFq }$. Recall that $\bS$ denotes a Sylow $d$-torus of $(\bK_1,\vFq )$, $\wt C:=\Cent_\wM(\bS)$, $\wt N:=\NNN_\wM(\bS)$ and $N:=\NNN_M(\bS)$.
\begin{prop}[Extension map \wrt $\Cent_\wM(\bS)\unlhd \NNN_\wM(\bS)$]\label{prop_requwtLambda} We keep \Cref{ass6_2} and follow \Cref{not:6:3}.
	 There exists some $\Lin(\wt N/N)\rtimes \wh N$-equivariant extension map $\II{Lambdatilde}@{\wt \Lambda}$ \wrt $\wt C\unlhd \wt N$.
\end{prop}
\begin{proof} We wish to apply \Cref{Ext_f} with $X=\wt C$, $A=\wt N$, $\wh A=\wh N\wt C$, and $X_0=C$, $A_0=N$, $\wh A_0=\wh N$. The conclusion of \Cref{Ext_f} provides exactly the extension map we need with the properties claimed.
	We now review the assumptions of \Cref{Ext_f}. 
	
	We have $C\unlhd \wt C$ with an abelian quotient since $\wt M=M\wt T$. The other group theoretic assumptions are clear.
	
	To check that maximal extendibility holds with respect to $C\unlhd \wt C$, notice first that $\wt C =\bK_3^{\nu F_q}$ where $\bK_3$ is the connected reductive group  $\bK_2\Cent_{\bK_1}(\bS) \Z(\wbG)$ with Frobenius endomorphism $\nu F_q$. Then $[\bK_3,\bK_3] =[\bK_2,\bK_2]$ since $\Cent_{\bK_1}(\bS) $ is a torus. Therefore maximal extendibility holds for $[\bK_2,\bK_2]^{\nu F_q}\unlhd \wt C$ thanks to \cite[Thm 15.11]{CE04}). But then we also get maximal extendibility for $C\unlhd \wt C$ by \Cref{ExtCrit}(b) and the inclusion $\bK_2^{\nu F_q}\leq C$.

	We have an $\wh N$-equivariant extension map \wrt $C\unlhd N$  for $\Irr(C\mid \TT(K_2))$ thanks to \Cref{prop_Lambda}(b) above.
	
	So we have all the required assumptions and \Cref{Ext_f} gives our claim.
\end{proof}

\subsection{Turning to groups related to Sylow \texorpdfstring{$d$}{d}-tori}\label{sec_6.B}
The aim of the following is the proof of \Cref{thm:ndreg}, namely the proof of Conditions \Ad{} and \Bd{} for $(\bG,F)$ in the case where $d\geq 3$ is not doubly-regular for $(\bG,F)$, thus completing the proof of \Cref{thm:dreg}. 

First we choose integers $l_1, l_2, \eps_1, \eps_2$ as in \Cref{lem6_11} with regard to $d$ and $(\bG,F)$ and determine a corresponding group $M$ as in \Cref{not_groupM}. For this group $M$ we see that the statements from the two preceding sections apply, where characters of $\NNN_M(\bS)$ were studied for some Sylow $d$-torus $\bS$ of $(\bG,\vFq )$. We see in \Cref{lem6_11} that $\bS$ is a Sylow $d$-torus of $(\bG,F)$ and that $\NNN_M(\bS)=\NNN_\GvF(\bS)$. 

Then we establish in \Cref{lem_iota} an isomorphism $\iota$ between $\GvF$ and $\GF$, as well as between $\GvF \UE(M)$ and $\GF \UE(\GF)$ in most cases. Denoting $N:=\NNN_M(\bS)$, $\wt N:=\NNN_{\wt M}(\bS)$, $\bS':=\iota(\bS)$, $N':=\NNN_{\GF}(\bS')$ and $\wt N':=\NNN_{\GF}(\bS')$, we show how
$\iota$ maps the $\wt N$-transversal in $\Irr(N)$ from \Cref{thm6_5} into a $\wt N'$-transversal in $\Irr(N')$ as required by \Ad . 

\medskip

\begin{theorem}\label{thm:ndreg}
	Let $(\bG,F)$ be as in \ref{sub2E} such that $\GF=\twepsDlq$ and $d$ an integer such that $d\geq 3$, $a_{(\bG,F)}(d)\geq 2$ and $d$ is not doubly regular for $(\bG,F)$ in the sense of \ref{def_dreg}. 
	Let $\bS'$ be a Sylow $d$-torus of $(\bG,F)$, $ C':=\Cent_\GF(\bS')$, $\wt C':=\Cent_\wGF(\bS')$, $N':=\NNN_{\GF}(\bS')$, $\wh N':=\NNN_{\GF \UE(\GF)}(\bS')$ and $\wt N':=\NNN_\wGF(\bS')$. Then the following hold.
	\begin{thmlist}
		\item There exists an $\wh N'$-stable $\wt N'$-transversal $\TT(N')$ in $\Irr(N')$, such that every $\psi\in \TT(N')$ extends to its stabilizer in $\wh N'$.
		\item Maximal extendibility holds \wrt $N'\unlhd \wt N'$.
		\item There exist some $\wh N'$-stable $\wt C'$-transversal $\TT'$ in $\Irr(C')$ and an $\wh N'$-equivariant extension map \wrt $ C'\unlhd N'$ for $\TT'$.
		\item There exists some $\Lin(\wt N'/N')\rtimes \wh N'$-equivariant extension map \wrt $\wt C'\unlhd \wt N'$.
	\end{thmlist}
	In particular, Conditions \Ad{} and \Bd{} from \Cref{ssec_2D} hold for $\GF$.
\end{theorem}

We have first the following.

\begin{lem}\label{lem6_11} 
We assume the situation of \Cref{thm:ndreg}.
\begin{thmlist}
\item We can choose integers
$l_1,l_2,\eps_1,\eps_2$, an element $v\in \obG$ as in \Cref{not_groupM} and \Cref{def_M}, and groups $\bG_1,\bG_2$, $\bM$ and $M$ such that: 
	\begin{asslist} 
		\item $M$ satisfies \Cref{ass6_2} for $d$. 
		\item Let $\nu$ be the automorphism of $\bG$ and $\wbG$ associated to $v$ as in \Cref{def_M}. If $\bK_1\in \{\bG_1,\bG_2\}$ is given as in \ref{not:6:3} and $\bS$ is a Sylow $d$-torus of $(\bK_1,\vFq )$, then $\bS$ is a Sylow $d$-torus of $(\bG,\vFq )$. 
			\end{asslist}
		\item Then $\Cent_\bG(\bS)\leq \bM^\circ$ and the group $M$ satisfies additionally 
  \[\NNN_{\bG}(\bS)^\vFq\leq M=\bM^\vFq,\ \NNN_{\wbG}(\bS)^\vFq\leq \wt M, \und \] \[\NNN_{\bG\UE(\bM)}(\bS)^\vFq\leq M \UE(\bM)=\bM^\vFq\UE(\bM) .\]
   %and therefore $\NNN_{\GvF}(\bS)=\NNN_M(\bS)$.
\end{thmlist}
\end{lem}

 Note that (b) implies ${C:=\Cent_{M^\circ}(\bS)}={\Cent_{\bG}(\bS)^{\nu F_q}=\Cent_{M}(\bS)}\unlhd N=\NNN_\bG(\bS)^{\nu F_q}$, whence the consistency with the general notation of Section 2.D.

\begin{proof} The assumptions of \Cref{thm:ndreg} clearly are satisfied only in the third case of \Cref{tricho}.  In this case, there exist
$j\in\{1,2\}$, $v\in\obG$, $l_1,l_2>0$ and $\eps_1,\eps_2\in\{\pm 1\}$ with $l_1+l_2=l$, $\eps=\eps_1\eps_2$ as in Notation \ref{not_groupM} and \Cref{def_M} such that  for the thereby defined groups $\bG_1,\bG_2$ and $M$, a Sylow $d$-torus $\bS$ of $(\bK_1,\vFq )$ is a Sylow $d$-torus of $(\bG,\vFq )$, where $\bK_1=\bG_j$. 
This gives (i) and (ii) of (a).
	
	The proof of $\NNN_{\bG}(\bS)^\vFq\leq M=\bM^\vFq$ in (b) uses the description of minimal $d$-split Levi subgroups $\Cent_\bG(\bS)$ of $(\bG,\vFq)$ and their associated relative Weyl groups $(\NNN_{ \bG}(\bS)/\Cent_{ \bG}(\bS))^\vFq$ in \cite{S10b} and \cite[Ch. 3.5]{GM}. Recall $\bM^\circ = \bK_1 .\bK_2$ is a central product of $\vFq$-stable groups. 
	Since $d$ is doubly regular for $(\bK_1,\vFq )$ we see that $\bT_1:=\Cent_{\bK_1}(\bS)$ is a torus and 
	\[\Cent_{\bM^\circ}(\bS)=\bT_1. \bK_2.\]    
	From the description of the root system of the Levi subgroup $\bC:=\Cent_{ \bG}(\bS)$ (see \cite[\S 7]{S10b} and \cite[Example 3.5.15]{GM}) we get first 
	$$\Cent_{\bM^\circ}(\bS)=\bC.$$ 
	It essentially remains to show that the relative Weyl groups in $\GvF$ and $M$ have the same order. Denoting $d_0$ and $a_{(\bG,\vFq)}(d)$ as in the proof of \Cref{tricho}, we abbreviate the latter as $a$. The relative Weyl groups being insensitive to the center of $\bG$, we can use the considerations of \cite[Ch. 3.5]{GM} in classical groups $\bG/\spannh =\SO_{2l}(\FF)$. From the end of the description in \cite[Example 3.5.15]{GM} (where $d_0$ is denoted as $e$) we get 
	\begin{align}\label{eq63}
	|\NNN_{\bG}(\bC)^\vFq/
	{ \bC}^\vFq| ={(2d_0)^{a}\cdot a!}  \ \ \und \ \ |\NNN_{K_1}(\bK_1\cap \bC)/
	{(K_1\cap \bC)}|=\frac 12{(2d_0)^{a}\cdot a!}.
	\end{align} 
	
	Note that $\NNN_{\bG}(\bC)^\vFq =\NNN_{\bG}(\bS)^\vFq$ since $\bS$ is the only Sylow $d$-torus of $\Z^\circ(\bC)$ and for the same reason $\NNN_{K_1}(\bK_1\cap \bC)=\NNN_{K_1}(\bS)$. The latter implies in turn that the second equality in (\ref{eq63}) above yields \begin{align}\label{eq64}
		|\NNN_{\bM^\circ}(\bS)^\vFq/
		{ \bC}^\vFq|=|(\NNN_{\bM^\circ}(\bS)/
		{ \bC})^\vFq|=|\left (\NNN_{\bK_1}(\bS)/
		{ (\bK_1\cap\bC})\right)^\vFq|=\frac 12{(2d_0)^{a}\cdot a!}.
	\end{align}
	
 By the theory of Sylow $d$-tori (\cite[Thm~25.11]{MT}) all Sylow $d$-tori of $(\bM^\circ,\vFq )$ are $M^\circ $-conjugate and hence $|\NNN_M(\bS)/\NNN_{M^\circ}(\bS)|=|M:M^\circ|=2$. Combining this with (\ref{eq64})  and the first equality of (\ref{eq63}), we get $|\NNN_{\bM}(\bS)^\vFq|=|\NNN_{\bG}(\bS)^\vFq|$ and therefore $\NNN_{\bG}(\bS)^\vFq =\NNN_{\bM}(\bS)^\vFq \leq M$ as claimed.
 
Let  $\wt{\mathbf{Y}}$ be some $F$-stable maximal torus of $\bM^\circ\Z(\wbG)$ (and hence of $\wbG$) containing $\bS$. We have $\wbG^\vFq =\wt{\mathbf{Y}}^\vFq \bG^\vFq $ and therefore  $\NNN_{\wbG}(\bS)^\vFq=\wt{\mathbf{Y}}^\vFq \NNN_{\bG}(\bS)^\vFq\leq (\wt{\mathbf{Y}} \bM)^\vFq =(\Z(\wt{\bG}) \bM)^\vFq =\wt M$ by the inclusion proved before.
 
By the definition of $\uE(\bM)$, we see easily that $\bK_1$ is $\uE(\bM)$-stable and hence the $K_1$-orbit of $\bS$ is $\uE(\bM)$-stable. Hence 
\[ \NNN_{\bG\UE(\bM)}(\bS)^{\vFq}\leq \NNN_{\bG}(\bS)^{\vFq} K_1 \UE(\bM)\leq M \UE(\bM) .\] 
This implies  $\NNN_{\bG\UE(\bM)}(\bS)^{\vFq}\leq M \UE(\bM)$. 
\end{proof}
%%%%%%%%%%%%%%%%%%%%%%%%%%%%%%%%%%%%%%%%%%%%%%%%%%%%%%
Results from \ref{ssec6_A} concern subgroups of $\GvF$ and for the proof of \Cref{thm:ndreg} we construct an isomorphism between $\GvF$ and $\GF$. Recall that $ \UE(\bG)$ is the group of abstract automorphisms of $\wbG$ generated by $F_p$ and $\gamma$.
\begin{lem}\label{lem_iota}
	Recall $\UE(\GF)$ is the subgroup of $\Aut(\GF)$ obtained by restriction of $\UE(\bG)$. 
	\begin{thmlist}
		\item Assume $(\eps,\eps_1,\eps_2)=(1,-1,-1)$ or equivalently $v\in \bG\setminus\{1_\bG\}$. 
    Let $x\in \bG$ be an element such that $(vF_q )^x=F_q=F$. (Such an element exists according to Lang's Theorem.) Then conjugation with $x$ defines an isomorphism
		\[ \iota: \wbG\rtimes \UE(\bG) \lra \wbG\rtimes \UE(\bG)\ \ \text{ by } \ \  y\mapsto y^x=y[y,x]\]
		such that $\iota(\bG^{\vFq })=\bG^{F}$ and $\iota(\wbG^{\vFq})=\wbG^{F}$. 
		Recall $ \UE(M)= \UE(\GvF)=\spa{F_p,\gamma}\leq \Aut(\wbG^{\vFq})$ from \Cref{EMdef}.
		\begin{enumerate}
			\item[(i)] If $\{|Z_1|,|Z_2|\}\neq \{4\}$, then $\iota(\GvF \UE(\bM))=\GF  \UE(\bG)$ with $\iota( vF_q )=F_q$ inducing an isomorphism $\GvF \UE(M)/\spa{h_0,v F_q }\cong \GF \UE(\GF)/\spannh$.
			\item[(ii)] If $\{|Z_1|,|Z_2|\}=\{4\}$, then the groups of automorphisms of $\GF$ induced by $\iota(\GvF \underline E(M))$ and $\GF \uE(\bG)$ are conjugate by a diagonal automorphism. Moreover $\UE(\GF)$ is cyclic.
		\end{enumerate}
		
		\item Assume $(\eps,\eps_1,\eps_2)=(1,1,1)$ or $(\eps,\eps_1,\eps_2)=(-1,-1,1)$ with $\{|Z_1|,|Z_2|\}\neq \{4\}$. Recall $\UE(M)\leq \Aut(\GvF)$ defined as in the preceding case. The identity map $\iota:\wbG\rtimes \uE(\bG)\lra \wbG\rtimes \uE(\bG)$ induces an isomorphism between $\GF \UE(M)$ and $\GF \UE(\GF)$. 
		
		\item Assume that \Cref{case2} holds so that $\vFq =F$. Let $\wh t \in \bG$ be some element with $F(\wh t)=\wh t h_0$ and recall $\uE(M)=\spa{F_p^2,\wh t \gamma}\leq \bG^F\rtimes \Aut(\bG^F)$ from \Cref{EMdef}. 
		Then the outer automorphism groups of $\GF$ induced by $\UE(M)$ and $\UE(\GF)$ are conjugate by a diagonal automorphism. Moreover, some inner automorphism of $\wGF$ induces an isomorphism between $\GF \UE(M)$ and $\GF \UE(\GF)$.
  %, where $\iota:\wbG\rtimes \underline E(\bG)\lra \wbG\rtimes \underline E(\bG)$ is the identity map.
		%\[\iota(\GF \underline E(M))=\begin{cases}
			%\GF \underline E(\bG)^{\wt t}& \text{, if } \eps=-1 \und \{|Z_1|,|Z_2|\}=\{4\},\\ \GF \underline E(\bG)& \otw .\end{cases}\]
	
 \item In all cases there exists an isomorphism
 \[ \iota': \wbG\rtimes \UE(\bG) \lra \wbG\rtimes \UE(\bG)\ \ \text{ by } \ \  y\mapsto y^{x'}=y[y,x']\]
 induced by conjugation with an element of $x'\in\wbG$, 
 such that 
 \[\iota' (\wbG^\vFq \UE(\bM))=\wbG^F \UE(\bG) ,\ 
 \iota' (\wbG^\vFq )=\wbG^F \und \iota' (\bG^\vFq \UE( \bM))=\bG^F \UE( \bG).\] 
Then one of the following holds: 
\begin{enumerate}
\item[(d.i)] $\iota'$ induces an isomorphism between $\GvF \UE(M)$ and $\GF \UE(\GF)$;
\item[(d.ii)] $\iota'$ induces an isomorphism between $\GvF \UE(M)/\spa{h_0,vF_q}$ and $\GF \UE(\GF)/\spannh$; or
\item[(d.iii)] $\UE(\GF)$ is cyclic. 
\end{enumerate}
\end{thmlist}
\end{lem}

\begin{proof}
	For part (a) we assume $(\eps,\eps_1,\eps_2)=(1,-1,-1)$ and therefore $F=F_q$ and $v\in \bG$. We get clearly $\iota (v F_q )=F_q$ in the semidirect product $\wbG\rtimes \UE(\bG)$ with $\iota(\wbG)=\wbG$, $\iota(\bG^{v F_q})=\GF$ and $\iota(\wbG^{v F_q })=\wbG^{F}$. 
	\Cref{tab:properties_c1} states 
\begin{align}\label{eq6.x} [F_p, v]=1\und [\gamma_1,v]=\begin{cases} 1&\text{if }\{|Z_1|,|Z_2|\}\neq \{4\},\\
		h_0 & \text{if } \{|Z_1|,|Z_2|\}= \{4\}.\end{cases}\end{align}
	As in \cite[Prop. 3.6(b)]{S21D2} or using the Three-Subgroup Lemma \cite[(8.7)]{Asch}, $[F_p, v]=1$ implies $\iota( F_p)\in\GF F_p$. 
	
	In the first case of (\ref{eq6.x}) above where $[\gamma_1,v ]=[\gamma_1,vF_q ]=1$, we then have $[[x,F],\gamma_1]=[[F,\gamma_1],x]=1$ and get analogously $[[\gamma_1,x],F]=1$ or equivalently $\iota( \gamma_1)\in \gamma_1\GF= \GF \gamma_1$. 
	This ensures the statement in (a.i) as follows:
	The group $\UE(M)$ is defined as a subgroup of $\Aut(\bG^{v F_q})$ hence $\UE(M)=\UE(\bM)/\spa{F_q^2}$, since $\eps_1=-1$. Now $\spa{F_q^2,v F_q }=\spa{v F_q , h_0}$ because of $(v F_q )^2=h_0 F_q^2$. Note $\iota(\spa{v F_q ,h_0})=\spa{F_q,h_0}$. This implies then the stated isomorphism. 
	
	In order to get (a.ii) we now assume $(\eps_1,\eps_2)=(-1,-1) \und \{|Z_1|,|Z_2|\}=\{4\}$.
	The latter leads to $2\nmid l_1 l_2$ and $2\nmid f$ according to \Cref{ZG_order}, which then ensures that $\UE(\GF)$ is cyclic. As above $\iota(F_p)\in \GF F_p$. By construction, we also have $\iota(\gamma_1)= \gamma_1g$ for $g:=[\gamma_1,x]\in \bG$.  We get \begin{align}\label{eq6xx}
	\iota(\bG^{v F_q} \underline E(M))=\GF \spa{F_p,\gamma_1 g}.
	\end{align} From $[\gamma_1, vF_q ]=h_0$ in (\ref{eq6.x}) and $\iota(v F_q )=F=F_q$ we obtain 
	\[ h_0=\iota(h_0)=\iota([\gamma_1, vF_q])=[\gamma_1g , F_q]=[g , F_q].\] 
	Therefore, $\iota(\gamma_1)=\gamma_1g$ induces on $\bG^{F_q}$ the product of $\gamma_1$ and a diagonal automorphism of $\bG^{F_q}$ associated with $h_0[\Z(\bG),F]$. Let $\wt t\in \bG$ with $F_q(\wt t)=\wt t \h_\ul(\varpi)$. Since $2\mid l=l_1+l_2$ we have $[F_p,\h_\ul(\varpi)]=1$ by \ref{ZG_order}(c.i) and therefore $[[\wt t,F_q],F_p]=[[F_q,F_p],\wt t]=1$. Then by the Three-Subgroup Lemma again we get $F_p^{\wt t}\in \GF F_p$. This leads to 
\begin{align}\label{eq6y} (\GF\spa{F_p,\gamma_1})^{\wt t}= \GF \spa{F_p,\gamma_1 g}\end{align} since we also have $\wt t^{\gamma_1}\in \GF\wt t g$. The latter can be seen from the action of $\gamma_1$ on the group of diagonal outer automorphisms of $\GF$ corresponding to its action on $\Z(\GF)$, a group of order 4 where only $h_0$ and 1 are fixed (see \ref{ZG_order}(b)).
We now get the statement in (a.ii) from (\ref{eq6xx}) and (\ref{eq6y}).
	
	For part (b) the assumptions imply that $\GvF=\GF$ and $\UE(\GF)=\UE(M)$. 
		Part (c) is clear from \Cref{lem:5_actions_C2}.

Now in part (d) we define $\iota'$ as $\iota$, whenever $\iota(\GvF \UE(\bM))=\GF \UE(\bG)$. It remains to consider the cases of (a.ii) and (c) and define $\iota'$ as the product of $\iota $ with conjugation with some element $t\in \bG$ with $\calL_F(t)=\h_\ul(\varpi)$. Then $\iota'$ satisfies
\[\iota' (\wbG^\vFq \UE(\bM))=\wbG^F \UE(\bG) ,\ 
 \iota' (\wbG^\vFq )=\wbG^F \und \iota' (\bG^\vFq \UE( \bM))=\bG^F \UE( \bG).\] 
In the case of (a.1), $\iota'$ induces an isomorphism
between $\GvF \UE(M)/\spa{h_0,vF_q}$ and $\GF\UE(\GF)/\spannh$ as stated in (d.ii). In case of (a.2), $\UE(\GF)$ is cyclic. 
In case of (b), $\iota'$ is the identity (map) and hence $\GvF \UE(M)$ and $\GF \UE(\GF)$ are isomorphic as stated in (d.i). 
In the case of part (c), $\UE(\GF)$ is also cyclic. 
\end{proof}

In the next step, we can finally verify Conditions \Ad{} and \Bd{}. 
\begin{proof}[Proof of \Cref{thm:ndreg}]
Let $\eps\in\{\pm 1\}$ such that $\GF=\twepsDlq$. According to \Cref{lem6_11} we can choose the groups $M$ and $\bK_1$, as well as $v\in \obG$ and the corresponding $\nu\in \bG\sqcup\{ \gamma \}$, so that Assumption \ref{ass6_2} is satisfied. Let $\bS$ be a Sylow $d$-torus of $(\bK_1,\vFq )$. Then $\bS$ is also a Sylow $d$-torus of $(\bG,\vFq )$ and $\NNN_{\GvF}(\bS)\leq M$, see \Cref{lem6_11}(c). 
	
Let $ \iota': \wbG\rtimes \underline E(\bG) \lra \wbG\rtimes \UE(\bG) $ be the isomorphism of \Cref{lem_iota}(d) with $\iota '(\GvF)=\GF$. Then $\iota'(\bS)$ is a Sylow $d$-torus of $(\bG,F)$ and we can assume $\iota'(\bS)=\bS'$ by $\GF$-conjugacy of Sylow $d$-tori. 
Set $\wt N:=\NNN_{\wbG^{\vFq }}(\bS)$ and $\wt N'=\NNN_\wGF(\bS')=\iota'(\wt N)$.
Analogously, set $\wh N^\flat:=\NNN_{M \UE(\bM) }(\bS)$ and $(\wh N')^\flat=\NNN_{\GF \UE(\bG)}(\bS')$.
According to \Cref{lem6_11}(b), we have $\NNN_{M\UE(\bM)}(\bS)=\NNN_{\bG^\vFq \UE(\bM)}(\bS)$ and hence $(\wh N')^\flat=
\iota' (\wh N^\flat)$.

Let $\TT(N)\subseteq \Irr(N)$ be the set from \Cref{thm6_5}, hence an $\wh N$-stable $\wt N$-transversal. Note that $\wh N$ is a quotient of $\wh N^\flat$. 
The set $\TT(N')$, the image of $\TT(N)$ under $\iota'$ is then a $(\wh N')^\flat$-stable $\wt N'$-transversal in $\Irr(N')$. This shows the first part of (a) of our theorem.

Let $\chi\in\TT(N')$. Then $\chi$ extends to its stabilizer in $\wh N' $ for $\wh N'=\NNN_{\GF \UE(\GF)}(\bS')$, whenever $\UE(\GF)$ and hence $\wh N'/N'$ is cyclic, see \Cref{ExtCrit}(a). If $\wh N'_\chi/N'$ is non-cyclic, then $h_0\in \ker(\chi)$. Then $\iota'$ maps $\wh N/\spa{vF_q,h_0}$ to $\wh N'/\spa{h_0}$. Let $\chi_0\in\TT(N)$ be the character mapped to $\chi$ via $\iota$. Then $\chi_0$ extends to its stabilizer in $\wh N/\spannh$. By the isomorphisms of \Cref{lem_iota}(d.i) or (d.ii), we see that also $\chi$ extends to its stabilizer in $\wh N'$.
	
%If $\iota$ is the identity and hence $\GF=\GvF$, the groups $\GF \UE(M)$ and $\GF \UE(\GF)$ are $\wGF$-conjugate. Hence $\wh N$ and $\wh N'$ are $\wt N$-conjugate. But also otherwise we see that the action of $\iota(\NNN_{\GvF \uE(M)}(\bS))$ and the one of $\NNN_{\GvF \uE(\bG)}(\bS')$ are $\wt N'$-conjugate. 	

Since \maex holds \wrt $N\lhd \wt N$ according to \Cref{cor6_6}, $\iota'(N)=N'$ and $\iota'(\wt N)=\wt N'$, \maex also holds \wrt $N'\lhd \wt N'$. This is (b).

For the part (c) note that
$\NNN_{\wbG^\vFq}(\bS)\leq \wt M$ from \Cref{lem6_11}(b) implies that 
$\wt C'=\iota(\Cent_{\wbG^\vFq}(\bS))=\iota(\wt C)$ and $\wt N'=\iota(\wt N)$, where $\wt C:=\Cent_{\wt M}(\bS)$. 
Using similar applications of $\iota'$, part (c) follows from \Cref{prop_Lambda}.
Analogously, \Cref{prop_requwtLambda} implies part (d).
\end{proof}

%%%%%%%%%%%%%%%%%%%%%%%%%%%%%%%%%%%%%%%%%%%%%%%%%%%%%%%%%%%%%%%%%%%
\subsection{Proof of McKay's equality and Theorem B}
We conclude by drawing the consequences from the above. 
\begin{theorem}\label{thm:iMK_D}
Let $q$ be a power of a prime and let $\ell$ be a prime not dividing $ 2q$. Then (\textbf{iMK}) holds for the quasisimple groups  $\tD_{l,\sico}^{\eps}(q)$ ($l\geq 4, \eps\in\{\pm 1\}$) and $\ell$. 
\end{theorem}
\begin{proof} Let $(\bG,F)$ be as in \ref{not:3_16} such that $\GF=\twepsDlq$. Thanks to \Cref{thm:dreg}(c) we can assume $l\geq 5$. Then
the quotient $\GF/\Z(\GF)$ is simple non-abelian and $\GF$ is its universal covering group, see \cite[Thm 6.1.4]{GLS3}.  

Let $d=d_\ell(q)$ be the order of $q$ in $\FF^\times_\ell$. We can assume $d\geq 3$ by \Cref{iMKbyAdBd} and \Cref{rem_A1A2B1B2}.

Note that if $\mathbf{\Phi}_m$ is the $m$-th cyclotomic polynomial ($m\geq 1$), then $\ell\mid\mathbf{\Phi}_m(q)$ if and only if $m=d\ell^a$ for some $a\geq 0$, see \cite[Lem. 5.2(a)]{MaH0}. Recall that by the definition of the multiplicities $a_{(\bG,F)}(m)$ the following equation holds 
$$|\GF|=q^{l^2-l}\prod_{m\geq 1}^{}\mathbf{\Phi}_m(q)^{{}^{a_{(\bG,F)}(m)}}.$$

If $a_{(\bG,F)}(d)\leq 1$, then as recalled in \Cref{tricho}(ii), $a_{(\bG,F)}(d\ell^a)=0$ for any $a\geq 1$, and therefore a Sylow $\ell$-subgroup of $\GF$ is included in some Sylow $d$-torus of $\bG$. This $d$-torus has rank $a_{(\bG,F)}(d)\leq 1$, so the Sylow $\ell$-subgroups of $\GF$ are cyclic (this also accounts for the case where $\ell\nmid |\GF|$). According to \cite[Thm~1.1]{KoSp} the so-called inductive Alperin--McKay condition holds for every $\ell$-block of $\bG^F$, since such a block has a cyclic defect group. As seen already in the proof of \Cref{thm:dreg} this implies that (\textbf{iMK}) holds for $\GF$ and $\ell$.

We now assume that $a_{(\bG,F)}(d)\geq 2$ and we can check (\textbf{iMK}) by establishing Conditions \Ad{} and \Bd{} from \ref{cond_Ad} and \ref{cond_Bd} thanks to \Cref{iMKbyAdBd}. If $d$ is doubly regular for $(\bG,F)$ the claim follows from \Cref{thm:dreg}.

In the remaining cases, $\ell\nmid q$, $d$ is not doubly regular for $(\bG,F)$ and $a_{(\bG,F)}(d)\geq 2$. By \Cref{tricho}, the assumptions of \Cref{thm:ndreg} are satisfied and according to this result, \Ad\ and \Bd\ do hold for the group $\GF$.
\end{proof}

Theorem A is clearly a consequence of Theorem B, so we concentrate on the latter. 

\begin{proof}[Proof of Theorem B]
Combining the above and \Cref{CS1319} we know that the universal covering group of any finite simple group satisfies (\textbf{iMK}) for any prime. \Cref{RossiMK} then implies that any finite group satisfies (\textbf{iMK}) for any prime. This means that for any finite group $X$, any prime $\ell$ and any Sylow $\ell$-subgroup $S\leq X$, there is a $\Gamma=\Aut(X)_S$-stable subgroup $N$ such that  $\NNN_X(S)\leq N\leq X$ with $N\neq X$ whenever $\NNN_X(S)\neq X$, and a $\Gamma$-equivariant bijection \begin{align}\label{eq6.4}
	\Irrl(X)\to \Irrl(N)\ \ \ \  \text{such that} \ \ \ \ (X\rtimes\Gamma_\chi ,X,\chi)\geq_c (N\rtimes\Gamma_{\chi'} ,N,\chi')
\end{align} 
for each $\chi\mapsto\chi '$ by the above bijection. But Theorem B claims that this is true for $\NNN_X(S)$ in place of $N$. We show this by induction on the index $|X:\NNN_X(S)|$, the case of a normal $S$ being trivial.

By induction one may assume that there is a $\Gamma '=\Aut(N)_S$-equivariant bijection $$ \Irrl(N)\to \Irrl(\NNN_X(S))$$ with $(N\rtimes\Gamma'_{\chi'} ,N,\chi')\geq_c (\NNN_X(S)\rtimes\Gamma'_{\chi ''} ,\NNN_X(S),\chi'')$ for each $\chi'\mapsto\chi ''$. The latter bijection is also $\Gamma$-equivariant since $\Gamma '$ in $\Aut(N)$ contains the image $\Gamma^0$ of $\Gamma$ and the second $\geq_c$ relation implies $(N\rtimes\Gamma^0_{\chi'} ,N,\chi')\geq_c (\NNN_X(S)\rtimes\Gamma^0_{\chi ''} ,\NNN_X(S),\chi'')$ by restriction (\Cref{rem_chartrip}(c)). Then by the Butterfly \Cref{Butterfly} $(N\rtimes\Gamma_{\chi'} ,N,\chi')\geq_c (\NNN_X(S)\rtimes\Gamma_{\chi ''} ,\NNN_X(S),\chi'')$. Combined with the first $\geq_c$ relation (\ref{eq6.4}), it gives by transitivity of the $\geq_c$ relation \[(X\rtimes\Gamma_{\chi} ,N,\chi)\geq_c (\NNN_X(S)\rtimes\Gamma_{\chi ''} ,\NNN_X(S),\chi'')\] 
for each $\chi\mapsto\chi'\mapsto\chi''$. So we get the claim of Theorem B by composition of the two bijections we have.
\end{proof}

%%%%%%%%%%%%%%%%%%%%%%%%%%%%%%%%%%%%%%%%%%%%%%%%%%%%%%%%%%%%%%%%%%%
%%%%%%%%%%%%%%%%%%%%%%%%%%%%%%%%%%%%%%%%%%%%%%%%%%%%%%%%%%%%%%%%%%%

%%%%%%%%%% REFERENCES

\printindex%%%%%%%%%%%%%%%%% INDEX


\begin{thebibliography}{BDT20}
	
\bibitem[A76]{A76} {\sc J.L.~Alperin}, The main problem of block theory. \emph{Proceedings of the Conference on Finite Groups (Univ. Utah, Park City, Utah, 1975)}, Academic Press, New York-London (1976), 341--356.

\bibitem[A87]{A87} {\sc J.L.~Alperin}, Weights for finite groups. \emph{Proc. Sympos. Pure Math. \bf47-1} (1987), 369--379.


\bibitem[Asch]{Asch} {\sc M.~Aschbacher}, \emph{Finite Group Theory}. Cambridge University Press, Cambridge, 2000.



\bibitem[B]{CedricSL2} {\sc C.~Bonnaf{\'e}}, {\em Representations of $\SL_2(\FF_q)$}. Springer, London, 2011. 
% 
\bibitem[B06]{Cedric} {\sc C.~Bonnaf{\'e}}, {\em Sur les Caract\`eres des Groupes R\'eductifs Finis \`a Centre Non Connexe: Applications aux Groupes Sp\'eciaux Lin\'eaires et Unitaires}. \emph{Ast\'erisque} {\bf306} (2006).
%



\bibitem[BM92]{BM92} {\sc M. Brou\'e and G.~Malle}, Théorèmes de Sylow génériques pour les groupes réductifs sur les corps finis. \emph{Math. Ann. \bf 292} (1992), 241 --262.

\bibitem[BMM93]{BMM93} {\sc M. Brou\'e, G.~Malle and J. Michel}, Generic blocks of finite reductive groups.
\emph{Astérisque  \bf212} (1993), 7--92. 
%
\bibitem[BM97]{BrMi}
{\sc M. Brou\'e and J.  Michel}, Sur certains \'{e}l\'{e}ments r\'{e}guliers des groupes de {W}eyl et les
vari\'{e}t\'{e}s de {D}eligne-{L}usztig associ\'{e}es.
\emph{Progr. Math. \bf141}, Birkhäuser, Boston, (1997), 73--139. 


\bibitem[Br90a]{Br90a} {\sc M. Brou\'e}, Isométries parfaites, types de blocs, catégories dérivées. \emph{Astérisque \bf181--182} (1990), 61--92. 


\bibitem[Br90b]{Br90b} {\sc M. Brou\'e},  Isométries de caractères et équivalences de Morita ou dérivées.
\emph{Publ. Math. Inst. Hautes Études Sci. \bf 71} (1990), 45--63.


\bibitem[Br06]{Br06} {\sc M. Brou\'e}, On blocks of finite reductive groups. \emph{Oberwolfach Report \bf15} (2006), 969--974. 



\bibitem[Br]{BroueBook} {\sc M. Brou\'e},  {\it Introduction to Complex Reflection Groups and their Braid Groups.} Lecture Notes in Mathematics \textbf{1988}, Springer, Berlin, 2010.
%


\bibitem[Ca94]{Ca94} {\sc M.
Cabanes}, Unicité du sous-groupe abélien distingué maximal dans certains sous-groupes de Sylow.
\emph{C. R. Acad. Sci. Paris \bf318} (1994), 889--894.

%%
%% \bibitem[CE99]{CE99}
%%{\sc M.~Cabanes and M.~Enguehard}, On blocks of finite reductive groups and twisted induction. 
%%\emph{Adv. Math. \bf 145}, 189–229 (1999).

 \bibitem[CE99]{CE99} {\sc M.~Cabanes and M.~Enguehard}, On fusion in unipotent blocks. \emph{Bull. London Math. Soc. \bf31} (1999), 143--148.

 \bibitem[CE]{CE04}  {\sc M.~Cabanes and M.~Enguehard}, \emph{Representation Theory of Finite  Reductive Groups}.  Cambridge University Press, Cambridge, 2004.
 
\bibitem[CS13]{CS13} {\sc M.~Cabanes and B.~Sp\"ath}, Equivariance and extendibility in finite reductive groups with connected center. \emph{Math. Z. \bf275} (2013), 689--713.

\bibitem[CS17a]{CS17A} {\sc M.~Cabanes and B.~Sp\"ath}, Equivariant character correspondences and inductive McKay condition for type A. 
\emph{J. Reine Angew. Math. \bf 728 } (2017), 153--194.
 
\bibitem[CS17b]{CS17C} {\sc M.~Cabanes and B.~Sp\"ath}, The inductive McKay condition for type  $\tC$. \emph{Represent. Theory \bf21} (2017), 61--81.

\bibitem[CS19]{CS18B}
{\sc M.~Cabanes and B.~Sp\"ath}, Descent {e}qualities and the inductive {M}c{K}ay condition for types  $\tB$ and  $\tE$. {\it Adv. Math. \bf 356},  (2019), Paper No. 106820, 41 pp.    


\bibitem[CS24]{CS22}
{\sc M.~Cabanes and B.~Sp\"ath}, On semisimple classes and component groups in type  $\tD$. \emph{ Vietnam J. Math. \bf52-2} (2024), 435--444.    
%


\bibitem[C]{Ca85} {\sc R.W.~Carter}, \emph{Finite Groups of Lie Type. Conjugacy classes and complex characters}. Wiley, New York, 1985.


\bibitem[ChR08]{ChR08}
{\sc J.~Chuang and R.~Rouquier,}
Derived equivalences for symmetric groups and $\frak{sl}_2$-categorifications.
 \emph{Ann. of Math. (2) \textbf{167}}  (2008), 245--298.





\bibitem[CoNo79]{CoNo79} {\sc J. Conway and S. Norton}, Monstrous Moonshine. \emph{Bull. Lond. Math. Soc. \bf11} (1979), 308-- 339.

\bibitem[CR13]{CR13}
{\sc C.~Craven and R.~Rouquier,}
Perverse equivalences and Broué's conjecture.
\emph{Adv. Math. \bf248} (2013), 1--58.

\bibitem[Da92]{Da92} {\sc E.C. Dade,} Counting characters in blocks. I.
\newblock \emph{Invent. Math. \bf109 } (1992), 187--210.

%


\bibitem[DM20]{DiMi2}
{\sc F.~Digne and J.~Michel}, \newblock \emph{Representations of Finite Groups of Lie Type.} 2nd Edition.
\newblock Cambridge University Press, Cambridge, 2020.

%

\bibitem[FeSp23]{FS23} {\sc Z. Feng and B. Sp\"ath,} Unitriangular basic sets, Brauer characters and coprime actions. \emph{Represent. Theory \bf27} (2023), 115--148.

\bibitem[FoSr82]{FS82} {\sc P.~Fong and B. Srinivasan,} The blocks of finite general linear and unitary groups. \emph{Invent. Math. \bf69} (1982), 109--153. 

\bibitem[FoSr86]{FS86} {\sc P.~Fong and B. Srinivasan,} Generalized Harish-Chandra theory for unipotent characters of finite classical groups. \emph{J. Algebra \bf 104} (1986), 301--309.  
%
\bibitem[FoSr89]{FS89} {\sc P.~Fong and B. Srinivasan}, The blocks of finite classical groups. \emph{J. Reine Angew. Math. \bf396} (1989), 122--191.

%


\bibitem[GM]{GM} {\sc M.~Geck and G.~Malle}, \emph{The Character Theory of Finite Groups of Lie Type: A Guided Tour}. Cambridge University Press, Cambridge, 2020.

\bibitem[GLS]{GLS3}
{\sc D.~Gorenstein, R.~Lyons and R.~Solomon}, \emph{The Classification of the Finite Simple Groups. Number 3}. American Mathematical Society, Providence RI, 1998.



%\bibitem[GLL76]{GLL76} {\sc J.A. Green, G.I. Lehrer and G. Lusztig.} On the degrees of certain group characters. \emph{Quart. J. Math. Oxford Ser. (2) \bf27} (1976), 1--4.



\bibitem[Is73]{Isa73}
{\sc I.M.~Isaacs}, Characters of solvable and symplectic groups. \emph{Amer. J. Math. \bf85} (1973), 594--635.

\bibitem[Is]{Isa}
{\sc I.M.~Isaacs}, \emph{Character Theory of Finite Groups}. Academic Press, New York, 1976.


\bibitem[IMN07]{IMN} {\sc I.M.~Isaacs, G.~Malle and G.~Navarro}, A reduction theorem for the McKay conjecture. \emph{Invent. Math. \bf170} (2007), 33--101.


\bibitem[IN02]{IN_Annals} 
{\sc I. M. Isaacs and G.~Navarro},
New refinements of the McKay conjecture for  arbitrary finite groups, 
\emph{Ann. of Math. (2) \textbf{156}} (2002), 333--344.


%\

%

\bibitem[KM19]{KM19} {\sc R. Kessar and G. Malle,} Local-global conjectures and blocks of simple groups. Groups St Andrews 2017 in Birmingham, \emph{ London Math. Soc. Lecture Note Ser.} \textbf{455}, Cambridge University Press, Cambridge, (2019), 70--105.






\bibitem[KnRo89]{KnRo89}  {\sc R. Kn\"orr and G. R. Robinson}, Some remarks on a conjecture of Alperin. \emph{J. London Math. Soc. (2) \bf39} (1989), 48--60.

\bibitem[KS16]{KoSp}{\sc S. Koshitani and B. Sp\"ath}, The inductive Alperin-McKay and blockwise Alperin weight conditions for blocks with cyclic defect groups and odd primes. \emph{ J. Group Theory \bf 19} (2016), 777--813.


\bibitem[L78]{Lehrer}
{\sc  G.I.~Lehrer}, On a conjecture of Alperin and McKay. {\it Math. Scand. \bf43}  no. 1, (1978/79), 5--10. 



\bibitem[Mac71]{Mac71}  {\sc I. Macdonald}, On the degrees of the irreducible representations of the symmetric groups. \emph{Bull. London Math. Soc. \bf3} (1971), 189--192.
%
%

\bibitem[Ma07]{MaH0}
{\sc G.~Malle}, Height 0 characters of finite groups of Lie type. \emph{Represent. Theory \bf11} (2007), 192--220.
%
\bibitem[Ma08]{ManonLie} {\sc G.~Malle}, The inductive McKay condition for simple groups not of Lie type. \emph{Comm. Algebra \bf36} (2008), 455--463.

%


\bibitem[Ma17a]{Ma17a} {\sc G.~Malle}, Cuspidal characters and automorphisms. {\it Adv. Math. {\bf 320}} (2017), 887--903.

\bibitem[Ma17b]{Ma17b} {\sc G.~Malle}, Local-global conjectures in the representation theory of finite groups.
\emph{EMS Ser. Congr. Rep.}
(2017), European Mathematical Society, Zürich, 519--539.

\bibitem[MS16]{MS16} {\sc G.~Malle and B.~Sp\"ath}, Characters of odd degree. {\it Ann. of Math. (2) \bf 184} (2016), 869--908.

\bibitem[MT]{MT} {\sc G. Malle and D. Testerman}, \emph{Linear Algebraic Groups and Finite Groups of Lie Type}. Cambridge University Press, Cambridge, 2011.

\bibitem[MK71]{McK71}  {\sc J.~McKay}, A new invariant for finite simple groups. \emph{ Notices of the AMS}, Vol 18, issue no 128, (1971),  p. 397.

\bibitem[MK72]{McK} {\sc J.~McKay}, Irreducible representations of odd degree. \emph{J. Algebra \bf20} (1972), 416--418.

\bibitem[MK80]{McK80}  {\sc J.~McKay}, Graphs, singularities, and finite groups. \textit{The Santa Cruz Conference on Finite Groups}, \emph{Proc. Sympos. Pure Math. \bf37} (1980), 183--186.




\bibitem[N04]{N_Annals}{\sc G.~Navarro}, The McKay conjecture and Galois automorphisms, \emph{
Ann. of Math. (2) \bf160} (2004), 1129--1140.

\bibitem[N]{Navarro_book}{\sc G.~Navarro}, \emph{Character Theory and the McKay Conjecture}. Cambridge University Press, Cambridge, 2018. 

\bibitem[NS14]{JEMS_NS}
{\sc G.~Navarro and B.~Sp\"ath}, On Brauer's height zero conjecture. {\it J. Eur. Math. Soc. \bf 16} (2014),  695--747.

\bibitem[NT13]{NavarroTiep_Annals_pprime}
{\sc G.~Navarro and P.H.~Tiep},
Characters of relative $p'$-degree over normal subgroups. 
\emph{Ann. of Math. (2) \bf178} (2013), 1135--1171.


\bibitem[Ok00]{Ok00}
{\sc T. Okuyama,}
Remarks on splendid tilting complexes.
Representation theory of finite groups and related topics (Kyoto, 1998)
\textit{Sūrikaisekikenkyūsho Kōkyūroku} , no. \textbf{1149} (2000), 53--59.

\bibitem[O76]{Olsson}
{\sc  J.B.~Olsson}, McKay numbers and heights of characters. {\it Math. Scand. \bf38} (1976), 25--42. 

\bibitem[Ro23a]{Rossi_McKay}
{\sc D.~Rossi}, The McKay Conjecture and central isomorphic character triples. {\it J. Algebra \bf618}  (2023), 42--55.

\bibitem[Ro23b]{Ro23b}
{\sc D.~Rossi}, The Brown complex in non-defining characteristic and applications. ArXiv:2303.13973 (2023), 34 p.



\bibitem[Rou23]{Rou23} {\sc R. Rouquier},  Modular representations of finite groups and Lie theory. \emph{J. Algebra \bf656} (2024), 446--485.


\bibitem[Ru22]{Ru22a}
{\sc L.~Ruhstorfer}, Jordan decomposition for the Alperin-McKay conjecture.
\emph{Adv. Math. \bf394} (2022), Paper No. 108031, 30 pp.

\bibitem[Ru24]{Ru22b}
{\sc L.~Ruhstorfer}, The Alperin–McKay and Brauer’s height zero conjecture for the prime 2. \emph{Ann. of Math. (2) \bf 201} (2025), 379--457.

%

\bibitem[S07]{S07}{\sc B.~Sp{\"a}th},
\newblock \textit{Die McKay-Vermutung f\"ur quasi-einfache Gruppen vom Lie-Typ.}\newblock Dissertation, TU Kaiserslautern (2007), \url{https://kluedo.ub.uni-kl.de/files/1844/Spaeth_McKay.pdf}.

\bibitem[S09]{S09} {\sc B.~Sp{\"a}th}, \newblock The {M}c{K}ay conjecture for exceptional groups and odd primes.
\newblock {\em Math. Z.} {\bf 261} (2009), 571--595.

\bibitem[S10a]{S10a}
{\sc B.~Sp{\"a}th},
\newblock Sylow {$d$}-tori of classical groups and the {M}c{K}ay conjecture. {I}.
\newblock {\em J. Algebra} {\bf 323} (2010), 2469--2493.

\bibitem[S10b]{S10b} 
{\sc B.~Sp{\"a}th},
\newblock Sylow {$d$}-tori of classical groups and the {M}c{K}ay conjecture. {II}.
\newblock {\em J. Algebra} {\bf 323} (2010), 2494--2509.

\bibitem[S12]{S12}{\sc B.~Sp{\"a}th},
\newblock Inductive {M}c{K}ay condition in defining characteristic.
\newblock {\em Bull. Lond. Math. Soc.} {\bf 44} (2012), 426--438.

%
%%
\bibitem[S17]{S17}{\sc B.~Sp{\"a}th},
\newblock A reduction theorem for Dade's projective conjecture. 
\newblock {\em J. Eur. Math. Soc. (JEMS) \bf 19} (2017), 1071--1126.

\bibitem[S18]{S18}{\sc B.~Sp{\"a}th}, Reduction theorems for some global-local conjectures. {\it  Local representation theory and simple groups,}
23--61, EMS Ser. Lect. Math., Eur. Math. Soc., Zürich, 2018. 

\bibitem[S23a]{S21D1}{\sc B.~Sp{\"a}th},
\newblock Extensions of characters in type $\tD$ and the inductive McKay condition, I. 
\newblock {\it Nagoya Math. J. \bf 252} (2023), 906--958.

\bibitem[S23b]{S21D2}{\sc B.~Sp{\"a}th}, \newblock  Extensions of characters in type $\tD$ and the inductive McKay condition, II.
\newblock \emph{Preprint} arXiv:2304.07373, (2023), 62 p.

\bibitem[Sp70]{Sp70}{\sc T. A.~Springer},                                                                                                                 
Cusp forms for finite groups.\emph{ Lecture Notes in Math.,}  \textbf{131} Springer (1970), 97--120.

\bibitem[Sp74]{Sp74}{\sc T. A.~Springer}, 
\newblock  Regular elements of finite reflection groups. 
\newblock {\em Invent. Math.} { \bf25} (1974),  159--198.

\bibitem[Ta18]{Ta18}{\sc J. Taylor}, 
\newblock Action of automorphisms on irreducible characters of symplectic groups.
\newblock {\em  J. Algebra} {\bf  505} (2018), 211--246.

\bibitem[Ti14]{Ti14} {\sc P.H. Tiep,} Representation of finite groups: conjectures, reductions, and applications.
\emph{Acta Math. Vietnam. \bf39} (2014), 87--109.

\bibitem[Tits66]{T} {\sc J.~Tits}, 
Normalisateurs de tores. I. Groupes de Coxeter étendus. \emph{J. Algebra \bf4} (1966), 96--116.

\bibitem[Tu08]{Tu08} {\sc A. Turull},
Strengthening the McKay conjecture to include local fields and local Schur indices.
\emph{J. Algebra \bf319} (2008), 4853--4868.

\bibitem[Uno04]{Uno04} {\sc K. Uno},
Conjectures on character degrees for the simple Thompson group.
\emph{Osaka J. Math. \bf41} (2004), 11--36. 

%                             
\bibitem[Wo78]{Wo78}{\sc T.R.~Wolf}, Characters of $p'$--degree in solvable groups.  \emph{Pacific J. Math. \bf74} (1978), 267--271. 




\end{thebibliography}
\end{document}